\documentclass[11pt,twoside]{amsart}
\usepackage{float}
\usepackage[utf8]{inputenc}
\usepackage[english]{babel}
\usepackage[T1]{fontenc}
\usepackage{mathtools}
\usepackage{enumitem}

\usepackage{tabularx}
\usepackage{graphicx}
\usepackage{tikz}
\usepackage[all]{xy}
\usepackage{tcolorbox}
\usepackage{amsmath,
    amsfonts,
    amssymb,
    mathrsfs,
    stackrel}
\usepackage{comment}
\usepackage{bookmark}
\usepackage{hyperref}
\usepackage{cleveref}
\hypersetup{colorlinks,
    citecolor=blue,
    linktocpage}

\usepackage[normalem]{ulem}

\DeclareFontFamily{U}{wncy}{}
\DeclareFontShape{U}{wncy}{m}{n}{<->wncyr10}{}
\DeclareSymbolFont{mcy}{U}{wncy}{m}{n}
\DeclareMathSymbol{\Sha}{\mathord}{mcy}{"58} 
\newcommand\Ash{\rotatebox[origin=c]{180}{$\Sha$}}

\newcommand\cA{{\mathcal A}}
\newcommand\cB{{\mathcal B}}
\newcommand\cC{{\mathcal C}}
\newcommand\cD{{\mathcal D}}
\newcommand\cE{{\mathcal E}}
\newcommand\cF{{\mathcal F}}
\newcommand\cG{{\mathcal G}}
\newcommand\cH{{\mathcal H}}
\newcommand\cI{{\mathcal I}}
\newcommand\cJ{{\mathcal J}}

\newcommand\cL{{\mathcal L}}
\newcommand\cM{{\mathcal M}}
\newcommand\cN{{\mathcal N}}
\newcommand\cO{{\mathcal O}}
\newcommand\cP{{\mathcal P}}
\newcommand\cQ{{\mathcal Q}}
\newcommand\cR{{\mathcal R}}
\newcommand\cS{{\mathcal S}}
\newcommand\cT{{\mathcal T}}

\newcommand\cV{{\mathcal V}}

\newcommand\cX{{\mathcal X}}
\newcommand\cY{{\mathcal Y}}
\newcommand\cZ{{\mathcal Z}}

\newcommand\fF{{\mathfrak F}}

\newcommand\fU{{\mathfrak U}}

\newcommand\bC{{\mathbb C}}
\newcommand\bD{{\mathbb D}}

\newcommand\bF{{\mathbb F}}
\newcommand\bG{{\mathbb G}}

\newcommand\bK{{\mathbb K}}

\newcommand\bM{{\mathbb M}}
\newcommand\bN{{\mathbb N}}

\newcommand\bP{{\mathbb P}}
\newcommand\bQ{{\mathbb Q}}
\newcommand\bR{{\mathbb R}}
\newcommand\bS{{\mathbb S}}

\newcommand\bU{{\mathbb U}}
\newcommand\bV{{\mathbb V}}

\newcommand\bZ{{\mathbb Z}}

\newcommand\hX{{\widehat X}}
\newcommand\hY{{\widehat Y}}

\newcommand\wE{{\widetilde E}}

\newcommand\wU{{\widetilde U}}

\newcommand\wW{{\widetilde W}}
\newcommand\wX{{\widetilde X}}
\newcommand\wY{{\widetilde Y}}
\newcommand\wZ{{\widetilde Z}}

\newcommand\oA{{\overline A}}

\newcommand\oF{{\overline F}}
\newcommand\oG{{\overline G}}

\newcommand\oK{{\overline K}}

\newcommand\oV{{\overline V}}

\newcommand\oX{{\overline X}}
\newcommand\oY{{\overline Y}}
\newcommand\oZ{{\overline Z}}

\newcommand\bfD{{\bf D}}

\newcommand\bfM{{\bf M}}

\newcommand\bfZ{{\bf Z}}

\DeclareMathOperator{\Aut}{Aut}

\DeclareMathOperator{\codim}{codim}

\DeclareMathOperator{\coker}{coker}

\DeclareMathOperator{\GL}{GL}

\DeclareMathOperator{\mult}{mult}

\DeclareMathOperator{\rk}{rk}

\DeclareMathOperator{\Sp}{Sp}

\DeclareMathOperator{\Spec}{Spec}

\DeclareMathOperator{\Supp}{Supp}

\DeclareMathOperator{\vol}{vol}

\DeclareMathOperator{\pr}{pr}
\DeclareMathOperator{\lct}{lct}

\newtheorem{theorem}{Theorem}[section]

\newtheorem{conjecture}[theorem]{Conjecture}
\newtheorem{corollary}[theorem]{Corollary}
\newtheorem{lemma}[theorem]{Lemma}

\newtheorem{proposition}[theorem]{Proposition}

\newtheorem{maintheorem}{Theorem}

\theoremstyle{definition}

\newtheorem{construction}[theorem]{Construction}

\newtheorem{definition}[theorem]{Definition}
\newtheorem{example}[theorem]{Example}

\newtheorem{notation}[theorem]{Notation}

\newtheorem{question}[theorem]{Question}
\newtheorem{remark}[theorem]{Remark}

\newtheorem{warning}[theorem]{Warning}

\newcommand{\twopartdef}[4]
{
	\left\{
		\begin{array}{ll}
			#1 & \mbox{if } #2 \\
			#3 & \mbox{if } #4
		\end{array}
	\right.
}

\newcommand{\twobytwo}[4]
{
	\begin{pmatrix}
		#1 & #2 \\
		#3 & #4
	\end{pmatrix}
}

\newcommand{\supp}{\operatorname{Supp}}

\newcommand{\NN}{\mathbb{N}}

\newcommand{\ZZ}{\mathbb{Z}}
\newcommand{\QQ}{\mathbb{Q}}
\newcommand{\RR}{\mathbb{R}}
\newcommand{\CC}{\mathbb{C}}

\newcommand{\PP}{\mathbb{P}}

\newcommand{\Deflt}{\mathrm{Def}^{\mathrm{lt}}}

\newcommand{\DefltL}{\mathcal{D}^{\ell}}
\newcommand{\Deflti}{\mathcal{D}(i)} 
\newcommand{\Xreg}{X^{\mathrm{reg}}}
\newcommand{\Xalb}{X^{{\rm Alb}}}

\def\O#1.{\mathcal {O}_{#1}}			
\def\pr #1.{\mathbb P^{#1}}						
\def\xrar#1.{\xrightarrow{#1}}			
\def\K#1.{K_{#1}}								
\def\bM#1.{\mathbf{M}_{#1}}						
\def\ses#1.#2.#3.{0\to #1\to #2\to #3 \to 0}

\newcommand{\rar}{\rightarrow}	
\newcommand{\drar}{\dashrightarrow}	

\newcommand{\Def}{\mathrm{Def}}

\title{Boundedness of some fibered $K$-trivial varieties}
\author[Engel]{Philip Engel}
\address{Department of Mathematics, Statistics, 
and Computer Science, University of Illinois in Chicago (UIC),
851 S Morgan St, Chicago, IL 60607, USA}
\email{pengel@uic.edu}
\author[Filipazzi]{Stefano Filipazzi}
\address{Department of Mathematics, Duke University,
120 Science Drive,
117 Physics Building,
Campus Box 90320,
Durham, NC 27708-0320,
USA}
\email{stefano.filipazzi@duke.edu}
\author[Greer]{Fran\c{c}ois Greer}
\address{Department of Mathematics, 
Michigan State University, 619 Red Cedar Rd, 
East Lansing, MI 48824, USA}
\email{greerfra@msu.edu}
\author[Mauri]{Mirko Mauri}
\address{Institut de Math\'ematiques 
de Jussieu-Paris Rive Gauche,
Universit\'e Paris Cit\'e,
Place Aur\'elie Nemours, 75013 Paris, France}
\email{mauri@imj-prg.fr}
\author[Svaldi]{Roberto Svaldi}
\address{
Dipartimento di Matematica 
'Federigo Enriques', Universit\`a 
degli Studi di Milano,
Via Cesare Saldini 50,
20133 Milan, Italy}
\email{roberto.svaldi@unimi.it}

\date{\today}

\begin{document}

\begin{abstract} 
We prove that irreducible
Calabi--Yau varieties of a fixed dimension,
admitting a fibration 
by abelian varieties or primitive symplectic
varieties of a fixed analytic deformation class,
are birationally bounded. 
We prove that there are only finitely
many deformation classes of primitive 
symplectic varieties of a fixed dimension, 
admitting a Lagrangian fibration.
We also show that fibered
Calabi--Yau $3$-folds are bounded.
Conditional on the generalized abundance or hyperk\"ahler
SYZ conjecture, our results prove that there are only
finitely many deformation classes of
hyperk\"ahler varieties, of a fixed dimension,
with $b_2\geq 5$.
\end{abstract}

\setcounter{tocdepth}{2}
\maketitle

\tableofcontents

\section{Introduction}

Projective varieties 
with trivial canonical class (for short, $K$-trivial 
varieties) play a crucial role in the birational 
classification of algebraic varieties. 
In this paper, we explore several boundedness 
phenomena concerning $K$-trivial varieties. 
A central question of the field 
is:

\begin{question}\label{quest:bound}
Are there finitely many deformation classes of 
projective $K$-trivial varieties
of a fixed dimension $d$?
\end{question}

In dimension $d=1$, the answer is affirmative, as elliptic 
curves have connected, 1-dimensional moduli.
Already in dimension $d=2$, some subtleties appear. It is true that any two smooth projective K3 or abelian 
surfaces $X_1$ and $X_2$
are deformation equivalent---there is a 
smooth projective morphism $\cX = \cX(X_1, X_2) 
\to \cT$
over a connected, quasi-projective base $\cT$, containing 
$X_1$ and $X_2$ as fibers. But there is no 
{\it boundedness} in the algebraic category: 
indeed, there exists no finite type family 
$\cX\to \cT$ of projective varieties, whose
fibers contain all K3 or abelian surfaces.
Starting from dimension $d\geq 3$,
Question \ref{quest:bound} remains largely open.

It is 
difficult to get a handle on the general 
$K$-trivial variety $X$. Historically, the most progress
has been made in the presence of some additional structure,
especially a {\it fibration} $f\colon X\to Y$; 
that is,
a proper morphism of normal varieties, 
with connected fibers, 
whose base satisfies $0<\dim Y<\dim X$. 
When the fibers are abelian or symplectic, 
we control the birational geometry of $X$ via the period
map from $Y$ to the relevant period space.
We refer to Definitions
\ref{def:CY} and \ref{defn:bounded}
for rigorous
notions of $K$-trivial varieties 
and the particular types of boundedness under 
consideration. 

\begin{maintheorem}\label{thm:icy-ab}
Irreducible Calabi--Yau varieties, 
of fixed dimension $d\geq 3$, fibered 
in 
\begin{enumerate}
\item[{\rm (AV)}] abelian varieties, or 
\item[{\rm (PS)}] a fixed
deformation class of 
primitive symplectic variety,
\end{enumerate}
are birationally 
bounded.
\end{maintheorem}

\begin{corollary}\label{thm:cy3}
Fibered Calabi--Yau $3$-folds are bounded.
\end{corollary}

Theorem \ref{thm:icy-ab} was 
known for elliptic Calabi--Yau $3$-folds,
by work in the 1990's 
of Gross \cite{Gross1994, Gross97} and
Dolgachev--Gross \cite{Gross92}, and 
for elliptic Calabi--Yau $d$-folds with 
rational section, by recent work
of Birkar, di Cerbo, and Svaldi \cite{BDCS}.
Corollary \ref{thm:cy3} was proven recently 
for elliptic Calabi--Yau $3$-folds by
Filipazzi, Hacon, and Svaldi \cite{FHS2024}. 
Even birational boundedness
of Calabi--Yau $3$-folds fibered by abelian or K3 surfaces was open until now.

\begin{maintheorem}\label{thm:ps}
Lagrangian-fibered primitive symplectic varieties,
of fixed dimension $2d$,
lie within
a finite number of locally trivial\footnote{A deformation
 is {\it locally trivial} if every point has
 an analytic neighborhood which is a trivial
 deformation. It is an equisingularity condition, local on the source space and not on the target of the deformation.}
deformation equivalence classes
of primitive symplectic variety.
\end{maintheorem}

Note that,
despite Theorem \ref{thm:ps}, Lagrangian-fibered
primitive symplectic varieties do not form
a bounded family, just as for K3 surfaces. 
The corollaries of
Theorem \ref{thm:ps} are particularly
strong. According to the hyperk\"{a}hler SYZ 
conjecture (\ref{SYZ conjecture}), any 
primitive symplectic variety with $b_2 \geq 5$
is deformation equivalent to a primitive 
symplectic variety with a Lagrangian fibration.
\begin{corollary} If the hyperk\"{a}hler SYZ conjecture (\ref{SYZ conjecture})
 is true,
the number of locally trivial
deformation classes of primitive symplectic
variety, of fixed dimension $2d$, 
with $b_2\geq 5$, is finite.
\end{corollary}

In brief, the proof strategy for the main theorems is 
as follows: first, we bound the base $Y$ via the canonical 
bundle formula; second, we bound the polarization type 
of the general fiber by using the Zarhin trick and the Kuga--Satake 
construction in families. In the case of an abelian 
fibrations, these steps yield a bound on the Albanese 
fibration, or equivalently, the classifying map to a 
moduli space of abelian varieties. In order to reconstruct 
the original fibration out of the Albanese, we invoke and 
refine the Tate--Shafarevich theory. 
We refer to \S~\ref{outline} 
for a detailed outline of the proof. 

\subsection{Previous work}
By the Beauville--Bogomolov decomposition theorem
(Thm.~\ref{thm:decomp}), any
$K$-trivial variety
is built 
from three fundamental building blocks:
abelian varieties, irreducible Calabi--Yau varieties, and irreducible symplectic varieties (which are special primitive symplectic varieties).  Roughly, the moduli space
is completely understood in the abelian case, and the only remaining
problem in the symplectic case concerns the finiteness and/or possible classification
of analytic deformation classes. On the other hand, 
less is understood about the moduli or Hodge theory 
of Calabi--Yau varieties.

We
briefly survey 
past results on 
Question~\ref{quest:bound}:
\medskip

{\bf Calabi--Yau varieties.}
 In this article, Calabi--Yau varieties are 
$K$-trivial varieties with canonical singularities and 
$h^i(X,\cO_X)=0$ for $0<i<\dim X$ (Def.~\ref{def:CY}). 
There are at least 30108 known 
distinct topological types of smooth 
Calabi--Yau 3-folds, compiled
by Kreuzer--Skarke \cite{KSCYdata},
though figures as large as 
$\approx 10^{400}$ have been proposed 
\cite{chandra}. 
The question of whether all smooth Calabi-Yau $3$-folds form a bounded family is a hard unsolved problem; 
Reid and Friedman studied interconnections
between moduli spaces of Calabi--Yau varieties 
via conifold transitions
\cite{reid, friedman91}, and
Wilson studied topological invariants
\cite{Wilson2021, Wilson2022} in special cases. 

Matsusaka's big theorem \cite{matsusaka, matsusaka2},
resp.~work of 
Martinelli, Schreieder, and Tasin \cite{MST20},
implies boundedness
under the hypothesis that there is an ample,
resp.~big and nef, line bundle of bounded volume.
The main progress on unconditional
boundedness results has been 
the elliptically-fibered $3$-fold case, 
discussed above. 
The analysis of the known examples suggests that smooth 
Calabi--Yau $3$-folds with 
large Picard rank are elliptic, cf.
\cite[p.~3]{wati} and \cite{Gross1994};
in view of Corollary~\ref{thm:cy3}, this would imply 
that the second Betti number 
of $K$-trivial $3$-folds is bounded. 
Hunt \cite{hunt} claims a bound
on the Euler characteristic of fibered
Calabi--Yau $3$-folds, but the proof
appears to be incomplete; see \cite[Rem.~1.3]{ohno}.

\medskip

{\bf Symplectic varieties.}
Irreducible symplectic manifolds, 
or compact hyperk\"{a}hler manifolds, are simply connected,
smooth $K$-trivial varieties 
admitting a holomorphic symplectic
form which is unique up to scaling. 
Beauville described the 
two known infinite series \cite[\S~6, 7]{B1983},
of type 
K3$^{[n]}$ and Kum$_n$, each of complex
dimension $2n$.
Besides these, only two additional 
sporadic examples are known: 
OG6 and OG10,
constructed by O'Grady \cite{OG6, OG10}, 
in dimensions $6$ and $10$ respectively. 
Alexeev observed
that there is at least
one {\it primitive symplectic} variety
(see Def.~\ref{def:CY}(PS)) for every
absolutely irreducible representation of 
a finite group $G$ on a lattice $L$ 
\cite[Thm.~5]{alexeev2022root}. 

It is unknown whether there
are finitely many deformation
classes of compact hyperk\"ahler manifolds in fixed dimension. A birational version of Matsusaka's big
theorem, due to Charles \cite{charles}, proved that
symplectic varieties of fixed dimension
with a line bundle of bounded positive top 
self-intersection are birationally bounded. 
By work of Huybrechts and Kamenova
\cite{Huy03, Kam2018}, there are only finitely
many deformation classes of hyperk\"ahler manifolds
with fixed Fujiki constant and given disciminant of the
Beauville--Bogomolov--Fujiki (BBF) lattice. In \cite{Sawon2016}, 
Sawon uses period maps to bound Lagrangian-fibered 
hyperk\"{a}hler manifolds, 
under strong assumptions 
on the fibration (fixed base, bound on the 
generic polarization, semistability in codimension 1, 
and existence of a section); 
see also \cite{Sawon2008, Kamenova2024}.

Some explicit topological
bounds on primitive symplectic varieties are known:
the second Betti number of a $4$-dimensional 
primitive symplectic variety with quotient 
singularities is at most 23 by 
\cite{guan.betti, Fu-Menet}; 
moreover, for a smooth hyperk\"ahler
4-fold, $b_2\leq 8$ 
or $b_2=23$ by \cite{guan.betti}. 
It is conjectured that $b_2\leq 23$ 
also for a smooth primitive symplectic 6-fold; 
see \cite{Kurnosov2016, Sawon2022} 
for partial results. 
Note that at the moment, there are only two known 
examples of primitive symplectic varieties with 
$b_2(X) \geq 24$: OG10 and a recent
example by Liu--Liu--Xu of a singular primitive 
symplectic variety of dimension 42; cf.~\cite[Rem.~4.17]{LLX2024}.
A conjectural bound for the second Betti number 
of primitive symplectic manifolds is 
proposed by Kim and Laza \cite{KimLaza2020}.

\subsection{Additional results} 
In the course of the proof, 
we obtain a number of interesting
results concerning boundedness and 
the minimal model program. The first
is a form of the effective b-semiampleness 
conjecture of Prokhorov and Shokurov
\cite[Conj.~7.13]{PS09},
extending results of Fujino \cite{Fuj03}:

\begin{maintheorem}\label{thm:b-semi}
Let $f \colon X \rar Y$ be a fibration
whose general fibers are either 
abelian varieties of dimension $g$, 
or primitive symplectic varieties of second
Betti number $b_2$.
Let $(Y, B, \bfM)$ be the generalized 
pair induced on $Y$. Then, there exists a constant $I$, 
depending only on $g$ or $b_2$ respectively, 
such that $I\bfM$ is free for a 
suitable birational modification $Y' \to Y$.
\end{maintheorem}

Next is a general result
concerning the relative boundedness
of spaces of abelian fibrations, which
is not particular
to an assumption of $K$-triviality. 
We say that a fibration $f\colon X\to Y$
has {\it Picard type} if any flat 
family of 
cohomology classes of degree $2$ on the smooth fibers 
is of $(1,1)$-type (Def.~\ref{def:pic-type}). 

\begin{maintheorem}\label{thm:bound-ab}
Let $(\cY,\cB,\cM)/\cT$ be a finite type
family of projective varieties $\cY\to \cT$,
with two families of $\bQ$-divisors $\cB$, $\cM$. 
Abelian fibrations 
$f\colon X\to \cY_t$ for some closed point 
$t\in \cT$, of fixed dimension, 
of Picard type, admitting a rational section, and 
inducing $(\cY_t,\cB_t,\cM_t)$,
are birationally bounded. 
\end{maintheorem}

We also have the following result
on the birational class of the relative Albanese fibration,
generalizing results of 
\cite[Cor.~29]{grassi-wen},  \cite[Cor.~6.4]{sacca},
\cite[\S~6]{kollar_new2}:

\begin{maintheorem}\label{thm:alb}
Let $f\colon X\to Y$ be an abelian fibration of a $K$-trivial
variety $X$. Then there is a birational model 
$f^{\rm Alb}\colon X^{\rm Alb}\to Y$ of its relative Albanese
variety over $Y$, for which $X^{\rm Alb}$ is $K$-trivial.
Furthermore, $f$ and $f^{\rm Alb}$ induce the same 
generalized pair on the base $Y$.

Additionally, if $X$ is (primitive)
symplectic, then $X^{\rm Alb}$ is (primitive) symplectic, and if 
$X$ is Calabi--Yau, then $X^{\rm Alb}$ is Calabi--Yau.
\end{maintheorem}

We prove the following topological result,
concerning the smooth locus of an irreducible 
symplectic variety (Def.~\ref{def:CY}(IS)).

\begin{maintheorem}\label{thm:pi1}
Let $X$ be an irreducible symplectic variety
admitting a Lagrangian fibration. 
Then $\pi_1(X^{\rm reg})$ is finite.
\end{maintheorem}

We also prove a structure
theorem concerning the birational classes of abelian 
torsors; see Theorem 
\ref{identify-N1-2} 
for a precise formulation. 
To prove this result,
we build up the theory of the Tate--Shafarevich 
group in the presence of multiple fibers
(see \S~\ref{sec:ts}).

\subsection{Outline}\label{outline}
We now summarize the contents 
of the paper in more detail.

\subsubsection*{Section \ref{sec:prelim}}
Here we introduce the fundamental
notions of the paper, such as $K$-trivial, Calabi--Yau, and
primitive symplectic varieties (Def.~\ref{def:CY}), 
b-divisors and generalized pairs (\S~\ref{b-div.subs}), 
boundedness, and birational boundedness notions 
for fibrations (\S~\ref{sec:bd-def}). We also state 
the  {\it canonical bundle formula} 
(Thm.~\ref{base-same-sing}), which is of central importance
to the paper. On the base of a relatively $K$-trivial
fibration $f\colon X\to Y$ of a klt variety $X$, we naturally
have the structure of a {\it generalized pair} $(Y,B,\bfM)$ 
(cf.~Def.~\ref{lc-trivial.def}),
consisting of an effective $\bQ$-divisor $B_Y$ (Constr.~\ref{coefficientB}) measuring
singularities of $f$ over codimension $1$ points of $Y$, 
and a $\bQ$-b-divisor class $\bfM$ (cf.~\S~\ref{hodge-moduli}), which is the class
of the Hodge $\bQ$-line bundle. There is an equivalence
$$K_X\sim_\bQ f^\ast (K_Y+B_Y+\bfM_Y).$$

\subsubsection*{Section \ref{sec:hodge}}
We discuss the theory of the period mapping for
abelian fibrations (\S~\ref{av-fib}) and primitive
symplectic fibrations (\S~\ref{sec:moduli-ps}). 
To every fibration $f\colon X\to Y$
of abelian $g$-folds whose fibers are polarized by
an ample class of type ${\bf d}$, we associate
a {\it classifying morphism} 
\begin{align*} \Phi^o\colon Y^o &\to \cA_{g,{\bf d}} \\
y&\mapsto {\rm Aut}^0(X_y)
\end{align*}
from the locus $Y^o\subset Y$ over which $f$ is smooth,
to the DM stack of abelian varieties with a 
distinguished origin. Passing to the coarse space,
there is a {\it period map} 
$\Phi\colon Y\dashrightarrow \oA_{g,{\bf d}}$ to the Baily--Borel
compactification (Def.~\ref{def:classifying-morphism}).
An important and subtle point
in the paper is that {\it these two pieces of data
are not equivalent}; $\Phi^o$ contains more information
(cf.~Ex.~\ref{no-B-ex}).

We introduce the central notion of abelian fibrations of
Picard type (Def.~\ref{def:pic-type}). 
Abelian fibrations of Calabi--Yau varieties
and Lagrangian fibrations both fall into this class.
We improve on the existing literature about 
lattice-polarized abelian varieties and their moduli
(\S~\ref{sec:lattice_av}).

Then, we introduce a key tool of the 
paper---the {\it Zarhin trick}. This is a morphism 
$\cA_{g,{\bf d}}\to \cA_{8g}$
which assigns to any ${\bf d}$-polarized abelian $g$-fold 
a {\it principally polarized} abelian $8g$-fold (\S~\ref{sec:Ztrick}).
It is not so surprising that the Zarhin trick should
arise in bounding families of abelian varieties.
For instance, it appears in Zarhin's 
extension \cite[Thm.~1]{zarhin85} of Faltings' proof 
\cite[Thm.~6]{faltings} (for fixed polarization type
${\bf d}$) of the {\it Shafarevich conjecture},
that  there are finitely many abelian 
varieties over $\bQ$ with
good reduction outside a finite set of primes.

We also review the {\it Kuga--Satake} construction (\S~\ref{sec:kuga}) 
$\cF_\Lambda\to \cA_{g,{\bf d}}$
which assigns to a polarized Hodge structure of weight 
$2$ (for example, the period of a primitive 
symplectic variety), a polarized Hodge
structure of weight $1$. We clarify some aspects
of the functoriality of 
this construction, and adapt it
to a setting more suitable for
proving boundedness results.

Finally, we prove Theorem \ref{thm:b-semi} 
on effective $b$-semiampleness 
(Thm.~\ref{thm:weakeffective}), by bounding
the multiple of the Hodge bundle which is
very ample on the target of the period mapping, 
relying on Hirzebuch--Mumford proportionality and 
the fact that the Baily--Borel compactification 
has log canonical singularities. 

\subsubsection*{Section \ref{sec:bd}}
Here we prove boundedness results for abelian fibrations
with a rational section---these are easier to classify,
as they are pulled back from the universal abelian variety
over $\cA_{g,{\bf d}}$ along the classifying morphism
$\Phi^o$. 
The Zarhin trick plays a crucial role.

We fix a finite type family
of bases $(\cY,\cB, \cM)/\cT$ as in the statement
of Theorem \ref{thm:bound-ab}, and consider
all abelian fibrations $X\to \cY_t$ whose
boundary and moduli divisors are $\cB_t$ and $\cM_t$
for some closed point $t\in \cT$.
By an easy volume argument, our first result 
(Prop.~\ref{bd-map}), is that the composition
of the period map to $\oA_{g,{\bf d}}$
with the Zarhin trick to $\oA_{8g}$ lies
in a bounded family of rational
maps $\cY_t\dashrightarrow \oA_{8g}$.
We lift these to the stack $\cA_{8g}$
on an appropriate open subset 
$\cY_t^o\subset \cY_t$ (Prop.~\ref{bd-stack}),
and prove that 
it is possible to ``undo'' the Zarhin trick
(Prop.~\ref{bd-unzarhin}) on the level of $\bZ$-local
systems.
So the $\bZ^{16g}$-local
system on $\cY_t^o$ pulled back from
$\cA_{8g}$ determines, up to finite ambiguity,
the $\bZ^{2g}$-local system pulled back
from $\cA_{g,{\bf d}}$ (Lem.~\ref{local-sys-undo}).
 Under the {\it Picard type} hypothesis, 
this local system recovers
a polarization type ${\bf d}$ on
the general fiber. But
without the Picard type hypothesis,
this recovery procedure fails (Ex.~\ref{pic-ex}).

We deduce Theorem \ref{thm:bound-ab}, which can
be viewed as a relative boundedness result
over a finite type family of bases,
for abelian fibrations with section. A crucial
result following from the existing MMP
literature \cite{BDCS} is 
that the effective $b$-semiampleness 
(Thm.~\ref{thm:b-semi}), together with the hypothesis
that the total space is
irreducible Calabi--Yau, provides boundedness
of the base in codimension $1$ (Lem.~\ref{cor:bases}).
Via the Kuga--Satake construction, we deduce
Theorem \ref{thm:icy-ab}(PS), and weaker versions
of Theorems \ref{thm:icy-ab}(AV) and \ref{thm:ps}, 
which assume $f\colon X\to Y$ has a rational section.

More importantly, we birationally bound 
the {\it Albanese fibration} 
$f^{\rm Alb}\colon X^{\rm Alb}\to Y$
of any abelian fibration $f\colon X\to Y$
of an irreducible Calabi--Yau or primitive
symplectic variety---this is
the pullback of the universal abelian variety
over $\cA_{g,{\bf d}}$ along the classifying
morphism $\Phi^o$. Alternatively, $f^{\rm Alb}$
is the datum of the variation of weight $1$
Hodge structure induced by $f$.

\subsubsection*{Section \ref{sec:ts}}
Over the generic point of $Y$, an abelian fibration $f \colon X \to Y$ is a torsor over the Albanese fibration
$f^{\rm Alb}\colon X^{\rm Alb} \to Y$.
To deduce the full forms
of Theorems \ref{thm:icy-ab}(AV) and \ref{thm:ps}
from \S~\ref{sec:bd} requires an in-depth 
understanding of the {\it Tate--Shafarevich group}: 
A group which parameterizes these abelian torsors.
Our approach is largely inspired by 
\cite{Gross92, Gross97,
AR2021, abasheva, sacca, dutta, yoonjoo}. 
Many of these papers  have
imposed the restrictive hypothesis 
that $f\colon X\to Y$ is Lagrangian and/or
without multiple fibers, conditions which
we remove.

Thus, we begin a technical study, which occupies
nearly half of the present paper. For a reader wishing
to skip this study, the final upshot is that
the group parameterizing the relevant
abelian torsors is:
\begin{enumerate}
\item[(i)] finite and varies constructibly, for abelian fibrations
of irreducible Calabi--Yau varieties.
\item[(ii)] an extension 
of a finite, constructibly-varying
group by a quotient of $\bC$, 
for abelian fibrations 
of primitive symplectic varieties.
\end{enumerate} 
The continuous group $\bC$
parameterizes special locally
trivial analytic deformations.

Our broad strategy
is thus: 
\begin{enumerate} 
\item\label{step1}
We improve the singularities
of the fibration
by passing to a Galois branched cover $Z\to Y$, with Galois group $G$.
Specifically, we may produce a branched cover
for which the base change of $X^{\rm Alb}$ 
admits a $K$-trivial semistable reduction over 
all codimension $1$ points in $Z$.
\item\label{step2}
By ensuring that $Z\to Y$ ramifies appropriately
over the support of the boundary divisor $B\subset Y$, 
we can ensure that there is a big\footnote{A Zariski open set 
is said to be {\it big} if its 
complement has codimension $\geq 2$.}
open set 
$Z^+\subset Z$ over which the base changes of $X$ 
and $X^{\rm Alb}$ are \'etale-locally birational.
\item\label{step3} 
By Hartogs-type extension results,
it suffices to understand how $X$ is a twist
of $X^{\rm Alb}$ in the analytic topology,
over the open set $Z^+$. This gives access
to the exponential exact sequence for the relevant
group schemes, providing a link to topology.
\item\label{step4} 
The multiple fibers of $f\colon X\to Y$
are encoded by the failure of $f$
and $f^{\rm Alb}$ to be \'etale-locally birational,
or equivalently by the failure of their base changes
to be $G$-equivariantly \'etale-locally birational.
\end{enumerate}

Section \ref{local-sys-undo} begins (Prop.~\ref{nicest-model})
by producing the finite branched cover $Z\to Y$, 
semistable model as in (\ref{step1}), and big
open set $Z^+\subset Z$ as in (\ref{step2}),
via the Mumford construction 
\cite{mumford-construction, alexeev-nakamura}. 
These relations
are encoded in the all-important diagram
\begin{align*}
\xymatrix{V^+\ar[rd]_{h^+} &
W^+ \ar@{<-->}[l]_{\textrm{\'et-loc}} \ar[d]^{g^+} \ar[r]^{/G}
& X^+\ar[d]^{f^+} \\
& Z^+ \ar[r]^{/G} & Y^+
}
\end{align*}
where $g^+\colon W^+\to Z^+$ is the normalized
base change of the restriction $f^+\colon X^+\to Y^+$
of $f$ to an appropriate big open $Y^+\subset Y$,
and $h^+\colon V^+\to Z^+$ is the relatively
$K$-trivial, semistable (i.e.~Kulikov) model of $(g^+)^{\rm Alb}$.

The smooth locus of the morphism $h^+$
is a $G$-equivariant semiabelian group scheme
$\cP\to Z^+$ (Def.~\ref{def:p}), which like $f^{\rm Alb}$,
can be viewed as the datum of the variation
of mixed Hodge structure of $f$, or
more specifically Deligne's
$1$-motives \cite[Sec.~10]{deligne1974},
extending over codimension $1$ points. 
Then our variant
of the Tate--Shafarevich group is the 
$G$-equivariant \'etale cohomology
$$\Sha_G 
\coloneqq H^1_G(Z^+,\cP).$$

Central to analyzing this group is the exponential
exact sequence (in the analytic topology)
\begin{align*}
0\to \Gamma
\to \mathfrak{p}
\xrightarrow{\rm exp} 
\cP^0\to 1,\end{align*}
and its associated long exact sequence;
here $\cP^0\subset \cP$ is the fiberwise
identity component
and $\mathfrak{p}$ is the Lie algebra of $\cP$.
Using Hartogs-type results and 
the interplay between the \'etale and analytic topology,
our first non-trivial result (Prop.~\ref{local-no-twist})
is to show that the open set $Z^+\subset Z$ of
(\ref{step2}) can be chosen independently of $f$.
Here and elsewhere, the critical tool is that $\Sha_G$
is torsion (Prop.~\ref{torsion-lemma-1});  
see Raynaud \cite{ray70}.

An innovation which we expect will be useful
in future studies is the introduction of the
{\it multiplicity group} $\Ash^G$ (Def.~\ref{def:pre-mult},
Def.~\ref{mult-alg}, Prop.~\ref{alg-mult}) which
measures structure of the multiple fibers of $f$ 
in codimension $1$, by a {\it multiplicity
class} $m(f)\in \Ash^G$. The class $m(f)$
obstructs $f\colon X\to Y$ from being a twist
of its Albanese fibration, in the usual sense.
This is encoded by an exact sequence 
$$0\to \Sha\to \Sha_G\to \Ash^G\to 0.$$

Section \ref{sec:big-good} is dedicated to proving
that the analytification map $\Sha_G\to 
\Sha_{G,\,{\rm an}}$ is injective (Prop.~\ref{sha-inj})
and surjective onto torsion (Thm.~\ref{ray-mult}),
justifying our restriction to the big open set
$Z^+$ in (\ref{step3}). 
Analyzing the exponential exact sequence 
(\S~\ref{ts3}) and repeatedly applying torsion-ness,
we come to the exact sequence
$$0 \to \bQ N^1/ N^1 \to
H^1_{G}(Z^+,\cP^0)  \to  N^2\to 0$$
where $N^2$ is a finite group, 
$N^1$ is a free abelian group of finite
rank, and 
$\bC N^1\supset \bQ N^1\supset N^1$
is a finite-dimensional vector space
identified with $H^2(X,\cO_X)/H^2(Y,\cO_Y)$ 
(Thm.~\ref{identify-N1-2}). A number of technical
tools are required to achieve this, such as 
the running of an MMP to produce models
of $X$ for which $X^+\subset X$ is 
a big open set (Prop.~\ref{contract-f-exc}), 
the use of intersection cohomology to relate the 
cohomologies of $X$ and $X^+$ (Prop.~\ref{h2-surjection}),
and Koll\'ar's vanishing results (Prop.~\ref{ss-degen}),
see \cite{Ko86:KollarHigherI, Kol86}.
One concludes 
$\bC N^1 = 0$ when $X$ is irreducible Calabi--Yau
and $\bC N^1 = \bC$ when $X$ is primitive symplectic.

This completes our study of Tate--Shafarevich
twists of a fixed fibration $f\colon X\to Y$.

\subsubsection*{Section \ref{sec:alb}}

We analyze here the geometry of the relative Albanese
variety. In \S~\ref{bir-alb}, 
we prove Theorem \ref{thm:alb},
again crucially relying on the fact that $\Sha_G$ is torsion
to show that $X$ admits a generically finite morphism
to $X^{\rm Alb}$ which, roughly, is multiplication
by $n$ on the fibers (Prop.~\ref{prop.map.alb}).
This provides a method to transport holomorphic forms 
from $X$ to $X^{\rm Alb}$.

In \S~\ref{sec:ho}, the Albanese fibration 
of a Lagrangian fibration 
is examined explicitly
(Prop.~\ref{magic-table}),
and we describe behavior of this 
fibration in codimension $1$,
with respect to the Hwang--Oguiso 
classification (Thm.~\ref{ho-thm})
of multiple fibers of Lagrangian 
fibrations \cite{HO2009, HO2011}.

\subsubsection*{Section \ref{sec:ts2}}

We complete in this section the 
proofs of Theorems \ref{thm:icy-ab}(AV)
and \ref{thm:ps}, by removing the hypothesis that 
the fibration $f\colon X\to Y$ has a rational section.
To do so, we must analyze how the 
Tate--Shafarevich group $\Sha_G$ varies
in families. 

We begin by showing in
\S~\ref{sec:constructibility} that 
the non-divisible part of the group $\Sha_G$
varies in a constructible manner---this
is rather subtle,
since $\Sha_G$
itself need not vary constructibly.
Because $\Sha_G$ is finite for
abelian-fibered irreducible Calabi--Yau varieties, it equals
its own non-divisible part.
We deduce Theorem \ref{thm:icy-ab}(AV).

It remains to analyze the case where 
$f\colon X\to Y$ is a Lagrangian fibration
of a primitive symplectic variety. Here,
the divisible group $\bQ N^1/N^1$
sits inside a corresponding group
of analytic Tate--Shafarevich twists
$\bC N^1/N^1\subset \Sha_{G,\,{\rm an}}$.
We employ the Lagrangian structure
to prove that the classes in 
$\bC N^1/N^1$ extend from $Y^+$ to the
entire base $Y$, following \cite{AR2021}. 
This allows us to 
interpolate the countably many algebraic
twists in $\bQ N^1/N^1$ by a single, proper
analytic family over $\bC \simeq \bC N^1$. 
We deduce Theorem \ref{thm:ps}.

\subsubsection*{Section \ref{sec:bir-bd-to-bd}}
Following \cite{FHS2024}, we argue that boundedness
of the base implies boundedness of the total space,
at least in the abelian- and K3-fibered setting. We deduce
Corollary \ref{thm:cy3}, and a result 
(Cor.~\ref{cor:fano-base}) that abelian-fibered
irreducible Calabi--Yau varieties over a Fano base
are bounded.

\subsubsection*{Section \ref{sec:pi1}}

In this section, we prove Theorem \ref{thm:pi1}, that the fundamental group 
 of the regular locus of an 
 irreducible symplectic variety admitting a Lagrangian fibration is finite.
The key ingredient is an  orbifold version 
(Prop.~\ref{prop:exact sequence}) 
of Nori's exact sequence of homotopy groups for a 
fibration \cite[Lem.~1.5]{Nori1983}, 
see also Campana's work \cite[Prop.~11.7]{Campana2011}.  
This provides the exact sequence
$$
\mathbb Z ^{\dim X} \to \pi_1(X^{\rm reg}) 
\to \pi_1^{\rm orb}(Y,\Delta_f) \to 1
$$
where $\pi_1^{\rm orb}(Y,\Delta_f)$ is the orbifold 
fundamental group of the multiplicity pair 
$(Y,\Delta_f)$ (Def.~\ref{defn:orbifoldpi1}).
Since $X$ is irreducible symplectic, it follows that 
$(Y,\Delta_f)$ is a pair of Fano type (Prop.~\ref{prop:basepair}).
In turn, by \cite{Braun2021}, $\pi_1^{\rm orb}(Y,\Delta_f)$ is finite.
Thus, $\pi_1(X^{\rm reg})$ is virtually abelian.
Lastly, we utilize the irreducibility of $X$ and the identity 
$H^1(\widetilde X^{\rm reg},\QQ)=H^1(\widetilde X,\QQ)=0$ for 
every quasi-\'etale cover $\widetilde X \to X$ 
(Lem.~\ref{lem:virtuallyabelianfinite}) to deduce that 
$\pi_1(X^{\rm reg})$ is indeed finite.

\smallskip

\subsubsection*{Appendix \ref{sec:deformation}}

Finally, in the appendix, we show that a Lagrangian
fibration of a primitive 
symplectic variety deforms (see Conjecture \ref{gen-abund}) over the locus where the pullback of
an ample line bundle from the base remains 
in the Picard group. As a result, Theorem \ref{thm:ps} also
holds analytically.

\smallskip

\subsection*{Acknowledgements} 
We would like to thank 
Valery Alexeev, Ben Bakker, Dori Bejleri,
Harold Blum, Andrea Di Lorenzo,
Yajnaseni Dutta, Osamu Fujino,
Cécile Gachet,
Mark Gross, Andres Herrero, Moritz Hartlieb, 
Daniel Huybrechts,
Giovanni Inchiostro,
Yoon-Joo Kim, János Kollár, Radu Laza, Zsolt Patakfalvi,
Vasily Rogov, Stefan Schreieder, Calum Spicer, 
Salim Tayou, and Claire Voisin 
for useful mathematical
discussions.
A special thanks goes to Gabriele Di Cerbo for 
many discussions and ideas that contributed 
to starting this project.
We are also grateful to Lie Fu, Antonella Grassi, Ljudmila Kamenova, Christian Lehn, and Joaquín Moraga 
for valuable comments on an earlier version of this paper. 

We thank EPFL, the Hausdorff Institute for 
Mathematics (HIM) in Bonn, and the American Institute 
of Mathematics (AIM) for providing a 
supportive research environment. In particular, the project advanced considerably during the Junior Trimester Program \emph{Algebraic geometry:~derived categories, Hodge theory, 
and Chow groups} at HIM, and the workshop \emph{Higher-dimensional log Calabi--Yau pairs} at AIM.

PE was supported by NSF grant DMS-2401104.
SF was supported by ERC starting grant \#804334.
FG was supported by NSF grant DMS-2302548.
MM was supported by \'{E}cole Polytechnique, and 
Université Paris Cité and Sorbonne Université, 
CNRS, IMJ-PRG, F-75013 Paris, France. 
RS was supported by the ``Programma
per giovani ricercatori Rita Levi Montalcini'' of MUR and by PSR 2022 – Linea 4 of the University of Milan. He is a member of the
GNSAGA group of INDAM. 

\section{Preliminary material}\label{sec:prelim}

Throughout the paper, we work over the  field of complex numbers 
$\bC$.

\subsection{{\it K}-trivial varieties}
Let $X$ be a projective variety. 
We denote by 
$\Omega_X^{[p]}=(\bigwedge^p \Omega^1_X)^{\ast \ast}$
the reflexive hull of the 
$p$-th
exterior product
of the sheaf of K\"ahler differentials. 
If $X$ is normal and $j \colon X^{\mathrm{reg}} 
\hookrightarrow X$ is the inclusion of the regular locus, 
then  $\Omega_X^{[p]}=j_* \Omega_{X^{\mathrm{reg}}}^{p}$.
If further, $X$ has rational singularities, 
then $\Omega_X^{[p]}=\pi_* \Omega_{\widetilde{X}}^{p}$ 
for any resolution of singularities 
$\pi \colon \widetilde{X} \to X$, see 
\cite[Cor.~1.8]{KS2021}. 
When 
$p=\dim X$, 
$K_X$ denotes any Weil divisor such that 
$\mathcal O_X(K_X) \simeq \Omega^{[\dim X]}_X$.

Recall that a {\it quasi-\'etale} cover 
is a finite surjective 
morphism that is \'{e}tale 
in codimension 1.

\begin{definition}
\label{def:CY} 
A \emph{$K$-trivial 
variety} is a normal projective 
variety $X$ with canonical
singularities such that
$K_X\sim 0$.
We say a $K$-trivial variety $X$ is
\begin{enumerate}
    \item[(CY)]  \noindent {\it Calabi--Yau} if 
    $H^0(X,\Omega_X^{[k]})=0$ for all $0<k<\dim X$;
    \item[(ICY)] {\it irreducible Calabi--Yau} if 
    all  quasi-\'etale covers $X'\to X$ are Calabi--Yau;
    \item[(PS)] {\it primitive symplectic} if 
    $H^0(X,\Omega^{[1]}_X)=0$ and $H^0(X,\Omega^{[2]}_X)=\bC\sigma$ for 
    some $2$-form $\sigma$ which is non-degenerate
    on $X^{\rm reg}$; 
    \item[(IS)] {\it irreducible symplectic} if 
     all  quasi-\'etale covers $X'\to X$ are primitive symplectic;  
    \item[(SIS)]  {\it strict irreducible symplectic} if 
    $X$ is irreducible symplectic, and the algebraic fundamental group of its regular locus $X^{\mathrm{reg}}$ is trivial, i.e., $\hat{\pi}_1(X^{\mathrm{reg}})=1$; and
    \item[(AV)] an {\it abelian variety} if
    $\dim H^0(X,\Omega_X^{[1]})=\dim X$.
\end{enumerate}
We say $X$ is \emph{$K$-torsion} 
if it is a normal projective 
variety 
$X$ 
with klt
singularities such that
$c K_X\sim 0$ 
for some 
positive integer $c$.
More generally, the same definitions can be
made for compact K\"{a}hler spaces; 
see \cite[Sec.~3, p.~346]{Grauert1962} 
or \cite[Def.~1]{BGL2022}.
\end{definition}

In the algebraic category, $K_X$ being torsion 
is equivalent to $K_X\equiv 0$; see \cite{Nak04}.

The singular analogue of the 
Beauville--Bogomolov theorem \cite{B1983}, 
proved in increasing degrees of generality in 
\cite{GrebKebekusKovacsPeternell, DruelGuenancia, 
Druel2018, Guenancia2016, GrebGuenanciaKebekus, 
HoringPeternell, Campana2021, BGL2022}, 
asserts the following decomposition property 
for varieties (even K\"ahler spaces) with numerically
trivial canonical bundle.

\begin{theorem}\label{thm:decomp}
Any compact K\"{a}hler space with numerically trivial 
canonical class and klt singularities admits a 
quasi-\'{e}tale cover which can be decomposed into a 
product of complex tori, irreducible Calabi--Yau 
varieties, and strict irreducible symplectic varieties.
\end{theorem}

Hence, the building blocks of $K$-trivial varieties are 
the three classes 
(AV), (ICY), (SIS)
of Definition \ref{def:CY}.

\begin{remark}
There is a great deal variation in the literature
regarding these definitions. 

Since $X$ has a global algebraic volume form, 
$X$ has canonical singularities
if and only if it has rational singularities, 
see \cite[Th\'{e}or\`{e}me 1]{Elkik1981}. 
So in Definition \ref{def:CY}, one may replace canonical 
with rational singularities.

The definitions of 
primitive and irreducible symplectic varieties stabilized
after the work of Bakker and Lehn \cite{BL2022}. 
A reflexive 2-form 
$\sigma \in  
H^0(X,\Omega^{[2]}_X)=
H^0(X^{\mathrm{reg}},\Omega^{2}_{X^{\mathrm{reg}}})$ 
on a normal variety  $X$ 
is \emph{symplectic} if its restriction to the regular 
locus is closed non-degenerate. 
When $X$ is projective with rational singularities, 
then any algebraic 2-form on 
$X^{\mathrm{reg}}$ 
is automatically closed, see \cite[Thm.~1.13]{KS2021}. 
Therefore, in Definition \ref{def:CY}(PS) the 2-form 
$\sigma$ is always symplectic. 
Let us recall also that if 
$X$
has rational singularities, any symplectic 
form on $X^{\mathrm{reg}}$ extends to a regular 2-form 
on any resolution of $X$ by \cite[Cor.~1.8]{KS2021}. 
So Definition \ref{def:CY}(PS) is equivalent 
 to \cite[Def.~3.1]{BL2022}. 

If $X$ is IS, then the algebra 
$H^0(X',\Omega^{[*]}_{X'})$ is generated by a 
symplectic form $\sigma\in H^0(X',\Omega^{[2]}_{X'})$.
By Theorem \ref{thm:decomp}, Definition \ref{def:CY}(IS) 
is equivalent to \cite[Def.~8.16]{GKP2016}. In particular, 
smooth IS varieties are hyperk\"ahler. 
\end{remark}

\begin{remark}
The word {\it irreducible} in Definition
\ref{def:CY} aims to signal that any Beauville--Bogomolov 
decomposition of $X$ admits a unique factor; such decomposition 
is unique if $X$ is strict, i.e., it does not admit further 
\'{e}tale covers $X' \to X$. 
In general, 
there is no distinguished or maximal 
choice for  the quasi-\'{e}tale cover in the 
Beauville--Bogomolov  decomposition (e.g., a complex 
torus does not have it), 
but it does exist for irreducible
symplectic varieties. Indeed, any IS variety is 
covered by a SIS variety, by the results of 
Greb--Guenancia--Kebekus
\cite[Cor.~13.3, Cor.~13.2]{GrebGuenanciaKebekus} 
and Campana \cite[Cor.~5.3]{Campana1995}.
It is not known if such a cover exists for irreducible 
Calabi--Yau varieties, but if it does, it should be 
called strict Calabi--Yau.
For instance, a smooth, simply connected 
Calabi--Yau variety is strict. 
\end{remark}

\begin{proposition}[Matsushita]\label{prop:base} 
Let $f \colon X \to Y$ be a fibration from a 
primitive symplectic variety $X$ 
with $\dim X=2g$. Then
\begin{enumerate}
\item $f$ is a Lagrangian fibration---$f$ 
is equidimensional of relative 
dimension $g$, and the 
symplectic form vanishes on the regular locus of any 
irreducible component of any fiber of $f$, 
endowed with its reduced structure;
\item the general fiber of $f$ is a 
projective abelian variety; and 
\item \label{item:base} $Y$ is a Fano variety 
with $\QQ$-factorial klt singularities 
and Picard number 1.
\end{enumerate}

\end{proposition}
\begin{proof}
See \cite[Thm.~2]{Matsushita}, 
\cite[Thm.~1]{Matsushita2000}, 
\cite[Thm.~3.1]{Matsushita2015}, 
\cite[Thm.~3]{Schwald2020}\footnote{In
{\it loc.cit.}~the definitions of 
irreducible and primitive symplectic varieties 
are essentially opposite
ours.}, and \cite[Thm. 2.10]{KL2022}.
\end{proof}

By \cite[Thm.~1.1(3)]{BL2022}, 
any primitive symplectic variety $X$ of dimension $2d$
with $b_2(X) \geq 5$ deforms in a locally trivial manner 
to a primitive  symplectic variety $X'$ admitting 
a non-trivial nef line bundle $L$ with $c_1(L)^{d+1}=0$.
Conjecturally \cite{LP20}, $L$ should be semiample:

\begin{conjecture}[Generalized abundance 
for symplectic varieties]\label{gen-abund}
Let $X$ be a primitive symplectic variety, 
and let $L$ be a nef line bundle on $X$.
Then $L$ is semiample.
\end{conjecture}

Taking the linear system $|mL|$ for large $m$ as the fibration in Proposition \ref{prop:base},
\Cref{gen-abund} would imply the 
Strominger--Yau--Zaslov conjecture for $b_2 \geq 5$.

\begin{conjecture}[SYZ conjecture] \label{SYZ conjecture}
Every primitive symplectic variety $X$, with $b_2 \geq 5$,
admits a locally trivial deformation to a primitive 
symplectic variety with a Lagrangian fibration.
\end{conjecture}

\subsection{b-Divisors and generalized pairs}
\label{b-div.subs}

We refer to \cite{KM98} for the standard 
terminology in birational geometry.

\begin{definition}
Let $\mathbb{K}$ denote $\ZZ$, $\QQ$, or $\RR$.
Given a normal variety $Y$, 
a {\it $\mathbb{K}$-b-divisor} 
$\mathbf{D}$ is a (possibly infinite) sum 
of geometric valuations $\nu_i$ 
of $k(Y)$ with coefficients in 
$\mathbb{K}$, 
\begin{align*}
 \mathbf{D}= \sum_{i \in I} b_i \nu_i, \; b_i \in \mathbb{K},
\end{align*}
such that, given any normal variety 
$Y'$ birational to $Y$, 
only finitely many valuations 
$\nu_{i}$ have a center of codimension 1 on 
$Y'$. The {\it trace} $\mathbf{D}_{Y'}$ 
of $\mathbf{D}$ on $Y'$ is the 
$\mathbb{K}$-Weil divisor 
\begin{align*}
\mathbf{D}_{Y'} 
 \coloneqq  \sum b_i D_i
\end{align*}
where the sum is indexed over valuations $\nu_{i}$ 
that have divisorial center $D_i\subset Y'$.
Geometrically, 
a $\bK$-b-divisor $\mathbf{D}$ is an element
of the inverse system of $\bK$-Weil divisors
on all such $Y'$:
$$\bfD\in \lim_{Y' \to Y}\, \bK{\rm Div}(Y').$$
\end{definition}

Given a $\mathbb{K}$-b-divisor $\mathbf{D}$ 
over $X$, we say that $\mathbf{D}$ 
is a {\it $\mathbb{K}$-b-Cartier} 
if there exists a birational model 
$Y'$ of $Y$ such that $\mathbf{D}_{Y'}$ 
is $\mathbb K$-Cartier on $Y'$ and for any model 
$r \colon Y''  \rar Y'$, we have
$\mathbf{D}_{Y''} = r^\ast \mathbf{D}_{Y'}$.
When this is the case, we will say that 
$\mathbf{D}$ \emph{descends} to $Y'$
and we shall write 
$\mathbf{D}= \overline{\mathbf{D}}_{Y'}$
for the $\bK$-b-divisor which $\bfD_{Y'}$ determines.

We say that $\mathbf{D}$ is {\it b-effective}, 
if $\mathbf{D}_{Y'}$ is effective for any model 
$Y'$. We say that $\mathbf{D}$ is {\it b-nef} 
(resp.\ {\it b-semiample}, {\it b-free}), 
if it is $\mathbb{K}$-b-Cartier 
and, moreover, there exists a birational model 
$Y'$ of $Y$ such that 
$\mathbf{D}= \overline{\mathbf{D}}_{Y'}$ 
and $\mathbf{D}_{Y'}$ is nef 
(resp.\ semiample, Cartier and globally generated) on $Y'$.
Furthermore, we say that $\mathbf{D}$ is 
\emph{effectively b-semiample} if it is b-semiample, 
and there is positive integer $I$, only depending on 
the given setup, such that $I\mathbf{D}$ is b-free.
These notions extend analogously to the relative case.
In all of the above, if $\mathbb{K}= \bZ$, 
we will systematically drop it from the notation.

We recall the definition of generalized pairs, 
first introduced in~\cite{BZ16} which is a generalization 
of the classic notion of pair.

\begin{definition}
A {\em generalized sub-pair} $(Y,B, \mathbf{M})/S$ is the datum of:
\begin{enumerate}
\item a projective morphism
$Y \rar S$ of normal quasi-projective varieties;
\item an $\mathbb R$-Weil divisor $B$ on $Y$;
\item an $\mathbb R$-b-Cartier b-divisor $\mathbf{M}$ 
over $Y$ which descends to a Cartier divisor $\mathbf{M}_{Y'}$ 
on some birational model $Y' \rightarrow Y$, with 
$\mathbf{M}_{Y'}$ relatively nef over $S$.
\end{enumerate}
We require that $K_Y +B+ \mathbf{M}_Y$ 
is $\mathbb R$-Cartier. When $B$ 
is effective, we say that 
$(Y,B,\bM.)/S$ is a {\it generalized pair}.
Similarly, if $\bM. = 0$, we drop the word generalized, 
as this retrieves the usual notions of (sub-)pair.
\end{definition}

If $S = \Spec \CC$, we omit it from the notation.
We refer the reader to \cite{FS20} and the references 
therein for more details about generalized pairs.

\subsection{Canonical bundle formula} \label{sec:cbf}
We recall the notion of lc-trivial fibration.

\begin{definition} A {\it fibration} is a
surjective, projective morphism $f\colon X\to Y$ of
normal quasi-projective
varieties, with connected fibers,
satisfying $0<\dim Y<\dim X$.
A {\it rational fibration} 
$f\colon X\dashrightarrow Y$ is 
a rational dominant map of normal quasi-projective
varieties, which is a fibration
over a non-empty open set $U\subset Y$.
The datum of a rational fibration 
is equivalent to 
that of a projective variety
$X_\eta \to \eta$ over the
generic point $\eta$ of $Y$. 
\end{definition}

\begin{definition} 
\label{lc-trivial.def}
Let $(X, \Delta)$ be a sub-pair with coefficients in $\QQ$.
A fibration $f \colon X \rar Y$ of quasi-projective 
varieties is \emph{lc-trivial} if
\begin{itemize}
    \item[(i)] $(X,\Delta)$ is sub-log canonical over the 
    generic point of $Y$;
    \item[(ii)] $\rk f_\ast  \O X. 
    (\lceil \mathbf{A}^\ast (X,\Delta)\rceil)=1$, where 
    $\mathbf{A}^\ast (X,\Delta)$ is the b-divisor defined in 
    \cite[\S~1.3]{Amb04};
    \item[(iii)] there exists a $\QQ$-Cartier $\QQ$-divisor $L_Y$ on $Y$
    such that $\K X. + \Delta \sim_\QQ f^\ast  
    L_Y$.\footnote{Observe that lc-trivial stands for (relatively) 
    trivial log canonical divisor, i.e., assumption (iii), and not 
    to the type of the singularities in assumption (i).} 
\end{itemize}
\end{definition}

\begin{remark}
Condition (ii) is automatic if $\Delta$ is effective 
over the generic point of $Y$.
\end{remark}

\begin{definition}
We say that a fibration $f\colon X \to Y$ is a 
{\it $K$-trivial}, {\it abelian}, or {\it primitive 
symplectic} if the general
fiber of $f$ is a $K$-trivial, abelian, or 
primitive symplectic variety, respectively.
\end{definition}

A $K$-trivial fibration 
$f \colon X \to Y$ is also an lc-trivial fibration, 
up to choosing a 
suitable sub-boundary $\Delta$ on $X$ 
(whose support does not dominate $Y$). 
By \cite{HX13}, see also \cite[Thm 3.1]{LLX2024}, 
up to a birational modification of $X$, 
we can suppose that 
the sub-boundary $\Delta$ is trivial over a 
big open set of $Y$
as follows:

\begin{proposition}\label{prop:bestKtrivialmodel}
Let $Y^{o}$
be an open subset of
a normal quasi-projective variety $Y$, and let 
$f \colon (X^{o},\Delta^{o}) \to Y^{o}$ be an lc-trivial 
fibration such that $(X^{o},\Delta^{o})$ is log canonical 
and every log canonical center of $(X^{o},\Delta^{o})$ 
dominates $Y^{o}$. Then 
there exists a projective morphism 
$f \colon X \to Y$ and a boundary divisor 
$\Delta$ on $X$ such that:
\begin{enumerate}
\item $X^{o} = X \times_{Y} Y^{o}$ is an open 
subset and $\Delta^{o}=\Delta|_{X^{o}}$;
\item $(X,\Delta)$ is log canonical, and $K_{X}+\Delta$ 
is semiample over $Y$; and
\item $(X,\Delta) \to Y$ is an lc-trivial fibration over $U$, 
i.e., $K_{X_{U}} +\Delta|_{X_U} \sim_{\bQ,f} 0$, where $U$ is the 
maximal big open set over which $f \colon X \to Y$ is 
equidimensional, and $X_{U} \coloneqq  f^{-1}(U)$.
\end{enumerate}
Furthermore, $(X,\Delta)$ can be chosen so that
\begin{itemize}
\item[(4)] if $\Delta^{o}=0$, $\Delta=0$ holds; and
\item[(5)] if $(X^{o},\Delta^{o})$ is klt (resp.~canonical, 
resp.~terminal, resp.~klt/canonical/terminal and 
$\mathbb Q$-factorial),
so is $(X,\Delta)$.
\end{itemize}
\end{proposition} 

\begin{proof}
By \cite[Cor.~1.2]{HX13}, there exists a projective 
morphism $\widetilde{f} \colon \widetilde{X} \to Y$ and 
an lc pair $(\widetilde{X}, \widetilde \Delta)$, where 
$X^{o} = \widetilde{X} \times_{Y} Y^{o}$ is an 
open subset and $\widetilde \Delta|_{X^{o}}=\Delta^{o}$.
Furthermore, by the proof of \cite[Cor.~1.2]{HX13}, we 
observe that $\widetilde \Delta = 0$ if $\Delta^{o}=0$ holds, 
and $(\widetilde{X}, \widetilde \Delta)$ is klt (resp.~canonical) 
if so is $(X^{o},\Delta^{o})$.

Up to a $\mathbb Q$-factorial dlt modification of 
$(\widetilde{X},\widetilde{\Delta})$, 
$(\widetilde{X},\widetilde{\Delta})$ admits a good minimal 
model over $Y$ by \cite[Thm.~1.1]{HX13}. By \cite[Prop.~2.5]{Lai2011} 
(klt case) or \cite[Thm.~1.7]{HH2020} (lc case), a 
$(K_{\widetilde{X}}+\widetilde{\Delta})$-MMP with scaling 
of an ample divisor terminates with a minimal model 
$(X,\Delta)$ of $(\widetilde{X},\widetilde{\Delta})$ over $Y$. 
By the lc-triviality assumption, this MMP is the identity 
along $X^{o}$. We observe that this MMP preserves the 
class of singularities of $(\widetilde{X},\widetilde{\Delta})$ 
(i.e., lc, klt, canonical, or terminal).
This is clear for lc or klt singularities.
Furthermore, this MMP is the identity over $X^o$, and,
by the proof of \cite[Cor.~1.2]{HX13}, every prime 
component of $\widetilde{\Delta}$ intersects $X^o$.
No component of $\widetilde{\Delta}$ is contracted by 
the MMP, and so canonical or terminal singularities 
are preserved as well.

Furthermore, this minimal model is good, i.e., $K_X+\Delta$ 
is relatively semiample over $Y$, see \cite[Thm.~1.1]{HX13} 
and \cite[Prop.~2.4]{Lai2011} (notice that the latter reference 
is phrased for klt pairs, but its proof goes through 
verbatim for lc pairs). Lastly, in case $(X^{o},\Delta^{o})$ 
is $\QQ$-factorial klt, we may replace $X$ with a small 
$\QQ$-factorialization, which is then the identity along $X^{o}$.
This settles (1), (2), (4), and (5).

Now, we restrict $f$ to $f_U \colon X_U \to U$.
By construction, we have $K_{X_U}+\Delta|_{X_U} 
\sim_{\QQ,f_U} G_U$, where $G_U$ is vertical over $U$.
For (3), we should take $G_U=0$. 

Note that it suffices to show the claim for 
$f_{U^{\rm reg}} \colon X_{U^{\rm reg}} \to U^{\rm reg}$.
Indeed, since $f_U$ is equidimensional, $X_{U^{\rm reg}}$ 
is a big open subset of $X_U$. Thus, if 
$K_{X_{U^{\rm reg}}}+\Delta|_{X_{U^{\rm reg}}} 
\sim_{\QQ,f_{U^{\rm reg}}} 0$ holds, then so does 
$K_{X_U}+\Delta|_{X_U} \sim_{\QQ,f_{U}} 0$.
Hence, we may assume that $U$ is smooth;
this allows us to pull back any prime divisor in 
$U$ to $X_U$ while preserving the linear 
equivalence relative to $U$. 

Now, by perturbing $G_U$ by the pullback of a suitable 
$\QQ$-divisor in $U$, we may assume that $G_U \leq 0$ 
and $G_U$ is of insufficient fiber type in the sense 
of \cite[Def.~2.9]{Lai2011}.
But then, by \cite[Lem.~3.3]{Bir12}, $G_U=0$ follows,
so property (3) is satisfied too.
\end{proof}

\begin{corollary}\label{cor:ktrivial-model} 
Every 
$K$-trivial fibration $f\colon X\to Y$
admits a birational model which has canonical
total space, and $K_{X_U}\sim_{\bQ,f_U}0$ for
a big open set $U\subset Y$.
\end{corollary}

The {\it canonical bundle formula} is a broad
term for the formula for the $\bQ$-divisor
$L_Y$ satisfying $K_X+\Delta\sim_\bQ f^{\ast}L_Y$. It
was proven in increasing generality in
\cite{Kod1,Kod2,Fuj86,Kaw98,FM00,Amb04,
Amb05,FG14,Fil20,ACSS}. The form we require
is:

\begin{theorem} \label{base-same-sing}
Let $f \colon (X, \Delta) \to Y$ be an lc-trivial fibration. 
Then 
$$
K_X+\Delta \sim_\QQ f^\ast (K_Y + B_Y + \bfM_Y)
$$
where $B_Y$ and $\bM Y.$ are the {\rm boundary}
and {\rm moduli} divisors. 
Furthermore, $(Y,B_Y, \bfM)$ has the structure
of a generalized sub-pair. 
We say that $f\colon (X,\Delta)\to Y$
{\rm induces} 
$(Y,B_Y, \bfM)$. 

When $(X,\Delta)$ is a 
klt (resp.\ lc) pair, then 
$(Y,B_Y, \bfM)$ is a klt (resp.\ lc) generalized pair. 
\end{theorem}

We now discuss the definitions of $B_Y$
and $\bfM$. For further details, we refer to 
\cite{Amb04, Amb05, FG14}. 

\begin{construction}\label{coefficientB}
The \emph{boundary divisor}
$B_Y$ is a 
$\QQ$-divisor measuring
 the singularities of $(X,\Delta+f^{\ast}P)$, for all prime
divisors $P\subset Y$, over the generic point of $P$.
We define the log canonical threshold 
of $f^{\ast}P$ with respect to $(X,\Delta)$ by
$$
\lct_f(P) \coloneqq \inf_E \,
\frac{a(E; X,\Delta)}{\mult_E(f^{\ast}P)}
$$
where the infimum 
is taken over all the divisors $E$ over $X$ 
which dominate $P$, and $a(E; X, \Delta)$ is the
log-discrepancy of $E$ with respect to $(X, \Delta)$.
If $C\subset Y$ is a general curve
transverse to $P$, then $\lct_f(P)$ agrees 
with the log canonical threshold of the 
central fiber of $f\vert_C$ 
with respect to $(X, \Delta)|_{f^{-1}(C)}$. 
We have the following formula for
$B_Y$:
\begin{equation}\label{eq:lcg}
B_Y=\sum_{P\subset Y} (1-\lct_f(P))P,
\end{equation}
where the sum ranges over all prime divisors
 $P\subset Y$. 
 If $\Delta$ is effective, then so is $B_{Y}$ 
 and it dominates the \emph{multiplicity divisor} $\Delta_{f}$, i.e.,
\begin{equation}\label{eq:multiplicitycouple}
 B_{Y} \geq \Delta_f  \coloneqq  \sum_{P\subset Y} 
 \bigg(1 - \frac{1}{m_f(P)}\bigg) P,
\end{equation}
where $m_f(P)$ is the multiplicity 
of the 
fiber of $f$ over the generic point of $P$, i.e.,
if $f^{\ast} P = \sum_{i \in I} m_i D_i$ 
over the generic point of $P$, then 
$m_f(P)  \coloneqq  \gcd_{i \in I}\{m_i\}$. 
\end{construction}

The \emph{moduli divisor}
$\bM Y.$ represents a pseudo-effective 
$\mathbb Q$-linear series measuring the 
variation in moduli of the fibers of $f$. 
It is the trace of a $\bQ$-b-Cartier
b-divisor
$\bM.$ which is b-nef. For a $K$-trivial fibration, we recall 
the definition of
$\bfM$ in Section \ref{hodge-moduli}. 
We refer the reader for instance to 
\cite{Kollar07} or \cite[\S~7.5]{PS09} 
for the general case.

\begin{example}\label{elliptic-ex}
Let $f \colon S \rar C$ be a relatively minimal, 
smooth elliptic surface.
Kodaira's canonical bundle formula gives
$$
\K S/C. \sim_\bQ f^{\ast}(B_C +
j^\ast \O \mathbb P ^1.(\tfrac{1}{12})),
$$
where 
$j \colon C \rar \mathbb P ^1$ 
is the $j$-invariant function, see \cite{Kod1,Kod2}. 
Here $\bM C.  \coloneqq  j^\ast 
\O \mathbb P ^1.(\frac{1}{12})$ is semiample and is 
the pullback of the Hodge $\bQ$-line bundle 
$\lambda$ on the compactified
moduli space $\bP^1$ of elliptic curves. Recall
that Kodaira classified the singular fibers 
$S_p$ of an elliptic fibration, into types listed 
in Table \ref{b-vals}.
The coefficients of $B_C=\sum n_p[p]$ of the 
corresponding point $p\in C$ are given.
The coefficient is $0$ exactly when the fiber
is of $I_n=I_n(1)$ 
type, i.e., the elliptic fibration
is semistable.

\begin{table}
\renewcommand{\arraystretch}{1.4}
\begin{tabular}{|c|c|c|c|c|c|c|c|c|}
\hline
$I_n(m)$ & $II$ & $III$ & 
$IV$ & $I_n^{\ast}$ & 
$II^{\ast}$  & $III^{\ast}$  & $IV^{\ast}$ \\
\hline
 $\tfrac{m-1}{m}$ & $\tfrac{1}{6}$ & $\tfrac{1}{4}$ & 
 $\tfrac{1}{3}$ &  $\tfrac{1}{2}$ &
$\tfrac{5}{6}$ & $\tfrac{3}{4}$ & $\tfrac{2}{3}$ \\
\hline
\end{tabular}
\vspace{0.5 cm}\\
\caption{Kodaira fibers and their contribution to the boundary.}
\label{b-vals}
\end{table}

\end{example}

\subsection{Variations of Hodge structure and moduli
divisors}\label{hodge-moduli}

We discuss the relationship between variations
of Hodge structure and moduli divisors. 

\begin{definition}
A {\it $\bZ$-polarized variation of Hodge structure}
$(\bV,\cF^\bullet,\psi)$
of weight $k$ on a smooth
quasi-projective variety $Y$ consists of a 
$\bZ$-local system $\bV\to Y$, a
bilinear form $\psi\colon \bV\otimes \bV\to 
\underline{\bZ}_Y$, and a filtration 
$\bV\otimes \cO_Y=\cF^0\supset \cF^1\supset\cdots$ 
by holomorphic subbundles, such that:
\begin{enumerate}
\item $(\bV_y, \psi_y, \cF^\bullet_y)$ defines
a polarized weight $k$ Hodge structure 
for all $y\in Y$; and
\item $\nabla(\cF^p)\subset \cF^{p-1}\otimes \Omega^1_Y$
where $\nabla$ is the flat connection $\bV\otimes \cO_Y$.
\end{enumerate}
\end{definition}

A $\bZ$-PVHS has a well-defined 
{\it Hodge type} $(h^{p,q})_{p+q=k}$. See 
\cite{voisin, carlson} for general
reference.

\begin{example}
Suppose $f\colon X\to Y$ is a smooth projective
morphism and $L$ is a relatively 
ample line bundle on $X$. Then,
the relative primitive $k$-th cohomology
$(R^kf_*\underline{\bZ}_X)_{\rm prim}$ with respect
to the polarization $L$ is a local system
underlying a $\bZ$-PVHS.
\end{example}

Suppose the rank of $\bV$ is $N$
and let $*\in Y$ be a base point.
There is an associated {\it monodromy representation} 
$\rho\colon\pi_1(Y,*)\to \GL_N(\bZ)$,
well-defined up to conjugacy in $\GL_N(\bZ)$,
corresponding to a choice of frame of $\bV_*$.
Let $Y\hookrightarrow \oY$ be an snc compactification
with boundary divisor $\partial = 
\oY\setminus Y = \sum \partial_i$, 
where $\partial_i$ are the irreducible 
components of $\partial$.
The monodromy about each boundary divisor 
$\partial_i\subset \oY\setminus Y$ is 
quasi-unipotent, cf.~\cite[Lem.~4.5]{schmid}.

Let $\Gamma_N(3)\subset \GL_N(\bZ)$ be the finite
index subgroup of matrices which reduce to the
identity mod $3$.
It is a neat\footnote{A subgroup 
$\Gamma \subset \GL_N(\bZ)$ is \emph{neat} if the 
eigenvalues of any element generate a 
torsion-free subgroup in $\CC^\ast$.} subgroup
and in particular contains no finite 
order elements, see \cite[p.~253, p.~207 Lem.]{Mum70}. 
Then there is an \'etale cover $Z\to Y$
and a corresponding
pullback $\bZ$-PVHS
whose monodromy representation $\rho$ factors
through $\Gamma_N(3)$. Letting $\pi\colon
\oZ\to \oY$ be an snc resolution of the 
corresponding finite branched cover of $\oY$ 
(\'etale over $Y$), 
we have that $Z\to \oZ$ is an snc 
compactification with unipotent monodromies 
at the boundary.

By \cite[Prop.~5.2]{deligne1970}, 
\cite[Thm.~4.12]{schmid}, 
the vector bundles $\cF^p$
extend canonically to vector bundles
$\overline{\cF}\,\!^p$ on the whole $\oZ$. Furthermore,
these extensions are stable under
further generically finite base changes.
It follows that
$$
\bfM^p _{\overline{Y}} \coloneqq  \deg(\pi)^{-1}
\pi_*(\det \overline{\cF}\,\!^p)
$$
induces a well-defined $\bQ$-b-Cartier
$\bQ$-b-divisor $\bfM^p$,
whose Cartier index is bounded
above solely in terms of $N$, e.g., by 
$3^{N^2}> |\!\GL_N(\bZ/3\bZ)|$. 

\begin{definition} \label{def:moduli-ht}
We say that a $\bZ$-PVHS $(\bV,\cF^\bullet, \psi)$ 
of weight
$k$ is of {\it $K$-trivial type}
if $h^{k,0}=1$. In this case,
we define the {\it moduli $\bQ$-b-divisor} 
$\bfM  \coloneqq  \bfM^k$.
\end{definition}

\begin{example}\label{ex:moduliKtrivialfibration}
Let $f\colon X\to Y$ be a $K$-trivial 
fibration of relative dimension $k$. 
Let $Y^o\subset Y$ be an open
subset where the fibers admit a
simultaneous resolution 
$\widetilde{f}^o\colon \wX^o\to Y^o$.
Then $R^k\widetilde{f}^o_*\underline{\bZ}_{\wX^o}$
defines a $\bZ$-local system underlying
a $\bZ$-VHS over $Y^o$.
We have an identification
$f^o_*K_{X^o/Y^o}\simeq
\widetilde{f}^o_*K_{\wX^o/Y^o}$ 
because the fibers are canonical.
Hence, 
$R^k\widetilde{f}^o\underline{\bZ}_{\wX^o}$
is of $K$-trivial type.
Thus, Definition \ref{def:moduli-ht}
gives a $\bQ$-b-divisor $\bfM$ on $Y$.

The transcendental part
of the variation 
$R^k\widetilde{f}^o\underline{\bZ}_{\wX^o}$ is 
independent of the choice
of resolution, since if $\wX'\to \wX^o$ 
is a further blow-up of a smooth center
over the generic point of $Y$, 
we have by Mayer--Vietoris 
that
$$R^k\widetilde{f}'_*
\underline{\bZ}_{\wX'} = 
R^{k}\widetilde{f}^o_*
\underline{\bZ}_{\wX^o}\oplus \bV'$$ 
where $\bV'$ is a $\bZ$-VHS of weight $k$ 
with $h^{k,0}=0$. Thus, the transcendental
variation is well-defined, independent
of the resolution $\wX^o$, and so is $\bfM$.

\end{example}

For more general lc-trivial fibrations,
it is still possible to define $\bfM$ as
the Deligne extension of an appropriate
filtered piece of a variation of 
Hodge structure \cite[Lem.~5.2]{Amb04},
from which nefness follows by curvature
properties
\cite[Prop.~7.7]{griffithsIII},
\cite[Lem.~7.18]{schmid}.

 \begin{remark}\label{standard-snc}
 Let $f\colon X\to Y$ be a $K$-trivial fibration,
 and let $f'\colon X'\to Y'$ be some model.
 The moduli $\bQ$-b-divisor $\bfM$ descends
 if there exists an open set $U\subset Y'$
 for which 
 \begin{enumerate}
 \item $f'\colon X'\to Y'$ is 
 locally trivial
 over $U$; and
 \item $Y'$ is smooth with $Y'\setminus U$ an snc divisor.
 \end{enumerate}
 We call such a model $f'$ a {\it standard snc model}
 of $f\colon X\to Y$. 
 \end{remark}

\subsection{Effective b-semiampleness 
conjecture}\label{sec:eff-b-semi}

Let $f\colon X\to Y$ be an lc-trivial fibration.
It is expected that $\bM Y.$ is the pullback 
of an ample Hodge $\bQ$-line
bundle from a suitable projective
moduli space, see
\cite[Conj.~7.13]{PS09}.
Recently, Laza has also given a conjectural
Hodge-theoretic
description of this compactification as one of
``Baily--Borel'' type; 
it should satisfy the Borel extension property, 
and the boundary should be stratified
by moduli of lower-dimensional
$K$-trivial varieties,
see also \cite[\S~8.3.8]{Kollar07}.

However, the conjectural moduli map may not be defined 
everywhere, and in general the linear system 
$|k \bM Y.|$ is not base point free for 
any $k$. 
To remove the indeterminacy of the expected 
rational map, we should replace 
$f \colon X \to Y$ with a birational base change 
$f' \colon X' \to Y'$.
In general, we only know that $\bM {Y'}.$ 
is a 
$\mathbb Q$-effective nef divisor \cite{Amb05}.
It is conjectured that $\bM {Y'}.$ is indeed semiample, 
i.e.~$\bM.$ is expected to be b-semiample.

A first evidence of this fact is the case 
of elliptic fibrations. 
Generalizing 
the work of Kodaira to arbitrary elliptic 
fibrations, Fujita showed that $12 \bM {Y'}.$ 
is integral and $|12\bM {Y'}.|$ is free in 
\cite{Fuj86}.
Prokhorov and Shokurov 
conjectured that this is a general pattern:
In \cite[Conj.~7.13]{PS09}, they predicted that 
$\bM {Y'}.$ is \emph{effectively semiample}. Said 
otherwise, there exists a constant $I$, 
only depending on the type of general fiber, 
ideally only on its dimension, such that $I \bM  Y'.$ 
is integral and $|I \bM  Y'.|$ is free.

\begin{conjecture}[Effective b-semiampleness]
For any integer $d$, there is an integer 
$I = I(d)$ such that the following holds. For 
any lc-trivial fibration $f \colon X \rar Y$ of 
relative dimension $d$ (Def.~\ref{lc-trivial.def}), 
there is a commutative diagram
\begin{align*}
\xymatrix{
({X}',\Delta')  \ar[r]^{\phi} \ar[d]_{f'}&
(X,\Delta) \ar[d]^{f} \\
Y'  \ar[r]^{\psi} & Y
}
\end{align*}
with the following properties:
\begin{enumerate}
    \item the morphisms $\phi$ and $\psi$ are birational;
    \item $\K  X'. + \Delta' = \phi^\ast (\K X. + \Delta)$;
    \item $I \bM  Y'.$ is Cartier and $|I \bM Y'.|$ is a free linear series, and
    \item $\bfM$ descends to $\bM  Y'.$.
\end{enumerate}
\end{conjecture}

In general, the semiampleness of $\bM  Y'.$ is known 
in very few cases;
families of K3 surfaces or abelian varieties by 
Fujino \cite{Fuj03}, of primitive symplectic 
varieties by Kim \cite{Kim2024}, 
building on \cite{Fuj03}, and in relative dimension $2$
\cite{Fil20, babwild}.
We prove some effectivity
in Section \ref{sec:eff-semi-pf}.
In our cases, $\bM Y.$ is the 
pullback of a universal Hodge $\bQ$-line
bundle from a compactified period space.
But the effectivity
is not immediate,
because there are infinitely many
period spaces involved.

\subsection{Notions of boundedness}\label{sec:bd-def}

\begin{definition} \label{defn:bounded}
Let $\{X_i\}_{i\in I}$
be a collection of projective varieties. We
say that this collection is {\it bounded} if there exists
a projective morphism 
$\cX\to \cT$ 
of quasi-projective\footnote{In particular, 
the base $\cT$ is of finite type.} varieties, and
for all 
$i\in I$, 
there exists a closed point 
$t\in \cT$ 
and
an isomorphism 
$f_t \colon X_i \to \cX_t$. 
Similarly, we say the collection is 
{\it birationally bounded}, 
resp.\ {\it bounded in codimension 1},
if $f_t$ is merely assumed to be a birational map, 
resp.\ a birational map which is an 
isomorphism in codimension $1$. 
\end{definition}
\begin{example} 
Rational varieties of a fixed dimension are 
birationally bounded, but not bounded, e.g., 
there are infinitely many non-isomorphic 
Hirzebruch surfaces. 
\end{example}

\begin{definition} 
A collection
of fibrations, resp.~morphisms or rational
maps,
$\{X_i\to Y_i\}_{i\in I}$ 
is {\it bounded} if there exists 
a fibration, resp.~morphism or rational
map, $\cX\to \cY/\cT$
over a finite type base $\cT$, such that,
for all $i\in I$, there is a 
closed point $t\in \cT$ for which
$X_i\to Y_i$ and $\cX_t\to \cY_t$
are isomorphic.
A collection of fibrations is 
\begin{enumerate}
\item {\it generically bounded}
if we only require that $X_i\to Y_i$ and 
$\cX_t\to \cY_t$ are isomorphic over
the generic points of their bases;
\item {\it generically bounded with bounded base} 
if it is generically bounded and the 
birational map $Y_i\dashrightarrow 
\mathcal{Y}_t$ on the 
bases is an isomorphism; and
\item 
{\it birationally bounded} if
$X_i\to Y_i$ and $\cX_t\to \cY_t$ are merely birational
over the generic points of their bases.
\end{enumerate}
\end{definition}

Concerning boundedness of
morphism spaces, we also allow the 
targets $Y_i$ to be a fixed DM stack,
see e.g.~Proposition \ref{bd-stack}.

\begin{example}
Let $\bF_n \to \PP^1$ be the
ruling of a Hirzebruch surface. 
The collection of fibrations
$\{\bF_n \to \bP^1 \}_{n\in \bN}$
is generically bounded. 
All such agree, 
as fibrations, with 
$\bP^1 \times \bP^1 \to \bP^1$, 
over some Zariski
open set of $\bP^1$. 

\end{example}

We also require notions of boundedness for pairs 
and sub-pairs, analogous to Definition \ref{defn:bounded}.

\begin{definition}
A collection of sub-pairs 
$\{(Y_i, B_i)\}_{i \in I}$ 
is \emph{bounded} if there exist a projective 
morphism $\cY \to \cT$  of schemes of finite type 
and a $\bQ$-Weil divisor $\cB$ on $\cY$ such that, 
for all $i \in I$, there exist a closed point 
$t \in \cT$ and an isomorphism 
$\pi_t \colon \cY_t \to Y_i$ 
such that $(\pi_t)_*\cB_t = B_i$. 
The collection is \emph{bounded 
in codimension 1} 
if  $\pi_t$ is instead assumed 
to be a birational map 
and an isomorphism in codimension 1, again with 
the condition that 
$(\pi_t)_*\cB_t = B_i$.
\end{definition}

\section{Period mappings and period domains}
\label{sec:hodge}

\subsection{Abelian fibrations}\label{av-fib}
In this section,
we discuss the period map for abelian fibrations;
see, e.g., \cite[Ch.~7]{debarre_book}, 
\cite[Ch.~8]{birkenhake-lange}.

Let $f\colon X\to Y$ be an abelian fibration. Denote by
$Y^o\subset Y$ the open subset 
on which $f$ is smooth and $\partial  \coloneqq  Y\setminus Y^o$
for the complement. Write 
$\partial=\bigcup_i \partial_i$ for the irreducible 
components. Let $X^o  \coloneqq  f^{-1}(Y^o)$ and 
$f^o  \coloneqq  f\vert_{X^o}\colon X^o\to Y^o$.

Let $A$ 
be an abelian variety of dimension 
$g$.
A {\it polarization} 
$L$ 
on 
$A$ 
is an element 
$$L\in H^{1,1}(A)\cap H^2(A,\bZ)={\rm NS}(A)$$
that is also the first Chern class of an ample line bundle on 
$A$.
Since 
$H^2(A,\bZ)\simeq \wedge^2H_1(A,\bZ)^{\ast}$,
$L$ 
defines a symplectic form on 
$H_1(A,\bZ)\simeq \bZ^{2g}$. 
There is a unique sequence
of positive integers 
${\bf d}  \coloneqq  (d_1, \dots, d_g)$,
satisfying 
$d_i\mid d_{i+1}$,
such that the symplectic
lattice defined by $L$ on $H_1(A,\bZ)$ is isometric to
the symplectic lattice
$V_{g,{\bf d}}^*=\bigoplus_{i=1}^g\bZ e_i\oplus \bZ f_i$ 
defined by
$e_i\cdot e_j = f_i\cdot f_j=0$, and 
$e_i\cdot f_j=d_i\delta_{ij}$. In this case,
we say $(A,L)$ is {\it ${\bf d}$-polarized}.
Define $\Gamma_{g,{\bf d}}  \coloneqq  
{\rm Sp}(V^*_{g,{\bf d}})$ 
as the group of isometries
of $V^*_{g,{\bf d}}$. We define 
$V_{g,{\bf d}} \coloneqq 
{\rm Hom}(V_{g,{\bf d}}^*,\bZ)$ as the dual lattice, 
and note that $\cdot$ defines a $\bQ$-valued 
symplectic form on $V_{g,{\bf d}}$ which
we denote by the same symbol.

Let $\cH_g$ be the Siegel space: It is the 
analytic open subset 
$$\cH_g  \coloneqq  
\{F^1\in {\rm LGr}(V_{g,{\bf d}}\otimes \bC)\,:\, 
i\alpha\cdot \bar\alpha>0
\textrm{ for all }0\neq \alpha\in F^1\}$$
of the Lagrangian Grassmannian parameterizing 
polarized weight $1$
Hodge structures on $V_{g,{\bf d}}$.
We define $\cA_{g,{\bf d}}  \coloneqq  
[\Gamma_{g, {\bf d}}\backslash \cH_g]$,
where the square brackets denote the stack quotient.
As an orbifold, it is the DM moduli stack 
of ${\bf d}$-polarized abelian $g$-folds {\it with
a distinguished origin point}.
More precisely:
\begin{align*}\cA_{g,{\bf d}}(B)&=
\{
f\colon \cA\to  B
\textrm{ smooth
abelian }g\textrm{-fold fibration 
with section, }\,
\\
&
\cL\in 
H^0(B, R^2f_*\underline{\bZ}_\cA)
\textrm{ s.t.~}
\cL_b\in H^{1,1}(\cA_b)
\textrm{ and positive of
type }{\bf d},\,
\forall b\in B\}.\end{align*}
We may assume $d_1=1$ since there is 
a natural identification between the DM stacks of
${\bf d}$- and $(d_1^{-1}{\bf d})$-polarized
abelian varieties, given by  
$(\cA,\cL)\leftrightarrow (\cA,d_1^{-1}\cL)$.
Note that while $\cL$ may not be induced
by a global line bundle on $\cA$, 
one may argue as in Proposition~\ref{is-pic} 
that $d_1\cL$ is, for some integer $d_1>0$.

Let $L'$ be an ample line bundle
on $X$. Restricting
to the fibers of $f^o$
and dividing by $d_1$ gives
a primitive polarization $L$ fiberwise.
Thus, we get a
classifying morphism
\begin{align}\label{def:classify}
\begin{aligned}
\Phi^o\colon Y^o&\to \cA_{g,{\bf d}} \\
y & \mapsto {\rm Aut}^0(X_y)
\end{aligned}
\end{align}
to the DM stack of {\bf d}-polarized abelian $g$-folds with
origin. 
If $f\colon X\to Y$ is endowed with 
a section over $Y^o$, say $s\colon Y^o\to X^o$,
we may canonically 
identify ${\rm Aut}^0(X_y)$
with $X_y$ for $y\in Y^o$. 
In this circumstance, 
the abelian fibration $f^o\colon 
X^o\to Y^o$ is
recovered as the pullback of the universal family over 
$\cA_{g,{\bf d}}$ of abelian $g$-folds with origin.

\begin{definition}\label{defn:alb}
The {\it Albanese fibration}
$(f^o)^{\rm Alb}\colon 
(X^o)^{\rm Alb}\to Y^o$
is the pullback of the universal
family along $\Phi^o$.  Alternatively, 
$(X^o)^{\rm Alb} = (J_{X^o/Y^o})^*$, where
$J_{X^o/Y^o}\coloneqq \cF^0/(\cF^1+R^1f^o_*\underline{\bZ}_{X^o})$ 
for $\cF^0\supset \cF^1$
the Hodge filtration of the weight $1$ 
Hodge structure on $R^1f^o_*\underline{\bZ}_{X^o}$.
\end{definition}

In general,
$(X^o)^{\rm Alb}$ is only
birational to $X^o$ when
$f^o\colon X^o\to Y^o$ 
admits a rational section.
Over the generic point $\eta \in Y$, the fiber $(X^o)^{\rm Alb}_\eta$ of $(f^o)^{\rm Alb}$ coincides 
with the Albanese variety of the generic fiber $X_\eta$ of $f$.

We may compose the classifying morphism
with the coarsening map $\cA_{g,{\bf d}}\to 
A_{g,{\bf d}}$. We now recall
the Borel extension theorem \cite{borel-ext}.
Let $D$ denote a holomorphic disk with center $0$
and $D^{\ast}=D\setminus 0$ the punctured disk.

\begin{theorem}\label{borel-ext}
A holomorphic map $(D^{\ast})^k\times D^{n-k}\to A_{g,{\bf d}}$
to the coarse moduli space of ${\bf d}$-polarized 
abelian varieties (or more generally, any Shimura
variety) extends to a holomorphic map 
$D^n\to \oA^{\rm BB}_{g,{\bf d}}$
to the Baily--Borel compactification
\cite{BB66}.
\end{theorem}

We mostly work with Baily--Borel compactifications,
and so we henceforth drop the superscript.
If the discriminant 
$\partial\subset Y$ 
is snc, we conclude by Theorem \ref{borel-ext}
that the classifying morphism 
extends as a morphism to the coarse space 
$\Phi\colon Y\to \oA_{g,{\bf d}}.$

\begin{definition}\label{def:classifying-morphism}
To distinguish it from the {\it classifying morphism}
$\Phi^o\colon Y^o\to \cA_{g,{\bf d}}$ 
to the DM stack, we call the rational map
$\Phi\colon Y\dashrightarrow
\oA_{g,{\bf d}}$ to the coarse space
the {\it period map}. 
\end{definition}

The stack $\cA_{g,{\bf d}}$ supports a universal
polarized variation  
of Hodge structure $(\bV,\cF^1, \psi)$ of weight $1$,
and $\wedge^g\bV$ is a $K$-trivial variation
of Hodge structure, whose corresponding moduli 
$\bQ$-b-divisor (Def.~\ref{def:moduli-ht})
on $A_{g,{\bf d}}$
is called the {\it Hodge $\bQ$-line bundle}
$\lambda_g=\det(\cF^1)$.
For $g\geq 2$, the Baily--Borel
compactification may be
defined via the ring of modular
forms, as 
$$\oA_{g,{\bf d}}=\textstyle 
{\rm Proj}\,\bigoplus_{n\geq 0}
H^0(A_{g,{\bf d}}, \lambda^{\otimes n}_g)$$
and for $g=1$, a moderate growth
condition at the cusps is required.

\begin{lemma}\label{lem:proportionality} There exists 
a constant $a \in \NN$, only depending on $g$, such 
that $a \bM Y. \sim a \Phi^{\ast}(\lambda_g)$.
\end{lemma}
\begin{proof}
Up to a birational modification of $Y$, we can suppose that 
$\partial = Y \setminus Y^{o}$ is an snc divisor, and so the 
restricted
period map $\Phi\vert_{Y^o} \colon Y^{o} \to A_{g,{\bf d}}$ 
extends to a regular map $\Phi \colon Y \to \oA_{g,{\bf d}}$
 by Theorem \ref{borel-ext}.
Choose a Galois cover $\pi \colon Y'\to Y$ with Galois 
group $G$, and a normal variety $X'$ birational to the 
main component $X''$ of $X \times_{Y'} Y$ such that:
\begin{itemize}
\item $\partial ' \coloneqq \pi^{-1}(\partial) \subset Y'$ 
is snc, and $f' \colon X'\to Y'$ is smooth over 
$Y_{o}' \coloneqq Y' \setminus \Sigma'$; and
\item $f'$ is semistable in codimension 1, so that 
$R^k f'_\ast  \mathbb{Z}_{X'_{o}}$ has unipotent monodromy.
\end{itemize}
The setup is summarized in the following diagram:
\[
\xymatrix{
X' \ar[dr]_{f'} \ar[r]^{bir} 
& X'' \ar[d]^{f''} \ar[r] 
& X \ar[d]^{f} \\
& Y'  \ar[r]^{\pi} 
& Y.
}
\]
Recall that the moduli part for the fibration $f'$ is the 
Deligne's extension $\overline{\mathcal{F}}\,\!^g$ of the 
lowest piece of the Hodge filtration of the $\ZZ$-variation 
of Hodge structure $R^g f'_{*} \ZZ_{X'_{o}}$, i.e., 
$\cO_{Y'}(\bM Y'.) = \overline{\mathcal{F}}\,\!^g$,
see \S~\ref{hodge-moduli}.
Thus, there exist isomorphisms 
$\rho_{Y'_{o}} \colon \cO_{Y'_{o}}( \bM{Y'_{o}}.) 
\to \cO_{Y'_{o}}((\pi^{o})^*(\Phi^{o})^{*} \lambda_{g})$ 
and  $\rho_{Y^{o}} \colon \cO_{Y^{o}}(\bM {Y^{o}}.) 
\to \cO_{Y^{o}}((\Phi^{o})^* \lambda_{g})$, which are 
compatible, i.e., $\rho_{Y'_{o}} = \pi \circ \rho_{Y^{o}}$.
The isomorphism $\rho_{Y'_{o}}$ extends to the 
Galois invariant isomorphism
\[
\cO_{Y'}(\pi^*\bM Y.) \simeq \cO_{Y'}(\bM Y'.) = 
\overline{\mathcal{F}}\,\!^g 
\simeq \cO_{Y'}(\pi^* \Phi^*\lambda_{g}).
\]
Again by construction of moduli part 
(see \S~\ref{hodge-moduli}), the Cartier indices of 
the $\QQ$-Cartier $\QQ$-divisors $\bM Y.$ and 
$\lambda_g$ are bounded above solely in term of $g$, 
say that $a\bM Y.$ and $a\lambda_g$ are Cartier.
Taking Galois invariants, we obtain
\begin{equation*}
\cO_{Y}(a \bM Y.) \simeq 
(\pi_* \pi^* \cO_{Y}(a \bM Y.))^{G} 
\simeq (\pi_*\pi^* \cO_{Y}(a \Phi^* \lambda_{g}))^G 
\simeq \cO_{Y}(a \Phi^* \lambda_{g}).
\end{equation*} 
\end{proof}

\begin{remark}
The constant $a$ in \Cref{lem:proportionality} is 
necessary already when $g=1$. Consider for instance 
an isotrivial elliptic fibration over a curve. 
Its moduli part may be a torsion Cartier divisor, 
which therefore cannot be pulled back from $\oA_{1,1} 
\simeq \mathbb{P}^1$ via the $j$-map. In fact, when 
$g=1$, we have $a=12$, see \cite[(2.9)]{Fuj86}.
\end{remark}

It follows 
that $\bM. \sim_{\QQ}\Phi^{\ast}(\lambda_g)$
is $b$-semiample, 
cf.~\cite{Fuj03}. 
In particular, 
for $f\colon X\to Y$ an arbitrary 
abelian fibration, $\bfM$ descends to $Y$ once the 
period map $\Phi\colon Y\dashrightarrow \oA_{g,{\bf d}}$
is a morphism. 
Conversely, if $\bfM$ descends to $Y$,
consider a model $\pi\colon Y'\to Y$ on which 
$\Phi'\colon Y'\to \oA_{g,{\bf d}}$
is a morphism. Curves $C$ contained in the fibers of
$\pi$ satisfy $C\cdot \bM Y'.=0$. Thus, $\Phi'$
factors through $\pi$ and descends
to a morphism $\Phi$ on $Y$.

\begin{definition} 
We call $\Phi^{-1}(\oA_{g,{\bf d}}\setminus 
A_{g,{\bf d}})\subset Y$ the {\it Hodge-theoretic
boundary} $\partial_{\rm H}\subset \partial$.
\end{definition}

\begin{remark}
The Hodge-theoretic boundary
is independent of the choice
of polarization on $X$: For instance,
it consists exactly of the union of the components
$\partial_i\subset \partial$ about which the local
monodromy of $\bV  \coloneqq  R^1f^o_*\underline{\bZ}_{X^o}$
on $Y^o$ has infinite order.
\end{remark}

\begin{lemma}\label{section-extend}
Let $f\colon X\to Y$ be an
abelian fibration over a smooth base $Y$. 
Any rational
section $s\colon Y\dashrightarrow X$ is
a morphism over $Y^o$.
\end{lemma}

\begin{proof}
The lemma follows from the fact that
we can resolve indeterminacy of $s$
over $Y^o$ via blow-ups of smooth centers.
Since $Y^o$ is smooth,
the fibers of such a blow-up are rationally chain-connected.
The image of a rationally chain-connected
variety in an abelian variety is a point.
Hence, no blow-ups are necessary to resolve
indeterminacy.
\end{proof}

\begin{lemma}
\label{extend2}
Let $Y$ be smooth
and $U \subset Y$ be a 
big
open set.
Let $f_U \colon X_U \to U$ be a smooth
abelian fibration. 
Then the local system
$R^1f_{U*}\underline{\bZ}$ and 
the $\bZ$-PVHS it underlies extend to $Y$. 

If, additionally, $f_U$ has a section, then
there is a 
unique smooth abelian fibration 
$f \colon X \to Y$ 
with section
that extends $f_U$, 
i.e., $f \vert_U=f_U$.
\end{lemma}

\begin{proof} 
Any $\bZ$-PVHS
on a big open set $U$ extends 
uniquely to a $\bZ$-PVHS
on $Y$, since $Y$ is smooth
and $Y\setminus U$ has codimension
$\geq 2$. Then $f \colon X \to Y$ is the Albanese 
fibration associated to the classifying morphism 
of the $\bZ$-PVHS.
\end{proof}

\begin{proposition}\label{extend3}
Let $f\colon X\to Y$ be an abelian
fibration for which $K_X\sim_{\bQ,f} 0$ and 
$\bfM$ descends to the smooth base $Y$.
Let $Y^o\subset Y$ be an
open set over which $f$ is smooth 
and let $B_Y$ denote the boundary 
divisor induced by the canonical bundle formula.
The local system 
$R^1f^o_*\underline{\bZ}$ on $Y^o$ 
and the $\bZ$-PVHS it underlies
 extend along the open set
$U \coloneqq  Y\setminus 
(\partial_{\rm H}\cup \Supp B_Y)$.

If additionally,
$f$ has a rational section, the smooth abelian
fibration $f|_{Y^o}$ with section admits an 
extension to $U$, which coincides with $f$ 
if $X$ has klt singularities.
\end{proposition}

\begin{proof} 
Let $\partial_i$ be an irreducible
divisor in $U\setminus Y^o$.
Restricting $f$ to a general arc transverse
to $\partial_i$, we obtain a family $g\colon
\cX\to (D,0)$
over a disk, for which $K_\cX\sim_{\QQ,g} 0$, and 
$\cX^{\ast}=g^{-1}(D^{\ast})$ is a smooth abelian 
fibration. The hypothesis 
that $\partial_i\not\subset \Supp B_Y$
implies that the central fiber $\cX_0$ is slc.
Also $K_{\cX_0}\sim_{\QQ} 0$ holds by adjunction.
Since $\partial_i\not\subset 
\partial_{\rm H}$ we may additionally
conclude that the dimension
of the source (cf.\ \cite[\S~4.5]{Kollar2013}) 
of $\cX_0$ is the dimension
of the general fiber and thus $\cX_0$
is not just slc, but a canonical variety; 
cf., e.g., \cite[Thm 4.1.10]{NX2016}. 
Since $\cX_0$ has du Bois singularities 
by \cite[Ch.~6]{Kollar2013}, 
$h^{k}(\mathcal{X}_t, \mathcal{O}_{\mathcal{X}_t})$ 
is constant for any $t \in D$, see 
\cite{Steenbrink1980}. Hence, $\cX_0$
is an abelian variety and $g$ is a smooth abelian fibration. 

It follows that $f$ extends
to a smooth abelian fibration over the 
generic point of $\partial_i$.
Thus, $f$ extends smoothly to a big
open subset of $U$, and the result on 
$U$ follows from Lemma \ref{extend2}. 

Suppose now that $X$ has klt 
singularities. Let $f^+$
be a smooth extension of $f^o\colon X^o\to Y^o$ to an 
abelian fibration $f^+\colon X^+\to U$.
The birational map $f^{-1}(U) \to X^{+}$ 
is regular,  since $f^{-1}(U)$ has klt 
singularities and there 
are no rational curves in the fibers of $f^{+}$; 
cf.~\cite[Cor.~1.44]{Debarre2001}. 
Any  exceptional divisor of the morphism 
$f^{-1}(U) \to X^{+}$ has positive discrepancy 
with respect to $X^+$,
since $X^{+}$ is smooth (in particular, terminal), 
but both $f^{-1}(U)$ and $X^+$ are 
$K$-trivial over $Y^+$.
Hence, $f^+$ and 
$f$ must be isomorphic over $U$.
\end{proof}

\subsection{Fibrations of Picard type}

We define a class of $K$-trivial fibrations
for which we expect boundedness to hold.

\begin{proposition}\label{pure} 
Let $f\colon X\to Y$
be a K-trivial fibration, 
and let 
$f^o\colon X^o\to Y^o$ be a restriction 
over an open set $Y^{o}$ 
over which $f$ is locally trivial.
Then $R^2f^o_*\underline{\bZ}$ underlies
a variation of Hodge structure of pure weight $2$.
\end{proposition}

\begin{proof} The second cohomology group of a 
variety with rational singularities carries a 
pure Hodge structure, see e.g., \cite[Lem.~2.1]{BL2021}. 
The proposition follows, as the fibers of 
$f^o$ are canonical.
\end{proof}

Associated to the local system 
$\bU \coloneqq  
(R^2f^o_*\underline{\bZ})_{\rm tf}$
we have a monodromy representation 
$$\rho\colon \pi_1(Y^o)\to\GL_N(\bZ)$$ where
$N=\rk \bU$. The global sections
$H^0(Y^o,\bU)=\bU^\rho$ 
are the monodromy invariants 
and do not depend on the chosen
Zariski open set $Y^o$.

\begin{proposition}
A flat section 
$L\in \bU^\rho$ lies in $H^{1,1}$ of 
every fiber of $f^o\colon X^o\to Y^o$ 
if and only if it lies in $H^{1,1}$
of some fiber.
\end{proposition}

\begin{proof}
This follows from the theorem
of the fixed part \cite[Thm.~7.22]{schmid},
which implies $\bU^\rho$
admits a pure weight $2$ Hodge structure,
which is a sub-Hodge structure of every fiber.
\end{proof}

\begin{definition}\label{def:pic-type}
We say that a fibration
$f\colon X\to Y$ is of {\it Picard type} 
if $(R^2f^o_*\underline{\bZ})_{\rm tf}^\rho$
is of Hodge--Tate type (equivalently,
$h^{2,0}=0$ for this Hodge structure).
\end{definition}

\begin{remark}\label{rem:ps-picard}
Definition \ref{def:pic-type} says that
every globally flat class of degree
$2$ on the fibers is of $(1,1)$-type.
Thus, a fibration on a projective variety
with $(R^2f^o_*\underline{\bZ})_{\rm tf}^\rho$
of rank $1$ is automatically of Picard type. Moreover,
a fibration by primitive symplectic varieties 
is \emph{not} of Picard type if and only if it is 
isotrivial and the monodromy acts by 
symplectic automorphisms of the fiber.
\end{remark}

\begin{proposition}\label{is-pic}
Let $X$ be a projective 
variety with rational 
singularities and $h^2(X,\cO_X)=0$. 
Then any fibration $f\colon X\to Y$ 
is of Picard type.
\end{proposition}

\begin{proof}
Let $\wX\to X$ be a resolution of 
singularities which is a simultaneous
resolution of $f^o\colon X^o\to Y^o$.
Since $X$
has rational singularities, we have 
$h^2(\wX,\cO_\wX)=0$.
Thus $H^2(\wX,\bC)=H^{1,1}(\wX)$. 
Since $\wX^o$ is smooth, 
the mixed Hodge structure on $H^2(\wX^o,\bQ)$
has weights $\geq 2$.
The natural restriction map 
$$H^2(\wX,\bQ)\twoheadrightarrow W_2H^2(\wX^o,\bQ)$$ 
surjects onto the weight $2$ part
of the mixed Hodge structure.

Consider the Leray spectral sequence
for the smooth 
morphism $\widetilde{f}^o\colon \wX^o\to Y^o.$ 
By the relative hard Lefschetz theorem 
of Blanchard--Deligne,
this spectral sequence
degenerates at the $E_2$ page,
giving a surjection 
$$H^2(\wX^o,\bQ)\twoheadrightarrow 
H^0(Y^o,R^2\widetilde{f}^o_*\underline{\bQ}) =
H^0(Y^o,R^2f^o_*\underline{\bQ}\oplus \bU'_{\bQ}) 
= (\bU_\bQ)^\rho \oplus (\bU'_\bQ)^\rho,$$
where $\bU'$ is the local system
whose stalk at $y \in Y^o$ is generated by the 
cycle classes of the exceptional divisors of the
resolution $\wX^o_y \to X^o_y$, 
as in \Cref{ex:moduliKtrivialfibration}.
Since $\bU$ has pure weight $2$
by Proposition \ref{pure},
the weight $2$ part $W_2H^2(\wX^o,\bQ)$ surjects
onto $(\bU_\bQ)^\rho$. 
We conclude that there is a surjection
\begin{align}\label{11-lift}
H^{1,1}(\wX,\bQ)\twoheadrightarrow (\bU_\bQ)^\rho
\end{align}
and so the Hodge structure on the latter
is of Hodge--Tate type.
\end{proof}

\begin{remark}\label{pic-aut}
For a projective
variety with rational singularities, 
we have Hodge 
symmetry: $h^2(X,\cO_X)=h^0(X,\Omega_X^{[2]})$.
So any fibration from a CY variety 
of dimension $\geq 3$ is of Picard type.
The same holds for fibrations with primitive
symplectic total space $X$. In that case, 
$h^2(X,\cO_{X})=h^2(\wX,\cO_{\wX})=1$, but the
surjection $H^2(\wX,\bQ)\twoheadrightarrow 
(\bU_\bQ)^\rho$ sends
$H^{2,0}(\wX)$ to zero,
since the fibration is Lagrangian,
so the proof of Proposition \ref{is-pic} still applies.
In fact, in this case, more is true:
$(R^2f^o_*\underline{\bZ})_{\rm tf}^\rho$
has rank $1$, see \cite[Lem.~2.2]{Matsushita2016}.
\end{remark}

\subsection{Moduli of lattice-polarized abelian varieties}
\label{sec:lattice_av}

Fix a reference free $\bZ$-module $V$
of rank $2g$. 

\begin{definition} The {\it period
domain of complex tori} is
$$
\bD  \coloneqq  
\{F^1\in {\rm Gr}(g,V_\bC)\,:\,F^1\cap \oF\,\!^1=0\}
$$
i.e., the space of weight $1$ (unpolarized)
Hodge structures on $V$.
\end{definition}

Given a point $F^1\in \bD$, we have a map
$V_\bC^{\ast}\to (F^1)^{\ast}$. We define a complex
torus as the quotient 
$A=A(F^1)  \coloneqq  (F^1)^{\ast}/V^{\ast}.$ Thus,
we have canonical identifications
$H_1(A,\bZ)= V^{\ast}$, $H^1(A,\bZ)= V$.
We have that $\bD$ is a fine
moduli space, in the analytic category,
for complex tori $A$, with 
an isomorphism of free $\bZ$-modules
$\varphi\colon 
H^1(A,\bZ)\to V$, i.e., a marking.

\begin{definition}
Let $\Lambda\subset \wedge^2V$
be a nonzero saturated sublattice of alternating
forms on $V^{\ast}$.
The {\it period domain
of $\Lambda$-polarized complex tori} is
$$
\bD_\Lambda  \coloneqq  
\{F^1\in \bD\,:\,\Lambda\subset H^{1,1}  
\coloneqq  F^1\wedge \oF\,\!^1\}.
$$
\end{definition}

Note that $\bD_\Lambda$ is a locally
closed set in the real Zariski topology
on ${\rm res}_{\bC/\bR}\,{\rm Gr}(g,V_\bC)$,
so it has finitely many connected components. 
Let
$$\Sp_\Lambda  \coloneqq  
\{\gamma\in\GL(V)\,:\,(\wedge^2\gamma)(L)=L
\textrm{ for all }L\in \Lambda\}.
$$ 
Then $\Sp_\Lambda$ acts naturally
on $\bD_\Lambda$.

\begin{proposition}\label{pol-const} 
Let $\bD_{\Lambda,0}$ be a connected
component of $\bD_\Lambda$. Then if $L\in \Lambda$
defines a polarization 
for some $F^1\in \bD_{\Lambda,0}$,
it defines a polarization for all $F^1\in \bD_{\Lambda,0}$.
\end{proposition}

\begin{proof}
Since $L$ is real by hypothesis, the condition
$L\in H^{1,1}$ amounts to the 
statement that the subspace
$K\coloneqq \ker(V_\bC^*\to (F^1)^*)\subset V^*_\bC$
is isotropic for $L$.

Note $L\in \Lambda$ can only define 
a polarization if $L$ is non-degenerate
on $V^{\ast}$.
The condition $V_\bC^* = K\oplus \oK$, 
together with the non-degeneracy of $L$
and that $K$ is isotropic,
imply that $L$ induces a non-degenerate
pairing between $K$ and $\oK$.
Equivalently, 
the hermitian form 
$$h_L(\alpha,\beta) 
\coloneqq  
iL(\alpha,\overline{\beta})$$
on $K$ is non-degenerate. Hence $h_L$ is
non-degenerate at all points of a connected component
$\bD_{\Lambda,0}\subset \bD_\Lambda$. It follows
that the signature of $h_L$ is the same
for all $F^1\in \bD_{\Lambda,0}$. The proposition follows,
as $L$ defines a polarization for 
$F^1\in \bD_\Lambda$ if and only if $h_L$ is 
positive-definite.
\end{proof}

\begin{definition} A {\it polarizable}
connected component of 
$\bD_{\Lambda,0}\subset \bD_\Lambda$ is a connected 
component on which there exists some $L\in \Lambda$
which defines a polarization on some
(equivalently, any) $F^1\in \bD_{\Lambda,0}$. 
We define ${\rm Sp}_{\Lambda,0}
\subset {\rm Sp}_\Lambda$ as the 
stabilizer of this connected component.
\end{definition}

\begin{proposition} \label{shimura-var}
The action of $\Sp_{\Lambda,0}$ on a polarizable component
$\bD_{\Lambda,0}$
is properly discontinuous, and the quotient
is a Shimura variety.
\end{proposition}

\begin{proof}
Since $\bD_{\Lambda,0}$ is polarizable,
we have a closed embedding of period domains
$\bD_{\Lambda,0}\hookrightarrow \cH_g$
into the Siegel space associated to
the polarization $L$ of type ${\bf d}$.
Then ${\rm Sp}_{\Lambda,0}$ can be identified
with a finite index subgroup
of the stabilizer of the Mumford--Tate
subdomain of $\cH_g$ corresponding
to the condition $\Lambda\subset H^{1,1}$.
\end{proof}

In fact it follows that there is a DM stack,
defined as in \S~\ref{av-fib},
$$
\cA_{g,\Lambda}  \coloneqq  
\big[
\Sp_\Lambda\backslash  \!\!\!\!\bigsqcup_{\rm polarizable} \!\!\!\!\bD_{\Lambda,0}\big]
$$
of $\Lambda$-polarizable abelian varieties with origin.
It serves as an appropriate target
for period maps of abelian fibrations
 of Picard type, with rational section.
 When $\Lambda = \bZ L$, 
 for $L$ of type ${\bf d}$,
 $$\bigsqcup_{\rm polarizable} 
 \!\!\!\!\bD_{\Lambda,0}= 
 \cH_{g,{\bf d}}\sqcup \overline{\cH}_{g,{\bf d}}$$
 and $\cA_{g,\Lambda} = \cA_{g,{\bf d}}
 \sqcup \overline{\cA}_{g,{\bf d}}$
 with the two components 
 corresponding to 
 when $h_L$ is positive-, 
 resp.~negative-definite (that 
 is, when $L$ or $-L$ is ample).
On the other connected components
of $\bD_\Lambda$ where $h_L$ has indeterminate
signature $(r,s)$, i.e.~$r,s\geq 1$, 
Proposition \ref{shimura-var} is false;
the action of $\Gamma_{g,{\bf d}}$ is {\it not}
properly discontinuous.

Associated to an abelian fibration 
$f\colon X\to Y$, write $\bV \coloneqq 
R^1f_*^o\underline{\bZ}$ for the corresponding
rank $2g$ local system on $Y^o$.

\begin{proposition}\label{pic-per}
Let $f\colon X\to Y$ be an abelian
fibration of Picard type, 
which is smooth over an open set $Y^o\subset Y$.
Let $\Lambda  \coloneqq  H^0(Y^o,\wedge^2\bV)
\subset \wedge^2\bV_*$
for a basepoint $*\in Y^o$.
There is a well-defined classifying morphism
\begin{align*} \Phi^o\colon Y^o&\to \cA_{g,\Lambda} \\
y&\mapsto {\rm Aut}^0(X_y).
\end{align*}
If furthermore, $f$ admits a rational section 
and $Y^o$ is smooth,
then $f^o\colon X^o\to Y^o$ is the pullback
of the universal family along $\Phi^o$.
\end{proposition}

\begin{proof}
The first part of the 
proposition follows from the definition
of Picard type, since $\Lambda$ lies in $H^{1,1}$
of every fiber, and the fact that 
the monodromy representation
$\rho$ factors through $\Sp_{\Lambda,0}$
because $\Lambda=(\wedge^2\bV)^\rho$. 

To see the second part, note
that the rational section
is regular over $Y^o$ by Lemma
\ref{section-extend} and $\cA_{g,\Lambda}$
is the DM stack parameterizing 
$\Lambda$-polarizable abelian varieties,
with origin. We conclude that $f^o\colon X^o\to Y^o$
is the pullback of the universal family
as in \S~\ref{av-fib}.
\end{proof}

\begin{remark}\label{pol-const-2}
By Proposition \ref{shimura-var},
each connected component of $\cA_{g,\Lambda}$
maps finitely 
to some $\cA_{g,{\bf d}}$ (as a DM stack),
with ${\bf d}$ depending only on $\Lambda$
and the connected component.
\end{remark}

\subsection{The Zarhin trick}\label{sec:Ztrick}

Fix ${\bf d}  \coloneqq  (d_1, \dots, d_g)$ a sequence
of $g$ positive integers, each dividing the next,
with $d_1=1$. By Section \ref{av-fib}, the moduli stack
of ${\bf d}$-polarized abelian varieties $(A,L)$
with origin is the quotient 
$\cA_{g,{\bf d}}=[\Gamma_{g,{\bf d}}\backslash \cH_g]$.
Recall that a polarized abelian variety $(A,L)\in 
\cA_{g,\bf d}$ has a dual
abelian variety $(A^{\ast},L^{\ast})\in 
\cA_{g, {\bf d}^{\ast}}$
where $${\bf d}^{\ast} = 
(d_g/d_g,  d_g/d_{g-1},\cdots, d_g/d_2, d_g).$$

We recall the following theorem, due to Zarhin, 
cf.~\cite[Thm.~1.1]{zarhin}:

\begin{theorem}\label{thm:zarhin}
Let $(A,L)$ be a polarized abelian variety. Then 
$Z(A)=(A\oplus A^{\ast})^{\oplus 4}$ admits a principal
polarization $L_Z$ for which there is a natural embedding
$${\rm End}(A,L)\hookrightarrow {\rm End}(Z(A),L_Z).$$
\end{theorem}

Furthermore, the construction in Theorem 
\ref{thm:zarhin} is functorial and can 
be performed in families. Thus it defines a morphism
$$Z\colon \cA_{g,{\bf d}}\to \cA_{8g}$$
of DM stacks, lifting on universal covers to a map of 
period domains $\wZ\colon \cH_g\hookrightarrow \cH_{8g}$ and 
corresponding to an inclusion 
$\Gamma_{g,{\bf d}}\hookrightarrow {\rm Sp}(16g,\bZ)$.
Note though that the Zarhin trick is not ``unique''
in the sense that the map $Z$ does depend on some
additional choices of initial data \cite[\S~4]{zarhin}: 
A solution of \begin{equation}\label{zarhin-choice}
a^2+b^2+c^2+d^2\equiv -1\textrm{ mod }
\prod_{i=1}^g d_i.
\end{equation} 

For an abelian fibration of Picard type, the period
map
$\Phi\colon Y\dashrightarrow A_{g,\Lambda}$
(see Prop.~\ref{pic-per})
lands in a connected component $A_{g,\Lambda,0}$.
We fix, for once and all, a chain
of maps
\begin{align}
\label{inclusions}
\xymatrix{
\cA_{g,\Lambda,0} 
\ar[r] 
&
\cA_{g,{\bf d}}
\ar[r]^{Z}
&
\cA_{8g}
}
\end{align}
for all possible $\Lambda$, and connected components.
This requires choosing some $L\in \Lambda$
of type ${\bf d}$, defining a polarization
over $\cA_{g,\Lambda,0}$ (cf.~Rem.~\ref{pol-const-2})
and choosing a Zarhin trick (\ref{zarhin-choice})
for  each ${\bf d}$.
Denote the composition
of $\Phi$ with these inclusions, on coarse spaces,
by $$\Phi_Z\colon Y\dashrightarrow A_{8g}.$$
Similarly, $\Phi_Z^o\colon Y^o\to \cA_{8g}$
is defined as the composition of the classifying 
morphism with the above maps, on DM stacks. Then
$\Phi_Z$ is the composition of $\Phi_Z^o$
with the coarsening map.

\begin{definition}\label{z-period}
The {\it Z-period map}
is $\Phi_Z$. The {\it Z-classifying morphism}
is $\Phi_Z^o$.
\end{definition}

\subsection{Moduli of primitive symplectic
varieties}\label{sec:moduli-ps}

Let $f\colon X\to Y$ be a primitive symplectic 
fibration.
Since the fibers of $y\in Y^o$ are PS varieties,
the Hodge structure on $H^2(X_y,\bZ)$
is pure of weight $2$ and admits a natural
symmetric $\bZ$-valued 
bilinear form, the 
{\it Beauville--Bogomolov--Fujiki} 
(or {\it BBF}) {\it form} \cite{Namikawa:deformation, 
Kirschner, Matsushita2015, Schwald2020, BL2021}. 
Thus, we have a local system 
$\bU_{\rm np}  \coloneqq  
(R^2f^o_*\underline{\bZ})_{\rm tf}$
(np for ``non-polarized'')
admitting a bilinear form $\psi$. Let 
$$\Lambda  \coloneqq  H^0(Y^o,\bU_{\rm np})=
(\bU_{\rm np})^\rho$$
denote the monodromy-invariant classes.
These form a sub-local system which is a constant
Hodge sub-structure $\underline{\Lambda}$ 
of weight $2$ on the fibers
over $Y^o$. It is nonzero because $f$ is projective.

Define
a saturated sub-$\bZ$-local system 
$\bU\subset \bU_{\rm np}$
by setting $\bU \coloneqq
\underline{\Lambda}^\perp
\subset \bU_{\rm np}$
where the perpendicular is taken 
with respect to the BBF form
$\psi$. When $f$ is of Picard type, 
$\bU$ underlies a $\bZ$-PVHS
of type $(1,m,1)$ for some
$0\leq m\leq b_2(X_y)-3$.

Let $(M,\cdot)$, $(M_{\rm np},\cdot)$ be 
lattices isometric to $(\bU_*,\psi_*)$, $(\bU_{{\rm np},*},\psi_*)$ respectively, 
for some
basepoint $*\in Y^o$. So $\Lambda^\perp=M$,
$M^{\perp}=\Lambda$
are mutually perpendicular, saturated
sublattices of $(M_{\rm np},\cdot)$.
The {\it Type IV domain} $\bD_M$ 
associated to a lattice $M$
of signature $(2,m)$
is one of the two connected components
of $\bP\{x\in M_\bC\,\big{|}\,
x\cdot x=0,\,x\cdot \bar x>0\}$.
Then $\bD_M$ parameterizes the 
space of polarized weight $2$
Hodge structures on $M$ of type $(1,m,1)$.
Let $O^*(M)$ denote the subgroup of 
the isometries
of $M$ which admit an extension to 
$M_{\rm np}$ acting
trivially on $\Lambda\subset M_{\rm np}$ and
preserving the connected component $\bD_M$.
Define $$\cF_\Lambda \coloneqq 
[O^*(M)\backslash \bD_M].$$ It serves as the natural 
classifying space of Hodge
structures for $\Lambda$-polarized
PS varieties. While $\bD_M$
depends only on the abstract
isometry type of $M$, the subgroup
$O^*(M)\subset O(M)$ depends implicitly
on $\Lambda$ and
the embedding $M\subset M_{\rm np}$.
Let $F_\Lambda$ and $\oF_\Lambda$ 
be the coarsening and
its Baily--Borel compactification, 
respectively.
We have a classifying morphism
for the Hodge structures
$$\Phi^o\colon Y^o\to \cF_\Lambda.$$

By the Borel extension theorem (Thm.~\ref{borel-ext}),
$\Phi^o$ extends as a coarse period map 
$\Phi\colon Y\to \oF_\Lambda$ when 
$f\colon X\to Y$
is a standard snc model. There is a Hodge 
$\bQ$-line bundle $\lambda_\Lambda$ on the coarse
space associated to the tautological 
$\bZ$-PVHS of weight $2$ and $K$-trivial type
over $\cF_\Lambda$. 
There is a constant $a \in \NN$, depending only
on $b_2(X)$, 
so that \begin{equation}\label{ps-lambda}a \bM Y. 
\sim a  \Phi^{\ast}(\lambda^{\otimes d}_\Lambda),\end{equation}
where $2d$ is the fiber dimension of $f$. 
This follows from the same argument of 
\Cref{lem:proportionality}, where we further 
apply \cite[Lem. 3.1]{Kim2024} on $Y'$ (notation as 
in the proof of \Cref{lem:proportionality}), which 
accounts for the occurrence of the $d$-th 
power---$H^{2d,0}(X_y)\simeq H^{2,0}(X_y)^{\otimes d}$ 
for general $y\in Y$.

The {\it inertia stratification}
of $F_\Lambda$ is the locally closed 
stratification by loci in $F_\Lambda$
which have a fixed inertia group
in the DM stack $\cF_\Lambda$.

\begin{definition}
Let $\Phi^{-1}(F_\Lambda)$ be the maximal 
open set in $Y$ where the period map $\Phi$ is a morphism.
The {\it inertia jumping locus} 
$\partial_{\rm I}\subset Y$ is the 
closure in $Y$ of the locus
of points $y \in \Phi^{-1}(F_\Lambda)$ 
such that $\Phi(y)$ lands in a deeper
stratum of the inertia stratification
than the general point of $Y$. 
\end{definition}

\Cref{extend4} 
is a weak analogue 
of Proposition \ref{extend3} for 
primitive symplectic fibrations.

\begin{proposition}\label{extend4}
Let $f\colon X\to Y$ be a primitive symplectic
fibration, such that $K_X \sim_{\QQ,f} 0$ 
and $\bfM$ descends onto the smooth base $Y$.
Let $Y^o\subset Y$ be an
open set over which $f$ is locally trivial
and $B_Y$ be the boundary divisor.
The local system 
$\bU_{\rm np}=(R^2f^o_*\underline{\bZ})_{\rm tf}$ 
on $Y^o$ and the $\bZ$-VHS it underlies
extend to the open set 
$U \coloneqq  Y\setminus 
( \partial_{\rm H}\cup \partial_{\rm I}\cup \Supp B_Y)$.
\end{proposition}

\begin{proof}
Let $\partial_i$ be an irreducible divisor in 
$U \setminus Y^o$, and $D$ be a small analytic  
neighborhood of a general point $0 \in \partial_i$. 
By restricting $f$ to $D$, we
obtain a family $g\colon
\cX\to (D,0)$, for which $K_\cX\sim_{\QQ,g} 0$, and 
$\cX^{\ast}=g^{-1}(D^{\ast})$ is a locally trivial 
fibration of primitive symplectic varieties over $D^{\ast} \coloneqq D \setminus \partial_i$.
As in the proof of Proposition \ref{extend3}, 
if $\partial_i\not\subset 
\partial_{\rm H} \cup \Supp B_Y$, then $\cX_t$ is a $K$-torsion variety for any $t \in D$. Following
the argument of
\cite[Sec.~4, Proof of Thm.~0.7]{KLSV},
$\cX_t$ is in fact a primitive symplectic variety. Set
$Z  \coloneqq  \cX_0$. 

Choose a $\QQ$-factorial terminalization 
$\pi \colon Z' \to Z$. Since 
$\pi_* \mathcal{O}_{Z'} = \mathcal{O}_{Z}$ and 
$R^1\pi_* \mathcal{O}_{Z'} = 0$, there exists a 
commutative square
\begin{align*}
\xymatrix{
\mathcal{Z}' \ar[r] \ar[d]_-{h'} & \mathcal{Z}\ar[d]^-{h}\\
\Def(Z') \ar[r]^-{p} & \Def(Z)
}
\end{align*}
for the universal deformations 
of $Z'$ and $Z$, see \cite[Prop.~11.4]{KM92},
\cite[Lem.~4.6]{BL2022}. Note that
\begin{enumerate}
\item $\Def(Z')$ and $\Def(Z)$ are smooth of the 
same dimension by \cite[Thm.~1]{Namikawa2006};
\item
\label{locally_trivial}
$\mathcal{Z}'$ is a locally trivial 
fibration by \cite[Main Thm.]{Namikawa2006};
\item
\label{surjective} 
$p$ is a surjective branched Galois covering, and 
its Galois group $G_{\partial_i}$ acts faithfully 
on $H^2(Z', \ZZ)$ via monodromy operators preserving 
the Hodge structure;
see \cite[Lem.~1.2, \S~6.1.1]{Markman2010} or 
\cite[Thm.~3.5]{LMP24}.
\end{enumerate}
Consider the classifying map $\phi \colon (D, 0) \to \Def(Z)$, and let $D'$ be the normalization of an 
irreducible component of $D \times_{\Def(Z)} \Def(Z')$. They sit in the commutative square
\[
\xymatrix{
(D',0) \ar[r]^{p_D} \ar[d]_-{\phi'}& (D,0)\ar[d]^-{\phi}\\
\Def(Z') \ar[r]^-{p} & \Def(Z).
}
\]
Let $f'\colon (\phi')^{\ast}\cZ'\to (D',0)$ be the
pullback deformation along $\phi'$.

If $p_D$ 
is not an isomorphism, then $\phi '(0)$
is fixed 
by some element in $G_{\partial_i}$ that does not fix 
the whole image $\phi'(D')$, so 
$\partial_i \subset \partial_{\rm I}$
by \eqref{surjective}.
Otherwise, 
$p_D$ is an isomorphism,
so we can identify $D'$ with $D$.
By \eqref{locally_trivial}, 
the sheaf $R^2f'_*\underline{\bZ}$ is a trivial 
local system on $D'$, and $\bU_{\rm np}|_{D^{\ast}}$ 
is a sub-$\bZ$-local system of 
$(R^2f'_*\underline{\bZ})|_{D^{\ast}}$. 
Hence, $\bU_{\rm np}|_{D^{\ast}}$ extends over $D$.

It follows that $\bU_{\rm np}$ extends
to the generic point of $\partial_i$.
Thus, $\bU_{\rm np}$ extends to a big
 open subset of $U$, hence to all of $U$. 
\end{proof}

\begin{definition}\label{ps-moduli}
Let $\cM_\Lambda^{\rm q}$ be the moduli stack
of $\Lambda$-quasipolarized 
PS varieties. It parametrizes
PS varieties $X$ which are locally-trivially
deformation equivalent to a fixed one,
and which admit a primitive embedding 
$j\colon
\Lambda\hookrightarrow {\rm NS}(X)$ such that
$j(\sigma)$ consists of big and nef classes, where 
$\sigma\subset \Lambda_\bR$ is a choice of 
{\it small cone}, cf.~Def.~\ref{def:small-cone}.
Define $\cM_\Lambda\subset \cM_\Lambda^{\rm q}$
as the open substack for which $j(\sigma)$
consists of ample classes. It is 
the moduli stack of $\Lambda$-polarized
PS varieties.
\end{definition}

\begin{definition}\label{def:small-cone}
Let $\cC^+\subset \Lambda_\bR$ 
denote the positive cone of a primitive,
hyperbolic sublattice $\Lambda\subset M_{\rm np}$, and $M_{\rm np}^{\rm term}\supset M_{\rm np}$ be the BBF lattice of a $\bQ$-factorial terminalization.
A {\it potential wall} is a hyperplane in $\cC^+$
of the form $\beta^\perp\cap \cC^+$ where
$\Lambda+\bZ \beta\subset M_{\rm np}^{\rm term}$ 
is a hyperbolic sublattice of $M_{\rm np}^{\rm term}$, 
for which $0>\beta^2\geq -B$
for the constant $B$ of \cite[Prop.~7.7]{LMP24}. 
A {\it small cone} $\sigma\subset \cC^+$ is
a connected component of the complement of the 
locally polyhedral arrangement defined by the 
collection of all potential walls 
\cite[Def.~4.2, Def.~4.9, Prop.~4.13]{AE2025}. 
\end{definition}

These definitions require $b_2(X^{\rm term})\geq 5$ to 
apply \cite[Prop.~7.7]{LMP24}, but are 
valid/vacuous also when $\rk \Lambda=1$.
Note that in the 
special case $b_2(X^{\rm term})=4$, $\rk \Lambda=2$,
the moduli is rigid, so fibrations
in such PS varieties are necessarily isotrivial.
The $b_2(X^{\rm term})\leq 4$ 
cases will be treated separately in the proof
of Theorem \ref{bd-families}(PS).

\begin{remark} Walls arise from extremal 
contractions, see \cite[Prop.~7.2]{LMP24}.
For K3 surfaces, we have $\beta^2=-2$,
and all walls correspond to 
divisorial contractions of 
quasi-polarized 
K3 surfaces, to polarized
K3 surfaces with ADE singularities. 
But for higher-dimensional PS varieties, 
there may be both divisorial 
and small contractions. 
\end{remark}

There are only
finitely many orbits of small cones
under the action
of isometries of $\Lambda$ extending trivially
to $M_{\rm np}$.
Given any
embedding $\Lambda\hookrightarrow {\rm NS}(X)$,
an element $\lambda\in \sigma$ is big and nef,
resp.~ample, on $X$
if we and only if any other $\lambda'\in \sigma$
is. So there are only finitely many 
possible moduli spaces of 
$\Lambda$-(quasi)polarized
PS varieties.
Combining \cite[Cor.~10]{MST20} and 
\cite[Thm.~4.7]{AE2025}, one may deduce
$\cM_\Lambda^{\rm q}$ is a finite type DM stack.
Observe that we have a surjective
classifying morphism
$\Phi^o_\Lambda\colon
\cM_\Lambda^{\rm q}\to \cF_\Lambda$ assigning to
$[X]\in \cM_\Lambda^{\rm q}$ the Hodge structure 
$[H^2(X)_{\rm prim}]\in \cF_\Lambda$ 
in the corresponding stack of weight $2$ $\bZ$-PVHS;
it is quasi-finite by local Torelli.

\begin{definition} Choose a small cone
$\sigma\subset \Lambda_\bR$ 
and suppose that 
$\beta\in M_{\rm np}\setminus \Lambda$
gives a potential wall intersecting 
(equivalently, containing) $\sigma$. 
The corresponding {\it Heegner
divisor of potential wall type} is the image
in $\cF_\Lambda=[O^*(M)\backslash \bD_M]$ 
of the hyperplane 
section
$\{[x]\in \bD_M\,\big{|}\,x\cdot \beta=0\}$. 
\end{definition}

The bound $\beta^2\geq -B$ of 
\cite[Prop.~7.7]{LMP24}, together with
the condition that $\beta^\perp\cap \cC^+$
is non-empty, imply that 
there are only finitely many Heegner divisors
of potential wall type.

Note that 
$\cM_\Lambda$ is separated---supposing that $g\colon \cX\to (D,0)$ is a family of $\Lambda$-polarized PS varieties over a disk,
any other family of $\Lambda$-polarized
PS varieties $g'\colon \cX'\to (D,0)$ that is isomorphic to $g$ over the punctured disk $D^*$, is isomorphic
to $g$ itself. 
This holds
because $\cX$ is the relatively ample model, 
for any $L\in \sigma$.
On the other hand, $\cM_\Lambda^{\rm q}$ is in
general not separated, since there can 
be multiple locally trivial 
fillings of a punctured family 
$\cX^*\to D^*$ for which
$L\in \sigma$ extends to a big and nef class 
on the central fiber. Such fillings will be
related by birational modifications along
$L$-trivial curve classes, a phenomenon
which is already present in moduli of 
quasipolarized K3 surfaces.

\begin{definition}\label{wall-jump}
The {\it wall jumping locus} 
$\partial_{\rm W}\subset Y$ is the closure
in $Y$ of the locus of points $y\in Y$ where
the number of branches
of Heegner divisors of potential
wall type passing through $\Phi^o(y)$ increases.
\end{definition}

\begin{proposition} \label{extend5}
Let $f\colon X\to Y$ be a primitive symplectic
fibration over a smooth base $Y$, 
whose general fiber has $\bQ$-factorial
terminalization with $b_2\geq 5$, 
such that $K_X \sim_{\QQ,f} 0$.
Let $Y^o\subset Y$ be an
open set over which $f$ is locally trivial.
Let 
$\Psi^o\colon Y^o\to \cM_\Lambda$
be the
classifying morphism to the moduli stack
of $\Lambda$-polarized
PS varieties, associated to a small cone 
$\sigma$ containing a relative polarization $L$. 
Then $\Psi^o$ extends
along $U \coloneqq  
Y\setminus 
(\partial_{\rm H}\cup 
\partial_{\rm I}\cup \partial_{\rm W}
\cup \Supp B_Y)$.
\end{proposition}

\begin{proof}
Let $U'\subset U$ be the big open subset
with equidimensional fibers for $f$, and let
$D\ni 0$ be a small analytic
neighborhood of $0\in U'$ with $D^*\coloneqq D\cap Y^o$.
Following the proof of
Proposition \ref{extend4},
there is a locally trivial family
$f'\colon \cX'\to D$ of $\bQ$-factorial
terminal PS varieties, where 
$f'\colon \cX'\to D$ is a simultaneous
$\bQ$-factorial terminalization 
of the fibers of $f\vert_{D}\colon \cX\to D$,
$\cX\coloneqq f^{-1}(D)$. 
Let $\pi\colon \cX'\to \cX$ be this
birational morphism, and consider
$\pi^*(L\vert_{\cX})$,
which is relatively big and nef over $D$. The
condition that $U'\subset U$ maps into 
the complement of $\partial_{\rm W}$ grants
that the curve classes intersecting 
$\pi^*(L\vert_{\cX})$ to be zero are the same on all 
fibers of $f'\colon \cX'\to D$. Hence 
$f\vert_D\colon \cX\to D$ must be locally trivial,
as $L$ is relatively ample.
Thus, we may choose $Y^o\subset Y$ to contain $U'$. 

Let $M_\Lambda$ be the coarse space of $\cM_\Lambda$.
The normality of $U$ implies that 
$U'\to M_\Lambda$ (the composition
of $\Psi^o\vert_{U'}$ with the coarsening map), 
admits an extension $U\to M_\Lambda$.
This follows because the Hodge-theoretic
classifying morphism extends to
$\Phi^o\colon U\to \cF_\Lambda$ 
by Proposition \ref{extend4}, and the restriction
of the morphism $M_\Lambda\to F_\Lambda$ to
the complement of $\partial_{\rm W}$ is finite.
Thus, we have a morphism $U\to M_\Lambda$ which
lifts over a big open set to the stack $\Psi^o\vert_{U'}\colon
U'\to \cM_\Lambda$.
We deduce, by 
the smoothness of $U$, and 
an appropriate extension of \cite[Lem.~2.4.1]{AO2002}
to higher dimension, 
cf.~\cite[Rem.~2.4.2]{AO2002} or \cite[Lem.~4.5]{mj2025}, 
that $\Psi^o\vert_{U'}$ extends uniquely to a classifying
morphism $\Psi\colon U\to \cM_\Lambda$.
\end{proof}

\subsection{Kuga--Satake construction} 
\label{sec:kuga}

We refer to \cite[Ch.~4 and Ch.~6 \S 4.4]{huybrechts2016}
for proofs of claims below.
Let $(M,\cdot)$ be a lattice of 
signature $(2,m)$.
Recall that the {\it Clifford algebra} of $M$ is
the following quotient of the tensor algebra:
$${\rm Cl}(M)  \coloneqq  
(\textstyle \bigoplus_{p\geq 0} M^{\otimes p})/
(v\otimes v - (v\cdot v)).$$
There is a filtration of ${\rm Cl}(M)$
induced by tensor degree, and we have
a canonical isomorphism 
${\rm gr}_p {\rm Cl}(M)\simeq \wedge^p M$ with the 
wedge product. Thus, ${\rm Cl}(M)$ is a 
free $\bZ$-module of rank $2^{m+2}$. 

Given $[x]\in \bD_M$, we may 
suppose
$x=e_1+ie_2$ with
$e_1,e_2\in M_\bR$ an oriented orthonormal
basis of $\langle x,\bar x\rangle\cap M_\bR$.
Then $J_{[x]}  \coloneqq  e_1 e_2\in {\rm Cl}(M_\bR)$
is independent of the 
basis and satisfies $J^2_{[x]}=-1$. 
Left multiplication by $J_{[x]}$ induces a 
complex structure on ${\rm Cl}(M_\bR)$. 

\begin{definition}  \label{ks-primary}
The {\it primary Kuga--Satake torus $KS'(x)$} is
$$KS'(x) \coloneqq {\rm Cl}(M_\bR)/{\rm Cl}(M),$$ 
endowed with the complex structure defined by $J_{[x]}$.
\end{definition}

Thus $J_{[x]}$ defines
a weight $-1$ Hodge structure on
${\rm Cl}(M)=H_1(KS'(x),\bZ)$ and 
a weight $1$ Hodge structure on
${\rm Cl}(M)^{\ast}=H^1(KS'(x),\bZ)$.
We denote the corresponding
period map $\mathfrak{Ks}' \colon \bD_M \to \bD$ 
where the target is the period domain 
of weight $1$ unpolarized Hodge 
structures on ${\rm Cl}(M)^{\ast}$, 
see \S~\ref{sec:lattice_av}.

To show that $KS'(x)$ is in fact an abelian 
variety, we define a polarization. 
Fix orthogonal vectors $f_1, f_2\in M$.
Define an element of $\wedge^2{\rm Cl}(M)^{\ast}$ 
(i.e., an alternating form on ${\rm Cl}(M)$) by 
\begin{equation}\label{ks-pol}
(v,w)\mapsto {\rm tr}(f_1f_2v^{\ast}w),
\end{equation}
where ${\rm tr}(v)$ 
stands for the trace of the endomorphism of 
$\rm Cl(M)$ given by left multiplication by $v 
\in \rm Cl(M)$,
and $^{\ast}$ is the anti-automorphism defined by
$(m_1\cdots m_k)^{\ast}= m_k\cdots m_1$, $m_i\in M$.
When $f_1\cdot f_1$ and $f_2\cdot f_2$ are 
both positive,
(\ref{ks-pol}) defines a polarization on the Kuga--Satake
torus, for all periods in the connected
component $\bD_M$ containing $[f_1+i f_2]$. 

In families, the Kuga--Satake construction
is quite subtle to define. Let $$KS_M'\to \bD_M$$ 
denote the pullback of the universal family of
unpolarized complex tori along the map 
$\mathfrak{Ks}'$. There is a natural
lift of the action of $O^*(M)$ on $\bD_M$
to $KS_M'$ by declaring that the
$O^*(M)$ action on $M^{\otimes p}$ is induced 
by the action on $M\simeq {\rm Cl}^1(M)$. 
{\it But this is may not be an algebraic 
family of Kuga--Satake varieties over 
$\cF_\Lambda$}. The issue
is that the polarization (\ref{ks-pol})
does not descend to the quotient
$[O^*(M) \backslash KS_M']$.

In order to fix this issue, we introduce the notion of 
secondary Kuga--Satake variety. To this end, we first 
recall some generalities about the Pin group. 
If ${\rm Cl}(M_\bQ)^*$ is the group of units of the 
Clifford algebra ${\rm Cl}(M_\bQ)$, then
\[
{\rm GPin}(M_\bQ) \coloneqq 
\{\gamma \in {\rm Cl}(M_\bQ)^* \,|\, 
\gamma M_{\bQ} \gamma^{-1} \subset M_{\bQ}\}.
\]
is the Clifford group.
It admits the orthogonal representation 
\begin{align*} 
\phi \colon \rm{GPin}(M_\bQ)&\to O(M_\bQ) \\
\gamma &\mapsto (v\mapsto (-1)^{\deg\gamma}\gamma v 
\gamma^{-1},\,\textrm{ for }
v\in {\rm Cl}^1(M_\bQ)\simeq M_\bQ),
\end{align*}
For instance, for any $m\in M_\bQ$ with $m \cdot m \neq 0$, 
i.e., $m\in M_\bQ \cap {\rm Cl}(M_\bQ)^*$, the 
transformation $\phi(m)$ coincides with the 
orthogonal reflection about the vector $m$. 
Since $(M, \cdot)$ is a non-degenerate quadratic space, 
$O(M_\bQ)$ is generated by reflections due to the 
Cartan--Dieudonné theorem 
(see e.g.~\cite[\S 5.3.9]{Procesi2007}).
So $\phi$ is surjective and 
sits in the exact sequence
\begin{equation}
1 \to \bQ^* \to {\rm GPin}(M_\bQ) 
\xrightarrow{\phi} O(M_\bQ)  \to 1. 
\end{equation}
The Pin group is the subgroup of the Clifford 
group of spinor norm 1, i.e.,
\begin{align*}
{\rm Pin}(M_\bQ) & \coloneqq 
\{\gamma \in {\rm GPin}(M_\bQ)\,| \, \gamma^*\gamma=1\} \\
&= \{\gamma = m_{1} \cdots m_{k} \in {\rm{Cl}}(M_{\bQ})\,| 
\, m_{i} \in M_{\bQ}, \,\gamma^*\gamma=1\},
\end{align*}
sitting in the exact sequence
\begin{equation}
1 \to {\rm Pin}(M_\bQ) \to 
{\rm GPin}(M_\bQ) \xrightarrow{N} \bQ^*  
\end{equation}
where $N(\gamma)= \gamma^*\gamma$ is the spinor norm. 
Since the spinor norm of elements in 
$\bQ^* = {\rm GPin}(M_\bQ)^0$ 
lies in $(\bQ^*)^2$, we obtain an exact sequence
\begin{equation}\label{eq:exactsequencePin}
1 \to \{\pm 1\} \to {\rm Pin}(M_\bQ) \xrightarrow{\phi} O(M_{\bQ}) \xrightarrow{N} \bQ^*/(\bQ^*)^2.
\end{equation}

Let
${\rm Pin}(M) \coloneqq \phi^{-1}(O(M))$ 
be the pre-image of the integral
isometries $O(M)\subset O(M_\bQ)$.
Elements of ${\rm Pin}(M)$ act 
on ${\rm Cl}(M_\bQ)$ by left multiplication,
and ${\rm Pin}(M)_\bZ \coloneqq 
{\rm Pin}(M)\cap {\rm Cl}(M)$ acts integrally, i.e., 
invertibly on ${\rm Cl}(M)$.

\begin{lemma}\label{lem:finiteindex} The following 
subgroups have finite index:
\begin{enumerate}
\item $\phi({\rm Pin}(M)) \subseteq O(M)$;
\item ${\rm Pin}(M)_\bZ \subseteq {\rm Pin}(M)$;
\item ${\rm Cl}(M) \subseteq 
\langle {\rm Pin}(M){\rm Cl}(M)\rangle$.
\end{enumerate}
\end{lemma}
\begin{proof} The exact sequence 
\eqref{eq:exactsequencePin} induces the short exact sequence 
\[1 \to \phi({\rm Pin}(M)) \to O(M) \to N(O(M)) 
\coloneqq {\rm im}(O(M) \xrightarrow{N} \bQ^*/(\bQ^*)^2) \to 1.\]
Since $O(M)$ is finitely generated and 
$\bQ^*/(\bQ^*)^2$ is a 2-torsion group, 
then $N(O(M))$ is a finite group, which implies (1). 
By construction, ${\rm Pin}(M)_\bZ$ is an arithmetic 
subgroup of ${\rm Pin}(M_{\bQ})$, so 
$\phi({\rm Pin}(M)_\bZ)$ is an arithmetic 
subgroup of $O(M_{\bQ})$ by \cite[Thm.~8.9]{Borel1969}. 
Since, by construction, $\phi({\rm Pin}(M)_\bZ)$ is 
contained in the arithmetic subgroup $O(M)$, 
we conclude that $\phi({\rm Pin}(M)_\bZ)$ has finite 
index in $O(M)$, so (2) holds. As a result, 
${\rm Pin}(M){\rm Cl}(M)$ generates an overlattice of 
finite index of the lattice 
${\rm Pin}(M)_{\ZZ}{\rm Cl}(M) = {\rm Cl}(M)$, 
which gives (3).
\end{proof}

\begin{definition}\label{secondary-ks}
The {\it secondary Kuga--Satake variety} $KS''(x)$
is $$ KS''(x) \coloneqq {\rm Cl}(M_\bR)/
\langle {\rm Pin}(M) {\rm Cl}(M)\rangle ,$$
endowed with the complex structure defined by $J_{[x]}$.
\end{definition}

By \Cref{lem:finiteindex}(3), 
there is a natural isogeny $KS'(x) \to KS''(x)$ between 
primary and secondary Kuga--Satake 
varieties. 
The advantage
of the secondary Kuga--Satake variety is that
now, the entire group ${\rm Pin}(M)$
acts integrally on $H_1(KS''(x),\bZ)$. As above,
we define $\mathfrak{Ks}''\colon \bD_M\to \bD$
to  be the period map to unpolarized 
complex tori assigning
$[x]\mapsto KS''(x)$.
Note that the polarization (\ref{ks-pol}) on $KS'(x)$
naturally defines a polarization also on $KS''(x)$.

\begin{lemma}\label{ks-pol-prop} The action
of ${\rm Pin}(M)$ on $\bD_M$ via $\phi$ lifts 
to an action $v\mapsto \gamma v$
on the universal secondary Kuga--Satake torus
$KS_M''\to \bD_M$ respecting
the polarization (\ref{ks-pol}).
\end{lemma}

\begin{proof}
See also \cite[Sec.~5.4]{rizov}, 
and \cite[Thm.~4.2.6]{pozzi}.
There are two 
things to check. The first is that the polarization
(\ref{ks-pol}) is preserved by the stated action.
The second is that the action
$v\mapsto \gamma v$ on 
$\langle {\rm Pin}(M) {\rm Cl}(M)\rangle $
induces a Hodge isometry from $KS(x)$ to 
$KS(\phi(\gamma)\cdot x)$, for all $x\in \bD_M$. 
To check the first, observe:
$$(v,w)\mapsto {\rm tr}(f_1f_2(\gamma v)^*(\gamma w)) 
= {\rm tr}(f_1f_2v^*\gamma^*\gamma w) = 
{\rm tr}(f_1f_2v^*w),$$
since $\gamma^*\gamma=1$.
To check the second, we show that
$v\mapsto \gamma v$
sends $H^{1,0}(KS(x))\to H^{1,0}(KS(\gamma\cdot x))$.
Equivalently, 
$$J_{[\gamma\cdot x]} = 
(\phi(\gamma) \cdot e_1)(\phi(\gamma) \cdot e_2) = 
\gamma (e_1e_2)\gamma^{-1} = 
\gamma J_{[x]}\gamma^{-1}.$$ 
The middle 
equality holds, since
$\phi(\gamma)\cdot m = (-1)^{\deg \gamma} 
\gamma m \gamma^{-1}$.
\end{proof}

Let $M_{\rm np}$ be an ambient lattice. As $M$ runs over 
all primitive sublattices of $M_{\rm np}$, 
the index $[O(M) : \phi({\rm Pin}(M))]$ 
is unbounded in general. To restore a universal bound, we 
must pass to $O^*(M)$.
For every
primitive sublattice $M\subset M_{\rm np}$, recall that 
$O^*(M) \subset O(M)$ denotes the subgroup of isometries 
of $M$ which admit an extension to $M_{\rm np}$ acting 
trivially on $M^{\perp}$, cf.~\S~\ref{sec:moduli-ps}. 
In particular, there is a natural inclusion 
$O^*(M)\subset O(M_{\rm np})$
given by extending by the identity map
on $M^\perp$, which fits into a commutative
diagram with the natural
inclusion ${\rm Pin}(M)\subset {\rm Pin}(M_{\rm np})$
sending $m_1\cdots m_k\mapsto m_1\cdots m_k$.

Given the orthogonal representations 
$\phi \colon {\rm Pin}(M) \to O(M)$ and $\phi_{\rm np} 
\colon {\rm Pin}(M_{\rm np}) \to O(M_{\rm np})$, 
define the group \[{\rm Pin}^*(M) \coloneqq 
\phi^{-1}(\phi_{\rm np}({\rm Pin}(M_{\rm np})) 
\cap O^*(M)) \subset {\rm Pin}(M).\]

\begin{lemma}\label{bounded-index}
The image of ${\rm Pin}^*(M)$ in $O^*(M)$ has 
finite index, uniformly bounded in a manner 
depending only on  $M_{\rm np}$. 

Furthermore, we may choose a finite
index subgroup ${\rm P}^*(M)\subset {\rm Pin}^*(M)$
of index bounded solely in terms of $\rk M$,
so that ${\rm P}^*(M)$ maps isomorphically
onto its image in $O^*(M)$.
\end{lemma}
\begin{proof}
Since $\phi({\rm Pin}^*(M)) = 
\phi_{\rm np}({\rm Pin}(M_{\rm np})) \cap O^*(M)$, the required 
uniform bound is given by the index 
$[O(M_{\rm np}) : \phi_{\rm np}({\rm Pin}(M_{\rm np}))]$. 

The second statement follows by taking, say,
the level $3$ subgroup 
${\rm P}^*(M)\subset {\rm Pin}^*(M)$,
whose index is bounded above by $3^{(\rk M)^2}$
and which does not contain the kernel 
$\{\pm 1\}$ of the map $\phi\colon 
{\rm Pin}(M)\to O(M)$, so maps isomorphically
onto its image via $\phi$.
\end{proof} 

Combining Lemmas \ref{ks-pol-prop}
and \ref{bounded-index}, we get a lift
of the ${\rm P}^*(M)$-action on $\bD_M$
to an action on the universal secondary 
Kuga--Satake torus $KS_M''\to \bD_M$ 
which preserves a polarization.
Thus, there is a nontrivial
``generic Néron--Severi lattice'' 
$\Lambda''\subset \wedge^2\langle
{\rm Pin}(M) {\rm Cl}(M)\rangle^{\ast}$,
containing a polarization,
which is are ${\rm P}^*(M)$-invariants 
(for the left action)
of the N\'eron--Severi lattice
of $KS''(x)$, for $[x]\in \bD_M$ very general.
We conclude that there is an embedding 
$\bD_M\hookrightarrow \bD_{\Lambda'',0}$
into a connected component of the
space of $\Lambda''$-polarizable complex tori,
which is equivariant with respect
to the action of ${\rm P}^*(M)$ on the source
and an appropriate action
of ${\rm Sp}_{\Lambda'',0}$
on the target. So there is a morphism
of DM stacks \begin{align}\label{secondary-ks-map}
\mathfrak{Ks}''\colon 
\cF_\Lambda''
\to \cA_{g'',\Lambda'',0}\end{align}
where $g''=2^{m+1}$, 
$\Lambda''=\Lambda''(M,\cdot)$ depends only
on the lattice $(M,\cdot)$ and the lift 
${\rm P}^*(M)$, and $\cF_\Lambda'' \coloneqq 
[{\rm P}^*(M)\backslash \bD_M]$ is an
$O^*(M)/{\rm P}^*(M)$-cover of the DM stack $\cF_\Lambda$.

Pushing forward the corresponding weight 1 
$\bZ$-PVHS along the \'etale cover $\cF_\Lambda''\to \cF_\Lambda$
of DM stacks, we get a morphism of DM stacks
\begin{align}\label{final-ks-1}
\mathfrak{Ks}\colon \cF_\Lambda\to \cA_{g,\Lambda_g,0}
\end{align}
where $g = 2^{m+1} [O^*(M):{\rm P}^*(M)]$,
and $\Lambda_g$ is the invariant 
N\'eron--Severi lattice,
which is non-empty since we may push forward
any element of $\Lambda''$.
On points, 
we have 
\begin{align}\label{final-ks-2}\mathfrak{Ks}\colon 
[x] 
\mapsto KS(x) \coloneqq  \!\!\!\!\! 
\bigoplus_{\gamma \,\in\, O^*(M)/{\rm P}^*(M)}
\!\!\!\!\! KS''(\gamma\cdot x).\end{align}

\begin{remark}\label{ks-differences}
To define the Kuga--Satake family over 
$\cF_\Lambda$, we had
to choose a lift of the action of ${\rm P}^*(M)$ 
on $\bD_M$ to
the universal secondary Kuga--Satake torus
$KS_M''\to \bD_M.$ There were two
natural choices,
the ``orthogonal'' and the ``pin''
lifted actions:
\begin{align*}\textrm{orthogonal:}&\,\, v\mapsto 
(-1)^{\deg \gamma} \gamma v\gamma^{-1} \\
\textrm{pin:}&\,\, v\mapsto 
\gamma v,
\end{align*}
for $\gamma\in {\rm P}^*(M)$,
$v\in {\rm Cl}(M)$.

For the pin action,
the invariant N\'eron--Severi lattice 
$\Lambda_g$
is quite 
large---for instance, it naturally
contains the entire subspace of 
$\alpha \in {\rm Cl}(M)$ for which 
$\alpha^*=-\alpha$, 
identified with the space of alternating
two-forms $(v,w)\mapsto {\rm tr}(\alpha v w)$.
On the other hand, it is unclear to us whether,
in general, the orthogonal lifted action preserves
a polarization at all (in which case, the resulting
family of complex tori over $[{\rm P}^*(M)\backslash \bD_M]
$ would not even be algebraic). 

Thus, we always will take the ``pin'' form
of the universal Kuga--Satake variety, at the expense
of passing to a finite cover of $\cF_\Lambda'' = 
[{\rm P}^*(M)\backslash \bD_M]$ whose degree is bounded
solely in terms of the ambient lattice 
$M_{\rm np}\supset M$, and then pushing forward,
back down to $\cF_\Lambda=[O^*(M)\backslash \bD_M]$.
The final pushforward allows us
to define the morphism $\mathfrak{Ks}$ directly from
$\mathcal{F}_{\Lambda}$ to a moduli space of abelian 
varieties, without passing to the finite cover 
$\mathcal{F}''_{\Lambda}$ 
(cf.~\cite[comment after Ch.~6, 
Prop.~4.10]{huybrechts2016}), 
at the price of increasing the dimension of the 
abelian varieties.
\end{remark}

\begin{definition} \label{ksz-period}
Consider a classifying morphism $\Phi^o\colon 
Y^o\to \cF_\Lambda$ for a $\bZ$-PVHS of type $(1,m,1)$
over $Y^o$ with monodromy
valued in $O^*(M)$. 
The {\it Z-classifying
morphism} $\Phi^o_{Z}$ is the composition 
of the classifying
morphism 
$\Phi^o$
with the chain of morphisms
of DM stacks
\begin{align*}
\xymatrix{
\cF_\Lambda
\ar[r]^{\mathfrak{Ks}} 
&
\cA_{g,\Lambda_g,0}
\ar[r]
&
\cA_{g,{\bf d}}
\ar[r]^{Z} 
&
\cA_{8g}.
}
\end{align*}
The first morphism 
is defined in 
(\ref{final-ks-1}) and (\ref{final-ks-2}); 
the second morphism comes from Remark \ref{pol-const-2}; 
the last morphism is the Zarhin trick 
defined in \S~\ref{sec:Ztrick}. 
If $Y\supset Y^o$ is a 
compactification, the coarsening of $\Phi^o_Z$,
viewed as a rational map $\Phi_{Z}\colon 
Y\dashrightarrow
\oA_{8g}$
is the {\it Z-period map}.
\end{definition}

\begin{remark}\label{bb-finite}
The chain of morphisms in Definition~\ref{ksz-period} 
(see also, Def.~\ref{z-period}) descends to the coarse
spaces and the Baily--Borel compactifications.
The resulting morphisms are finite,  
because the maps $\bD_M^+\to \bD_{\Lambda_g,0}^+\to 
 \cH^+_g\to  \cH^+_{8g}$ 
of rationally extended period domains 
(i.e., prequotients of the Baily--Borel 
compactifications, see \cite[I.~\S~3]{BB66})
are inclusions. 
\end{remark}

\subsection{Effective b-semiampleness}\label{sec:eff-semi-pf} 
We prove effective b-semiampleness, 
see \S~\ref{sec:eff-b-semi}, for lc-trivial 
fibrations in abelian or primitive symplectic
varieties with bounded $b_2$.

\begin{theorem}[Weak effective b-semiampleness, 
Theorem \ref{thm:b-semi}] 
\label{thm:weakeffective}
Let $f \colon (X,\Delta) \rar Y$ be an 
lc-trivial fibration
whose general fibers are either abelian 
varieties of dimension $g$, or primitive symplectic 
varieties of second Betti number $b_2$.
Let $(Y, B_Y, \bM.)$ be the generalized pair 
induced on $Y$. Then, there exists a constant $I$, 
depending respectively on $g$ or on $b_2$ only, 
such that $I\bM Y'.$ is free for a 
suitable birational modification $Y' \to Y$.
\end{theorem}

\begin{proof}
Associated to $f\colon X\to Y$, 
we have a rational
period map $$\Phi\colon Y\dashrightarrow 
\oA_{g,{\bf d}}
\textrm{ or }\Phi\colon 
Y\dashrightarrow \oF_\Lambda$$
where ${\bf d}$ is the (primitivized)
type of a relative polarization $L_X$ 
in the abelian case,
and $M=\Lambda^\perp\subset H^2(X_y,\bZ)_{\rm prim}$ 
is the primitive BBF lattice 
with respect to a relative
lattice polarization\footnote{For 
the purposes of this proof,
we may simply take $r=1$, 
restricting a relative polarization $L_X$. 
But for compatibility with Def.~\ref{ksz-period}
and \S~\ref{sec:bd}, we include the flexibility 
of allowing $r>1$.} of rank $r$
in the primitive symplectic case,
for some $y\in Y$ general. In either case, by 
\Cref{lem:proportionality} or \eqref{ps-lambda} 
respectively, the effective semiampleness of $\bM Y.$ 
will follow from the effective very-ampleness of the Hodge 
bundle $\lambda$ on $\oA_{g,{\bf d}}$ or $\oF_\Lambda$: 
There exists an integer $k$, depending only on $g$ or 
$b_2$, respectively, such that $\lambda^{\otimes k}$ 
is very ample. 

Imposing a level $3$ structure amounts to taking a finite 
Galois cover $q \colon \overline{S} \to \oA_{g,{\bf d}}$ or 
$q \colon \overline{S} \to \oF_\Lambda$ of degree 
$e\leq 3^{g^2}$ 
or $3^{(b_2-1)^2}$; cf.~\S~\ref{hodge-moduli}. In particular, 
if  $\lambda_{S}$ is the Hodge bundle on $\overline{S}$, we 
have $e q^*\lambda \sim e \lambda_{S}$. By 
\cite[Props.~3.4 and 4.2]{Mum77} and 
\cite[Thm.~1.2]{Fuj03}, provided that 
$\dim \overline{S}\geq 2$,\footnote{The 
case of $\dim \overline{S}=1$ is 
treated in \cite{Fuj86} (for elliptic curves), or 
reducing to $\dim \overline{S}\geq 2$ in the 
primitive symplectic case by adding $(1,1)$-classes, 
e.g., by blowing-up a multisection of $X$ as in 
\cite{Kim2024}.} the following fundamental 
facts hold:
\begin{itemize}
\item $\overline{S}$ has lc singularities, and 
\item $\omega_{\overline{S}}$ is Cartier ample 
and proportional to the Hodge bundle on 
$\overline{S}$, i.e., 
\[
\lambda^{\otimes j(g+1)}_{S}\simeq 
\omega_{\overline{S}}^{\otimes j}
\text{ or }\lambda^{\otimes j(b_2-3)}_S
\simeq \omega_{\overline{S}}^{\otimes j} 
\text{ for any }j\geq 0.
\]
\end{itemize} 
As a result, it suffices to show that a fixed multiple
of $\omega_{\overline{S}}$ is very ample. Note that 
the dimension of $\overline{S}$ is either $g(g+1)/2$ 
or $b_2-3$, and is hence bounded. By the effective 
basepoint-free theorem for lc varieties (see 
\cite[Thm. 1.1]{Fuj09}), 
there exists a bounded $l \geq 1$ such that 
$|l K_{\overline{S}}|$ is free.
The Kodaira vanishing theorem for lc varieties gives
\begin{align*}
0=
H^{i}(X, \mathcal{O}_{\overline{S}}
(K_{\overline{S}} + 
(l(\dim \overline{S}-i)+1)K_{\overline{S}})
)
=H^{i}(X, \mathcal{O}_{\overline{S}}
((l\dim \overline{S}+2)
K_{\overline{S}} -i(l K_{\overline{S}}))
)
\end{align*}
for all $i>0$; see \cite[Thm.~1.1]{Fuj15}. This means 
that $(l \dim \overline{S}+2) K_{\overline{S}}$ is 
$0$-regular with respect to $l K_{\overline{S}}$, 
which is  ample and globally generated;
see \cite[Def.~1.8.4]{Laz}.
By \cite[Ex.~1.8.22]{Laz}, we conclude that 
$(l \dim \overline{S}+l+2) K_{\overline{S}}$ is very 
ample, so we may take $k=l \dim \overline{S}+l+2$. 
\end{proof}

\begin{proposition}\label{lambda-formula}
Let $\Phi_Z'\colon Y'\to \oA_{8g}$ be a resolution
of indeterminacy of the Z-period map 
$\Phi_Z\colon Y\dashrightarrow \oA_{8g}$ of an
abelian or primitive symplectic fibration.
There exists a constant $a \in \ZZ$, 
depending on $g$ only,
so that we have a linear equivalence
\begin{align}\label{lambda-eqn}
a (\Phi_Z')^{\ast}\lambda_{8g} \sim \twopartdef{8a\bM Y'.}
{f\textrm{ is abelian,}}
{4gd^{-1}a\bM Y'.}{f\textrm{ is primitive symplectic}.}
\end{align}
Here $2d$ is the dimension of the general 
fiber of a primitive symplectic fibration $f$,
and $g =g(M_{\rm np})\coloneqq 2^{b_2-r-1} 
[O^*(M):{\rm P}^*(M)]$ is bounded 
solely in terms of $M_{\rm np}$.
\end{proposition}

\begin{proof} Write $\lambda_g$
or $\lambda_M$ to denote, respectively, the 
Hodge line bundles on
$\oA_{g,{\bf d}}$ and $\oF_\Lambda$.
The Zarhin map $Z\colon \oA_{g,{\bf d}}\to \oA_{8g}$
satisfies 
$Z^{\ast}(\lambda_{8g}^{\otimes a_g})
\simeq
(\lambda_g)^{\otimes 8a_g}$ for some
constant $a_g\in \bZ$ depending only on $g$.
For instance, this follows as in 
\Cref{lem:proportionality} by passing
to neat subgroups of $\Gamma_{g,{\bf d}}$ and
${\rm Sp}(16g,\bZ)$, whose index $a_g$ 
depends only on $g$, and then observing
$$H^{8g,0}((A\oplus A^{\ast})^{\oplus 4})
\simeq H^{g,0}(A)^{\otimes 4}\otimes 
H^{g,0}(A^{\ast})^{\otimes 4}.$$
So in the case of abelian fibrations, 
$a(\Phi_Z')^*\lambda_{8g}\sim 8a\bM Y'.$ 
for some $a\in \bZ$ depending only
on $g$, by \Cref{lem:proportionality}.

By again passing to neat subgroups of bounded
index, we see that the secondary (or primary)
Kuga--Satake map $\mathfrak{Ks}''
\colon 
\cF''_\Lambda
\to \cA_{g'',{\bf d}''}$
(see Def.~\ref{secondary-ks})
satisfies $$(\mathfrak{Ks}'')^{\ast}(\lambda_{g''}^{\otimes a_{g''}})
\simeq (\lambda_\Lambda'')^{\otimes a_{g''}g''/2},$$
for some $a_{g''}\in \bZ$,
where $\lambda_\Lambda''$ is the Hodge line
bundle on $\oF_\Lambda''$---for $x\in \cF_\Lambda''$, 
the scalar action on $H^{1,0}(KS''(x))$ by 
$c\in \bC^{\ast}$
scales $H^{1,0}(KS''(x))^{\otimes 2}\supset H^{2,0}(x)$ 
by $c^2$ and scales $H^{g,0}(KS''(x))$
by $c^{g''}$. In turn, we claim that
$\mathfrak{Ks}^*(\lambda_g^{\otimes a_g}) 
\sim (\lambda_\Lambda)^{\otimes a_gg/2}.$
This follows from the fact that 
$KS(x)\simeq 
\bigoplus_{\gamma\in 
O^*(M)/{\rm P}^*(M)} KS''(\gamma\cdot x)$
and so, for some $a_g\in \bZ$ depending
only on $g$,
$$\mathfrak{Ks}^*(\lambda_g^{\otimes a_g}) \sim
\det \pi_*((\lambda_\Lambda'')^{\otimes a_g})\sim
(\lambda_M)^{\otimes a_g[O^*(M)\,:\,{\rm P}^*(M)]}$$
where $\pi\colon \oF_\Lambda''\to 
\oF_\Lambda$ is the quotient
map by the $O^*(M)/{\rm P}^*(M)$-action.
Combining these results, we get 
$a(\Phi_Z')^*\lambda_{8g}
\sim 4gd^{-1}a\bM Y'.$ for some integer $a$
depending only on $g$, by (\ref{ps-lambda}).
\end{proof}

\begin{remark}
\Cref{lambda-formula} suggests an alternative proof to 
\Cref{thm:b-semi}. The linear equivalence 
\eqref{lambda-eqn} reduces the 
effective $b$-semiampleness of $\bM Y.$ to the semiampleness 
of the Hodge line bundle on the single variety 
$\overline{A}_{8g}$. In the abelian case, this suffices 
to reprove \Cref{thm:b-semi}. In the primitive 
symplectic case, the constant 
$g \coloneqq 2^{b_2-r-1}[O^*(M):{\rm P}^*(M)]$ depends 
a priori on the unpolarized
BBF lattice, while the 
proof of \Cref{thm:b-semi} relies 
simply on a bound on $b_2$. 
\end{remark}

\begin{corollary}
Fix $\dim X$. Abelian fibrations $f\colon X\to Y$
satisfy the effective b-semiampleness conjecture.
If the generalized abundance Conjecture \ref{gen-abund}
for symplectic varieties holds, then
primitive symplectic fibrations $f\colon X\to Y$
satisfy the effective b-semiampleness conjecture.
\end{corollary}

\begin{proof}
Fixing the dimension $\dim X$ also bounds 
the dimension of a fiber of $f$, and so the 
abelian case is immediate from 
Theorem \ref{thm:weakeffective}.

For the primitive symplectic case,
we split into two subcases. If $b_2\leq 4$
for the general fiber, we
apply Theorem \ref{thm:weakeffective}.
So suppose $b_2\geq 5$. Then the BBF form
on the general fiber has signature $(3,n)$
for $n\geq 2$. It follows from Meyer's Theorem
\cite{meyer} that the BBF form represents the integer
$0$. By the surjectivity of the period mapping 
\cite[Thm.~1.1]{BL2022},
there is some primitive symplectic variety $Z$,
deformation equivalent to a general fiber of $f$,
admitting a nef line bundle $L$ with $L\cdot L=0$.
By Conjecture \ref{gen-abund},
$|L|$ defines a Lagrangian fibration.
Then Theorem \ref{thm:ps}
implies that $b_2(Z)$ is bounded above. Since $b_2$
is a topological invariant, we conclude 
$b_2$ is bounded for the general fiber of $f$.
Now, we apply Theorem \ref{thm:weakeffective}.
\end{proof}

\section{{\it K}-trivial fibrations
over a family of bases}\label{sec:bd}

\subsection{Boundedness of the relative polarization}
Given a finite type family $\cY\to \cT$,
and a fixed target variety 
$A$, 
we can encode a space 
$\cS$ 
parameterizing rational maps
$\Psi_s\colon \cY_t\dashrightarrow A$, 
$t\in \cT$, 
for $s\in \cS$,
by the closure of the 
universal graph. 
That is, we have a
forgetful finite type
morphism $\cS\to \cT$, not necessarily
surjective,
and a universal graph closure
$$\overline{\Gamma \Psi}\subset 
(\cY\times_\cT \cS)\times A.$$
To lighten the notation, we notate such a universal
rational map $\Psi\colon \cY\dashrightarrow A/\cS$, 
with the map $\cS\to \cT$ implicit.
In this section, unless stated otherwise,
$$(\cY,\cB,\cM)\to \cT$$ will denote
a family $\cY\to \cT$ of normal 
projective varieties,
where $\cB$, $\cM$ are $\bQ$-divisors on $\cY$,
which are $\bQ$-divisors on all fibers $\cY_t$, $t\in \cT$.

We shall frequently pass to locally closed stratifications
of $\cS$, $\cT$. 
For simplicity, we will notate the new
bases again by $\cS$, $\cT$.

\begin{definition}\label{def:indicates} 
Let us fix a closed point 
$t\in \cT$. 
We say that a $K$-trivial
fibration $f\colon X\to \cY_t$ 
{\it indicates} $(\cY_t ,\cB_t,\cM_t)$ (cf.~``induces'' in 
Thm.~\ref{base-same-sing}) if
\begin{enumerate}
\item $K_{X_U}\sim_{\bQ, f_U}0$ for a big open
set $U\subset \cY_t$; 
\item $\supp B_{\cY_t}=\supp \cB_t$ and $\bM \cY_t.\equiv \cM_t$. 
\end{enumerate}
\end{definition}

\begin{remark}
Our definition of indicating $(\cY_t,\cB_t,\cM_t)$ employs
numerical equivalence $\equiv$ to avoid issues concerning
the non-constructibility of loci where 
$\bM \cY_t.\sim_\bQ \cM_t$. For instance, fix
a fibration $f\colon X\to Y$ over a base $Y$ with
${\rm Pic}^0(Y)\neq 0$, inducing $(Y,B,\bfM)$.
Set $\cT={\rm Pic}^0(Y)$,
$\cY \coloneqq Y\times \cT$, $\cB = B\times \cT$, and 
choose a family of $\bQ$-divisors with
$\cM_t \sim \bM Y.\otimes \cL_t$ where 
$\cL_t$ is the line bundle parameterized by 
$t\in \cT={\rm Pic}^0(Y)$. Then replacing $\equiv$
with $\sim_\bQ$ in Definition \ref{def:indicates}
would mean that $f$ only indicates $(\cY_t,\cB_t,\cM_t)$
at the torsion points of ${\rm Pic}^0(Y)$.
\end{remark}

\begin{lemma}\label{bir-vol}
Fix $Y$ a normal projective variety,
and a $\bQ$-divisor $D$ on $Y$. 
The set of rational maps
$\Phi\colon Y\dashrightarrow \bP^n$ for which
$D\equiv\Phi^\ast \cO_{\bP^n}(1)$ is bounded. 
\end{lemma}

\begin{proof} Fix a very ample line bundle $L$ on $Y$.
Consider the very ample line bundle 
$L\boxtimes \cO_{\bP^n}(1)$ on $Y\times \bP^n$. 
Let $Y' \coloneqq \overline{\Gamma \Phi} 
\subset Y \times \mathbb P^n$ 
be 
the closure of the graph and let $\pi\colon Y'\to Y$ 
be the restriction of the projection onto 
the first factor to $Y'$. 
Then 
$(L\boxtimes \cO_{\bP^n}(1))\vert_{Y'}= 
\pi^\ast L+D'$
for some effective Cartier divisor $D'$ on $Y'$ 
satisfying $\pi_*D'\equiv D$. Hence, as the 
volume is a numerical invariant for 
Weil
$\bR$-divisors
\cite[Thm.~3.5.iv]{FKL}
and 
$\vol(D') \leq \vol(\pi_\ast D')$, we have
\begin{align*}
\vol_{Y'} (L\boxtimes \cO_{\bP^n}(1)|_{Y'})
= \vol_{Y'} (\pi^\ast L+D')
\leq 
\vol_{Y} (L+\pi_*D')
=
\vol_{Y} (L+D).
\end{align*}
Since the volume of $L+D$ is bounded by assumption, the
result now follows from generalities on Chow schemes.
\end{proof}

\begin{proposition}\label{bd-map} 
Fix $g>0$ and $(\cY,\cB,\cM)\to \cT$
as above.
There is a finite type space of rational maps
$\cY\dashrightarrow \oA_{8g}/\cS$ containing
all Z-period maps $\Phi_Z\colon \cY_t\dashrightarrow 
\oA_{8g}$
for some abelian or primitive symplectic
fibration $f\colon X\to \cY_t$ 
indicating $(\cY_t,\cB_t,\cM_t)$.
\end{proposition}

Here,
in the primitive symplectic case,
we take 
$g=2^{b_2-r-1}[O^*(M):{\rm P}^*(M)]$,
as in the proof of Proposition \ref{lambda-formula}.

\begin{proof}
Let $\cL$ be a line bundle on $\cY$ that is relatively very ample
over $\cT$.
The graph closure of such a rational Z-period
map $\Phi_Z$ has bounded
volume with respect to the ample
line bundle $\cL_t\boxtimes \lambda_{8g}^{\otimes I_g}$
on $\cY_t\times \oA_{8g}$
because $\Phi_Z^{\ast}(\lambda_{8g})\equiv8\cM_t$
or $4gd^{-1}\cM_t$ in the respective cases, see 
(\ref{lambda-eqn}). This holds even in the case
when $\Phi_Z$ is rational, because on a model
$\pi\colon Y'\to Y$ on which $\bfM$ descends (i.e., a 
resolution of indeterminacy of $\Phi_Z$), we
have $\bM Y.=\pi_*\bM Y'.$ and that $\bM Y'.$ is 
pulled back from $\oA_{8g}$. 
The result now follows from a relative form of 
Lemma \ref{bir-vol}.
\end{proof}

\begin{proposition}\label{bd-stack}
Fix $g>0$ and $(\cY,\cB,\cM)\to \cT$.
There is a finite type space of morphisms
$\cY^o\to \cA_{8g}/\cS$
containing
all Z-classifying morphisms $\Phi_Z^o\colon 
\cY_t^o\to \cA_{8g}$
for some abelian or primitive symplectic
fibration $f\colon X\to \cY_t$ 
indicating $(\cY_t,\cB_t,\cM_t)$.
\end{proposition}

\begin{proof}
By Proposition \ref{bd-map}, the space of rational
Z-period maps $\Phi_{Z,s}\colon \cY_t\dashrightarrow 
\oA_{8g}$ for a fibration indicating 
$(\cY_t,\cB_t,\cM_t)$ lies in a bounded family. 
We first bound
the open set $\cY_s^o\subset \cY_t$ on which
the Z-classifying morphism $\Phi_{Z,s}^o$
(corresponding to the Albanese fibration, see
(\ref{def:classify}))
necessarily exists. Let 
$\partial_{{\rm H},8g} \coloneqq \oA_{8g}\setminus A_{8g}$.
In the abelian case, $\Phi_{Z,s}$
is induced by a classifying morphism $\Phi_{Z,s}^o$
on the open set 
\begin{align}\begin{aligned}\label{open1}
\cY_s^o 
& \coloneqq  \cY_t\setminus 
({\rm indet}(\Phi_{Z,s})\cup 
\Phi_{Z,s}^{-1}(\partial_{{\rm H}, 8g})
\cup \Supp \cB_t\cup (\cY_t)_{\rm sing}) \\
&\,\,=\cY_t\setminus ({\rm indet}(\Phi_{Z,s})\cup
\partial_{{\rm H},s}
\cup \Supp \cB_t\cup (\cY_t)_{\rm sing})
\end{aligned}\end{align}
by Proposition \ref{extend3},
applied to the restriction
of $f$ to $U\cap \cY_t\setminus 
({\rm indet}(\Phi_{Z,s})\cup (\cY_t)_{\rm sing})$,
where $U\subset \cY_t$ is a big open set with
$K_{X_U}\sim_{\bQ,f_U}0$. Here, we have extended
$\Phi_{Z,s}^o$ from $U\cap \cY_s^o$ to $\cY_s^o$
using that $U$ is big and $\cY_s^o$ is smooth.
Note that, for any $f$,
the moduli divisor 
$\bfM$ descends on the complement of the indeterminacy
locus of $\Phi_{Z,s}$. 

In the primitive symplectic case, first
pass to a locally closed stratification
of the space of maps, 
according to the generic inertia group 
$G\subset \Sp(16g,\bZ)$
of the image $\Phi_{Z,s}(\cY_t)$,
so that $G_s=G$ is constant
on a connected component. 
Let $\partial_{{\rm I},8g,G}\subset \oA_{8g}$
be the sublocus of the closure of 
the inertia stratum
for $G$ where the inertia jumps 
(to be larger than $G$).
Then $\Phi_{Z,s}$
is induced by a classifying morphism $\Phi_{Z,s}^o$
on the open set
\begin{align}\begin{aligned}\label{open2}
\cY_s^o &\coloneqq 
\cY_t\setminus 
({\rm indet}(\Phi_{Z,s})\cup 
\Phi_{Z,s}^{-1}(\partial_{{\rm H},8g})\cup 
\Phi_{Z,s}^{-1}(\partial_{{\rm I},8g,G})\cup 
\Supp \cB_t\cup (\cY_t)_{\rm sing}) \\
&\,\,\subset 
\cY_t\setminus ({\rm indet}(\Phi_{Z,s})\cup
\partial_{{\rm H},s}\cup
\partial_{{\rm I},s}\cup \Supp \cB_t\cup (\cY_t)_{\rm sing})
\end{aligned}\end{align}
by Proposition \ref{extend4}. The 
Hodge-theoretic boundary pulls back to
the Hodge-theoretic boundary under $\mathfrak{Ks}$
by \cite[Thm.~1.2]{schreieder2020kuga}.
The containment
holds because if the inertia at $y\in \cY_t$
for the period morphism $\Phi_s\colon \cY_t\to \oF_\Lambda$
jumps, then so does necessarily the inertia
at $\Phi_{Z,s}(y)$ because $\cF_\Lambda\to \cA_{8g}$
is induced by an embedding
of prequotients 
$\bD_M\hookrightarrow \cH_{8g}$.

Note that $f$ is not, a priori,
smooth over all of $\cY^o_s$.
Rather, it is critical for us
that the open sets $\cY_s^o\subset \cY_t$ 
from (\ref{open1}) or (\ref{open2}) on which 
the $Z$-period map lifts to a $Z$-classifying map
depend only on $\Phi_{Z,s}\colon 
\cY_t\to \oA_{8g}$ (and $\supp \cB_t$) but 
not on $f$ itself.
In summary,
we have shown that
there is a finite type family
of morphisms $$\cY^o\to A_{8g}/\cS$$
from quasi-projective subvarieties 
$\cY^o_s\subset \cY_t$ for which the period
map $\Phi_{Z,s}\colon \cY_t\dashrightarrow \oA_{8g}$ 
associated to $f$ is induced by {\it some} lift
to the DM stack $\Phi_{Z,s}^o\colon \cY_s^o\to \cA_{8g}$.

It remains to bound the set of lifts.
Stratify $\cS$ by topological type of 
the fibers $\cY^o_s$.
For those components of $\cS$ on which a lift
exists, the set of lifts of $\Phi_{Z,s}\vert_{\cY_s^o}\colon 
\cY^o_s\to A_{8g}$ to $\cA_{8g}$
is identified with the space of 
crossed homomorphisms 
$\pi_1(\cY^o_s)\to G$, with respect
to twisting by the pullback to $\cY^o_s$
of the $G$-gerbe over the inertia stratum.
In particular since $\pi_1(\cY^o_s)$
is finitely generated and $G$ is finite, 
these are finite in number.
So the number of lifts $\cY_s^o\to \cA_{8g}$ to the
DM stack is finite, and constant on connected
components of $\cS$. Alternatively, one can invoke \cite[Thm.~1.1]{Olsson2007}.
\end{proof}

We must now ``undo'' the Kuga--Satake and Zarhin
tricks. Then, using that the fibrations are
of Picard type, we will recover a polarization type 
from the local system alone;
see Example \ref{pic-ex} for why this restriction
is necessary.

\begin{proposition}\label{bd-unzarhin} 
Fix $g>0$
and $(\cY,\cB,\cM)\to \cT$.
\begin{enumerate}
\item[(AV)] There is a finite set 
of saturated sub-modules $\Lambda^{(i)}\subset\wedge^2(\bZ^{2g})$
and a finite type space of maps 
$\cY^o\to \bigsqcup_i \cA_{g,\Lambda^{(i)}}/\cS$
containing
all classifying morphisms $\Phi^o_s\colon 
\cY_s^o\to \cA_{g,\Lambda^{(i)}}$
for some abelian
fibration $f\colon X\to \cY_t$ of relative dimension $g$, of Picard type,
indicating $(\cY_t,\cB_t,\cM_t)$.\smallskip
\item[(PS)] 
Fix an analytic deformation
class of primitive symplectic varieties.
There is a finite collection 
$\Lambda^{(i)}\subset (M_{\rm np},\cdot)$
of primitive, signature $(1,r-1)$
sublattices and a finite type space of maps 
$\cY^o\to \bigsqcup_i \cF_{\Lambda^{(i)}}/\cS$
containing
all classifying morphisms $\Phi_s^o\colon 
\cY_s^o\to \cF_{\Lambda^{(i)}}$
for some primitive symplectic
fibration $f\colon X\to \cY_t$ of Picard type,
indicating $(\cY_t,\cB_t,\cM_t)$.
\end{enumerate}
\end{proposition}

Here $(M_{\rm np},\cdot)$ is the
BBF lattice of the specified analytic deformation class
of primitive symplectic variety. Then
$g=2^{b_2-r-1}[O^*(M):{\rm P}^*(M)]$ 
where $b_2=\rk M_{\rm np}$ and $r=\rk \Lambda$
is the rank of the monodromy-invariant sublattice
of $M_{\rm np}$, see \S~\ref{sec:moduli-ps}.

\begin{proof} We first treat the abelian-fibered case.
By Proposition \ref{bd-stack}, there is a finite
type space $\Phi^o_Z\colon
\cY^o\to \cA_{8g}/\cS$ containing all possible
Z-classifying morphisms.
There is a universal
rank $16g$ $\bZ$-local system 
over the stack $\cA_{8g}$
and pulling it back along $\Phi^o_Z$ gives a
rank $16g$ $\bZ$-local system $\bV_Z\to \cY^o.$ 
Given
a local system $\bV$, define $Z(\bV) \coloneqq 
\bV^{\oplus 4}\oplus (\bV^{\ast})^{\oplus 4}$.
We have that 
$(\cY^o_s, Z(R^1f^o_*\underline{\bZ}_{X^o}))$ and 
$(\cY^o_s,\bV_{Z,s})$ agree over the locus 
where $f$ is smooth.

\begin{lemma}\label{local-sys-undo} 
Let $Y^o$ be a quasi-projective
variety and let $\bV_Z\to Y^o$ be a semisimple
$\bZ$-local
system of rank $16g$. 
There are only finitely many semisimple
$\bZ$-local systems $\bV\to Y^o$ of rank $2g$
for which $Z(\bV)=\bV_Z$.
\end{lemma}

\begin{proof}
Suppose that there exists a local system 
$\bV$ on $Y^o$ which satisfies
$Z(\bV)=\bV_Z$. Let 
$\rho_{\bV_Z}\colon \pi_1(Y^o)\to \GL_{16g}(\bZ)$
and $\rho_{\bV}\colon \pi_1(Y^o)\to \GL_{2g}(\bZ)$ 
be the monodromy representations of the respective local systems.
By a theorem of Procesi \cite{procesi},
the $\GL_{2g}(\bQ)$-conjugacy class 
of a semisimple representation
of a finitely generated group $\Pi$ is uniquely determined
by the traces ${\rm tr}\,\rho_{\bV}(\gamma_k)$ of a finite
generating set $\gamma_k\in \Pi$, which depends
only on the group $\Pi$. Let $\{\lambda_i\}_{i=1,\,\dots,\,2g}$
denote the eigenvalues of $\rho_{\bV}(\gamma)$. 
Then, the eigenvalues of $\rho_{\bV_Z}(\gamma)$ are
necessarily
\begin{align}\label{eigens}
\{\lambda_i, \lambda_i,\lambda_i,\lambda_i
,\lambda_i^{-1},\lambda_i^{-1},
\lambda_i^{-1},\lambda_i^{-1}\}_{i=1,\,\dots,\,2g}.
\end{align}
In fact, if for any $\gamma\in \Pi$, 
the eigenvalues of $\rho_{\bV_Z}(\gamma)$
are not of the form (\ref{eigens}), then no $\bV$
exists for which $Z(\bV)=\bV_Z$. If the eigenvalues
are of the form (\ref{eigens}), then there 
are $\leq 2^{2g}$ possibilities for the set of eigenvalues 
of $\rho_{\bV}(\gamma)$:
$$\{\lambda_1^{\pm 1}, \,\dots,\,\lambda_{2g}^{\pm 1}\}.$$
Thus, $\rho_{\bV_Z}$ determines, up to finite
ambiguity,\footnote{This ambiguity is legitimately
present. For instance, we have $Z(A)=Z(A^*)$ or
$Z(A\oplus B)= Z(A\oplus B^*)$ for two abelian
varieties, or fibrations, $A$ and $B$.
} the eigenvalues, and in turn the traces
of $\rho_{\bV}(\gamma_k)$ for Procesi's
generators $\gamma_k\in \Pi$.
We deduce that $\bV_Z$ determines, up to finite
ambiguity,
the $\GL_{2g}(\bQ)$-conjugacy class
of the monodromy representation
of any $\bV$ satisfying $Z(\bV)=\bV_Z$ (if one exists).
The lemma follows from the fact that the
$\bQ$-conjugacy class contains at 
most finitely many, possibly zero,
$\GL_{2g}(\bZ)$-conjugacy classes, 
see e.g., \cite[Lem.~0.3]{deligne87}.
\end{proof}

The monodromy of a local system underlying
a $\bZ$-PVHS is semisimple \cite[Prop.~1.13]{deligne87}. 
The finite ambiguity 
in Lemma \ref{local-sys-undo} is constant on connected 
components of $\cS$, since we have stratified 
by topological type of $(\cY_s^o,\bV_{Z,s})$. 
Replacing $\cS$ with a finite cover,
it follows that there exists a finite type
family $(\cY^o,\bV)\to \cS$ consisting
of the quasi-projective
subvarieties $\cY^o_s\subset \cY_t$ endowed
with a rank $2g$ $\bZ$-local 
system $\bV\to \cY^o$, such that: For 
any abelian fibration $f\colon X\to \cY_t$
indicating $(\cY_t,\cB_t,\cM_t)$,
we have, for some $s\in \cS$, agreement between
$R^1f^o_*\underline{\bZ}$ and 
$\bV_s$ on $\cY_s^o$ (over the locus
where $f$ is smooth).
Define, for a connected component
$\cS^{(i)}\subset \cS$,
\begin{equation}\label{possible_lattices}
\Lambda^{(i)} \coloneqq H^0(\cY^o_s, \wedge^2\bV_s)
\subset \wedge^2\bV_{s,y}\simeq \wedge^2(\bZ^{2g})
\end{equation}
for some (or any) $s\in \cS^{(i)}$, $y\in \cY_s^o$.
We throw out any components $\cS^{(i)}$
for which $\Lambda^{(i)}=0$ since we are 
concerned with fibrations with
projective total space.

We now use that $f\colon X\to \cY_t$
is assumed to be of Picard type. By definition,
this implies that $\Lambda^{(i)}\subset H^{1,1}(X_y)$
for $y\in \cY_s^o$ and thus, the original
classifying morphism used to construct
the Z-classifying morphism for $f$ factors
through
$$
\Phi^o_s\colon \cY_s^o\to \cA_{g,\Lambda^{(i)}},
$$
see Proposition \ref{pic-per} and (\ref{inclusions}).

The possible 
saturated submodules $\Lambda^{(i)}$ depend on 
the connected components $\mathcal{S}^{(i)}$, 
which are finite in number. So, 
we may now reapply the arguments of Propositions
\ref{bd-map}, \ref{bd-stack} essentially
verbatim, replacing $\oA_{8g}$, $\cA_{8g}$ with 
$\oA_{g,\Lambda^{(i)}}$, $\cA_{g,\Lambda^{(i)}}$
to conclude that the space of possible
classifying morphisms $\Phi^o\colon 
\cY_s^o\to \cA_{g,\Lambda^{(i)}}$ lie
in a bounded family. Case (AV) follows. \smallskip

Now, we treat the primitive symplectic case.
The proof strategy is similar, with many
additional wrinkles. Proposition \ref{bd-stack}
and Lemma \ref{local-sys-undo} ensure that
the local system $\bV$ associated to the family
of Kuga--Satake varieties (\ref{final-ks-1}),
(\ref{final-ks-2}) over $Y^o$ is bounded.

Let $(\bU_{\rm np}, \psi)$ be a $\bZ$-local system admitting 
a non-degenerate quadratic form $\psi$ on a smooth quasi-projective 
variety $Y^o$, and  $(\bU,\psi) \subseteq (\bU_{\rm np}, \psi)$
a $\bZ$-polarized sublocal system of signature $(2,m)$ 
such that $\bU^{\perp}$ is a trivial local system, with 
$(M_{\rm np},\cdot)\coloneqq
(\bU_{{\rm np},*},\psi_*)$
isometric to the unpolarized BBF lattice
of the PS varieties in question for a base 
point $*\in Y^o$.
Set $(M,\cdot)=(\bU_*,\psi_*) \subseteq M_{\rm np}$. Let 
${\rm P}^*(M) \subset O^*(M)$ the subgroup defined in 
\Cref{bounded-index}, which has index uniformly bounded 
above depending only on $M_{\rm np}$, namely by 
$$b \coloneqq 3^{(\rk M)^2}
[O(M_{\rm np}):\phi_{\rm np}({\rm Pin}(M_{\rm np}))].$$

The monodromy representation $\rho \colon \pi_1(Y^{o}) 
\to O^*(M)$, associated to the local system $\bU$ on $Y^o$, 
induces the representation
\[\rho^{\rm orth} \colon \pi_1(Y^{o}) \to \GL({\rm Cl}(M)) 
\simeq \GL_{2g''}(\bZ)\] by the universal property 
of the Clifford algebra. 
The representation $\rho^{\rm orth}$ determines
a local system ${\rm Cl}^{\rm orth}(\bU, \psi)$, 
which is a $\bZ$-local system of 
Clifford algebras of rank $2g'' =2^{m+2}$.

Let $Y''\to Y^{o}$ be the finite \'{e}tale cover corresponding to 
the subgroup  $$\pi_1(Y'') \coloneqq \rho^{-1}({\rm P}^{\ast}(M)) 
\subset \pi_1(Y^{o}).$$ For any 
$\gamma \in \pi_1(Y'')$, the transformation 
$\rho^{\rm orth}(\gamma)$ is given by
$v \mapsto (-1)^{\deg \,\rho(\gamma)} 
\rho(\gamma) v \rho(\gamma)^{-1}$, $v\in {\rm Cl}^1(M)$,
with $\rho(\gamma) \in {\rm P}^*(M)$. 
The group ${\rm P}^*(M)$ acts also by 
left multiplication on the overlattice
$\langle {\rm Pin}(M){\rm Cl}(M)\rangle \supseteq {\rm Cl}(M)$. 
This allows us to define a new representation 
$\rho'' \colon \pi_1(Y'') \to \GL(\langle {\rm Pin}(M)
{\rm Cl}(M)\rangle) \simeq \GL_{2g''}(\bZ)$, which induces 
\[
\rho^{\rm pin} \coloneqq 
{\rm Ind}_{\pi_1(Y'')}^{\pi_1(Y^o)}(\rho'')
\colon \pi_1(Y^{o}) \to \GL_{2g}(\bZ),
\]
where $2g=2^{m+2}[O^*(M):{\rm P}^*(M)]$. 
The representation $\rho^{\rm pin}$ corresponds to a $\bZ$-local 
system of rank $2g$ on $Y^{o}$, denoted 
${\rm Cl}^{\rm pin}(\bU,\psi)$. 

Geometrically, the dual representations of $\rho^{\rm orth}$ and 
$\rho''$ are the monodromy on $H^1$ of the family of primary and 
secondary Kuga--Satake varieties, i.e., $KS'_{Y^o} \to Y^o$ and 
$KS''_{Y''} \to Y''$ respectively, see \S~\ref{sec:kuga}, and 
$\rho^{\rm pin}$ is the dual monodromy of the universal weight 1
$\mathbb{Z}$-PVHS on $\mathcal{A}_{g, \Lambda_g, 0}$ pulled back 
along the morphism $\mathfrak{Ks} \colon Y^o \to 
\mathcal{A}_{g, \Lambda_g, 0}$, cf.~\eqref{final-ks-1}.

\begin{lemma}\label{orth-to-pinnew} Fix the lattice 
$(M_{\rm np}, \cdot)$ and the constant $b$ as above. 
Fix a semisimple local system $\bV$ on a quasi-projective 
variety $Y^{o}$. There are only finitely
many semisimple local systems $\bU$ such that:
\begin{enumerate}
\item\label{cond1} there exists a $\ZZ$-local system 
$\bU_{\rm np}$ containing $\bU$, admitting a non-degenerate 
quadratic form $\psi$ such that
\begin{enumerate}
\item $(\bU_{\rm np, \ast}, \psi_*)\simeq (M_{\rm np}, \cdot)$;
\item $\psi$ restricts to a polarization on $\bU$ of signature $(2, m)$;
\item $\bU^{\perp}$ is a trivial local system.
\end{enumerate}
\item ${\rm Cl}^{\rm pin}(\bU,\psi)^* \simeq \bV$.
\end{enumerate}
\end{lemma}

Note that the requirement \eqref{cond1} grants that the 
monodromy representation of $\bU$ takes values in $O^*(M)$, 
so that the local system ${\rm Cl}^{\rm pin}(\bU,\psi)$ 
can be defined.

\begin{proof}
Let $\rho_{\bV}$ be the monodromy representation of $\bV^*$.

 \textbf{Step 1.} There are only finitely many \'{e}tale covers 
 $Y'' \to Y^{o}$ of index bounded $\leq b$. Indeed, for any subgroup 
 $\Pi \subset \pi_1(Y^o)$ of index $\leq b$, the normal core 
 ${\rm Core}(\Pi) \coloneqq 
 \cap_{\gamma \in \pi_1(Y^o)} \gamma \Pi \gamma^{-1}$
 has index $\leq b!$, but since $\pi_1(Y^o)$ is finitely 
 generated, there exists only finitely many possible quotients 
 $\pi_1(Y^o) \to \pi_1(Y^o)/\Pi$, so $\Pi$ ranges among the 
 pre-images of subgroups of index $\leq b$ in these 
 finitely many finite quotients.
 
\textbf{Step 2.} For a given choice of $Y'' \to Y^{o}$ 
of degree $b' \leq b$, there are only finitely many 
representations $\rho''_{\bV} \colon \pi_1(Y'') \to 
\GL_{2g''}(\bZ)$ inducing $\rho_{\bV}$.
Indeed, by restricting to $\pi_1(Y'')$, we obtain  
\[\rho_{\bV}|_{\pi_1(Y'')}= {\rm Ind}_{\pi_1(Y'')}^{\pi_1(Y^o)}
(\rho''_{\bV})|_{\pi_1(Y'')} = (\rho''_{\bV})^{\oplus b'},\]
which determines $\rho''_{\bV}$ up to a finite ambiguity, 
since $\rho_{\bV}$ and $\rho''_{\bV}$ are semisimple.

\textbf{Step 3.} Fix $\rho''_{\bV}$ as in Step 2, satisfying 
the following additional property: There exists a primitive 
embedding $(M, \cdot) \hookrightarrow\ (M_{\rm np}, \cdot)$, 
and ${\rm P}^*(M)\subset O^*(M)$ of index $\leq b$ as in 
\Cref{bounded-index} such that  $\rho''_{\bV}$ is isomorphic 
to a representation $\pi_1(Y'') \to \GL(\langle {\rm Pin}(M)
{\rm Cl}(M)\rangle)$, assigning to each $\gamma \in \pi_1(Y'')$ 
the left multiplication by an element $p(\gamma)\in {\rm P}^*(M)$. 
Then there are only finitely many semisimplifications 
$(\rho^{\rm orth}_{\bV})^{\rm ss}$ of representations 
$\rho^{\rm orth}_{\bV} \colon \pi_1(Y^{o}) \to \GL({\rm Cl}(M)) 
\simeq \GL_{2g''}(\bZ)$ satisfying the following property: 
\begin{equation}\label{eq:semisimpleorth}
\rho^{\rm orth}_{\bV}(\gamma)\cdot v = 
\pm p(\gamma)vp(\gamma)^{-1} \quad \text{ for any }
\gamma \in \pi_1(Y'')\text{ and }v \in {\rm Cl}(M_{\bQ}).
\end{equation}
To this end, fix $\gamma \in \pi_1(Y'')$. Let 
$\lambda_{\rm max}$ be the largest norm of the eigenvalues 
of the left multiplication by ${p(\gamma)}$. Since 
$(p(\gamma) v)^* = v^* p(\gamma)^* = v^* p(\gamma)^{-1}$, 
the eigenvalues of the left multiplication by $p(\gamma)$ 
and right multiplication by $p(\gamma)^{-1}$ coincide.
 By \eqref{eq:semisimpleorth}, the
eigenvalues of $\rho^{\rm orth}_{\bV}(\gamma)$
have norm bounded above by $\lambda_{\rm max}^2$
(this bound does not depend
on the $\bZ$-structure of the representation).
As in Lemma 
\ref{local-sys-undo},
the integral conjugacy class of the 
semi-simplification\footnote{While it is true that 
in our situation, $\rho_\bV^{\rm orth}$ underlies a 
$\bZ$-VHS of weight $1$, it may fail to be polarizable
(cf.~Rem.~\ref{ks-differences}), and thus 
Deligne-Andr\'e semisimplicity may fail.}
$\rho^{\rm orth}_{\bV}\vert_{\pi_1(Y'') }^{\rm ss}$
is determined up to finite
ambiguity
by ${\rm tr}\,\rho^{\rm orth}_{\bV}(\gamma_i)\in \bZ$ 
for some finite
set of $\gamma_i\in\pi_1(Y'') $.
Since the norm of the eigenvalues of 
$\rho^{\rm orth}_{\bV}(\gamma_i)$
is bounded, so is its trace, and thus, 
there are only finitely
many possibilities 
for 
$\rho^{\rm orth}_{\bV}\vert_{\pi_1(Y'')}^{\rm ss}$
since the trace is integral.

The possible extensions of 
$\rho^{\rm orth}_{\bV}\vert_{\pi_1(Y'')}^{\rm ss}$
to a semisimple
representation $(\rho^{\rm orth}_{\bV})^{\rm ss}$ of $\pi_1(Y^o)$
are bounded. Since $\pi_1(Y'')\subseteq \pi_1(Y^o)$ 
has index $b'\leq b$, $\gamma^{b'}\in \pi_1(Y'')$
for all $\gamma\in \pi_1(Y^o)$. So
the eigenvalues of 
$(\rho^{\rm orth}_{\bV})^{\rm ss}(\gamma)$, with
$\gamma\in \pi_1(Y^o)$, are determined by the eigenvalues
of $(\rho^{\rm orth}_{\bV})^{\rm ss}(\gamma^{b'})$, up to 
$b'$-th roots of unity.

\textbf{Step 4.} Given $(\rho^{\rm orth}_{\bV})^{\rm ss}$, 
there are only finitely many semisimple local systems 
$\bU$ admitting a non-degenerate polarization $\psi$ 
such that $\rho^{\rm orth} = \rho^{\rm orth}_{\bV}$. 
Let $\rho \colon \pi_1(Y^{o}) \to O(\bU_*) \subset \GL_{m+2}(\ZZ)$ 
be the monodromy of $\bU$. 
Suppose that the eigenvalues
of $\rho(\gamma)$ are 
$\{\lambda_1,\dots,\lambda_{m+2}\}$,
which are all nonzero.
Then, since ${\rm Cl}^{\rm orth}(\bU)$
admits a filtration whose associated
graded 
is $\wedge^{\bullet}\bU$,
we conclude that the eigenvalues
of $\rho^{\rm orth}_{\bV}$ are 
\[\mu_{I} \coloneqq \prod_{i\in I}\lambda_i\]
where $I$ runs over all subsets of $\{1,\dots,m+2\}$.
It follows that 
$$\sup_I |\mu_I| \geq 
\sup_i |\lambda_i|.$$
We conclude
that 
${\rm tr}\,\rho(\gamma)=\sum \lambda_i$ 
are bounded, in a manner
depending only on $\gamma$.
As in Lemma \ref{local-sys-undo},
the integral conjugacy class of $\rho_\bU$
is determined up to finite
ambiguity
by ${\rm tr}\,\rho_\bU(\gamma_i)\in \bZ$ 
for some finite
set of $\gamma_i\in \pi_1(Y^o)$.
\end{proof}

Combining Lemmas \ref{local-sys-undo}
and \ref{orth-to-pinnew} (respectively, to unwind
the Zarhin and Kuga--Satake constructions on the level
of local systems), we conclude:
There is a finite type
family $(\cY^o,\bU^o)\to \cS$ of quasi-projective
subvarieties $\cY^o_s\subset \cY_t$ endowed
with a $\bZ$-local 
system 
$\bU^o\to \cY^o$ of rank $m+2\geq 2$,
such that, for 
any primitive symplectic 
fibration $f\colon X\to \cY_t$
indicating $(\cY_t,\cB_t,\cM_t)$,
we have agreement between
$$\bU  \coloneqq  
H^0(\cY_s^o, R^2f^o_*\underline{\bZ})^\perp\subset
(R^2f^o_*\underline{\bZ})_{\rm tf} \eqcolon \bU_{\rm np}$$
and $\bU_s^o$ on $\cY_s^o$ (over the locus
where $f$ is locally trivial).

We must now determine, up to finitely
many choices, the primitive embedding $\bU_s^o=\bU\hookrightarrow 
\bU_{\rm np}$ and the induced
polarization
$\psi\colon \bU\otimes \bU\to \underline{\bZ}_{\cY_s^o}$ 
from the data of $\rho_\bU$.
This is subtle, as the replacement
$\psi\mapsto m\psi$ for some $m>0$
leaves $\bU$ and the resulting 
semi-simplified
orthogonal
Kuga--Satake
local system 
unaltered, i.e., 
${\rm Cl}^{\rm orth}(\bU,\psi)^{\rm ss}=
{\rm Cl}^{\rm orth}(\bU,m\psi)^{\rm ss}$.
The analytic deformation
class of the general fiber of $f$
has been specified. This
specifies an unpolarized BBF lattice
$(M_{\rm np},\cdot)=(\bU_{{\rm np},*},\psi_{{\rm np},*})$
and a finite index monodromy group
${\rm Mon}\subset O(M_{\rm np})$,
see \cite[Thm.~8.2(1)]{BL2022}.\footnote{Except
possibly in the case $b_2=4$.}

Let us first note that we have an integral decomposition
of a finite index sub-$\bZ$-local system
$$\bU_{\rm tr}\oplus \bU_{\rm fin} \unlhd \bU$$
into the transcendental subvariation 
and a complement $\bU_{\rm fin}$
of Hodge--Tate type,
on which $\rho_\bU$ has finite monodromy,
because $\bU$ is a polarized variation.
Suppose the fibration $f$ is not isotrivial 
(we deal with this case later). Then,
by the semi-simplicity of $\rho_\bU$
and Deligne's theorem of the fixed
part, we can detect $\bU_{\rm tr}$ as the 
saturated sub-$\bZ$-local system
of $\bU=\bU_s$ which, over $\bQ$, 
is the unique 
summand of infinite monodromy.
So define 
$\rho_{\rm tr} \coloneqq \rho_\bU\vert_{\bU_{\rm tr}}$.

The Zariski closure of
${\rm im}(\rho_{\rm tr})$ in $\GL_{m+2}(\bR)$
is $O(2,m')$ or $U(1,m')$, up to finite
index \cite[2.2.1, 2.3.1]{zarhin83}. 
Here $\rk \bU_{\rm tr} = 2+m'$ or $2+2m'$
respectively.
Thus, the monodromy of 
$\rho_{\rm tr}$
is absolutely irreducible, and
it follows by Schur's lemma
that $\bU_{\rm tr}$
admits a $\rho_{\rm tr}$-invariant
bilinear form $\psi_0$ which is unique
up to scaling. We choose $\psi_0$
as the smallest scaling for which $\psi_0$
is integer-valued on $\bU_{\rm tr}$.
We necessarily have $\psi_{\rm tr} \coloneqq 
\psi\vert_{\bU_{\rm tr}}=N_{\rm tr}\psi_0$
for some $N_{\rm tr}>0$ an integer.

\begin{lemma} $N_{\rm tr}$ is bounded above.
\end{lemma}

\begin{proof} Set
$M=\bU_*$, $M_{\rm tr}=
\bU_{{\rm tr},*}$, $M_{\rm fin}=\bU_{{\rm fin},*}$.
The transcendental monodromy representation
$\rho_{\rm tr}\colon \pi_1(\cY_s^o,*)\to 
\GL(M_{\rm tr})$
reduces mod $N$ to a homomorphism
$$\rho_{\rm tr}\,(\textrm{mod }N)\colon 
 \pi_1(\cY_s^o,*)\to 
{\rm Aut}(\tfrac{1}{N}M_{\rm tr}/
M_{\rm tr})$$
whose image grows to infinity in size
as $N\to \infty$,
because ${\rm im}(\rho_{\rm tr})$
is infinite---e.g.~we may take
any fixed finite set in ${\rm im}(\rho_{\rm tr})$,
and choose $N$ to be larger than twice
the absolute value of the largest entry
of any matrix in the set. 
We choose some integer $N_0$
for which 
$|{\rm im}\,\rho_{\rm tr}\,(\textrm{mod }N)|>
[M_{\rm np}^{\ast}:M_{\rm np}]^
{[M_{\rm np}^{\ast}\,:\,M_{\rm np}]{b_2}}\cdot 3^{(b_2)^2}
$ for all $N\geq N_0$.

We have a decomposition
$M_{\rm tr}\oplus M_{\rm alg} \coloneqq 
M_{\rm tr}\oplus (M_{\rm tr})^{\perp}
\unlhd M_{\rm np}$ 
and an inequality
$$
[M_{\rm np}^{\ast} :M_{\rm np}]^{-1}\leq 
\frac{[M^{\ast}_{\rm tr}:M_{\rm tr}]}
{[M_{\rm alg}^{\ast}:M_{\rm alg}]}
\leq [M_{\rm np}^{\ast}:M_{\rm np}].$$
For instance, when $M_{\rm np}$ is unimodular,
the discriminant groups of saturated
complementary lattices have equal size
and are canonically isomorphic.
In the general case, $M^{\ast}_{\rm tr}/M_{\rm tr}$
and $M_{\rm alg}^{\ast}/M_{\rm alg}$ canonically
contain the subgroup 
$M_{\rm np}/(M_{\rm tr} \oplus M_{\rm alg})$
of index at most $[M_{\rm np}^{\ast}:M_{\rm np}]$
in either.
It follows that, 
for an action $\rho_{\rm tr}$ of a group
on $M_{\rm tr}$ to extend to an action 
$\rho_{\rm np}$ on $M_{\rm np}$,
the ratio in size between
the images
of $\rho_{\rm np}$ in 
${\rm Aut}(M^{\ast}_{\rm tr}/M_{\rm tr})$
and in ${\rm Aut}(M^{\ast}_{\rm alg}/M_{\rm alg})$
is at most 
$[M_{\rm np}^{\ast}:M_{\rm np}]^{[M_{\rm np}^{\ast}\,:\,M_{\rm np}]{b_2}}$.\footnote{Let
$A\subset B$ be a subgroup of an abelian group, of index $p$,
and let ${\rm Aut}_A(B)$ be the subgroup of automorphisms of
$B$ preserving $A$. Then, any element $\phi$ of 
the kernel of the natural restriction map ${\rm res}\colon
{\rm Aut}_A(B)\to {\rm Aut}(A)$
satisfies $\phi(pb) = pb$, for all $b\in B$. Hence 
$\phi(b)-b\in B[p]$ for all $b\in B$. Note $|B[p]|\leq p^{\rk B}$
and $\phi$ is uniquely determined from ${\rm res}(\phi)$
and its value on coset representatives of $A$ in $B$.
Thus, $|\ker({\rm res})|\leq {p^{p\rk B}}$. We apply this
to $A = M_{\rm np}/M_{\rm tr}\oplus M_{\rm alg}$ 
and $B=M_{\rm tr}^*/M_{\rm tr}$ or $M_{\rm alg}^*/M_{\rm alg}$ 
observing
that $\rho_{\rm np}$ maps to ${\rm Aut}_A(B)$ and
has some fixed image in ${\rm Aut}(A)$, given by the image
of $\rho_{\rm np}$ in ${\rm Aut}(M_{\rm np}/
M_{\rm tr}\oplus M_{\rm alg}).$}

The image of $\rho_{\rm alg}$
has size bounded above by $3^{(b_2)^2}$---the
image is finite in $\GL(M_{\rm alg})$,
because $\rho_{\rm alg}$ fixes some 
polarization class, and any finite subgroup
of $\GL(M_{\rm alg})$ injects into 
$\GL(M_{\rm alg}/3M_{\rm alg})$ by neatness
of the level $3$ subgroup. So
the image of $\rho_{\rm tr}$ acting
on the discriminant group is bounded 
above in size by 
$[M_{\rm np}^{\ast}:M_{\rm np}]^
{[M_{\rm np}^{\ast}\,:\,M_{\rm np}]{b_2}}
\cdot 3^{(b_2)^2}$.
Finally, since 
$\psi_{\rm tr}=N_{\rm tr}\psi_0$, 
$$\tfrac{1}{N_{\rm tr}}M_{\rm tr}/
M_{\rm tr}\hookrightarrow M_{\rm tr}^{\ast}/M_{\rm tr}$$
and so $|{\rm im}\,\rho_{\rm tr}\,(\textrm{mod }
N_{\rm tr})|
\leq [M_{\rm np}^{\ast}:M_{\rm np}]^
{[M_{\rm np}^{\ast}\,:\,M_{\rm np}]{b_2}}
\cdot 3^{(b_2)^2}$.
We conclude that $N_{\rm tr}< N_0$.
\end{proof}

Thus, we have bounded 
$\psi_{\rm tr} = \psi\vert_{\bU_{\rm tr}}$ 
when the variation
is not isotrivial.
If $b_2\geq 5$, 
there are only finitely many primitive
embeddings 
$\iota_{\rm tr}\colon
(M_{\rm tr}, \psi_{{\rm tr},*})
\hookrightarrow
(M_{\rm np},\cdot)$, up
to the action of 
${\rm Mon}\subset O(M_{\rm np})$ 
by post-composition, 
see, e.g., \cite[Prop.~1.5.1]{nikulin}.
In the $b_2=4$ case, we will only
bound these embeddings up to the action of 
$O(M_{\rm np})$.

Thus, the embedding of the perpendicular lattice 
$(M_{\rm alg},\cdot_{\rm alg})\hookrightarrow (M_{\rm np},\cdot)$
is bounded.
Let $G_{\rm alg}$ be 
the finite image of $\rho_\bU$
in $O(M_{\rm alg})$.
Let ${\rm Mon}_{\rm alg}$ be
the image in $O(M_{\rm alg})$ of the
subgroup of ${\rm Mon}$ acting
 trivially on $M_{\rm tr}$.
 It has finite index in $O(M_{\rm alg})$.
 So there are only finitely
 many ${\rm Mon}_{\rm alg}$-conjugacy classes 
 of finite subgroups of $O(M_{\rm alg})$,
 see \cite[\S~5(a)]{borel_finite}.
 Thus, up to post-composition with
 an element of ${\rm Mon}_{\rm alg}$,
 there are only finitely many possibilities
 for $G_{\rm alg}$.
 Note that $M$ is, over $\bQ$, the
 sum of $M_{\rm tr}$ and the non-trivial
 isotypic components of the action of 
 $G_{\rm alg}$.
 Hence, there are only finitely many
 extensions
 $\iota\colon M\hookrightarrow 
 (M_{\rm np},\cdot)$ of $\iota_{\rm tr}$
 giving finitely many possible embedded lattices
 $M^{(i)}\subset M_{\rm np}$ for $M$, 
 up to post-composition with ${\rm Mon}$. 

  The above argument fails when the fibration
  is isotrivial, because the monodromy
  $\rho_\bU$ has finite image. In this
  case, the period map is constant,
  and the monodromy naturally
  factors as
  $$\rho_\bU\colon \pi_1(\cY_s^o)\to {\rm Aut}(X_*)\to \GL(M)$$ 
  for $X_*=f^{-1}(*)$. Letting 
  $G={\rm im}(\rho_\bU)\subset {\rm Mon}$
  be the resulting finite group, we again
  have by \cite[\S~5(a)]{borel_finite}
  that there are only finitely many possibilities
  for $G$ up to ${\rm Mon}$- or $O(M_{
  \rm np})$-conjugacy---note that this also
  holds in the $b_2=3$ case, which is necessarily
  isotrivial.
  Since $f$ is of Picard type,
  $G$ must be acting not purely
  symplectically, i.e., $g^\ast \sigma\neq \sigma$
  for some $g\in G$.
  Thus, $H^{2,0}(X_*)$
  lies in the complexification of the direct
  sum of all non-trivial sub-representations
  of $G$ acting on $M_{\rm np}$, whose integral
  lattice equals $M^{(i)}$.  
  
  In summary we have argued that, whether
  the monodromy is finite or infinite,
  only finitely many 
  ${\rm Mon}$-orbits (or $O(M_{\rm np})$-orbits,
  when $b_2=4$) of inclusions
  $\bU_{s,*}=M^{(i)}
  \hookrightarrow (M_{\rm np},\cdot)$, 
  $s\in\cS^{(i)}$ can arise from
  a primitive symplectic fibration $f$
  inducing the local system $\bU_s$
  on $\cY_s^o$.
  Since we assumed
  that $f$ has Picard type,
  the classifying morphism lands in 
  $\Phi^o\colon \cY_s^o\to \cF_{\Lambda^{(i)}}$,
  where $\Lambda^{(i)} \coloneqq M^{(i)\perp}$.
  We conclude as in the abelian case.
\end{proof}

\begin{theorem}[cf.~Theorem \ref{thm:bound-ab}]\label{bd-families} 
Fix $g>0$ and $(\cY,\cB,\cM)\to \cT$.
\begin{enumerate}
\item[(AV)] There is a finite type family
$\cX^{\rm Alb}\to \cY/\cS$ of abelian fibrations
with rational section, such that 
for any
fibration $f\colon X\to \cY_t$ of abelian
$g$-folds, of Picard type,
indicating $(\cY_t,\cB_t,\cM_t)$,
there is some closed point $s\in \cS$, $s\mapsto t$,
for which 
$f^{\rm Alb}\colon X^{\rm Alb}\to \cY_t$
and $\cX^{\rm Alb}_s\to \cY_s$ are isomorphic
over the generic points of their bases.
\smallskip
\item[(PS)] 
There is a finite type family
$\cX\to \cY/\cS$ of primitive symplectic 
fibrations, such that 
for any primitive symplectic
fibration $f\colon X\to \cY_t$ of Picard type
indicating $(\cY_t,\cB_t,\cM_t)$,
with fibers in a fixed analytic deformation class,
there is some closed point $s\in \cS$, $s\mapsto t$,
such that the fibrations $f\colon X\to \cY_t$
and $\cX_s\to \cY_s$ are crepant birational
over the generic points of their bases.
\end{enumerate}
\end{theorem}

\begin{proof}
The statement (AV)
follows from Proposition \ref{bd-unzarhin},
and that an abelian fibration with section
is pulled back from the universal family
by its Hodge-theoretic classifying morphism
on the smooth locus.

Statement (PS) is a bit trickier. 
First, assume
that the $\bQ$-factorial terminalization
of fibers of $f$ satisfies $b_2\geq 5$. 
We define a {\it refined Hodge-theoretic
classifying morphism} as a lift of 
$\Phi^o_s\colon \cY_s^o\to 
\cF_{\Lambda^{(i)}}=[O^*(M^{(i)})
\backslash \bD_{M^{(i)}}]$ 
to the finite cover
$$
\widetilde{\Phi}^o_s\colon \cY_s^o\to 
[{\rm Fix}_{\rm Mon}(\Lambda^{(i)})
\backslash\bD_{M^{(i)}}],
$$ where ${\rm Fix}_{\rm Mon}(\Lambda^{(i)})$
denotes the elements of the monodromy
group ${\rm Mon}\subset O(M_{\rm np}$) 
fixing $\Lambda^{(i)}$.
Such lifts $\widetilde{\Phi}^o_s$ of $\Phi^o_s$ 
are finite, bounded in number, 
and vary constructibly in families, for
each $\Lambda^{(i)}\subset (M_{\rm np},\cdot)$
in Proposition \ref{bd-unzarhin}, since
we are lifting a bounded family
of maps $\Phi_s^o$ to a fixed finite (\'etale)
cover of the target. Note that the monodromy
of the fibration $f$ lands in ${\rm Mon}$,
fixing $\Lambda^{(i)}$, so the classifying
morphisms in question necessarily admit
such a lift. 

The morphism $\widetilde{\Phi}^o_s$ determines
the birational class of the general fiber
of the fibration $f$, 
since projective PS varieties 
with the same periods are birational 
\cite[Thm.~1.1(2)]{BL2022}.
But $\widetilde{\Phi}_s^o$
may not determine the birational 
class of the fibration; see Example
\ref{no-B-ex-2}.

Let $\partial_{\rm W}\subset \cY_s^o$
be the wall jumping locus 
(Def.~\ref{wall-jump}), which varies
constructibly in families of classifying
morphisms $\widetilde{\Phi}_s^o$.
By Proposition \ref{extend5}, any 
refined Hodge-theoretic
classifying morphism $\widetilde{\Phi}^o_s$
coming from a fibration of PS varieties
admits a lift over the open subset
$\cY_s^{oo} \coloneqq \cY_s^o\setminus
\partial_{\rm W}$ to a moduli-theoretic
classifying morphism
$$\Psi^{oo}_s\colon \cY^{oo}_s\to \cM_\Lambda$$
where $\cM_\Lambda$ is the DM stack
of $\Lambda$-polarized PS varieties
(Def.~\ref{ps-moduli})
of the specified deformation class, 
for an appropriate
choice of small cone $\sigma\subset \Lambda_\bR$ 
(Def.~\ref{def:small-cone}).

Since $\cM_\Lambda\to 
[{\rm Fix}_{\rm Mon}(\Lambda^{(i)})\backslash \bD_{M^{(i)}}]$ (or $\cM_\Lambda\to \cF_\Lambda$)
is quasi-finite, we deduce
that the lifts $\Psi_s^{oo}$ of $\widetilde{\Phi}_s^{o}\vert_{\cY_s^{oo}}$ 
(or of $\Phi_s^o\vert_{\cY_s^{oo}}$)
are finite in number and vary constructibly,
for each of the finitely many
possible choices of $\sigma$.
The fibration of $\Lambda$-polarized PS
varieties over $\cY_s^{oo}$ is determined
by $\Psi_s^{oo}$ and hence so is the 
birational class of $f\colon X\to \cY_t$.
We deduce the theorem for fibrations
whose fibers have $\bQ$-factorial
terminalizations with $b_2\geq 5$.

When $b_2=3, 4$ and $\rk \Lambda =1$,
$\cM_\Lambda$ (as in 
Def.~\ref{ps-moduli}) 
is still a separated
DM stack, since it is a substack
of the polarized $K$-trivial 
varieties of bounded volume (a primitive 
generator of $\Lambda$ gives the required
polarization) and fixed dimension.
Hence, the locus in $\cF_\Lambda$ 
over which the classifying morphism 
$\cM_\Lambda\to \cF_\Lambda$
is finite, is the complement
of some fixed 
Zariski-closed subset
$\partial_{\rm W}\subset \cF_\Lambda$
which does not contain the entire image
of $\Phi_s^o$. Now, Proposition \ref{extend5} 
and the same
proof as above apply, even though
we cannot explicitly describe
the locus $\partial_{\rm W}$.

Finally, when $b_2=4$ and $\rk \Lambda=2$,
the fibration $f$ is necessarily isotrivial,
and the monodromy of the fibration fixes all
N\'eron--Severi classes. Since $\Lambda$
has been bounded, we may choose a 
vector $L\in \Lambda$ of bounded positive norm.
Let $X_*$ denote the fiber over some 
general point. Then, the birational class
of $f$ is uniquely determined
by a representation $\pi_1(\cY_s^{oo},*)\to {\rm Aut}_\Lambda(X_*)$ where ${\rm Aut}_\Lambda(X_*)$
is the group of automorphisms of $X_*$ 
acting trivially on
$\Lambda$.

Since $\Lambda$
has been bounded, we may choose a primitive
vector $L\in \Lambda$ of bounded positive norm.
Then, the birational class of $X_*$ can be bounded, 
via the Torelli theorem for $\bZ L$-polarized
PS varieties, see also the discussion before Lemma
\ref{lemma_smaller_norm}. A birational
map $\phi\colon X_*\dashrightarrow X_*'$ gives an 
identification $\phi_*\colon {\rm Aut}_\Lambda(X_*)\to 
{\rm Aut}_\Lambda(X_*')$ because $\phi\circ 
\alpha \circ \phi^{-1}$,
$\alpha\in {\rm Aut}_\Lambda(X_*)$
is a birational isomorphism of $X_*'$ acting trivially
on ${\rm NS}(X')=\Lambda$, hence preserves an ample
class and is a morphism. Thus, $|{\rm Aut}_\Lambda(X_*)|$ is bounded. So,
the pullback of $f$ to some bounded degree
\'etale cover 
$\widetilde{\cY}_s^{oo}\to \cY_s^{oo}$ 
is a trivial fibration, and 
$\widetilde{\cY}_s^{oo}\times X_*\dashrightarrow
\widetilde{\cY}_s^{oo}\times X_*'$ are equivariantly birational.
The result now follows from finiteness of 
${\rm Aut}_\Lambda(X_*)$.
\end{proof}

\begin{remark}
In the proof of Theorem \ref{bd-families},
one may argue directly from the quasi-finiteness
of $\cM_\Lambda\to \cF_\Lambda$. But the 
current formulation is conceptually clearer, 
because the map $\cM_\Lambda\to 
[{\rm Fix}_{\rm Mon}(\Lambda^{(i)})
\backslash \bD_{M^{(i)}}]$ 
is $1$-to-$1$ onto its image (on coarse spaces). 

In general, though,
it is not an isomorphism 
over any open substack.
Rather, there is an open substack
of $[{\rm Fix}_{\rm Mon}(\Lambda^{(i)})
\backslash \bD_{M^{(i)}}]$ over which
$\cM_\Lambda$ is a $G$-gerbe, 
where $G$ is the group of automorphisms
of the general $X$, $[X]\in \cM_\Lambda$,
which act trivially on $H^2(X,\bZ)_{\rm tf}$.
Unlike for K3 surfaces,
this group may be nontrivial
\cite[Cor.~3.3]{BSNWS}.
Note that $G$ is finite
since it preserves a polarization. 
Over the open substack where $\cM_\Lambda$
is a $G$-gerbe,
a lift of the Hodge-theoretic
classifying map 
$\Phi^o\colon Y^o\to \cF_\Lambda$ to a 
moduli-theoretic classifying map
$\Psi^o\colon Y^o\to \cM_\Lambda$
is determined by a crossed homomorphism
valued in $G$, as in
Proposition \ref{bd-stack}.
\end{remark}

    \begin{remark}[Shafarevich conjecture for symplectic varieties] 
    \Cref{bd-families}(PS) should be compared with \cite[Thm.~5.1]{Fu2025}. Although the general 
    strategy of the proof is similar, e.g., both use 
    versions of the Kuga--Satake construction in families,
    they differ at the technical level, and none of 
    the results implies immediately the other. 
    The set-up of {\it loc.~cit.}~is that 
$f^o \colon X^o \to Y^o$ is a (not necessarily 
projective) locally trivial 
fibration over a fixed 
smooth variety $Y^o$ with a fixed
isomorphism type of a $\bQ$-factorial terminal fiber $X_*$
over a base point $*\in Y^o$. They prove finiteness of such fibrations up to birational equivalence.

Note that, in \Cref{bd-families}, the locus where the fibration $f \colon X \to Y$ is locally trivial is not bounded by assumption; the bound on the generalized pair indicated by such fibrations is sufficient. Observe also that \cite[Thm.~5.1]{Fu2025} does not require the Picard type assumption (Def.~\ref{def:pic-type}) in \Cref{bd-families}(PS), but it is somehow implicit in the choice of the isomorphism type of $X_*$. Indeed, by Remark
\ref{rem:ps-picard}, 
a fibration by PS varieties is of Picard
type unless it is isotrivial, and the
monodromy representation 
$\rho\colon \pi_1(Y^o)\to {\rm Aut}(X_*, \sigma_*)$
factors through the symplectic automorphism group
of $X_*$. In the Picard type case, we have
birational boundedness by \Cref{bd-families}(PS). 
In the non-Picard type case, the fibrations
with a fixed isomorphism type of $X_*$ correspond
to representations $\rho$ above, which are finite
in number. Finally, the condition that 
$X_*$ have a fixed isomorphism type
rigidifies the classifying map, see
e.g.~\cite[Prop.~3.6]{peters}, 
\cite[\S~10]{kerr-pearlstein}, 
\cite[Thm.~1.7]{javlitt}, \cite[\S 4]{Fu2025},
so that boundedness 
reduces to finiteness.
\end{remark}

\subsection{Boundedness of the base}

The following lemma is well known to experts.
We include it for the reader's convenience.

\begin{lemma}\label{lemma:gentopair}
Let $(Y,B_Y,\bM.)$ be a generalized klt (resp.~lc) pair.
Assume that $c\bM.$ is b-free for some positive integer $c$.
Then, we may choose $0 \leq \Gamma_Y \sim_\QQ \bM Y.$ such 
that $(Y,B_Y+\Gamma_Y)$ is a klt (resp.~lc) pair and the 
coefficients of $\Gamma_Y$ are in $c^{-1} \mathbb{Z}$. 
\end{lemma}

\begin{proof}
Let $\pi \colon Y' \rar Y$ be a log resolution 
of $Y$ where $\bM.$ descends.
We write
\begin{align*}
\K Y'.+B_{Y'}+\bM Y'.=\pi^\ast (K_Y+B_Y+\bM Y.).
\end{align*}
Since $Y'$ is smooth and $(Y,B_Y,\bM.)$ is 
generalized klt (resp. lc), $(Y',B_{Y'})$ is 
sub-klt (resp. sub-lc).
By assumption, $|c\bM Y'.|$ is a free linear series.
By Bertini's theorem, choose $D' \in |c\bM Y'.|$ 
such that the class of singularities of $(Y',B_{Y'})$ 
is the same as those of $(Y',B_{Y'}+D')$. Then, we 
can set $\Gamma_Y  \coloneqq  \pi_\ast (c^{-1}D')$.
\end{proof}

\begin{lemma}\label{cor:bases}
Let  $f \colon X \rar Y$ be a fibration from a projective 
$K$-trivial variety $X$ to a rationally connected base $Y$. 
Let $(Y,B_Y,\bM.)$ be the generalized pair structure induced 
on $Y$ via $f$. Assume that $c\bM.$ is b-free for 
some positive integer $c$.

Then we may choose $0 \leq \Gamma_Y \sim_\QQ \bM Y.$ such 
that the pairs $(Y,B_Y+\Gamma_Y)$ arising in this way 
are log bounded in codimension 1.
\end{lemma}

\begin{proof}
Since $K_X \sim 0$ holds, we have 
$f^\ast (K_Y + B_Y + \bM Y.) \sim 0$ by \cite[\S~7.5]{PS09} 
and $K_Y + B_Y + \bM Y. \sim 0$ by \cite[Prop.~5.3]{FM20}.
Furthermore, $(Y,B_Y,\bM.)$ is generalized klt, see 
\cite[Prop.~4.16]{Fil20}.
Then, by Lemma \ref{lemma:gentopair}, we may choose an 
integral divisor $0 \leq c\Gamma_Y \sim  c\bM Y.$ such that:
\begin{enumerate}
\item $(Y,B_Y+\Gamma_Y)$ is klt;
\item $c(K_Y + B_Y + \Gamma_Y) \sim 0$; and
\item $Y$ is rationally connected. 
\end{enumerate}
We conclude by \cite[Thm.~1.5]{BDCS}.
\end{proof}

\begin{corollary}\label{cor:basesLagrbound}
Let $f \colon X \to Y$ be a Lagrangian fibration from a projective 
primitive symplectic variety of dimension $2d$.
Let $(Y,B_Y,\bM.)$ be the generalized pair structure induced on 
$Y$ via $f$. Then, we may choose $0 \leq \Gamma_Y \sim_\QQ \bM Y.$ 
such that the pairs $(Y,B_Y+\Gamma_Y)$ are bounded.
\end{corollary}
\begin{proof}
We apply Lemma \ref{cor:bases}.
Since $\rho(Y)=1$ and $Y$ is 
$\QQ$-factorial by Proposition
\ref{prop:base}, all flops are trivial. So
boundedness of the pairs
in codimension $1$ implies
boundedness.
\end{proof}

\begin{remark}
In the context of Corollary~\ref{cor:bases}, if 
the general fiber of $f$ is smooth, the b-Cartier 
index of $\bM.$ is controlled independently of our 
work by \cite[Thm.~3.1]{FM00}, see also 
\S~\ref{hodge-moduli}.
Since $K_Y+B_Y+\bM Y.\sim 0$ and $c\bM Y.$ is 
b-Cartier for some positive integer $c$,
we may deduce the weaker statement that $Y$ is 
bounded modulo flops by \cite[Thm.~1.3]{HJ22}.
\end{remark}

\begin{example}\label{rc-base-ex}
The conclusion of \Cref{cor:bases} 
holds in the following key examples:
\begin{enumerate}
\item $X$ is an irreducible Calabi--Yau variety of fixed
dimension, and the general fiber of $f$ is primitive 
symplectic or abelian. The irreducibility assumption 
grants that $Y$ is automatically rationally 
connected\footnote{If $Y$ were 
not rationally connected, the image of 
its MRC fibration would be a non-uniruled variety
of positive dimension by \cite[Cor.~1.4]{GHS2003}. 
However, by \cite[Thm.~14]{KL2009}, a quasi-\'{e}tale 
cover of $X$ would split as the product of two 
Calabi--Yau varieties of positive dimension, which is 
a contradiction by the irreduciblity of $X$.
Hence, $Y$ must be rationally connected.} 
and so the lemma applies, because $c$ is bounded
by Theorem \ref{thm:weakeffective}. 
\item $X$ is a Calabi--Yau 3-fold. If $f$ is a fibration by
K3 or abelian surfaces, the result is automatic as $Y=\PP^1$.
Elliptic Calabi--Yau 3-folds are instead fibered either 
onto rational or singular Enriques surfaces, see 
\cite[Prop.~2.2]{grassi91} 
and
\cite[Prop.~2.3]{Gross1994}.
In the former case, the above proof applies.
In the latter case, $B_Y = 0 \sim_{\QQ} \bM Y.$ and we can apply \Cref{cor:bases} thanks to \cite[Thm.~6.3]{FHS2024}.
\item $X$ is a primitive symplectic variety. 
See Proposition \ref{prop:base}.
\end{enumerate}
In the last two cases, the pairs $(Y,B_Y+\Gamma_Y)$ are 
actually bounded, 
since all flops are trivial either for 
dimensional reasons or because $\rho(Y)=1$ and 
$Y$ is $\QQ$-factorial by 
\Cref{prop:base}.
\end{example}

\begin{theorem}[cf.~Theorem \ref{thm:icy-ab}(PS)]\label{bd-families-2} 
Fix $d>0$. 

\begin{enumerate}
\item[(AV)]
Abelian fibrations
$f\colon X\to Y$ of Picard type admitting a rational section over
a rationally connected base,
such that $X$ is $K$-trivial
of dimension $d$,
are birationally bounded.\smallskip

\item[(PS)]
Primitive 
symplectic fibrations $f\colon X\to Y$
of Picard type, over a
rationally connected base, with fibers of a 
fixed analytic deformation class,
such that $X$ is $K$-trivial
of dimension $d$, are birationally bounded.
\end{enumerate}

In both cases, the birational map on bases and
on total spaces can
be assumed isomorphisms in codimension $1$.
\end{theorem}

Note that the hypotheses of the above theorem 
are satisfied by any irreducible
Calabi--Yau variety $X$ endowed with an 
abelian or primitive symplectic fibration 
$f \colon X \to Y$.

\begin{proof}
By Lemma \ref{cor:bases} and 
Theorem \ref{thm:weakeffective}, 
there is a finite
type family $(\cY,\cB,\cM)\to \cT$
such that:
For any abelian or primitive symplectic
fibration $f\colon X\to Y$ with 
$\dim X=d$ and $X$ a $K$-trivial variety, 
over a rationally connected base $Y$,
there is a closed point $t\in \cT$ and a 
birational map $\pi\colon \cY_t\dashrightarrow Y$, 
which is an isomorphism in codimension $1$, 
for which $\pi_*(\cB_t)=B_Y$ and 
$\pi_*(\cM_t)\sim_\bQ \bM Y.$. 

Now $f \circ \pi^{-1}\colon X
\dashrightarrow \cY_t$
is a rational fibration which,
over a big open set of $\cY_t$ is a morphism
and $K$-trivial. Furthermore, it induces
$\cB_t$ and $\cM_t$ as the boundary and moduli
divisors. By \cite[Prop.~2.9]{Fil24}, we can
birationally modify $f$ to a fibration
$f' \colon 
X' \to \cY_t$ of a $K$-trivial variety,
which indicates $(\cY_t,\cB_t,\cM_t)$,
see Definition~\ref{def:indicates}.
The result now follows from Theorem \ref{bd-families},
since $f\circ \pi^{-1}$,
and in turn $f$, are 
generically bounded.
\end{proof}

\begin{corollary}
Primitive symplectic varieties of dimension $2d$,
admitting a Lagrangian fibration with rational section,
are birationally bounded.

There are only finitely many analytic
deformation classes of primitive symplectic varieties
containing one admitting a Lagrangian fibration
with rational section.
\end{corollary}

\begin{proof} The first part follows from
Example \ref{rc-base-ex}(3) and Theorem \ref{bd-families-2}.
Now, the second part follows from property 
\eqref{surjective} in the proof 
of Proposition \ref{extend4}
and the fact that birational
$\bQ$-factorial terminal primitive symplectic
varieties are deformation
equivalent \cite[Thm.~6.16]{BL2022} (via
locally trivial deformations).
\end{proof}

\begin{remark}
If $f \colon X \to Y$ is a Lagrangian fibration from an 
irreducible symplectic manifold with reduced fibers in 
codimension 1 and such that the integral homology of 
$X$ is indivisible, then Bogomolov--Kamenova--Verbitsky 
\cite{BKV} show that $f$ deforms to a Tate--Shafarevich 
twist with a meromorphic section; see also 
\cite[Thm.~C]{AR2021}. Koll\'{a}r \cite{kollar_new2} has 
recently showed that a Tate--Shafarevich twist of an 
arbitrary $K$-trivial locally projective klt abelian 
fibration whose general fiber is indivisible as a 
homology class admits a section. However, it is unclear 
under which assumptions the original fibration and its 
twists are deformation equivalent \cite[(13) and (20)]{kollar_new2}. At least in the smooth case, the 
indivisibility of the general fibers is necessary 
for the previous results; see
\cite[\S 1.5]{BKV} and \cite[Comment 9]{kollar_new2}. 

An alternative source of abelian fibrations with rational 
section is the Albanese fibration $f \colon X^{\rm Alb} \to Y$ 
of any given abelian fibration $f\colon X \to Y$ 
(Def.~\ref{defn:alb}), unique up to birational modifications. 
In \cite[Cor.~6.4]{sacca}, Sacc\`{a} showed that if $f$ is 
a Lagrangian fibration with reduced fibers in codimension 
1, then $X^{\rm Alb}$  can be chosen $\bQ$-factorial, 
terminal and symplectic. We generalize the result in 
\Cref{thm:alb} to arbitrary $K$-trivial varieties. 
In particular, we drop the assumption on reduced 
fibers. In general, an abelian fibration and its 
Albanese fibration are not deformation equivalent, 
even in the Lagrangian case, e.g., \cite[Ex.~6.7]{sacca}. 
Notwithstanding, we show that there are only finitely 
many deformation types of 
primitive symplectic varieties admitting a 
Lagrangian fibration whose Albanese fibration
lies in a fixed deformation class.
\end{remark}

\subsection{Examples and counterexamples}

We now give some examples which 
demonstrate that Propositions 
\ref{bd-map}, \ref{bd-stack}, 
\ref{bd-unzarhin} and Theorem \ref{bd-families}
are near-optimal statements, at least regarding 
boundedness over the generic point.

\begin{example}\label{no-M-ex}
Suppose we only bound $(\cY,\cB)\to \cT$ but not $\cM$.
We take $\cT={\rm pt}$ and $\cY=\bP^1$ and 
$\cB = 0$, the trivial divisor.
Now consider elliptic fibrations
$X\coloneqq\{y^2=x^3+ax+b\}$
where $a\in H^0(\bP^1,\cO(4d))$ and 
$b\in H^0(\bP^1, \cO(6d))$ are general. We have 
$\bfM_Y\sim \cO_{\bP^1}(d)$.
Letting $d\to \infty$, 
these elliptic surfaces induce $(Y,B_Y)=(\bP^1,0)$
but are not birationally
bounded, as $h^{2,0}(X)$
goes to infinity.
\end{example}

\begin{example}\label{no-B-ex}
Suppose we only bound $(\cY,\cM)\to \cT$ but not $\cB$.
Let $\partial$ be a set of $6d$ distint points on $\bP^1$.
Let $\wY\to \bP^1$ be the order $6$ cyclic 
cover of $\bP^1$ branched over $\partial$. Let $E$
be the elliptic curve with an order $6$ automorphism.
Consider the relatively minimal
model of the diagonal quotient 
$$X \coloneqq \wY\times E/C_6\dashrightarrow X'\to Y
\coloneqq E/C_6\simeq \bP^1.$$ 
Then $X'\to\bP^1$ is an isotrivial elliptic fibration
with section, that has a cuspidal fiber over each
point of $\partial$; it is given by 
an equation $X'=\{y^2=x^3+b\}$. We have
$B_Y=\tfrac{1}{6}\partial\sim \cO_{\bP^1}(d)$.
These fibrations induce $(Y,M_Y)=(\bP^1,0)$. But they 
are not birationally bounded, since $h^{2,0}(X')$
goes to infinity. Here Proposition \ref{bd-map}
goes through, which only requires control of $\cM$.
But Proposition \ref{bd-stack} fails as we 
do not control the open set $Y^o\subset Y$ 
on which the classifying morphism $\Phi^o\colon Y^o\to \cA_{1,1}$ exists.
\end{example}

\begin{example}\label{no-B-ex-2}
Let $F \coloneqq {\rm Kum}_n(A)$ be the Kummer
variety of an abelian surface $A$, i.e., the
kernel of the summation map 
${\rm Hilb}^{n+1}(A)\to A$. Inversion in 
the group law on $A$ induces an involution $\sigma$ on $F$,
acting trivially on $H^2(F,\bZ)$. Let $\tau$
be the hyperelliptic involution on a hyperelliptic
curve $C$ of genus $g$, and define $X \coloneqq F\times C/C_2$, where $C_2$ is generated by
the diagonal involution
$(\sigma,\tau)$.

Then $f\colon X\to \bP^1\simeq C/C_2$ 
is an isotrivial
family of PS varieties of ${\rm Kum}_n$-type,
with singular fibers over the $2g+2$ branch
points of $\tau$. Ranging over all $g\geq 0$,
these fibrations are unbounded. For example,
$K_X\simeq f^*(\cO_{\bP^1}(g-1))$. But since $\sigma$
acts trivially on $H^2(F,\bZ)$, the weight $2$
$\bZ$-VHS induced by $f$, for all $g\geq 0$,
is the same trivial $\bZ$-VHS.

Hence, the classifying morphism 
$\Phi_s^o\colon Y^o\to \cF_\Lambda$ to
the DM stack of Hodge structures is insufficient 
to bound $f$, and we really do require
the data of the lift $\Psi_s^o\colon Y^o\to 
\cM_\Lambda$ to the DM stack of $\Lambda$-polarized
PS varieties---this lift only exists
over the complement of the branch points.
Note that this example is not of Picard type, 
but it is also possible to furnish a
such a counter-example, by
performing the same construction
to the pullback along $C\to \bP^1$
of a ${\rm Kum}_n$-fibration, associated to
an abelian surface fibration
of Picard type over $\bP^1$.
\end{example}

\begin{example}\label{pic-ex}
The assumption on Picard type is necessary:
Take, for instance, the collection of 
abelian surface fibrations
$f\colon A\to {\rm pt}$ over a point.
The bases $(\cY,\cB,\cM)$
lie in a bounded family. Propositions \ref{bd-map},
\ref{bd-stack} go through: The induced Z-period 
maps, classifying morphisms
$$\Phi_Z,\, \Phi^o_Z\colon {\rm pt} \to A_{8}, 
\,\cA_{8}$$ lie in a bounded
family. But Proposition \ref{bd-unzarhin} fails;
we cannot recover the polarizing lattice
$\Lambda\subset \wedge^2(\bZ^4)$ from the local system,
because 
$H^0({\rm pt} , R^2f_*\underline{\bZ})=H^2(A,\bZ)$ 
is not of Hodge--Tate type. We are unable
to bound the moduli spaces $A_{2,{\bf d}}$, 
$\cA_{2,{\bf d}}$ through which $\Phi_Z$, $\Phi^o_Z$ 
factor.
\end{example}

\section{Tate--Shafarevich twists of abelian fibrations
with multiple fibers}
\label{sec:ts}

Given a family of bases $(\cY,\cB,\cM)\to \cT$,
Theorem \ref{bd-families} ensures that
there is a finite type family $\cX^{\rm Alb}\to \cY/\cS$,
containing a birational model of 
the Albanese fibration of any abelian fibration
$f\colon X\to \cY_t$ of Picard type,
indicating some $(\cY_t,\cB_t,\cM_t)$. Alternatively,
the datum of $f^{\rm Alb}\colon
X^{\rm Alb}\to \cY_t$ is determined by
some classifying morphism $\Phi_s^o\colon \cY_s^o
\to \cA_{g,\Lambda}$ for an 
appropriate $s\in \cS$ over $t\in \cT$,
and an appropriate polarization lattice $\Lambda$.

The goal of this section
is to understand the relation between
$f\colon X\to \cY_t$ and its Albanese
fibration $f^{\rm Alb}\colon X^{\rm Alb}\to \cY_t$. 
Thus, we study the geometric analogue
of the Tate--Shafarevich group.
After finite type base changes and 
passing to appropriate locally closed 
stratifications, 
our constructions can be performed 
relatively (Rem.~\ref{constructible-remark}),
so we will restrict here to the class of 
abelian fibrations $f\colon X\to Y$
inducing a fixed generalized pair $(Y,B,\bfM)$
and birational class of Albanese variety, 
encoded by $(Y,B,\bfM,\Phi^o)$.\footnote{Though
$\bfM$ is superfluous, as it is determined by $\Phi^o$.} Unlike in 
Section \ref{sec:bd}, we will now care about the 
coefficients of $B$, not just its support.

Throughout this section, we will
assume that $X$ is proper.

\subsection{Kulikov models}\label{ts1} 

\begin{definition}\label{def:Kulikov}
An abelian fibration is {\it Kulikov}
if the base and total space are smooth, 
the discriminant is a smooth divisor, and
the fibration is relatively $K$-trivial,
semistable, and analytically-locally 
trivial over the discriminant. That is, 
the analytic local equation of the total space 
is $x_1\cdots x_m=t$, where $t$ is a local
coordinate on the base cutting out the discriminant
divisor.
\end{definition}

\begin{proposition}\label{nicest-model}
Let $f\colon X\to Y$ be an abelian fibration
inducing $(Y,B,\bfM,\Phi^o)$. There is a 
finite Galois $G$-cover $Z\to Y$,
and a $G$-invariant, big open subset 
$Z^+\subset Z$,
satisfying the following condition: 
Let $g\colon W\to Z$ be the normalized
base change of $f$. \'Etale-locally over $Z^+$, 
the restriction $g^+\colon W^+\to Z^+$ of $g$
is translation-birational (Def.~\ref{def:tr-bir})
to a Kulikov
model $h^+\colon V^+\to Z^+$ of the Albanese fibration,
which can furthermore be chosen so that
the natural birational $G$-action is regular.

Finally, the cover $Z\to Y$
can be chosen in a manner 
depending only on $(Y,B,\bfM,\Phi^o)$.
\end{proposition}

\begin{remark}\label{remark-indep}
We will later prove (Prop.~\ref{local-no-twist})
that the big open set 
$Z^+\subset Z$
can also be chosen 
in a manner depending only on $(Y,B,\bfM,\Phi^o)$.
\end{remark}

\begin{remark}\label{remak-G-equiv}
An important and subtle point is that,
even though both $g^+\colon W^+\to Z^+$
and $h^+\colon V^+\to Z^+$ are \'etale-locally
birational, and both admit regular $G$-actions,
they are {\it not}, in general, $G$-equivariantly
\'etale-locally birational. This is
due to the presence of multiple fibers of $f$ 
in codimension $1$.
\end{remark}

\begin{proof}
As in (\ref{open1}), we 
have a classifying morphism $\Phi^o\colon Y^o\to 
\cA_{g,\Lambda}$ where $$Y^o \coloneqq Y\setminus 
({\rm indet}(\Phi)\cup \partial_{\rm H}\cup \supp B\cup  
Y_{\rm sing}).$$
We may pass to a finite Galois cover
$Y_1^o\to Y^o$ for which we have a morphism
$\Phi_1^o\colon Y_1^o\to 
\cA_{g,\Lambda}[3]$ to the moduli stack
of $\Lambda$-polarized abelian varieties with
full level $3$ structure. Let $Y_1^o\hookrightarrow Y_1$
be the normalization of $Y$ in $\bC(Y^o_1)$.
Since the level $3$ subgroup is neat,
$\cA_{g,\Lambda}[3]$ exists as a scheme,
there is a universal family of abelian
varieties over it, and $\Phi_1^o$
is uniquely determined by the corresponding
period map $\Phi_1$.

Let $\partial_i\subset \partial_{\rm H}$
be an irreducible component of codimension $1$.
The monodromy (on $H_1$) of a meridian
about $\partial_i$ is a unipotent
matrix $T_i\in {\rm Sp}_\Lambda[3]$, well-defined
up to conjugacy. Let
$N_i \coloneqq \log (T_i)\in \mathfrak{sp}_{\Lambda}[3]$ 
be its logarithm. Note that $N_i=T_i-I$ is integral
since $N_i^2=0$. There
is a toroidal extension $$
\cA_{g,\Lambda}[3]\hookrightarrow
\overline{\cA}_{g,\Lambda}
[3]^{\bigcup_i \bR_{\geq 0}N_i}$$
whose fan is the union of the rays spanned
by $N_i$ in the appropriate 
cone of positive semi-definite 
$g\times g$ matrices. 
See \cite{namikawabook, mumford-amrt} 
for reference on toroidal 
compactification of Siegel spaces.

The boundary of $\overline{\cA}_{g,\Lambda}
[3]^{\bigcup_i \bR_{\geq 0}N_i}$
consists of a disjoint collection
of smooth (non-proper)
divisors
$E_i$ of dimension $D-1$,
where $D \coloneqq \dim \cA_{g,\Lambda}[3]$.
Here $E_i$ admits a sequence of fibrations
\begin{align}\label{fibrations}
E_i\to F_i\to \cB_{g_i}\subset 
\overline{\cA}_{g,\Lambda}[3]\end{align}
over a 
Baily--Borel boundary stratum $\cB_{g_i}$
of abelian $g_i$-folds with appropriate
polarization and level structure, 
determined by the monodromy operator $N_i$. 
The fibration
$F_i\to \cB_{g_i}$ 
is a smooth, proper
fibration of abelian varieties, 
and $E_i\to F_i$ is an algebraic torus
bundle. Cases to keep in mind
are the $g_i=g-1$ case, where $E_i\simeq F_i$
and the $g_i=0$ case, where $E_i\simeq 
(\bC^\ast )^{D-1}$ and $F_i\simeq \rm pt$.

In the following, we consider $\cA_{g,\Lambda}$ as a parameter
space of weight $-1$ Hodge structures. Following Deligne
\cite[Sec.~10]{deligne1974}, a point in
$E_i$ defines the graded-polarized 
limit mixed Hodge structure, with (integral) weight
and Hodge filtrations
\begin{align*}
\mathscr{W}_\bullet & \coloneqq  \mathscr{W}_{-2}\subset \mathscr{W}_{-1}\subset \mathscr{W}_0 &&  
 \mathscr{W}_{-1}\coloneqq \ker N_i\\
 \mathscr{W}_{-2}&\coloneqq ({\rm im}\,N_i)^{\rm sat} && 
 N_i\colon (\mathscr{W}_0/\mathscr{W}_{-1})_\bQ \xrightarrow{\sim} (\mathscr{W}_{-2})_\bQ \\
\mathscr{F}^\bullet & \coloneqq  \mathscr{F}^{-1}\supset \mathscr{F}^{0} && 
\mathscr{F}^{-1}=(\mathscr{W}_0)_\bC.
\end{align*}
Any $1$-parameter
degeneration with monodromy $T_i$ has a limit
mixed Hodge structure of this type. It is 
$\Lambda$-polarized and has mixed Hodge
type
$\{(-1,-1),(-1,0),(0,-1), (0,0)\}$, with Hodge numbers 
$h^{-1,0}=h^{0,-1}=g_i$ and $h^{0,0} =h^{-1,-1}= g-g_i$.
The image under $E_i\to F_i$ records the truncated
mixed Hodge structure on $\mathscr{W}_{-1}$
of type $\{(-1,-1), (-1,0), (0,-1)\}$,
whose associated Albanese variety
\begin{align}\label{w-quotient}
(\mathscr{W}_{-1})_\bC/(\mathscr{F}^0\cap (\mathscr{W}_{-1})_\bC + \mathscr{W}_{-1})\simeq \cF^{-1}/(\cF^0+\mathscr{W}_{-1})\end{align}
is the unique semiabelian variety
occurring as the limit of the degeneration.
Thus, we can think of the sequence of fibrations
(\ref{fibrations}) Hodge-theoretically, as the maps
$$(\mathscr{W}_\bullet,\, \mathscr{F}^\bullet)\mapsto 
(\mathscr{W}_{\bullet \leq -1},\, \mathscr{F}^\bullet)\mapsto 
({\rm gr}_{-1} \mathscr{W}_\bullet, \,\mathscr{F}^\bullet).$$

Let
$Y_1' \coloneqq Y_1\setminus (\partial_{\rm H})_{\rm sing} 
\cup (Y_1)_{\rm sing}$
be the locus where $Y_1$ is smooth, and $\partial_{\rm H}$
is a smooth divisor.\footnote{In $Y_1^{\rm reg}$,
the Hodge-theoretic bundary
$\partial_{\rm H}$ is divisorial---it is 
detected by infinitude of the local monodromy.}
The period map $\Phi_1$ lifts to a morphism
\begin{align}
\label{period-lift}
\Psi_1\colon Y_1'\to \overline{\cA}_{g,\Lambda}
[3]^{\bigcup_i \bR_{\geq 0}N_i}.\end{align}
This is a consequence 
of the nilpotent orbit theorem \cite[Thm.~4.9]{schmid},
together with the fact
that $\bZ$-PVHS extend over 
codimension $2$ loci in a smooth
variety. 
Over the toroidal extension we can produce
a family $\overline{\cV}\to 
\overline{\cA}_{g,\Lambda}
[3]^{\bigcup_i \bR_{\geq 0}N_i}$
with smooth total space $\overline{\cV}$, which
is Kulikov over each boundary component $E_i$ via
a relative Mumford construction, see 
\cite[\S3]{alexeev-nakamura}, 
\cite{mumford-construction}, 
\cite{alexeev2002} for further
details, and \cite[Sec.~7]{AE2023}
for an analogous 
construction for
K3 surfaces.
We give an example in the $g_i=0$ case for 
simplicity of exposition;
otherwise one must perform the Mumford construction
relatively over the boundary stratum of 
$\overline{\cA}_{g,\Lambda}[3]$. 

\begin{construction}[$g_i=0$]\label{mumford-construction}
Conjugate $N_i\in \mathfrak{sp}_\Lambda[3]$ 
until it is of the $g\times g$ block
form $$\twobytwo{0}{M_i}{0}{0}$$ for 
$M_i\in {\rm Mat}_{g\times g}(\bZ)$ nondegenerate.
Assume $N_i$ is primitive (or otherwise, divide
it by its imprimitivity).
Let $L_i\subset \bZ^g$
be the lattice spanned by rows of $M_i$ and 
take an $L_i$-periodic tiling $\cT_i$ of $\bR^g$ into 
lattice simplices of volume $1/g!$. 
Viewing $\cT_i$ as a tiling of an affine hyperplane
$\bR^g\times \{1\}\subset \bR^g\times \bR$ at height
$1$, the cone ${\rm Cone}(\cT_i)$ over it
is a rational polyhedral fan in $\bR^g\times \bR$.

Since
the lattice simplices in $\cT_i$
are of the minimal possible
volume $1/g!$, each polyhedral cone 
$\sigma \in {\rm Cone}(\cT_i)$
is a standard affine cone, i.e.~${\rm Cone}(\cT_i)$
is a regular fan. 
The lattice
$L_i$ acts on $\bR^g\times \{1\}$ by translations,
which may be extended uniquely to linear transformations
of $\bR^g\times \bR$ preserving the fan, and also
preserving the projection to the second
factor $\bR$. Thus, there is an $L_i$-equivariant
morphism of fans $${\rm Cone}(\cT_i)\to \bR_{\geq 0}=\{0\}\cup \bR_{>0}.$$

Let $X(\fF)$ denote the toric
variety associated to a fan $\fF$.
Passing from fans to toric varieties,
$$
L_i\backslash X({\rm Cone}(\cT_i))\to \bC=X(\bR_{\geq 0})
$$
defines a degeneration of abelian varieties over an 
analytic neighborhood of $0\in \bC$,
whose central fiber is the $L_i$-quotient of the toric
boundary of $X({\rm Cone}(\cT_i))$. The action of $L_i$
is free, and $X({\rm Cone}(\cT_i))$ is smooth, because
the fan is regular. Thus, the quotient is also smooth.

The toric boundary of $X({\rm Cone}(\cT_i))$ is
an infinite periodic quilt of smooth toric varieties 
and $L_i$ freely translates the quilt.
The result is a $1$-parameter Kulikov degeneration;
in particular, the central fiber
is an snc variety and the total space is smooth.

There is a torus-worth $E_i\simeq (\bC^\ast )^{D-1}$ 
of twists of the $L_i$ action by translations
preserving the polarization $\Lambda$ on the 
general fiber, and performing this twisted construction
relatively gives the Kulikov extension $\overline{\cV}$
over the boundary divisor $E_i$.

By ensuring $\cT_i$ is the bending locus
of an appropriate convex PL function,
we can assume that 
$\overline{\cV}\to \overline{\cA}_{g,\Lambda}
[3]^{\bigcup_i \bR_{\geq 0}N_i}$ is 
relatively projective.
\end{construction}

Let $h_1'\colon V_1'\to Y_1'$ be the pullback of 
$\overline{\cV}$ along $\Psi_1$. 
Over components $\partial_i\subset \partial_{\rm H} 
\cap Y_1'$ for which $N_i$ is primitive, 
$V'_1$ is Kulikov since it is the pullback
along a transverse slice of the universal Kulikov 
model $\overline{\cV}$. When $N_i$ is imprimitive
of imprimitivity $k_i$,
the local analytic equation of $V_1'$ over 
$\partial_i$ is $x_1\cdots x_m=t^{k_i}$ where $t$
is a local coordinate cutting out $\partial_i$.
Thus, taking a complete
$k_i$-subdivision of $\cT_i$ we may
toroidally resolve the pullback to a Kulikov model.
Let $h_1\colon V_1\to Y_1$ be a 
compactification over $Y_1$.

In summary, we have found a particularly nice
model $h_1\colon V_1\to Y_1$
of the Albanese fibration of $f$, which is Kulikov
over
a big open set $Y_1'\subset Y_1$ that depends
only on $(Y,B,\bfM,\Phi^o)$. But, this may not be,
even analytically-locally over $Y_1$, bimeromorphic
to the base change of $f$. We argue that this failure
is entirely due to multiple fibers:

\begin{lemma}\label{finally-rid-of-B} 
Let $\cX, \cV\to (D,0)$ be
degenerations of abelian varieties with level $3$
structure, over an 
algebraic curve,
and suppose that
\begin{enumerate}
\item $\cX_t\simeq \cV_t$ for all $t\in D^\ast $;
\item $K_\cX\sim_{\bQ} 0$, $K_\cV\sim 0$; and
\item $\cV$ is semistable, i.e., $\cV$ is 
Kulikov.
\end{enumerate}
Suppose furthermore that the log canonical threshold
of $\cX_0\subset \cX$ satisfies
${\rm lct}(\cX_0)\geq \tfrac{1}{c}$ for some integer 
$c>0$. Then, after an order
$c!$ base change $(D',0)\to (D,0)$, the 
pullbacks $\cX'$ and $\cV'$ are 
\'etale-locally 
birational.
\end{lemma}

\begin{proof} We have 
$B_D\geq (1-\tfrac{1}{m_i})[0]$ where 
$m_i$ is a multiplicity of some component
of $\cX_0$. On the other hand, 
$(1-\frac{1}{c})[0]\geq B_D$ by hypothesis. 
Therefore $c\geq m_i$.

Let $(\cX')^\nu\to (D',0)$ be the normalization of
the order $c!$ base change. After this base change,
every irreducible component of
the central fiber is reduced.
So we have an \'etale section of 
$(\cX')^\nu\to (D',0)$. 
Choose an \'etale section
of $\cV\to (D,0)$. This induces an \'etale section
of $\cV'\to (D',0)$. The level $3$ structure 
rigidifies the fibers of $\cV'$, and so there is 
a unique 
\'etale-local birational
map
of the punctured families
$\cX'\dashrightarrow \cV'$ identifying the chosen 
sections over the punctured disk $(D')^\ast \subset D'$
and preserving the level structures.
\end{proof}

\begin{warning}
Lemma \ref{finally-rid-of-B} is false
in the Zariski topology. 
For instance, consider an elliptic K3 surface 
$\cX\to \bP^1$ with $24$ Kodaira type $I_1$ 
singular fibers, but admitting no section. Let 
$\cV\to \bP^1$ be its Albanese fibration. Then,
$\cX$ and $\cV$ are 
\'etale-locally isomorphic over an \'etale chart around any
$0\in \bP^1$,
but are {\it not} birational over any Zariski open set
containing $0$---if they were, $\cX$ would admit
a section. An example with level $3$ structure
is
furnished by base changing to an
appropriate branched
cover of $\bP^1$.
\end{warning}

Take a cyclic 
cover to a further finite base 
change $Z_1\to Y_1$, whose
branch order is $c!$
over the inverse image 
of $B$ in $Y_1$ and possibly
branched elsewhere to ensure the existence of the cyclic
cover, but in a bounded manner 
depending only on $(Y,B,\bfM,\Phi^o)$.
Here $c>0$
is an integer satisfying $1-\frac{1}{c}\geq 
\max\{\textrm{coefficients of }B\}$.
Let $Z\to Y$ be the Galois closure
of $Z_1\to Y$.
Let ${\rm Br}\subset Y_1$ be the
branch divisor of $Z\to Y_1$.
Let $\overline{h}\colon \overline{V}\to Z$ be
the normalized base change of $h_1\colon V_1\to Y_1$.

If $Z\to Y_1$ is ramified to 
order $k_i$ over 
a component $\partial_i\subset 
\partial_{\rm H}$
then the degeneration may no longer be 
generically Kulikov along the corresponding 
component. Set $Z'\subset Z$ as the 
inverse image of $Y'_1\setminus 
({\rm Br}\cup \partial_{\rm H})_{\rm sing}
\subset Y_1$.
So $Z'$ is smooth. There
is a simultaneous resolution $V'\to \oV\vert_{Z'}$
to a new family of Kulikov models, by replacing
the relevant tiling $\cT_i$ with a 
$k_i$-subdivision. 
Then $V'$ is again smooth, and the resolution
$h'\colon V'\to Z'$ is Kulikov, with smooth
discriminant.

Set $G$ to be the Galois group of $Z\to Y$. 
So far, our constructions have 
not depended on $f\colon X\to Y$ itself, 
but only on $(Y,B,\bfM,\Phi^o)$.  
Let $g\colon W\to Z$ be the normalized base
change of $f\colon X\to Y$. By Proposition 
\ref{prop:bestKtrivialmodel}, there is a 
birational modification of
$g\colon W\to Z$ which is relatively
$K$-torsion over a $G$-invariant
big open set $U$.
We define $Z'' \coloneqq Z'\cap U$.
Then, localizing about the generic point
of each divisor in the discriminant locus,
it follows from Lemma \ref{finally-rid-of-B}
that there is a big open subset $Z^+\subset Z''$
for which $g^+\colon W^+\to Z^+$ is \'etale-locally 
birational to the restriction
$h^+\colon V^+\to Z^+$ of $h'$.

We have thus constructed a family satisfying (2)
of Proposition \ref{nicest-model},
minus the $G$-equivariance of 
$h^+\colon V^+\to Z^+$.
Pulling $h^+$ back along the action of 
$a\in G$ on the base produces a new 
Kulikov model along $\partial_i$
corresponding to a pullback tiling 
$a^\ast \cT_i$ 
(Construction \ref{mumford-construction}). Taking
a finite ramified base change of $Z$ as necessary, 
we can scale the set 
$\bigcup_{a\in G} a^\ast  \cT_i\subset \bR^g$ 
until all vertices are integral, and then further
refine in a $G$-invariant manner 
to replace $\cT_i$ with a $G$-invariant tiling. Hence, 
we can ensure that $h^+\colon V^+\to Z^+$ is naturally
$G$-equivariant.
\end{proof}

We define $f^+\colon X^+\to Y^+$ as the restriction
of $f$ to the smooth, big open set $Z^+/G=Y^+\subset Y$. 
In summary, we have
a diagram of the form
\begin{align} \label{diagram}\begin{aligned}
\xymatrix{V^+\ar[rd]_{h^+} &
W^+ \ar@{<-->}[l]_{\textrm{\'et-loc}} \ar[d]^{g^+} \ar[r]^{/G}
& X^+\ar[d]^{f^+} \\
& Z^+ \ar[r]^{/G} & Y^+
}
\end{aligned}\end{align}
where $V^+\to Z^+$ is a Kulikov model
of the Albanese fibration, and $W^+$
is \'etale-locally birational to $V^+$.
Let $\partial^+\subset Z^+$ denote the discriminant
locus of $h^+\colon V^+\to Z^+$, a smooth divisor.

\begin{definition}\label{def:tr-bir}
Let $X\to U$ and 
$X'\to U$ be abelian fibrations
whose Albanese varieties
$X^{\rm Alb}=(X')^{\rm Alb}$
are identified
over the generic point ${\rm Spec}\,\bC(U)$. A
{\it translation-birational} map is a birational map 
$\varphi\colon X\dashrightarrow X'$ over $U$, 
such that 
$\varphi^{\rm Alb}={\rm id}_{X^{\rm Alb}}$.
If $X=X'$, this is the same as requiring
that $\varphi$ acts by translation
on a general fiber. If $X^{\rm Alb}=X'$,
it amounts to the statement that,
for $u\in U$ over which
the fibers are smooth, the morphism
$\varphi_u\colon X_u\to {\rm Aut}^0(X_u)$
induces the identity map $$(\varphi^{\rm Alb})_u\colon 
{\rm Aut}^0(X_u)\to {\rm Aut}^0({\rm Aut}^0(X_u))={\rm Aut}^0(X_u).$$
\end{definition}

\subsection{Variations on the Tate--Shafarevich group}
\label{ts2}

\begin{definition}\label{def:p}
The 
{\it translation sheaf} of 
$g^+\colon W^+\to Z^+$ is the sheaf $\cP$ of abelian groups in the \'{e}tale topology of $Z^+$ defined as follows:
$\cP(U)$ is the group
of translation-birational 
automorphisms of $W^+_U\to U$,
for an \'etale chart $U\to Z^+$.
See also \cite[Thm.~1.3(3)]{yoonjoo}.
\end{definition}

Note that $\cP$ is a sheaf of abelian groups.
The analytification of $\cP$ is similar;
its sections are the 
translation-bimeromorphisms 
over an analytic open set 
$U\subset Z^+$. 
The translation
sheaves of $g^+\colon W^+\to Z^+$
and $h^+\colon V^+\to Z^+$ are naturally identified, 
see (\ref{diagram}) and Definition \ref{def:tr-bir}. 

\begin{remark}\label{p-is-kul} 
We can choose Kulikov models
over $\partial^+\subset Z^+$ so that any 
\'etale-local
translation-birational automorphism
of $h^+\colon V^+\to Z^+$
in a neighborhood of a point 
$p\in \partial_i\subset \partial^+$ extends to a biregular map,
by ensuring the tiling $\cT_i$ of 
$\bR^{g-g_i}$
is $\bZ^{g-g_i}$-periodic 
rather than just $L_i$-periodic. But it is not clear that 
this is possible while additionally preserving
the $G$-equivariance. This is not a major sacrifice,
as the (non-proper) smooth locus 
$$(h^+)^{\rm sm}\colon (V^+)^{\rm sm}\to Z^+$$ 
of the morphism $h^+$ is 
independent of the choice of 
Kulikov model $V^+$---any two Kulikov models
are connected by toric modifications, which are
isomorphisms in the smooth locus and only change
the combinatorics of the triangulation of the 
dual complex.
Hence $\cP$ acts 
biregularly on $(V^+)^{\rm sm}$ and in fact,
$\cP\simeq (V^+)^{\rm sm}\to Z^+$ canonically,
since $V^+$ has a distinguished origin. 

Thus, the sheaf $\cP$ 
can simply be thought of as the sheaf of sections
of $h^+\colon V^+\to Z^+$
whose image lies in $(V^+)^{\rm sm}$. 
Note that a rational section of 
the Kulikov model is necessarily 
regular, see, e.g., \cite[6.4.2]{yoonjoo}.
So $\cP$ admits a natural algebraic structure, as
a group scheme of finite
type over $Z^+$.
\end{remark}

We have that $\cP\to Z^+$ 
is a group scheme over $Z^+$ in either 
the analytic or \'etale topology.
Let $\cP^0$ denote the fiberwise connected
component of the identity.
Then, over a point $z\in \partial^+\subset Z^+$ 
in the discriminant locus (this may not 
be the inverse of the discriminant locus of 
$Y^+$, e.g., base changes of multiple fibers 
with smooth reduction disappear from the discriminant 
locus in $Z^+$)
the semiabelian fiber $\cP^0_z$
fits into an exact
sequence \begin{align}\label{extension-exact}
0\to \bG\to \cP^{0}_z\to A\to 0\end{align}
where $A$ is an abelian $g_i$-fold and 
$\bG\simeq (\bG_m)^{g-g_i}$. In fact,
the moduli of $\cP_z^0$ 
is exactly recorded by ${\rm im}\,\Psi_1(z)\in F_i$ 
in the notation of \eqref{fibrations}
and ensuing discussion, and $\cP^0_z$ is the Albanese
of the corresponding mixed Hodge structure of type 
$\{(-1,-1), (-1,0), (0,-1)\}$, see 
\cite[Constr.~10.1.3]{deligne1974}. 
Here, $g-g_i=\rk(\ker(N_i))$,
where we (re)define $N_i\in \mathfrak{sp}_\Lambda[3]$ 
as the logarithm of  local monodromy about the given
component $\partial_i\subset \partial^+$.
Furthermore, the component group $\mu_z=\mu^{(i)}$ is 
finite abelian, depending only on $\partial_i$. It 
fits into an exact sequence
\begin{align*}
0\to \cP^{0}_z\to \cP_z\to \mu^{(i)} \to 0.\end{align*}
From the Mumford construction, we have an isomorphism 
$\mu^{(i)}\simeq  \bZ^{g-g_i}/L_i$ 
and $L_i={\rm im}\,N_i$. Thus, we have 
an exact sequence of group schemes 
\begin{align} \label{component-exact}
0\to \cP^{0}\to \cP\to \mu \to 0.
\end{align}
Here $\mu\to Z^+$ is a $G$-equivariant
constructible
group scheme of finite groups, supported
on the discriminant, and isomorphic to 
$\mu\simeq \bigoplus_i j^{(i)}_*\underline{\mu}^{(i)}$
where $j^{(i)}\colon \partial_i\hookrightarrow Z^+$
is the inclusion.
See Figure \ref{fig1} for a illustration.

\begin{figure}
\includegraphics[width=6in]{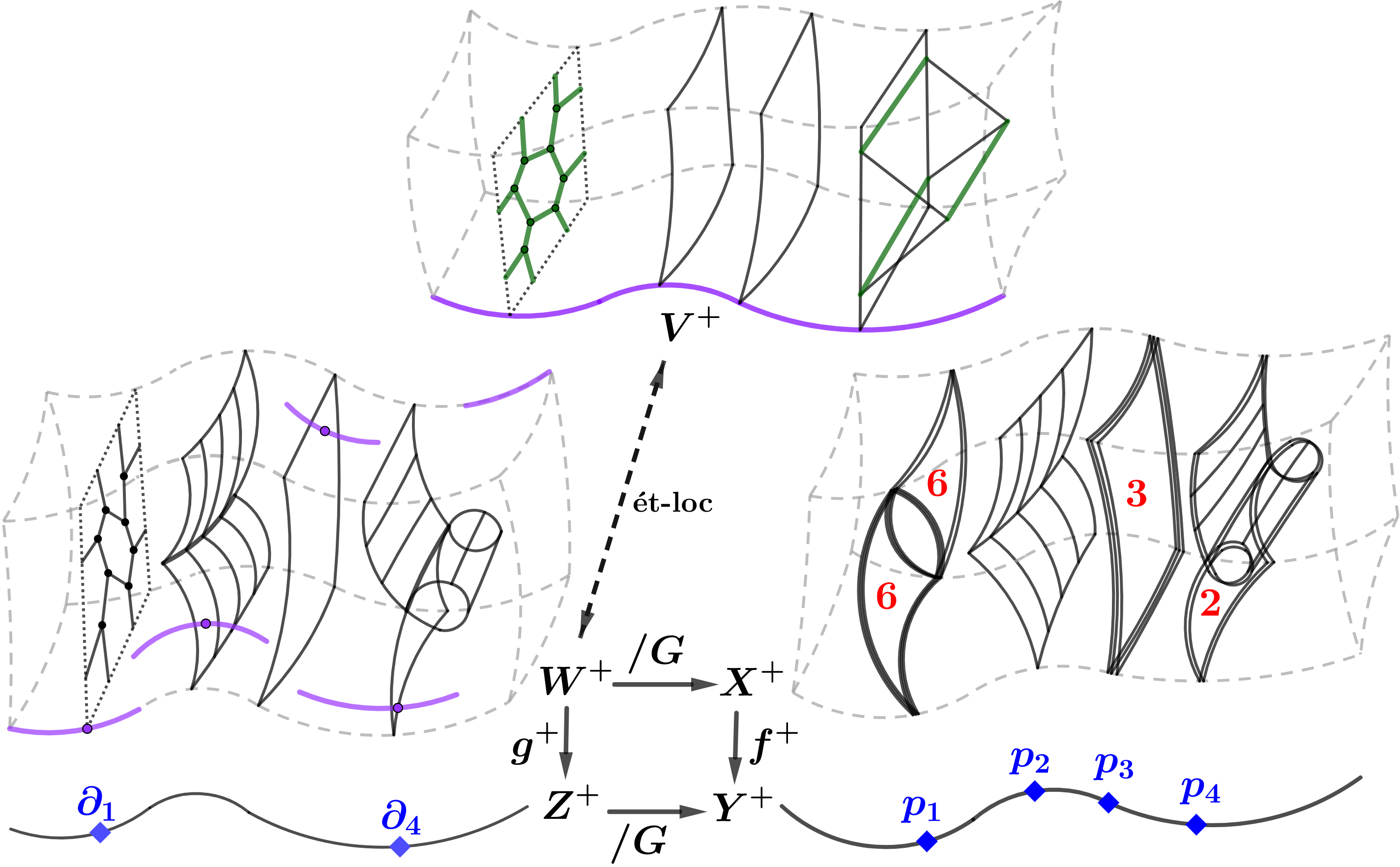}
\caption{{\bf Right:}
Abelian surface fibration
$f^+\colon X^+\to Y^+$ over a curve, 
with $4$ singular fibers
(multiplicities in red), over 
points $p_1$, $p_2$, $p_3$, $p_4$.
$B = \tfrac{5}{6}[p_1]+ \tfrac{1}{6}[p_2]
+\tfrac{2}{3}[p_3]+\tfrac{1}{2}[p_4]$. The $G$-cover
$Z^+\to Y^+$ is a cyclic $6$-to-$1$ cover, 
branched over $\{p_1,p_2,p_3,p_4\}$. {\bf Left:}
Normalized
base change $g^+\colon W^+\to Z^+$. 
Analytic local sections of $g^+\colon W^+\to Z^+$
in purple. {\bf Above:} Kulikov model 
$h^+\colon V^+\to Z^+$ of 
$(W^+)^{\rm Alb}$. Global section 
in purple. 
Group scheme $\cP$:
the complement
of the green singular locus of
$h^+$. Discriminant locus: the smooth
divisor $\partial^+ = \partial_1\cup 
\partial_4$.
Compact parts of corresponding
semiabelian groups: $g_1 = 0$ and $g_4=1$.
Component groups: 
$\mu^{(1)}\simeq \mu_2\oplus \mu_2$ and 
$\mu^{(4)}=\mu_3$. 
}\vspace{10pt}
\includegraphics[width=4in]{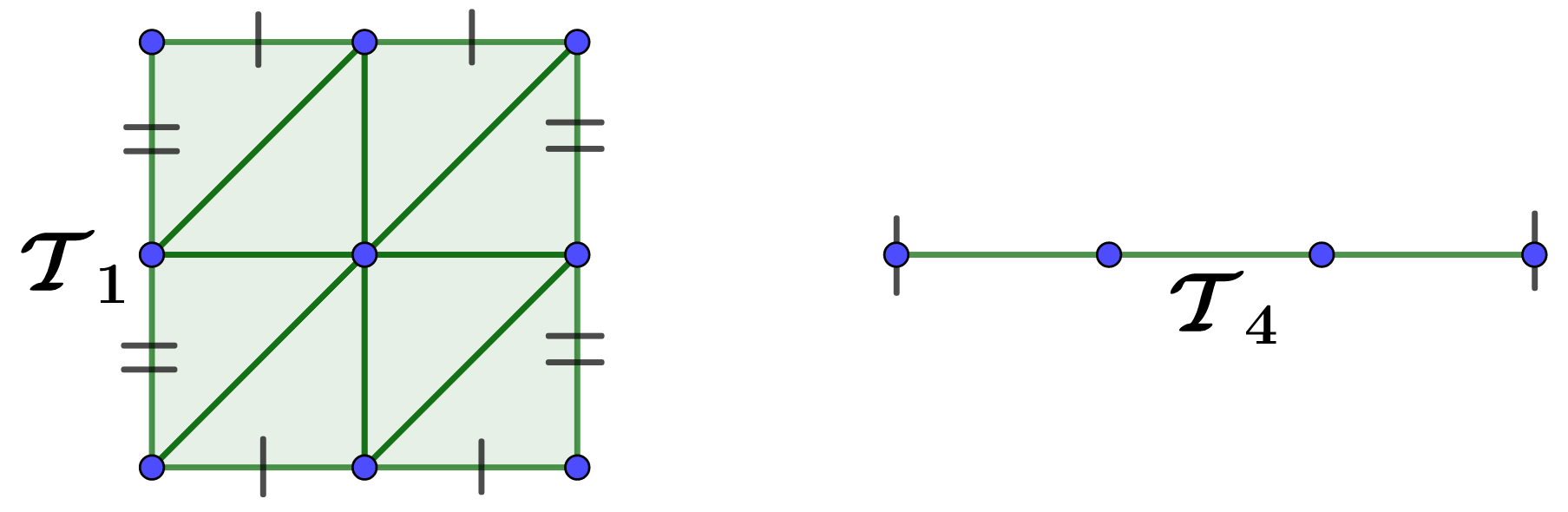}
\caption{{\bf Left} and {\bf Right:}
Mumford constructions
over $\partial_1$ and $\partial_4$. 
Monodromy matrices: $M_1 = {\rm diag}(2,2)$ 
and $M_4=(3)$. Standard simplicial tilings:
$\cT_1$ and $\cT_4$ of
$\bR^2/\bZ(2,0)\oplus \bZ(0,2)$
and $\bR/\bZ(3)$, respectively.}
\label{fig1}
\end{figure}

As $h^+\colon V^+\to Z^+$ is $G$-equivariant,
the sheaf $\cP$ and its neutral component $\cP^{0}$
admit a $G$-action
(i.e., $\cP$ is a $G$-equivariant sheaf
over the $G$-space $Z^+$). 
Note that the translation sheaves
of $h^+\colon V^+\to Z^+$ and 
$g^+\colon W^+\to Z^+$
are isomorphic as $G$-sheaves,  since
the $G$-actions only differ
by translation, under an \'etale-local
identification of $V^+$ and $W^+$.

We can think of $\cP$ and $\cP^0$ as sheaves on 
the orbifold $[Z^+/G]$. 
We define sheaves $P$, $P^0$ on
$Y^+$ as the pushforward of $\cP$, $\cP^0$ 
along the coarsening map
$${\rm co}\colon [Z^+/G]\to Z^+/G=Y^+$$ 
from the orbifold to its coarse
space. Thus, if $U\to Y^+$ is a small \'etale
or analytic open chart, $P(U) \coloneqq \cP(\wU)^G$
where $\wU$ is the inverse image of $U$
under the map $Z^+\to Y^+$.

\begin{definition} 
The {\it Tate--Shafarevich group}(s) 
are various first cohomology groups
of $\cP$ and its variants. Possible 
aspects to vary are:
\begin{enumerate}
\item\label{notation1} cohomology of 
the subsheaf $\cP^0$, $P^0$ or of
the whole group scheme $\cP$, $P$;
\item\label{notation2} $G$-equivariant or regular
cohomology of  $\cP$, $\cP^0$ or 
regular cohomology of $P$, $P^0$;
\item\label{notation3} localization at a point 
$z\in Z$ or $y\in Y$; and
\item\label{notation4} cohomology in the 
analytic or \'etale
topology.
\end{enumerate}
There are $20$ possibilities, with the most
important group being $\Sha_G \coloneqq H^1_G(Z^+,\cP)$.
\end{definition}

\begin{notation} Our notational conventions
are as follows: A superscript $^0$ or no superscript
indicate the possibilities for (\ref{notation1});
a subscript ``$G$'' or no subscript to
indicate (\ref{notation2}); a subscript
``$z$'' or ``$y$'' to indicate (\ref{notation3});
a subscript ``${\rm an}$'' or no subscript to
indicate (\ref{notation4}). Finally, 
if we take regular (but not $G$-equivariant)
cohomology upstairs on $Z^+$, we put a tilde 
over the $\Sha$ symbol.
For example: $\Sha^0_{G,\,{\rm an}} \coloneqq 
H^1_{G,\,{\rm an}}(Z^+, \cP^0)$ and 
$\widetilde{\Sha}_z \coloneqq \lim_{U\ni z} 
H^1(U\cap Z^+, \cP)$.

We only put the ``${\rm an}$'' 
subscript on the  cohomology or Sha symbols, i.e.,  
$H$ or $\Sha$, and notationally suppress 
the analytifications $Z^+_{\rm an}$, $Y^+_{\rm an}$, 
$\cP_{\rm an}^0$, $\cP_{\rm an}$, $P^0_{\rm an}$,
$P_{\rm an}$ of the spaces or objects
themselves. No symbol means \'etale cohomology.
\end{notation}

The exact sequence with the deepest geometric
significance for us will come from
the Leray spectral sequence for the coarsening
morphism ${\rm co}\colon [Z^+/G]\to Y^+$ and the sheaf
$\cP$.
It gives
\begin{align}\label{first-answer}
0\to H^1(Y^+, P)\to 
H^1_G(Z^+,\cP)\xrightarrow{m} H^0(Y^+,R^1{\rm co}_*\cP)
\xrightarrow{\rm ob} 
H^2(Y^+,P).\end{align}
Defining $\Sha \coloneqq H^1(Y^+,P)$,
$\Ash^{G,\,{\rm pre}} \coloneqq H^0(Y^+,R^1{\rm co}_*\cP)$, and 
$\ker({\rm ob}) \coloneqq \Ash^G\subset \Ash^{G,\,{\rm pre}}$,
we get the exact sequence \begin{align}\label{the-answer}
0\to \Sha\to \Sha_G\to \Ash^G\to 0\end{align}
whose meaning we will
interpret below (Prop.~\ref{alg-mult}).

\begin{proposition}\label{define-twist} 
Suppose that $f\colon X\to Y$
is an abelian fibration, inducing
$(Y,B,\bfM,\Phi^o)$. Let $Z^+$ be the big open
set of the $G$-cover $Z\to Y$
from Proposition \ref{nicest-model}.
Then, $g^+\colon W^+\to Z^+$ and 
$h^+\colon V^+\to Z^+$ differ via ``twisting''
by an element $t(f)\in 
\Sha_G \coloneqq H^1_G(Z^+, \cP)$.
Furthermore, if $f_1\colon X_1\to Y$ and 
$f_2\colon X_2\to Y$ inducing $(Y,B,\bfM,\Phi^o)$ 
are \'etale-locally translation-birational over $Y^+$,
the difference 
$t(f_1)-t(f_2)
\in \Sha_G$ 
lies in the subgroup
$$t(f_1)-t(f_2)\in \Sha=H^1(Y^+,P).$$
Finally, $f_1$ and $f_2$ are 
translation-birational over
$Y^+$ if and only if 
$t(f_1)-t(f_2)=0\in \Sha$.
\end{proposition}

\begin{proof}
By Proposition 
\ref{nicest-model}, the base change
$g^+\colon W^+\to Z^+$ admits
an \'etale-local 
translation-birational identification with
its Albanese $h^+\colon V^+\to Z^+$;
see Definition~\ref{def:tr-bir}. 
Choosing an \'etale cover 
$\{U_i\}_{i\in I}$ of $Z^+$ 
for which such identifications
$\varphi_i\colon (W^+)_{U_i}
\dashrightarrow (V^+)_{U_i}$ exist, we have 
translation-birational maps
\[
\xymatrix{
(V^+)_{U_i\cap U_j}\ar@{-->}[r]^{\varphi_i^{-1}} &
(W^+)_{U_i\cap U_j}\ar@{-->}[r]^{\varphi_j} &
(V^+)_{U_i\cap U_j}}
\] whose composition 
$t_{ij}=\varphi_j\circ \varphi_i^{-1}$ gives a 
translation-birational automorphism of 
$(V^+)_{U_i\cap U_j}$ and
hence a section of $\cP(U_i\cap U_j)$. 
Since $t_{ki}\circ t_{jk}\circ t_{ij}={\rm id}$,
the elements $t(f) \coloneqq \{t_{ij}\}$ define
an element of $\widetilde{\Sha}=H^1(Z^+,\cP)$.
The set of translation-birational
identifications $\varphi_i$ is naturally
a torsor over $\cP(U_i)$ under post-composition, and 
changing these identifications corresponds
to adding a $1$-coboundary to $t_{ij}$.
Refining the cover $\fU=\{U_i\}_{i\in I}$ 
as necessary so that the $G$-action
permutes the elements of $\fU$,
we have isomorphisms $(W^+)_{U_i}\to 
(W^+)_{a\cdot U_i}$ for all $a\in G$,
since our fibration is pulled back
from $f$. We also have isomorphisms
$(V^+)_{U_i}\to (V^+)_{a\cdot U_i}$
by construction. This endows
$\{t_{ij}\}$ with a $G$-equivariant 
structure, giving a class
$t(f)\in \Sha_G= H^1_G(Z^+,\cP)$.

Furthermore, note that if $f_1\colon X_1\to Y$ and 
$f_2\colon X_2\to Y$ are \'etale-locally
translation-birational over $Y^+$, then
we have lifts to $G$-equivariant
\'etale-local translation-birational 
maps between $W_1$ and $W_2$ over $Z^+$.
So we may choose  
$\{\varphi_i\}_{i\in I}$
as above, but additionally required
to be $G$-equivariant
birational maps over a $G$-equivariant open cover
of $Z^+$ of the form $\{\wU_i\}_{i\in I}$,
where $\{U_i\}_{i\in I}$ form an open
cover of $Y^+$ and $\wU_i$ is the 
inverse image $U_i$ under $Z^+\to Y^+$.
In turn, the corresponding twisting
element $t(f_1)-t(f_2)\in \Sha_G$ 
can be defined by 
a cocycle $\{t_{ij}\}$ with 
$t_{ij}\in \cP(\wU_i\cap \wU_j)^G = P(U_i\cap U_j)$,
and so defines an element $t(f_1)-t(f_2)\in\Sha=H^1(Y^+,P)$.

Finally, if $X_1$ and $X_2$ are 
translation-birational over $Y^+$ (or $Y$), 
this lifts to a $G$-equivariant translation-birational 
map of $W_1$ and $W_2$ over $Z$, and hence 
$t(f_1)-t(f_2)=0\in \Sha$. Conversely, if 
$t(f_1)-t(f_2)=0 \in \Sha$ then the cocycle
$\{t_{ij}\}$ corresponding to this difference
is a $1$-coboundary of the sheaf $P$ on $Y^+$. 
It follows that the $G$-equivariant 
identifications
$\varphi_{i,2}\circ \varphi_{i,1}^{-1}$ 
can be chosen to globalize to a $G$-equivariant
translation-birational map between $W_1$
and $W_2$ over $Z^+$. Passing to the quotient
by $G$, we conclude that $X_1$ and $X_2$ are 
translation-birational over $Y^+$.
\end{proof}

\begin{definition} We let $t(f)\in 
\Sha_G$ denote the element 
corresponding to the 
translation-birational equivalence 
class of $f$. 
We denote the image
of $t(f)$ in the local Tate--Shafarevich 
group at $z\in Z$ by $t_z(f)\in 
\widetilde{\Sha}_z \coloneqq \lim_{U\ni z} 
H^1(U\cap Z^+, \cP)$. 
\end{definition}

Note that we have
a forgetful map $\Sha=H^1_G(Z^+,\cP)\to
H^1(Z^+,\cP)=\widetilde{\Sha}$.

\begin{lemma}\label{torsion-lemma-1} 
The groups
$\widetilde{\Sha}$, $\Sha$, $\Sha_G$ 
are torsion. 

Furthermore, the pairs
$\widetilde{\Sha}^0$ and $\widetilde{\Sha}$, 
$\Sha^0$ and $\Sha$, 
$\Sha^0_G$ and $\Sha_G$ are related
by finite degree isogenies.\footnote{For us, it suffices to 
define an {\it isogeny} of abelian groups as a homomorphism
with finite kernel and cokernel.}
Here, 
$\widetilde{\Sha}^0 \coloneqq H^1(Z^+,\cP^0)$,
$\Sha^0 \coloneqq H^1(Y^+,P^0)$, and
$\Sha_G^0 \coloneqq H^1_G(Z^+,\cP^0)$.
\end{lemma}

\begin{proof} Given any $G$-sheaf $\cF$
on a $G$-space $S$, we have the Cartan--Leray spectral
sequence $$H^i(G,H^j(S,\cF))\implies H^{i+j}_G(S,\cF)$$
and the corresponding five-term exact sequence
$$0\to H^1(G, H^0(S,\cF))
\to H^1_G(S,\cF)
\to H^1(S,\cF)^G
\to H^2(G, H^0(S,\cF))
\to H^2_G(S,\cF).$$
Applied to $S=Z^{+}$ and $\cF =\cP$,
we get an exact
sequence
\begin{align*}
0\to H^1(G,H^0(Z^+,\cP))\to \Sha_G\to 
\widetilde{\Sha}^G.
\end{align*}
All higher cohomology
of a finite group is torsion, so it suffices
to prove that $\widetilde{\Sha}$ 
is torsion, which will
imply its $G$-invariants are too. 
The proof is the same as \cite[Prop.~6.32]{yoonjoo},
see references therein. By (\ref{the-answer}), in 
turn, $\Sha$ is torsion.

The second statements follow from 
taking the long exact sequence
(in either regular or $G$-equivariant cohomology)
of the component exact sequence (\ref{component-exact})
and observing that the cohomology of 
a constructible sheaf of 
finite abelian groups is finite abelian.
\end{proof}

Let $\mathfrak{p}$ denote the Lie algebra
of $\cP$.
It is a $G$-equivariant 
sheaf of Lie algebras over $Z^+$, 
and we have an exact sequence of sheaves of abelian groups
in the analytic category:
\begin{align}\label{ses}0\to \Gamma
\to \mathfrak{p}
\xrightarrow{\rm exp} 
\cP^0\to 1\end{align}
where $\Gamma$ 
is a constructible sheaf of finitely
generated, torsion-free $\bZ$-modules, 
whose rank away
from $\partial^+$ is $2g$ and whose
rank drops to $g+g_i$ over an irreducible 
component $\partial_i\subset \partial^+$. 

We continue with an analysis of the local Tate--Shafarevich
group for points $z\in Z'\setminus Z^+$. Recall from
the proof of Proposition \ref{nicest-model} that
$Z'\subset Z$ is a big open subset 
of the regular locus, 
depending only on $(Y,B,\bfM,\Phi^o)$,
for which (a model of) the Albanese fibration 
$h'\colon V'\to Z'$ is Kulikov and $Z^+\subset Z'$
is a possibly smaller big open set.

\begin{proposition} \label{local-no-twist}
For all $z\in Z'\setminus Z^+$, we have 
$t_z(f)=0$.
\end{proposition}

\begin{proof}
We have that $t_z(f)=0$ if and only if $g^+\colon W^+\to Z^+$
is \'etale-locally (in a neighborhood of $z\in Z$) 
birational to $h^+\colon V^+\to Z^+$, and 
this holds if and only if 
$g^+\colon W^+\to Z^+$ admits a rational section
over an \'etale neighborhood of $z\in Z$.

Consider some \'etale or analytic chart
$U\to Z$ about 
$z\in Z'\setminus Z^+$ and let $U^+=U \cap Z^+$. 
We will do an induction on the codimension of 
$Z'\setminus Z^+$. Once it is shown
that $t_z(f)=0$ for $z\in (Z'\setminus Z^+)^{\rm reg}$
we may replace $Z^+$ with $Z^+\cup (Z'\setminus Z^+)^{\rm reg}$
because $g\colon W\to Z$ is then \'etale-locally
birational to its Albanese
over this enlargement of $Z^+$, and therefore defines an element
of the Tate--Shafarevich group $\Sha_G$ over this enlargement,
as in Proposition \ref{define-twist}.

Thus consider $U=D^k\ni z$ some polydisk neighborhood of 
$z\in Z'$ for which $Z'\setminus Z^+$ 
is given by $\{0\}^\ell\times D^{k-\ell}$
for some $\ell\geq 2$. Then $U^+ = 
(D^\ell\setminus \{0\}) \times D^{k-\ell}$. We have
that $H^1_{\rm an}(U^+,\cP)$ computes the obstruction
to the existence of a meromorphic section of $W_{U^+}\to U^+$.
Taking the long exact sequence of (the restriction to $U^+$ of)
the short exact sequence (\ref{ses}), we have
$$
H^1(U^+,\Gamma)
\to H^1_{\rm an}(U^+,\mathfrak{p})
\to H^1_{\rm an}(U^+,\cP^0) \to
H^2(U^+,\Gamma).
$$

We now consider two cases, $z\notin \partial^+$
and $z\in \partial^+$. First suppose
$z\notin \partial^+$. Then $\Gamma\vert_{U^+}$ 
is a local system of rank $2g$ and $\cP^0=\cP$. 
But $\pi_1(U^+)=1$, since $U^+$ is 
homotopy-equivalent to a 
sphere $\bS^{2\ell-1}$, $\ell\geq 2$. So 
$\Gamma\vert_{U^+}\simeq \underline{\bZ}^{2g}$ is a trivial
local system, and 
$H^1(U^+,\Gamma)\simeq H^2(U^+,\Gamma)=0$.
Hence 
$H^1_{\rm an}(U^+,\mathcal P)\simeq 
H^1_{\rm an}(U^+,\mathfrak{p})$
is a $\bC$-vector space. 

Note that $t(f)\in \widetilde{\Sha}$
is torsion, by Lemma \ref{torsion-lemma-1}.
So $t_z(f)\in \widetilde{\Sha}_z$ 
is torsion and in turn the image 
$t_{z,\,{\rm an}}(f)\in \widetilde{\Sha}_{z,\,{\rm an}}
\coloneqq \lim_{U\ni z} H^1_{\rm an}(U^+,\cP)$ 
under the natural
analytification map $\widetilde{\Sha}_z\to 
\widetilde{\Sha}_{z,\,{\rm an}}$ is torsion.
But $\widetilde{\Sha}_{z,\,{\rm an}}
\simeq H^1_{\rm an}(U^+,\cP)$, 
a $\bC$-vector space, is torsion-free. We conclude 
that $t_{z,\,{\rm an}}(f)=0$.
Thus, $W_{U^+}\to U^+$ admits a meromorphic section
$s^+\colon U^+\dashrightarrow W_{U^+}$
for a sufficiently small analytic open set $U\ni z$. 
We require:

\begin{lemma}\label{mero-ext} 
Let $V\to U$ be a projective morphism
of analytic spaces and 
let $U^+\subset U$ be a big open set. 
A meromorphic section $s^+$ on the open
set $U^+$ is also a meromorphic section on $U$. \end{lemma}

\begin{proof}
By the projectivity hypothesis, it
suffices to prove the proposition for the projection 
$\bP^n\times U\to U$. Let $z_i$ be projective coordinates
on $\bP^n$. Then $w_{ij} \coloneqq (s^+)^\ast (z_i/z_j)$ 
is a meromorphic function on $U^+$. By the Levi extension
theorem (the meromorphic analogue of Hartogs' lemma), 
$w_{ij}$ is meromorphic on $U$. The lemma follows.
\end{proof}

So $s^+$ defines a meromorphic section 
$s\colon U\dashrightarrow W_U$. 
By Proposition \ref{mero-to-etale},
it follows that there is an \'etale 
chart $U'\to Z$ about $z\in Z$
for which $W_{U'}\to U'$ admits a rational section,
and thus $t_z(f)=0$ as desired.

Now we consider the second case $z\in \partial_i\subset 
\partial^+$. First,
we claim that the image of the localization
and analytification
map $\widetilde{\Sha}^0\to 
\widetilde{\Sha}^0_{z,\,{\rm an}} \coloneqq H^1_{\rm an}(Z^+,\cP^0)$
is zero. The former group
is torsion, and so it suffices
to prove that $\widetilde{\Sha}^0_{z,\,{\rm an}}$ is torsion-free.
In this case, $\Gamma\vert_{U^+}$ is a constructible
sheaf which has rank $2g$ over $U^+\setminus \partial^+$
and whose rank drops to $g+g_i$ over $U^+\cap \partial^+$.
We may assume $\partial^+$ is locally given by
$\{0\}\times D^{k-1}$ in the polydisk chart $D^k\ni z$.
Let $i\colon U^+\setminus \partial^+\hookrightarrow U^+$
denote the inclusion of the complement of the discriminant
and let $j\colon U^+\cap \partial^+\hookrightarrow U^+$
be the inclusion of the discriminant.

We observe that on $U^+$ there is
an exact sequence of constructible
sheaves \begin{align}\label{gamma-exact}
0\to \underline{\bZ}^{g+g_i}\to 
\Gamma\vert_{U^+}\to i_!\underline{\bZ}^{g-g_i}\to 0
\end{align}
arising 
from the fact that $\Gamma\vert_{U^+}\simeq 
i_*(\Gamma\vert_{U^+\setminus \partial^+})$
is the pushforward 
of a $\bZ^{2g}$-local system whose monodromy
about $\partial^+$ is $I+N_i$ (see also Prop.~\ref{isog}).
Since $N_i$ is nilpotent of order two, 
the subsheaf 
$\underline{\bZ}^{g+g_i} = \underline{H^0(U^+,\Gamma)}$
is the constant sheaf on the kernel of $N_i$ and
the quotient $i_!\underline{\bZ}^{g-g_i}$
is the extension by zero of the cokernel of $N_i$. 
As such, we may identify 
$\bZ^{g+g_i}\simeq 
\mathscr{W}_{-1}^{\rm lim}$ with
the weight $-1$ part of the limit MHS over
$\partial^+$, and similarly $\bZ^{g-g_i}\simeq 
\mathscr{W}_{0}^{\rm lim}/
\mathscr{W}_{-1}^{\rm lim}$.

Taking the long exact sequence, and noting that 
$H^1(U^+, \underline{\bZ}^{g+g_i})=
H^2(U^+, \underline{\bZ}^{g+g_i})=0$ (again, 
$U^+$ is homotopy equivalent to a real
$(2\ell-1)$-sphere because $Z'$
is smooth at $z$), 
we conclude that there are isomorphisms
$H^1(U^+, \Gamma)\simeq 
H^1(U^+, i_!\underline{\bZ}^{g_i})$ and
$$
H^2(U^+, \Gamma)\simeq \ker(
H^2(U^+, i_!\underline{\bZ}^{g_i})\to 
H^3(U^+, \underline{\bZ}^{g+g_i})).$$
Now, we consider the exact sequence of sheaves
$$0\to i_!\underline{\bZ}^{g-g_i}\to 
\underline{\bZ}^{g-g_i} \to j_*\underline{\bZ}^{g-g_i}\to 0.$$
Taking the long exact sequence in cohomology,
and noting $H^1(U^+, \underline{\bZ}^{g-g_i})=
H^2(U^+, \underline{\bZ}^{g-g_i})=0$ and
$H^0(U^+,\underline{\bZ}^{g-g_i})\to 
H^0(U^+, j_*\underline{\bZ}^{g-g_i})$ is an isomorphism,
we conclude that we have vanishing of 
$H^1(U^+, i_!\underline{\bZ}^{g-g_i}) = 0$ and that
$$
H^2(U^+, i_!\underline{\bZ}^{g-g_i}) \simeq 
H^1(\partial^+\cap U^+,
\underline{\bZ})^{\oplus (g-g_i)}=\twopartdef{\bZ^{g-g_i}}
{\codim_z Z'\setminus Z^+=2,}{0}
{\codim_z Z'\setminus Z^+\geq 3.}$$ 

When $\codim_z Z'\setminus Z^+=2$, 
a Čech complex computation verifies that
the third connecting homomorphism of the long exact sequence
associated to (\ref{gamma-exact}) is identified with the nilpotent operator $N_i$ is as follows: 
\[\xymatrix
{H^2(U^+, i_!\underline{\bZ}^{g-g_i}) \ar[r]\ar[d]^{\simeq} & 
H^3(U^+, \underline{\bZ}^{g+g_i})\ar[d]^{\simeq} \\
\mathscr{W}_0^{\rm lim}
/\mathscr{W}_{-1}^{\rm lim}\simeq \bZ^{g-g_i} \,\,
\ar[r]^{\,\,\,\,\,\,\,\,\,N_i} & 
\,\,\mathscr{W}_{-1}^{\rm lim}\simeq \bZ^{g+g_i}.
}\] 
And when $\codim_z Z'\setminus Z^+\geq 3$, it is 
isomorphic to the map $0\to 0$.
In fact, over $\bQ$, the map $N_i$ defines an isomorphism 
$(\mathscr{W}_0^{\rm lim}/\mathscr{W}_{-1}^{\rm lim})_\bQ
\to (\mathscr{W}_{-2}^{\rm lim})_\bQ
\subset (\mathscr{W}_{-1}^{\rm lim})_\bQ$ 
and so in either case, $\ker N_i = H^2(U^+,\Gamma)=0$.

Having proven $H^1(U^+,\Gamma)= H^2(U^+,\Gamma)= 0$, 
we conclude that the long exact sequence of (\ref{ses})
induces an isomorphism
$H^1_{\rm an}(U^+,\mathfrak{p})\simeq
H^1_{\rm an}(U^+,\cP^0)$.
So the map
$\widetilde{\Sha}^0\to 
\widetilde{\Sha}^0_{z,\,{\rm an}}$ is zero,
again since $H^1_{\rm an}(U^+,\mathfrak{p})$ is
a $\bC$-vector space, which is torsion-free.

Let us now consider the long exact sequence
associated to the component sequence 
(\ref{component-exact}).
We have again a dichotomy depending on
$\codim_z Z'\setminus Z^+=2$ or $\geq 3$:
$$
H^1_{\rm an}(U^+,\cP^0)\to H^1_{\rm an}(U^+,\cP)\to 
H^1(U^+\cap \partial^+, \underline{\mu}^{(i)})
\simeq (\mu^{(i)}\textrm{ or }0)\to H^2_{\rm an}(U^+,\cP^0),
$$
respectively. We also have a connecting homomorphism
$H^2_{\rm an}(U^+,\cP^0)\to H^3(U^+,\Gamma)$ and 
we consider the composition $H^1(U^+\cap \partial^+, 
\underline{\mu}^{(i)})\to H^3(U^+,\Gamma)$.
We have
$$
\coker(
H^2(U^+, i_!\underline{\bZ}^{g-g_i})\to 
H^3(U^+, \underline{\bZ}^{g+g_i}))=
\twopartdef{\coker(\bZ^{g-g_i}\xrightarrow{N_i}
\bZ^{g+g_i})}{\codim_z Z'\setminus Z^+=2,}
{0}{\codim_z Z'\setminus Z^+\geq 3,}
$$ as again, this coboundary
map is identified with the logarithic monodromy $N_i$.

Let $K$ denote this cokernel
(so $K\simeq \mu^{(i)}
\oplus \bZ^{2g_i}$
or $0$). We have an exact sequence
$$0\to K\to H^3(U^+,\Gamma)\to 
H^3(i_!\bZ^{g-g_i})\simeq (\bZ^{g-g_i}\textrm{ or }0)\to 0.$$ 
We conclude that $H^3(U^+,\Gamma)_{\rm tors}\simeq 
\mu^{(i)}\textrm{ or }0$. 
Thus, we have a diagram
\[\xymatrix{
H^1(U^+\cap \partial^+, 
\underline{\mu}^{(i)})\ar[r]\ar[d]^{\simeq} & 
H^3(U^+,\Gamma)_{\rm tors} \ar[d]^{\simeq} \\
\mu^{(i)}\textrm{ or }0 \ar[r] & \mu^{(i)}\textrm{ or }0
}\]
with the two cases
depending on $\codim_z Z'\setminus Z^+$. 
Again a Čech complex computation 
verifies that the bottom
horizontal map is an isomorphism when
$\codim_z Z'\setminus Z^+=2$. 
Indeed, a component of $\cP^0$ over $\partial_i$ 
corresponds to a vertex of the tiling 
$\cT_i$ (cf. \Cref{mumford-construction}) defining 
the Kulikov model, which corresponds to a 
cocharacter of the toric variety uniformizing 
the Kulikov model in the Mumford construction. 
In particular, such a cocharacter defines a section 
of the uniformization by construction, 
and it descends to a section of $\cP$. 
The coboundary map $\mu^{(i)}
\simeq
H^1(U^+\cap \partial^+,\underline{\mu}^{(i)})
\to H^2(U^+,\cP^0)$ assigns to a Čech $1$-cocycle with values
in $\mu^{(i)}$ the corresponding cocharacter. 
Then, the coboundary
$H^2(U^+,\cP^0)\to H^3(U^+,\Gamma)$ results from 
taking logarithms of these cocharacters in $\mathfrak{p}$. 
We get a $3$-cocycle valued in the cocharacter lattice 
of the Mumford construction, which is easily seen to 
represent the relevant element of $\mu^{(i)}\subset K$.

Therefore, the coboundary
map $H^1(U^+\cap \partial^+, \underline{\mu}^{(i)})
\to H^2_{\rm an}(U^+,\cP^0)$ in the component sequence
is necessarily injective. 
So $H^1_{\rm an}(U^+,\cP)\simeq H^1_{\rm an}(U^+,\cP^0)\simeq 
H^1_{\rm an}(U^+,\mathfrak{p})$.

Thus $t_{z,\,{\rm an}}(f)$ equals zero, and 
we may apply Lemma \ref{mero-ext} to conclude
that $V^+$ and $W^+$ are analytically-locally 
bimeromorphic in a neighborhood of $z$. 
In turn, $t_z(f)=0$ by algebraically approximating
a meromorphic section 
(Prop.~\ref{mero-to-etale}).
\end{proof}

As claimed in Remark \ref{remark-indep},
we have proven that the big open set $Z^+\subset Z$ over
which $W\to Z$ is \'etale-locally translation-birational to its Albanese
can be chosen to depend {\it only} on $(Y,B,\bfM,\Phi^o)$. 
More precisely, $Z^+=Z'$ agrees with the 
locus where the Albanese fibration $h\colon V\to Z$
admits a Kulikov model $h'\colon V'\to Z'$ 
and so we will henceforth
not distinguish between $Z^+$ and $Z'$.

\begin{corollary}\label{torsion-lemma-2} 
The groups $\Sha^0_G$ and $\Sha_G$ are related by
an isogeny of degree bounded above, by an integer 
depending only on $(Y,B,\bfM,\Phi^o)$.
The same holds for the groups $\Sha^0$ and 
 $\Sha$.
\end{corollary}

\begin{proof} It follows from Proposition \ref{local-no-twist}
that the big open
set $Z^+\subset Z$
can be bounded solely in terms of $(Y,B,\bfM,\Phi^o)$,
independently of $f$.
Therefore, in the $G$-equivariant long
exact sequence associated to the 
component exact sequence
$$
\textstyle 
H^0_G(Z^+, 
\mu) \to 
H^1_G(Z^+,\cP^0)\to H^1_G(Z^+,\cP)\to 
H^1_G(Z^+,\mu),$$ the first
and last groups are finite of size
bounded only in terms of $(Y,B,\bfM,\Phi^o)$. 
The same holds in the 
long exact sequence for regular cohomology
of $P$.
\end{proof}

\subsection{The multiplicity group}

In this section, we measure the failure of $V^+$ and $W^+$
to be {\it $G$-equivariantly} \'etale-locally 
translation-birational over $Z^+$. Equivalently,
we measure the failure of $f^+\colon X^+\to Y^+$ to be 
translation-birational to its Albanese, 
thus quantifying the multiple fibers of $f$ 
in codimension $1$, via a finite group which 
we call the {\it multiplicity group}.

Given $y\in Y^+$, let $U\ni y$ be an 
\'etale neighborhood,
and $\wU\subset Z^+$ 
be the inverse image of $U$ in $Z^+$, which is a 
$G$-invariant open set. We consider the group
$$
\Sha_{G,y} \coloneqq \lim_{U\ni y} \,H^1_G(\wU, \cP).
$$ 

\begin{proposition}\label{mero-to-etale}
Let $W\xrightarrow{g} Z\xrightarrow{\pi} Y$
be a sequence of projective morphisms
and let $y\in Y$ be a point. Suppose that
that there exists an analytic
open set $U\ni y$ of $Y$ such that
$g$ admits a meromorphic
section over $\wU \coloneqq \pi^{-1}(U)$.
Then, there exists an \'etale open set
$U'\ni y$ for which $g$ admits a 
rational section over $\wU' \coloneqq \pi^{-1}(U')$.
\end{proposition}

\begin{proof} 
By Hironaka's flattening theorem (cf.~\cite[Thm.~1]{Hironaka1975}), 
the graph of the meromorphic section 
$\sigma \colon \widetilde{U} \dashrightarrow W$ defines a meromorphic section 
$U \dashrightarrow H$ of the relative Hilbert scheme 
$H \coloneqq \mathrm{Hilb}(Z \times_Y W/Y)$. A local meromorphic 
section of the projective morphism $H \to Y$ is a map 
$U \to \bP^{n}$ given by $u \mapsto [g_0(u) \colon \ldots \colon g_{n}(u)]$, where $g_{i}$ are holomorphic functions on $U$. 
By Artin's approximation \cite{artin1969algebraic}, there 
exists an \'{e}tale open set $U'$ and a map 
$f_k \colon U' \to H \subset \bP^{n}$ given by $u \mapsto [f_{0,k}(u) \colon \ldots \colon f_{n, k}(u)]$ such that $f_{i, k} \equiv g_{i} \mod \mathfrak{m}_{y}^{k}$, where $f_{n, k}$ are now regular functions. 
Finally, let $H^{0}$ be the open set in $H$ parametrizing subschemes $\Gamma_{u} \subset Z_{u} \times W_{u}$ such that the projection $\Gamma_{u} \to Z_{u}$ restricts to an isomorphism on dense open subsets of both $\Gamma_{u}$ and $Z_{u}$.
For $k$ large enough, $f_{k}$ can be chosen in such a way that its image intersects the open set $H^{\circ}$, i.e., $f_k$ defines the required rational section of $g$ over $\widetilde{U}'$. 
\end{proof}

It follows as in Proposition 
\ref{local-no-twist} that for any
point $y\in Y^+$, there exists an
analytic open neighborhood 
$U\ni y$
for which $H^1_{\rm an}(\wU, 
\cP)=0$---this
holds since we may choose $U$ small
enough so that $\wU$ is a disjoint union
of contractible, 
Stein analytic open neighborhoods
of the points $z_i\in Z^+$ 
forming the pre-image of $y$.
Thus, $\{z_1,\dots,z_k\}$ admits an 
analytic neighborhood $\wU$ 
over which any (analytic)
Tate--Shafarevich
twist of $(V^+)_\wU\to \wU$ admits
a meromorphic section. We
may ensure $\wU$ is the inverse
image of some $U\ni y$ by 
intersecting
over its $G$-orbit.

Fix an 
\'etale open chart $U'\ni y$
and any element $t\in H^1(\wU', \cP)$,
where $\wU'\to U'$ is the preimage.
It corresponds to a twist $W_{\wU'}\to
\wU'$ of $(V^+)_{\wU'}\to \wU'$ (see 
e.g.~Prop.~\ref{sha-inj} for a construction).
By the previous paragraph and
Proposition 
\ref{mero-to-etale}, applied
to the sequence
of projective morphisms
$W_{\wU'}\to \wU'\to U'$,
there exists some possibly smaller
\'etale chart $U''\ni y$ for which
$t\vert_{\wU''}=0$,
where $\wU''\to U''$ is the preimage.
We conclude that $$\lim_{U\ni y} H^1(\wU,\cP)=0.$$
It follows that 
$\lim_{U\ni y} H^1(\wU,\cP)^G=0$
and again by Cartan--Leray, we conclude that 
$\Sha_{G,y}\simeq \lim_{U\ni y} H^1(G, 
\cP(\wU)).$
So the sheaf $\cH^1(G,\cP)$
on $Y^+$ which is the sheafification
of the presheaf
$$
U\mapsto H^1(G,\cP(\wU))
$$
over an \'etale
open $U\subset Y^+$,
has stalk $\Sha_{G,y}$ at $y\in Y^+$.
In terms
of (\ref{first-answer}), we have
$$\cH^1(G,\cP)=R^1{\rm co}_*\cP;$$
e.g., use the factorization $[Z^{+}/G] \to Y^+ \times BG \to Y^{+}$ and apply \cite[Lem.~B.2.2]{Shin2019}.
The support of the sheaf 
$\cH^1(G,\cP)$ is contained 
in the branch locus
of $Z^+\to Y^+$, since the higher
cohomology of a $G$-module of the form
$A[G]$ for an abelian
group $A$ vanishes.

\begin{definition}\label{def:pre-mult}
We call $\cH^1(G,\cP)$ the {\it multiplicity
sheaf}, and the {\it pre-multiplicity group}
is defined to be $\Ash^{G,\,{\rm pre}} \coloneqq 
H^0(Y^+, \cH^1(G, \cP))$.
\end{definition}

\begin{proposition}\label{mult-class} 
An abelian fibration $f\colon X\to Y$ 
inducing $(Y,B,\bfM,\Phi^o)$ defines an element 
$m(f)\in H^0(Y^+, \cH^1(G, \cP))$.
If $f_1\colon X_1\to Y$ and $f_2\colon X_2\to Y$
are two such fibrations, then $m(f_1)=m(f_2)$ 
if and only if
$f_1$ and $f_2$ are \'etale-locally 
translation-birational over $Y^+$.
\end{proposition}

This follows formally from (\ref{first-answer})
and Proposition \ref{define-twist},
but we find it instructive to explicitly construct
and geometrically interpret the map $m$.

\begin{proof}
Let $\{U_i\}_{i\in I}$ be a sufficiently fine \'etale 
open cover of $Y^+$. Over $\wU_i$
we have an \'etale-local translation-birational map
$\varphi_i \colon (W^+)_{\wU_i}
\dashrightarrow (V^+)_{\wU_i}$. 
On the normalized base change,
we have a natural $G$-action by automorphisms
sending fibers to fibers
$$
\rho_i^{\rm aff}\colon
G\to {\rm Aut}\,(W^+)_{\wU_i}$$ but acting
nontrivially on the base $Z^+$ and the 
open set $\wU_i\subset Z^+$.
Also, by construction, we have 
a $G$-action on the Kulikov model of the Albanese, 
$$\rho_i\colon 
G\to {\rm Aut}\,(V^+)_{\wU_i}.$$
Via $\varphi_i$ we can compare
them. The two actions 
$(\varphi_i^{-1})^\ast  \rho_i^{\rm aff}$ and $\rho_i$
act on smooth abelian fibers of $(V^+)_{\wU_i}$ 
with the same linear
part (the former action is only affine-linear whereas
the latter action is linear, as it preserves the origin
of the Albanese fibration). Thus, their
difference
is a translation-birational 
automorphism over $\wU_i$. That is,
we have a set-map
$$
m_i \coloneqq (\varphi_i^{-1})^\ast \rho_i^{\rm aff}\circ 
\rho_i^{-1}\colon G\to \cP(\wU_i)
$$
since
$\rho_i$ and $\rho_i^{\rm aff}$ define the same action 
on the base $Z^+$. Using that 
$\rho_i$ and
$\rho_i^{\rm aff}$
are representations of $G$, 
the $(m_i(g))_{g\in G}$ define 
a $G$-cocycle
in $Z^1(G,\cP(\wU_i))$. 
A different choice of
\'etale-local translation-birational 
map $\varphi_i$ corresponds to changing
the $m_i(g)$ by a $1$-coboundary. Hence 
$$m(f)(U_i) \coloneqq \{m_i(g)\}_{g\in G}$$ defines canonically 
an element of $H^1(G,\cP(\wU_i))$. Furthermore, 
$m(f)(U_i)=0$ if and only if $(W^+)_{\wU_i}$ and 
$(V^+)_{\wU_i}$ are $G$-equivariantly 
translation-birational.

Finally, we have that $m(f)(U_i)$ and $m(f)(U_j)$
agree on their intersection $U_i\cap U_j$ and hence
define a global section 
$m(f)\in H^0(Y^+,\cH^1(G, \cP))$ 
which enjoys the property
that $m(f)=0$ if and only if $g^+\colon W^+\to Z^+$
and $h^+\colon V^+\to Z^+$ are \'etale-locally, 
$G$-equivariantly, translation-birational 
at all points $y\in Y^+$.
Quotienting both sides
by $G$,  we conclude that $f^+\colon X^+\to Y^+$ and 
$V^+/G \to Z^+/G=Y^+$ are \'etale-locally
translation-birational if and only if $m(f)=0$. 
See also Proposition \ref{alg-mult}.
\end{proof}

\begin{definition} The {\it
multiplicity class} of $f$ is 
$m(f)\in \Ash^{G,\,{\rm pre}}$. 
\end{definition}

\begin{proposition}\label{mult:finite} 
The pre-multiplicity
group $\Ash^{G,\,{\rm pre}}$ is finite, 
of size bounded only in terms of 
$(Y,B,\bfM,\Phi^o)$. \end{proposition}

\begin{proof}
We first prove that
$\cH^1_{\rm an}(G,\cP)$
is a constructible sheaf of finite
abelian groups on $Y^+$. 
Taking the 
long exact sequence in 
$G$-cohomology of the usual
exact sequences (\ref{component-exact})
and (\ref{ses}), we get the following
exact sequences of sheaves on $Y^+$:
\begin{gather*}
\cH^0(G, 
\textstyle \mu)\to 
\cH^1_{\rm an}(G, \cP^0)\to 
\cH^1_{\rm an}(G, \cP) \to
\cH^1(G, 
\textstyle \mu), \\
\cH^1_{\rm an}(G, \mathfrak{p})\to 
\cH^1_{\rm an}(G, \cP^0) \to
\cH^2(G, \Gamma)\to
\cH^2_{\rm an}(G, \mathfrak{p}).
\end{gather*}
The first and last terms 
of the first exact sequence
are constructible sheaves of 
finite abelian groups on $Y^+$, and
the first and last terms
of the second exact sequence vanish
by Maschke's theorem, since 
$\mathfrak{p}$ is a sheaf of 
characteristic zero vector spaces.
Finally $\cH^2(G,\Gamma)$
is a constructible sheaf of 
finitely generated groups since $\Gamma$
is finitely generated, but also
a sheaf of torsion groups since 
$\cH^2(G,\Gamma_\bC)=0$.

By Lemma \ref{an-mult-is-mult}, the 
\'etale and analytic 
multiplicity sheaves
are equal.
Finally, the space of global sections 
of a constructible sheaf of 
finite abelian groups is itself 
finite abelian, so we conclude
the proposition.
\end{proof}

\begin{lemma}\label{an-mult-is-mult} 
$\cH^1_{\rm an}(G,\cP)=\cH^1(G,\cP)$. \end{lemma}

\begin{proof}
Since $\cH^1_{\rm an}(G,\cP)$ is a constructible
sheaf of finite groups, 
it again follows by an algebraic
approximation argument that
$\cH^1_{\rm an}(G,\cP)=\cH^1(G,\cP)$,
i.e.~any analytic-local affine-linear twist
$\rho^{\rm aff}(g)=\rho(g)+m(g)$ of 
$\rho\colon G\to {\rm Aut} 
(V^+)_{\wU}$ 
can be realized \'etale-locally.
Alternatively, we claim it is possible to represent
any element of $\cH^1(G,\cP)(U)$
or $\cH^1_{\rm an}(G,\cP)(U)$ by a 
$1$-cocycle $m=(m(g))_{g\in G}$ with 
$m(g)$ torsion. For instance, we have
an exact sequence of group schemes 
$$0\to  \cB^1(G,\cP)\to \cZ^1(G,\cP)\to \cH^1(G,\cP)\to 0$$
where $\cB^1$ are the $1$-coboundaries
and $\cZ^1$ are the $1$-cocycles, of $G$ valued
in $\cP$ (and a similar sequence analytically). 
Letting $m\in \cH^1(G,\cP)(U)$, there
is some integer $e$ so that $em=0$. Up to shrinking
$U$, we may choose
a lift of $m$ to $\widetilde{m}\in \cZ^1(G,\cP)(U)$.
We have $e\widetilde{m} \in \cB^1(G,\cP)(U)$. Letting $e'$
be an exponent for the component group scheme of
$\cB^1(G,\cP)$, we have that 
$e'em\in (\cB^1)^0(G,\cP)(U)$ is a section of the identity 
component of $\cB^1(G,\cP)$,
which is a divisible group scheme. Hence, 
$e'e\widetilde{m} 
= e'eb$ for some $b\in \cB^1(G,\cP)(U)$.
Then $\widetilde{m}-b$ is a lift of $m$
which is $(e'e)$-torsion.
It follows that any analytic or algebraic
affine-linear twist can also
be constructed algebraically and 
\'etale-locally, as an $(e'e)$-torsion twist.
\end{proof}

\begin{example}\label{multiple-ex1}
Let $f\colon X\to Y$ be an elliptic fibration
over a curve $Y$,
with a multiple fiber over $p\in Y$ of multiplicity $n$,
whose reduction is a smooth elliptic curve. Then, after a cyclic
base change $Z\to Y$ of order $n$ totally ramified over $p$,
the normalized pullback $g\colon Z\to Y$ 
has a smooth fiber over the point $q\mapsto p$. Thus, 
$g\colon W\to Z$ is \'etale-locally birational
to a Kulikov model---it is already a Kulikov model,
and admits an \'etale-local 
isomorphism near $q$ to its Albanese,
 by choosing an \'etale section.

But $g\colon W\to Z$ and 
 $g^{\rm Alb}\colon W\,\!^{\rm Alb}\to Z$
 are not $G=C_n$-equivariantly isomorphic (or birational),
 because $C_n$ acts on the fiber $g^{-1}(q)$
 by an $n$-torsion translation, whereas 
 $C_n$ acts on $E_q \coloneqq (g^{\rm Alb})^{-1}(q)$ trivially.
 We can identify the sheaf $\cH^1_{\rm an}(C_n,\cP)$ with a skyscraper
 sheaf at $p$ of the $n$-torsion subgroup $E_q[n]$, and 
 comparing the two $C_n$ actions, we see that the 
 multiplicity class lies in $m(f)\in \Ash^{G,\,{\rm pre}}=
 H^0(Y, \cH^1(C_n, \cP))=E_q[n]$.
 Thus, the multiplicity class $m(f)\in E_q[n]$ 
 in this example encodes exactly which $n$-torsion 
 element we should quotient $E_q$ by, to
 produce the multiple fiber $f^{-1}(p)$.
 \end{example}

Summarizing and combining
Propositions \ref{define-twist},
\ref{local-no-twist}, and \ref{mult-class}, 
we have shown:

\begin{corollary}\label{twist-cor}
Let $f\colon X\to Y$ be an abelian
fibration inducing $(Y,B,\bfM,\Phi^o)$. There is a big
open set $Z^+\subset Z$ depending only on $(Y,B,\bfM,\Phi^o)$,
such that $f$ defines a $G$-equivariant
twist class 
$t(f)\in \Sha_G=H^1_G(Z^+,\cP)$ mapping to a 
multiplicity class 
$m(f)\in \Ash^{G,\,{\rm pre}}$.
Furthermore, if $f_1\colon X_1\to Y$ and 
$f_2\colon X_2\to Y$ are two such fibrations
for which $m(f_1)=m(f_2)$, then there is a natural
lift of $t(f_1)-t(f_2)\in \Sha=H^1(Y^+,P)$ 
to a $G$-invariant twist class. 
This class equals
zero if and only if $f_1$ and $f_2$ are 
translation-birational over $Y^+$.

Similarly, if $m(f_1)_{\rm an}= m(f_2)_{\rm an}$,
then $t_{\rm an}(f_1)-t_{\rm an}(f_2)
=0\in \Sha_{\rm an}$ if and only if
$f_1$ and $f_2$ are translation-bimeromorphic 
over $Y^+$.
\end{corollary}

Thus, fibrations $f$ with the same multiplicity
class $m(f)$ naturally form a torsor over $\Sha$.
The question of 
whether a given class 
$m\in \Ash^{G,\,{\rm pre}}$ 
is realizable by an abelian fibration 
of an algebraic space 
is rather subtle:

\begin{example}\label{multiple-ex2}
Let $E\times \bP^1_{\bf y}\to \bP^1_{\bf y}$ 
be a trivial elliptic fibration over a projective line $\bP^1_{\bf y}$ with coordinates
${\bf y}=[y_0:y_1]$
and let $f\colon X\to \bP^1_{\bf y}$ be the result of
order $2$ log transforms \cite[Ch.~V.13]{barth2015compact} over $0$ and $\infty$,
so that $f$ has multiple fibers over $0, \infty$
with $f^{-1}(0)^{\rm red}$ and 
$f^{-1}(\infty)^{\rm red}$ elliptic curves
isogenous to $E$ by an order $2$ isogeny.
Note that $X$ is a compact complex manifold, but
may not be projective.
We have $B= \tfrac{1}{2}[0]+\tfrac{1}{2}[\infty]$
and $\bfM=0$.
After taking the order $2$ base change $\bP^1_{\bf z} \to \bP^1_{\bf y}$,
${\bf z} \mapsto {\bf z}^2 = {\bf y}$
with Galois group $C_2$, 
the normalized pullback becomes
a smooth analytic
genus $1$ fibration $g\colon W\to \bP^1_{\bf z}$.

As in Example \ref{multiple-ex1},
the sheaf $\cH^1(C_2, \cP)=E[2]_0\oplus E[2]_\infty$
is a sum of skyscraper sheaves,
and the pre-multiplicity group is $E[2]\oplus E[2]$.
Thus, there are $4\cdot 4=16$ possible multiplicity
structures, all of them realizable analytically. 
But we claim that only the 
diagonal subgroup 
$\Delta_{E[2]}\subset E[2]\oplus E[2]$
of the pre-multiplicity group
can be represented by an algebraic space.

To see why, first note that 
$V=W^{\rm Alb}\simeq E\times \bP^1_{\bf z}$ is a trivial
elliptic fibration and the exact sequence (\ref{ses})
is
\begin{equation} \label{ses-example}
0\to \underline{\bZ}^2\to \mathfrak{e}\to \cE\to 1
\end{equation}
where $\mathfrak{p}=\mathfrak{e}\simeq \cO_{\bP^1_{\bf z},\,{\rm an}}$
is the Lie
algebra of $\cP=\cE$, and
$\cE(U) \coloneqq {\rm Hol}(U,E)$
is a $C_2$-equivariant
sheaf on $\bP^1_{\bf z}$.
The long exact sequence in
analytic sheaf cohomology gives an isomorphism
$$\widetilde{\Sha}_{\rm an}=
H^1_{\rm an}(\bP^1_{\bf z}, \cE)\xrightarrow{\sim} 
H^2(\bP^1_{\bf z}, \underline{\bZ}^2)\simeq \bZ^2.$$

By Proposition \ref{sha-inj}, 
the analytification map $\widetilde{\Sha}=H^1(\bP^1_{\bf z},\cE)\to 
H^1_{\rm an}(\bP^1_{\bf z}, \cE)$ is injective, 
but on the other hand, $H^1(\bP^1_{\bf z},\cE)$
is torsion by Lemma \ref{torsion-lemma-1}. Thus 
$H^1(\bP^1_{\bf z},\cE)=0$. So if 
$f\colon X\to \bP^1_{\bf y}$ has an algebraic structure,
we must have that $g\colon W\to \bP^1_{\bf z}$ is a trivial
fibration $W\simeq E\times \bP^1_{\bf z}$. So in this
case, $X$ is a global $C_2$-quotient
of $E\times \bP^1_{\bf z}$, and therefore 
the multiplicity classes
in $E[2]_0$ and $E[2]_\infty$ must be equal.
\end{example}

\begin{definition}\label{mult-alg} 
Fix a base $(Y,B,\bfM,\Phi^o)$.
The {\it multiplicity group}
$\Ash^G\subset \Ash^{G,\,{\rm pre}}$ 
is the subset of $m\in \Ash^{G,\,{\rm pre}}$ 
for which there exists some $t(f)\in \Sha_G$ 
mapping to $m$, i.e., for which
there exists an abelian fibration $f\colon X\to Y$
of a projective variety realizing $m$.
\end{definition}

\begin{remark}\label{alg-space-remark} 
There is no difference replacing ``projective variety''
with ``proper algebraic space'' in Definition 
\ref{mult-alg}, by Artin's theorem \cite{artin1970algebraization}.
\end{remark}

\begin{proposition}\label{alg-mult} 
There is a natural map
$\Ash^{G,\,{\rm pre}}\xrightarrow{\rm ob} H^2(Y^+,P)$ 
and the kernel is $\Ash^G$. In particular,
$\Ash^G\subset \Ash^{G,\,{\rm pre}}$ is a subgroup.
\end{proposition}

Again, this follows formally
from (\ref{first-answer}). But for
clarity and geometric understanding, 
we construct the obstruction
map explicitly.

\begin{proof}
Fix an element $m\in \Ash^{G,\,{\rm pre}}=
H^0(Y^+,\cH^1(G,\cP))$ and suppose that $\{U_i\}_{i\in I}$
is an \'etale
cover of $Y^+$ for which $m$ is represented
by a collection of compatible
elements $m_i\in H^1(G, \cP(\wU_i))$. 
By definition, each $m_i$ gives a twist
of the linear $G$-action 
$\rho_{V^+}\colon G\to {\rm Aut}(V^+)_{\wU_i}$ 
to an affine linear $G$-action
$$
\rho_i\colon G\to {\rm Bir}(V^+)_{\wU_i}
$$
acting on the base in the same manner, with the same
linear part as $\rho_{V^+}$.
Since $m_i\vert_{U_{ij}}=m_j\vert_{U_{ij}}$,
there is
a $G$-equivariant translation-birational identification 
of the actions of $\{\rho_i\}_{i\in I}$ on the double
overlaps $\wU_{ij}=\wU_i\cap \wU_j$ (Prop.~\ref{mult-class}).
Choose
one such isomorphism $\varphi_{ij}\in \cP(\wU_{ij})^G=P(U_{ij})$.
The obstruction to globalizing the local $G$-actions
is then represented by the $2$-cocycle 
${\rm ob}_{ijk} \coloneqq \varphi_{ki} \circ 
\varphi_{jk}\circ \varphi_{ij}$ which defines a class 
${\rm ob}(m)\in H^2(Y^+,P)$.
We need ${\rm ob}(m)=0$ for the actions
$\rho_i$ to patch together, so we at least
have a containment $\Ash^{G,\,{\rm pre}}
\subset {\rm ker}({\rm ob})$.

Conversely,
we need to realize $m$ as the 
multiplicity class of a fibration 
of an algebraic space 
(cf.~Rem.~\ref{alg-space-remark})
whenever ${\rm ob}(m)=0$. In this case,
the affine linear $G$-actions 
$\rho_i$ are regular on $\cP_{\wU_i}$ and,
by vanishing of ${\rm ob}(m)$,
can be glued, in the
\'etale topology, to produce an algebraic space
$\cP(m)\to Z^+$ admitting a $G$-action. Taking
a $G$-equivariant compactification, we produce
an abelian fibration $g\colon W\to Z$.
The quotient by $G$ is an abelian fibration 
$f\colon X\to Y$ whose multiplicity class is $m$. 
Note that the construction of $\cP(m)$ is naturally 
ambiguous up to 
$\varphi_{ij}\mapsto \varphi_{ij}+t_{ij}$ where 
$\{t_{ij}\}$ define a $1$-cocycle in 
$P(U_{ij})=\cP(\wU_{ij})^G$, i.e., an element of 
$\Sha=H^1(Y^+,P)$.
The proposition follows.
\end{proof}

We may define, analogously,
the {\it analytic multiplicity group} 
$\Ash^{G,\,{\rm an}}$ as the kernel of
the map $$\Ash^{G,\,{\rm pre}}=
H^0(Y^+,\cH^1(G,\cP))\to 
H^2_{{\rm an}}(Z^+,P).$$ 

Consider the analytification
map on the exact sequence (\ref{first-answer}):
\begin{align}\label{an-exact}
\begin{aligned}\xymatrix{0\ar[r] & \Sha\ar[r] \ar[d] & 
\Sha_G\ar[r]\ar[d] & \Ash^{G,\,{\rm pre}}\ar[r]\ar[d] & H^2(Y^+, P)\ar[d] \\ 
0\ar[r] & \Sha_{\rm an}\ar[r] & \Sha_{G,\,{\rm an}}\ar[r] & 
\Ash^{G,\,{\rm pre}}\ar[r] & H^2_{\rm an}(Y^+,P).}\end{aligned}\end{align}
By Proposition \ref{sha-inj} below and Lemma
\ref{an-mult-is-mult},
the first, second, and third downward
maps are injective, but we will see that,
in general, the fourth is not injective.

\begin{example}\label{multiple-ex3} 
Example \ref{multiple-ex2} indicates
that even when local $G$-actions in the proof of 
Proposition \ref{alg-mult} fail to patch in the 
\'etale topology, they may successfully
patch in the analytic topology,
i.e., $\Ash^G\neq \Ash^{G,\,{\rm an}}$ in general.
Indeed, in Example
\ref{multiple-ex2}, the analytification
of exact sequences as in (\ref{an-exact}) gives
the following diagram:
\[\xymatrix{0\ar[r] & 
H^1(\bP^1_{\bf y}, E)\ar[r] \ar[d] & 
H^1_{C_2}(\bP^1_{\bf z}, \cE)\ar[r]\ar[d] & 
H^0(\bP^1_{\bf y}, \cH^1(C_2,\cE))
\ar[r]^{\,\,\,\,\,\,\,\,\,\,\,\,\rm ob}\ar[d]^{\sim} & 
H^2(\bP^1_{\bf y},E) \ar[d] \\ 
0\ar[r] & 
H^1_{\rm an}(\bP^1_{\bf y}, E)\ar[r]& 
H^1_{C_2,\,{\rm an}}(\bP^1_{\bf z}, \cE)\ar[r]
& H^0(\bP^1_{\bf y}, \cH^1(C_2,\cE))
\ar[r]^{\,\,\,\,\,\,\,\,\,\,\,\,\rm ob} & 
H^2_{\rm an}(\bP^1_{\bf y},E)}
\]
where by abuse of notation,
$E$ denotes the sheaf of 
functions on $\bP^1_{\bf y}$ valued
in $E$.

The bottom row is easily computed
using the exponential exact sequence \eqref{ses-example}.
Analyzing the long exact sequence
associated to the multiplication by $n$
map on $E$, we may also determine the top
sequence. The result is
\[\xymatrix{0\ar[r] & 
0 \ar[r] \ar[d] & E[2] \ar[r]\ar[d] & E[2]\oplus 
E[2] \ar[r]^{\,\,\,\,\,\,\,\,\rm ob}\ar[d] & E[\infty] \ar[d] \\ 
0\ar[r] & 
\bZ^2 \ar[r]& \bZ^2\oplus E[2] \ar[r]
& E[2]\oplus E[2]\ar[r]^{\,\,\,\,\,\,\,\,\,\,\,\,\rm ob} & 0}\] 
with the \'etale (upper) obstruction map
sending $(\alpha,\beta)\mapsto \alpha-\beta$.
So only the diagonal
subgroup of the pre-multiplicity
group is realizable algebraically,
as expected.
\end{example}

\subsection{Injectivity and surjectivity 
properties of analytification}\label{sec:big-good}
We now show why controlling 
the Tate--Shafarevich twists analytically 
over the 
big open set $Y^+\subset Y$ will
suffice for our purposes. More precisely,
we can check whether two algebraic twists
are the same by comparing them analytically
over $Y^+$.

\begin{proposition}\label{big-open-2}
Let $X$ and $X'$ be two normal, proper algebraic 
spaces (resp.~analytic spaces of Fujiki class 
$\cC$\footnote{A proper 
analytic space is of {\it Fujiki class} 
$\cC$ if it is bimeromorphic to a 
K\"ahler manifold.})
admitting surjective morphisms $f\colon X\to Y$, 
$f'\colon X'\to Y$ to a normal, 
proper target $Y$.
Let $Y^+\subset Y$ be
a big open set and define $X^+ \coloneqq f^{-1}(Y^+)$ and
$(X')^+ \coloneqq (f')^{-1}(Y^+)$. 
A bimeromorphism $\varphi\colon 
(X')^+\dashrightarrow X^+$ over the open
set $Y^+$ defines a birational (resp.~bimeromorphic) 
map $X'\dashrightarrow X$.
\end{proposition}

\begin{proof}
First, assume that $X$ and $X'$ are proper algebraic spaces.
Observe that $X$ admits a
bimeromorphic modification 
to a projective variety,
by Artin's theorem that 
proper algebraic spaces
are Moishezon \cite{artin1970algebraization}.
So assume
without loss of generality that $X$
and $X'$ are projective, and still
admit morphisms to $Y$.
Consider the map 
\begin{align*} 
\psi\colon Y^+&\dashrightarrow {\rm Chow}(X'\times X) \\
y&\mapsto \overline{\{(x,\varphi(x))\in X'\times X\,:\,f(x)=y\}}
\end{align*}
which assigns to a point $y\in Y^+$ the 
graph closure of the restriction of $\varphi$
to the fiber of $f'\colon X'\to Y$ over $y$. 
Since $\varphi$ is meromorphic over $Y^+$,
the map $\psi$ is also meromorphic.
By the Levi extension theorem 
\cite[p.~133, Thm.~4]{narasimhan},
and the projectivity
of the Chow variety of $X'\times X$,
it follows that $\psi$ extends to a meromorphic
map on $Y$, which by GAGA is a rational map.
It follows that the closure of the image of $\psi$ 
in ${\rm Chow}(X'\times X)$ is algebraic. So the union
of the algebraic cycles in $X\times X'$ parameterized 
by this image, which is the graph closure of $\varphi$,
is also algebraic. We conclude that
$\varphi$ is a rational map.

More generally,  if $X$, $X'$ are of Fujiki class $\cC$, 
then by \cite[Main Thm.]{barlet1989}
so is the Barlet space of $X'\times X$ 
(the analytic analog of the 
Chow variety).
Then, by 
\cite[Main Thm.]{ivash},
the Barlet
space still admits the Hartogs'/Levi
extension property. Let $\{U_i\}_{i\in I}$
be a Stein open cover of $Y$.
The proposition follows, by applying
{\it loc.cit.}~to $U_i\cap Y^+$, which
proves that $\psi$ extends meromorphically
to the envelope of holomorphy $U_i$.
\end{proof}

\begin{proposition}\label{sha-inj}
The analytification 
maps $\Sha_G^0\to \Sha_{G,\,{\rm an}}^0$ and
$\Sha_G\to \Sha_{G,\,{\rm an}}$
are injective.
\end{proposition}

\begin{proof} Let $t\in \Sha_G$
and suppose that $t_{\rm an}=0\in \Sha_{G,\,{\rm an}}$.
Associated to $t$, we can produce an
\'etale twist $g^+\colon W^+\to Z^+$ of the 
fibration $h^+\colon V^+\to Z^+$ as follows: Cover 
$Y^+$ by \'etale open sets $\{U_i\}_{i\in I}$ whose inverse
images $\wU_i\subset Z^+$ carry the cohomology 
class $t\in H^1_G(Z^+,\cP)$. Associated to $t$, 
we have an affine-linear twist $m(f)(U_i)$ of 
the $G$-action over $\wU_i$ and 
$G$-equivariant translation-birational maps
$t_{ij} \in \cP(\wU_{ij})$
satisfying the cocycle condition.

We cannot in general glue cylinders
$(V^+)_{\wU_i}$ and $(V^+)_{\wU_j}$ as we usually would,
because $t_{ij}$ is only birational. But $t_{ij}$
is biregular on the smooth loci of fibers
$\cP\simeq (V^+)^{\rm sm}$ and so we may construct
a twist $\cP(t)$ of the group scheme,
gluing $\cP_{\wU_i}$ and $\cP_{\wU_j}$
by the restriction of $t_{ij}$.
As a $G$-equivariant \'etale twist,
$\cP(t)$ 
naturally has the structure of an algebraic 
$G$-space, and so it admits a $G$-equivariant 
compactification $g\colon W\to Z$, which is
guaranteed to exist \cite{nagata}.

Set $f\colon X\to Y$ to be the quotient by $G$; 
then by construction, $t(f)=t$. 
Since $t_{\rm an}=0$, we
have that $f^+\colon X^+\to Y^+$ is 
translation-bimeromorphic to its Albanese fibration, 
by Corollary \ref{twist-cor}.
Since $Y^+\subset Y$ is a big open,
we conclude by Proposition \ref{big-open-2}
that $f$ is translation-birational to its Albanese
fibration. Hence $t=0$, as desired.

It remains to consider $\Sha^0_G$. 
By the component exact
sequence (\ref{component-exact}), 
we have
an exact sequence
$$0\to H^0_G(Z^+,\mu)/{\rm im}\,
H^0_G(Z^+,\cP)\to \Sha^0_G\to \Sha_G$$
and thus the injectivity of $\Sha^0_G\to 
\Sha^0_{G,\,{\rm an}}$ follows from
the statement that $$H^0_G(Z^+,\mu)/{\rm im}\,H^0_G(Z^+,\cP)
=H^0_{G,\,{\rm an}}(Z^+,\mu)/{\rm im}\,H^0_{G,\,{\rm an}}(Z^+,\cP).$$
Since $\mu$ is a constructible group scheme of finite
groups, we have $H^0_G(Z^+,\mu)=
H^0_{G,\,{\rm an}}(Z^+,\mu)$ by 
the \'etale-singular comparison theorem
\cite[Exposé XI, Thm.~4.4(iii)]{artin-comparison}.
We also have 
$H^0_G(Z^+,\cP)=H^0_{G,\,{\rm an}}(Z^+,\cP)$ because 
a $G$-invariant holomorphic section of $\cP\to Z^+$
is automatically algebraic, by 
Lemma \ref{mero-ext} and GAGA, 
applied to the compactification
$\cP\hookrightarrow V^+\hookrightarrow V$ 
over $Z$.
\end{proof}

\begin{remark}
As we saw in Example \ref{multiple-ex3}, the map
$H^2_G(Z^+,\cP)\to H^2_{G,\,{\rm an}}(Z^+,\cP)$ need
not be injective, even over a proper base $Z^+$.
So Proposition \ref{sha-inj} is a 
phenomenon particular to $H^1$ and its geometric
interpretation.
\end{remark}

\begin{proposition} \label{mult-surj}
The map $\Sha_G\to \Sha_{G,\,{\rm an}}$ 
surjects onto the torsion subgroup. \end{proposition}

\begin{proof} First we check the analogous
statement for the map $\Sha_G^0\to \Sha_{G,\,{\rm an}}^0$.
Since $\cP^0$ is a divisible group scheme,
we have an exact sequence $$0\to \cP^0[n]\to 
\cP^0\xrightarrow{\cdot n} \cP^0\to 0.$$
Taking the analytification and long exact
sequence in cohomology,
we obtain a surjection 
\begin{align*}
H^1_{G,\,{\rm an}}(Z^+,\cP^0[n])
\twoheadrightarrow
\Sha_{G,\,{\rm an}}^0[n].
\end{align*}
Since $\cP^0[n]$
is a finite group scheme, we have 
$H^1_{G,\,{\rm an}}(Z^+,\cP^0[n])=H^1_G(Z^+,\cP^0[n])$
by \'etale-singular comparison.
Thus, any element $t_{\rm an}\in 
(\Sha^{0}_{G,\,{\rm an}})_{\rm tors}$
lifts to an element of $H^1_G(Z^+,\cP^{0}[n])$ for some
$n$, and in turn, the image
$t\in \Sha^{0}_G=H^1_G(Z^+,\cP^{0})$ is 
an inverse image of $t_{\rm an}$.

To prove the original statement, it suffices
to prove that $\Sha_G$ and $\Sha_{G,\,{\rm an}}$
have the same image in 
$H^1_G(Z^+,\mu)\simeq H^1_{G,\,{\rm an}}(Z^+,\mu)$;
the isomorphism holds again by \'etale-singular comparison. Equivalently,
the kernels of the maps
\begin{align*}
H^1_G(Z^+,\mu)&\to H^2_G(Z^+,\cP^0) \\
H^1_{G,\,{\rm an}}(Z^+,\mu)&\to H^2_{G,\,{\rm an}}(Z^+,\cP^0)
\end{align*}
are equal. Since $\mu$ is a finite group scheme,
these coboundary maps factor through the coboundary
maps to
$H^2_G(Z^+,\cP^0[n])\simeq H^2_{G,\,{\rm an}}(Z^+,\cP^0[n])$,
once the exponent of $\mu$ divides $n$.
Thus, it suffices to prove that the maps
\begin{align*}
H^2_G(Z^+,\cP^0[n])&\to H^2_G(Z^+,\cP^0) \\
H^2_{G,\,{\rm an}}(Z^+,\cP^0[n])&\to H^2_{G,\,{\rm an}}(Z^+,\cP^0)
\end{align*}
have the same kernel, or equivalently, that
$\cdot n$ has the same cokernel on $\Sha^0_G$
and $\Sha^0_{G,\,{\rm an}}$. This follows
from the first paragraph and Proposition \ref{sha-inj},
which proves that $\Sha^0_G\simeq 
(\Sha^0_{G,\,{\rm an}})_{\rm tors}$.
\end{proof}

\begin{theorem}\label{ray-mult} 
$\Sha_G\simeq 
(\Sha_{G,\,{\rm an}})_{\rm tors}$
 and  $\Sha^0_G\simeq 
(\Sha^0_{G,\,{\rm an}})_{\rm tors}$.\end{theorem}

\begin{proof} This follows from 
Propositions \ref{sha-inj} and \ref{mult-surj}.\end{proof}

This theorem can be viewed as
a version of \cite[Prop.~XIII.2.3]{ray70}
that allows multiple fibers.
Note though that $\Ash^G\to \Ash^{G,\,{\rm an}}$
may not be surjective, despite both groups
being finite abelian. This is because non-torsion
elements of $\Sha_{G,\,{\rm an}}$ can map to torsion
elements in $\Ash^{G,\,{\rm an}}$. 
See Example \ref{multiple-ex3}.

\begin{remark}\label{real}
Let $\oG \coloneqq {\rm Gal}\,(\overline{\bC(Y)}/\bC(Y)).$ 
Consider the groups $\Sha_G$ for varying
groups $G$, and branched $G$-covers $Z^+\to Y^+$.
These form a filtration of the Weil--Ch\^atelet group
$H^1(\oG, \,X^{\rm Alb}_{\overline{\eta}})$,
$\eta= \Spec \bC(Y)$,
parameterizing
abelian torsors over the geometric generic 
point---each such torsor can be compactified to produce
an abelian fibration, whose multiple and non-semistable
fibers in codimension $1$ can be cleared by
an appropriate choice of $Z^+\to Y^+$ and is thus
represented by an element of $\Sha_G$ 
for some finite quotient 
$\oG\twoheadrightarrow G$.
It follows that 
$$H^1(\overline{G}, \,X^{\rm Alb}_{\overline{\eta}}) = 
\lim_{\oG\twoheadrightarrow G} \,\Sha_G.$$

More geometrically, for any effective Weil divisor 
$n_1B_1+\cdots +n_kB_k$ we define
a root stack $$Y^+(n_1B_1+\cdots+ n_kB_k),$$
where $Y^+\subset Y$ is a big open set 
of the smooth locus, and $Y^+\cap 
(\partial_{\rm H}\cup_i B_i)$ 
is a smooth divisor. 
Then, we may impose a root structure of 
order $n_i$ along $B_i\cap Y^+$. These root stacks
form an inverse system (via the support of the Weil
divisor, the choice of open set $Y^+$,
and the divisibility of the $n_i$). We
take the infinite root stack $\sqrt[\infty]{Y}\,\!^+$
as the inverse limit. Philosophically, 
$$``H^1(\overline{G},X^{\rm Alb}_{\overline{\eta}})=
H^1(\!\sqrt[\infty]{Y}\,\!^+,\cP)."$$

The advantage of the righthand side over the 
lefthand side is that (i)
the filtration of abelian torsors by their multiple
fiber structure is clarified 
(it comes from Leray of the pushforward 
map to the coarse space, see (\ref{the-answer})),
and (ii) $\sqrt[\infty]{Y}\,\!^+$ 
is an inverse limit of big opens, so
$$``H^1(\!\sqrt[\infty]{Y}\,\!^+,\cP)\simeq 
(H^1_{\rm an}(\!\sqrt[\infty]{Y}\,\!^+,\cP))_{\rm tors}."$$
\end{remark}

\begin{warning}
By Proposition \ref{alg-mult}
and Corollary \ref{twist-cor}, we have 
proven that birational classes of abelian 
fibrations $f\colon X\to Y$
inducing $(Y,B,\bfM,\Phi^o)$ are uniquely
represented by elements of the group $\Sha_G$.
But observe that not all elements of $\Sha_G$
produce an abelian fibration whose boundary equals
$B$. In fact, this is ill-defined---the divisor
$B$ depends on the choice of birational model
of $f\colon X\to Y$ while an element of
$\Sha_G$ only specifies
a birational class.

But, because we chose
$G$ and the cover $Z^+\to Y^+$ 
in terms of the support and 
coefficients of $B$,
we know that the birational class
of such an abelian fibration whose boundary
divisor is $B$ can be (uniquely) represented 
by some $t(f)\in \Sha_G$.
\end{warning}

\subsection{Global analysis of the Tate--Shafarevich group}
\label{ts3}

There is an induced
long exact sequence for the short exact sequence
(\ref{ses}) in analytic sheaf 
cohomology, and in $G$-equivariant analytic sheaf
cohomology. Since we have already related $\Sha_G$
and $\Sha^0_G$ by a bounded degree isogeny 
(Cor.~\ref{torsion-lemma-2}), 
we will focus on the diagram
\begin{align}\label{biggg}\begin{aligned}
\xymatrix{
& \Sha^0_G \ar[d]
& \\
H^1_{G,\,{\rm an}}(Z^+,\mathfrak{p}) 
\ar[r] & \Sha^0_{G,\,{\rm an}}
\ar[r] &  H^2_{G}(Z^+,\Gamma)
}
\end{aligned}\end{align}
where the vertical arrow is the analytification morphism.
The group
$H^2_{G}(Z^{+},\Gamma)$ 
is finitely generated
abelian. So the image of 
$\Sha^0_G$ in it,  
being both torsion (cf.~Lem.~\ref{torsion-lemma-1}) 
and finitely generated, is a finite abelian group 
$N^2\subset 
H^2_{G}(Z^{+},\Gamma).$
Let $\overline{\Sha}_{G,\,{\rm an}}^0$ denote the inverse
image of $N^2$ in $\Sha_{G,\,{\rm an}}^0$. 
Let $N^1$ denote
the image of $ 
H^1_{G}(Z^{+},\Gamma)$ in 
$H^1_{G,\,{\rm an}}(Z^+,\mathfrak{p})$.
Then, we have a diagram of the form
\begin{align}\begin{aligned}\label{bigger}
\xymatrix{
 & & \Sha_{G}^0\ar[d] 
 & &
 \\
0 \ar[r] & H^1_{G,\,{\rm an}}(Z^+,\mathfrak{p})/N^1
\ar[r] & \overline{\Sha}_{G,\,{\rm an}}^0
\ar[r] &  N^2 \ar[r] & 0
}\end{aligned}
\end{align}
with the bottom row short exact.

\begin{lemma}\label{lem:h1} We have that
$H^1_{G,\,{\rm an}}(Z^{+},\mathfrak{p})\simeq 
H^1_{\rm an}(Z^{+},\mathfrak{p})^G$ and this isomorphism
induces an isogeny of bounded degree between $N^1$ and 
${\rm im}(H^1(Z^{+},\Gamma)^G\to H^1_{\rm an}(Z^+,\mathfrak{p})^G)$.
\end{lemma}

\begin{proof} We again apply the Cartan--Leray
spectral sequence, but this time $H^i(G, H^0_{\rm an}
(Z^{+},\mathfrak{p}))$ vanishes for $i=1,2$
by Maschke's theorem,
since $H^0_{\rm an}(Z^+,\mathfrak{p})$
is a $\bC$-vector space. For $N^1$ it suffices
to observe that $H^i(G, H^0
(Z^{+},\Gamma))$ for $i=1,2$
are finite abelian of bounded size.
\end{proof}

\begin{lemma} \label{identify-N2}
$N^2=H^2_G(Z^+,\Gamma)_{\rm tors}$.
\end{lemma}

\begin{proof}
Note that $\Sha_G^0\twoheadrightarrow
(\Sha_{G,\,{\rm an}}^0)_{\rm tors}$ 
surjects (Prop.~\ref{mult-surj}) and
$(\Sha^0_{G,\,{\rm an}})_{\rm tors}\twoheadrightarrow
H^2_G(Z^+,\Gamma)_{\rm tors}$ 
surjects, because the next
term in the long exact sequence (\ref{biggg}),
$H^2_{G,\,{\rm an}}(Z^+,\mathfrak{p})$,
is torsion-free.
\end{proof}

Suppose that $t_1$ and $t_2\in \Sha^0_G$ 
map to the same element of $N^2$. Then their difference 
$t_1-t_2$ maps into 
$H^1_{\rm an}(Z^+,\mathfrak{p})^G/N^1$.
But since 
$H^1_{\rm an}(Z^+,\mathfrak{p})^G$
is a $\bC$-vector space, their difference
lands in $\bQ N^1 /N^1\subset \bC N^1/
N^1$ because
$\Sha^0_G$ is torsion. Here $\bC N^1$
means the $\bC$-span of $N^1$ as a subset of 
$H^1_{\rm an}(Z^+,\mathfrak{p})^G$
and not the abelian
group $N^1\otimes \bC$.
Hence, we have a diagram of the form
\begin{align}\begin{aligned}\label{smaller}
\xymatrix{
& & \Sha^0_G \ar[d] & & \\
0\ar[r] & \bC N^1/N^1
\ar[r] & 
\overline{\underline{\Sha}}_{G,\,{\rm an}}^0 
\ar[r] & N^2\ar[r] & 0
}
\end{aligned}
\end{align}
with the bottom row short exact,
for a subgroup 
$\overline{\underline{\Sha}}_{G,\,{\rm an}}^0
\subset \overline{\Sha}_{G,\,{\rm an}}^0$. 

\begin{remark}\label{alternate-thing} 
We could have taken the first 
term to be $\bQ N^1/N^1$ also.
Then the vertical map would 
instead be an isomorphism,
by Theorem \ref{ray-mult},
but such twists are all connected by
analytics twists in $\bC N^1/N^1$.
\end{remark}

\begin{remark}
The short exact sequence (\ref{smaller})
is a significant
improvement over (\ref{bigger}), since
$H^1_{\rm an}(Z^+,\mathfrak{p})^G$ is likely
an infinite-dimensional vector space,
whereas $\bC N^1$, being the $\bC$-span 
of a finitely
generated lattice, is necessarily 
finite-dimensional.
We will identify $\bC N^1$ 
in \S~\ref{sec:continuous-ts}.
\end{remark}

We now find an isogeny of the exact sequence
(\ref{ses}) to a simpler one. As the kernel of 
the exponential map, there is an isomorphism 
$\Gamma_z\simeq H_1(V_z,\bZ)$ for a 
smooth fiber $V_z$. 
In \cite[Prop.~4.4]{AR2021}, 
the polarization is used
to identify with $H^1(V_z,\bZ)$. We do the same
here.

\begin{proposition}\label{isog} 
A relative polarization 
of type ${\bf d}$ on a Kulikov fibration 
$h^+\colon V^+\to Z^+$ with section $s^+$
defines an isogeny, of 
degree bounded solely in terms
of ${\bf d}$, between \eqref{ses}
and the pushforward of the exponential exact sequence
$$
0\to R^1h^+_*\underline{\bZ}_{V^+} 
\to
R^1h^+_*\cO_{V^+} 
\to \cJ \to 1.
$$
Here $\cJ$ is an
extension of the relative Jacobian 
of the smooth fibration
to $Z^+$.
\end{proposition}

\begin{proof}
See \cite[Prop.~4.1]{AR2021},
\cite[Thm.~A]{dutta} for
the same result, in the special case
of Lagrangian fibrations. 
Let $N_{s^+}$ denote the pullback along the section $s \colon Z^{+} \to V^+$ of the normal
bundle of $s(Z^{+})\subset V^{+}$. 
We have
a natural isomorphism
$$
\mathfrak{p}=h^+_*T_{V^+/Z^+}\to 
N_{s^+}
$$
because the Lie algebra of infinitesimal
vertical automorphisms is identified with the relative
tangent space at the origin of the corresponding
Lie group. There exists an isomorphism $\cP\simeq (V^+)^{\rm sm}$ which sends the zero section of $\cP$ to $s^+(Z^+)$; see Remark \ref{p-is-kul}.

Let $\delta^+$
denote the inverse image of $\partial^+$, so that 
$\delta^+\subset V^+$ is an snc divisor, by the
hypothesis that $h^+\colon V^+\to Z^+$ is Kulikov.
Let $\Omega^1_{V^+/Z^+}(\log \delta^+)$ denote 
the sheaf of relative logarithmic 1-forms.
Then, we also have an isomorphism 
$$
h_*^+\Omega^1_{V^+/Z^+}(\log \delta^+)\to 
N_{s^+}^\ast.
$$
This isomorphism
is clear over the smooth fibration 
$h^o\colon V^o\to Z^o$, and can be 
extended to the Mumford construction.
The former is, by definition, the Deligne canonical
extension $\overline{\mathcal{F}}\,\!^1$ of the Hodge filtered
piece $\mathcal{F}^1=h^o_*\Omega^1_{V^o/Z^o}$. 
Thus, we see that 
$$\mathfrak{p}\simeq N_{s^+}\simeq (N_{s^+}^*)^*\simeq (\overline{\mathcal{F}}\!\,^1)^\ast .$$
With respect to the Deligne 
canonical extension of
the relative polarization $L$ of type ${\bf d}$,
which is monodromy-invariant and so extends
naturally to a non-degenerate
symplectic form on $\overline{\mathcal{F}}\,\!^0$,
the bundle $\overline{\mathcal{F}}\,\!^1\subset 
\overline{\mathcal{F}}\,\!^0$ is a Lagrangian subbundle.
So we have an identification
\begin{align}\label{ident}
\mathfrak{p}\simeq 
(\overline{\mathcal{F}}\!\,^1)^\ast \simeq_L 
\overline{\mathcal{F}}\,\!^0/
 \overline{\mathcal{F}}\!\,^1\simeq 
R^1h^+_*\cO_{V^+}
\end{align} 
with the Deligne canonical extension
of $\mathcal{F}^0/\mathcal{F}^1$. 

We have
$\cP^0 = \overline{\mathscr{F}}\,\!^{-1}/
(\overline{\mathscr{F}}\,\!^0 + \mathscr{W}_{-1})\simeq 
(\overline{\mathcal{F}}\!\,^1)^*/\mathscr{W}_{-1}$ 
where $\overline{\mathscr{F}}\,\!^\bullet$, 
$\mathscr{W}_\bullet$ are the Hodge and 
weight filtrations, see (\ref{w-quotient})---here
we interpret $\mathscr{W}_{-1}$ as a closed,
constructible subsheaf
of the Deligne canonical extension, given
by the topological closure of the flat sub-bundle of
integral sub-lattices of the smooth fibers.
Fiberwise, $\mathscr{W}_{-1}$ 
agrees with the weight $-1$
part of the limit MHS at any given point. 
Its fiber is, in particular, a rank $2g$ lattice 
at all points in $Z^o$, and drops to rank $g+g_i$
over $\partial_i$.
Applying the middle 
identification of
(\ref{ident}),  
\begin{align*}
\cP^0\simeq_L \overline{\mathcal{F}}\,\!^0/(\overline{\mathcal{F}}\!\,^1+L(\mathscr{W}_{-1}))\xrightarrow{\,\phi\,} 
\overline{\mathcal{F}}\,\!^0/(\overline{\mathcal{F}}\!\,^1+\mathcal{W}_1)\simeq \cJ\end{align*} where $\overline{\mathcal{F}}\,\!^\bullet$, $\mathcal{W}_\bullet$
denote the Hodge and weight 
filtration on the limit 
mixed Hodge structures on first cohomology, 
and $\phi$ is the quotient
map induced by the 
inclusion of the finite index sub-constructible
sheaf  
$L(\mathscr{W}_{-1})\hookrightarrow \mathcal{W}_1
\simeq R^1h^+_*\underline{\bZ}_{V^+}$ 
(see Lemma \ref{gysin-inj}).
Note that the fiberwise degree of $\phi$
is bounded by the type of $L$. 
The proposition follows.
\end{proof}

\begin{lemma}\label{gysin-inj} Let 
$h^+\colon V^+\to Z^+$ be 
a semistable fibration.
We have 
$R^1h^+_*\underline{\bZ}_{V^+} \simeq
\mathcal{W}_1$ where $\mathcal{W}_1$
is the weight $1$ part of the 
Deligne canonical
extension of first cohomology.
\end{lemma}

\begin{proof} 
Let $T \coloneqq (h^{+})^{-1}(U)$ 
be the preimage under $h^{+}$ of a small analytic 
open set $U\subset Z^+$. 
Let $T \to (U, \partial)$ be the restriction
of $h^+$, where $\partial\subset T$ is the discriminant. Choose general points 
$u\in U$, and $0\in \partial$.
Let $d$ be the relative dimension
and $N$ be the logarithm 
of monodromy on $H^1(T_u,\bZ)$ 
about $0\in \partial$. 
The Clemens--Schmid sequence 
\cite{Morrison1984}
$$H_{2d+1}(T_0,\bQ)=0 \to H^1(T_0,\bQ)
\to H^1(T_u,\bQ)\xrightarrow{N} H^1(T_u,\bQ)$$ 
yields the isomorphism
\[
R^1h^+_*\underline{\bQ}_{V^+}(U) = 
H^1(T_0,\bQ) \simeq \ker N_\bQ =
\mathcal{W}_{1, \bQ}(U).
\]
This gives the desired statement over $\bQ$,
and to check it integrally, 
it suffices to prove that the cokernel
of $H^1(T_0,\bZ)\simeq 
H^1(T,\bZ)\to H^1(T_u, \bZ)$ is torsion-free.
By the universal coefficient theorem, 
$H^2(T, T_u,\bZ)_{\rm tors}\simeq 
H_1(T, T_u,\bZ)_{\rm tors}$. The desired
vanishing now follows
from the surjectivity of 
$\pi_1(T_u)\to \pi_1(T)$,
which holds because
the components of $T_0$ are reduced.
\end{proof}

\begin{corollary}\label{identify-N1}
Consider an abelian fibration $f\colon X\to Y$
inducing $(Y,B,\bfM,\Phi^o)$. A choice of relative
polarization on $f^{\rm Alb}\colon X^{\rm Alb}\to Y$
defines an isogeny of finite index, bounded
solely in terms of $(Y,B,\bfM,\Phi^o)$ and 
the polarization type ${\bf d}$, 
between $$H^1_{\rm an}(Z^+,\mathfrak{p})^G/N^1 
\textrm{ and }
H^1_{\rm an}(Z^+,R^1h^+_*\cO_{V^+})^G/{\rm im}\, 
H^1(Z^+,R^1h^+_*\underline{\bZ}_{V^+})^G.$$
\end{corollary}

\begin{proof} Fix a relative polarization $L$
on the Albanese fibration, of type ${\bf d}$. 
By Proposition \ref{isog},
we have a $G$-equivariant
exact sequence $$0\to \Gamma\xrightarrow{L} 
R^1h^+_*\underline{\bZ}_{V^+}\to C\to 0$$
where $C$ is a constructible sheaf
of finite groups on $Z^+$,
locally constant on $Z^o$ and along each
$\partial_i\subset Z^+$, whose stalks have size 
bounded above,
solely in terms of ${\bf d}$ (e.g., by $\prod d_i$). 
Furthermore,
the map
$L$ extends to a $G$-isomorphism
$\mathfrak{p}\to R^1h_*^+\cO_{V^+}$. 
It follows that $L$ induces an isomorphism
$$\psi_L\colon 
H^1_{G,\,{\rm an}}(Z^+,\mathfrak{p})\to H^1_{G,\,{\rm an}}(Z^+,R^1h^+_*\cO_{V^+}),$$ and a map 
$H^1_G(Z^+,\Gamma)\to H^1_G(Z^+, R^1h^+_*\underline{\bZ}_{V^+})$
whose kernel and cokernel are bounded above in size, by the
size of $H^i_G(Z^+,C)$, $i=0,1$, respectively.
So the images
$N^1\subset H^1_{G,\,{\rm an}}(Z^+,\mathfrak{p})$ and 
${\rm im}\, 
H^1_G(Z^+,R^1h^+_*\underline{\bZ}_{V^+})\subset 
H^1_{G,\,{\rm an}}(Z^+,R^1h^+_*\cO_{V^+})$ are 
subgroups related by a bounded order isogeny, via the
isomorphism $\psi_L$. Finally, we apply Lemma \ref{lem:h1}
to pass from $G$-equivariant cohomology to $G$-invariants 
(at the expense of a further isogeny).
\end{proof}

\subsection{The continuous part of 
the Tate--Shafarevich
group} \label{sec:continuous-ts}
It remains to identify the finite-dimensional
vector space $\bC N^1$ which surjects
onto the identity component of
$\overline{\underline{\Sha}}_{G,\,{\rm an}}^0$.

\begin{proposition}\label{hi}
Let $\wW^+\to W^+$ be a $G$-equivariant resolution
of singularities and let 
$\widetilde{g}^+\colon \wW^+\to Z^+$ be the composition with $g^+$.
There is a commutative diagram
of analytic $G$-sheaves 
\[
\xymatrix{
R^1h^+_*\underline{\bC}_{V^+}
\ar[r]^{\sim} \ar[d]
& R^1\widetilde{g}^+_*\underline{\bC}_{\widetilde{W}^+} \ar[d] \\
R^1h^+_*\cO_{V^+} \ar[r]^{\sim}
& R^1\widetilde{g}^+_*\cO_{\widetilde{W}^+}.
}
\] 
More generally, $R^ih^+_*\cO_{V^+} \simeq 
R^i\widetilde{g}^+_*\cO_{\wW^+}$ and 
the images of the maps $R^ih^+_*\underline{\bC}_{V^+}\to
R^ih^+_*\cO_{V^+}$ and 
$R^i\widetilde{g}^+_*\underline{\bC}_{\wW^+}\to
R^i\widetilde{g}^+_*\cO_{\wW^+}$ are 
$G$-isomorphic.
\end{proposition}

\begin{proof}
The fibrations 
$h^+\colon V^+\to Z^+$ and 
$\widetilde{g}^+\colon \wW^+\to Z^+$ are locally 
translation-birational.
Over an analytic chart
$U$, the birational map 
$(V^+)_U\dashrightarrow (\wW^+)_U$ 
can be factored as a sequence
of blow-ups and blow-downs of smooth centers.
By the blow-up formula for cohomology, these
operations do not alter the sheaves
in the statement of the lemma.
Thus, we get the above commutative diagram over
$U$.

Now observe that if $U_i$ and $U_j$
are two such open sets, the 
translation-birational maps
$(V^+)_{U_i}\dashrightarrow (\wW^+)_{U_i}$ and
$(V^+)_{U_j}\dashrightarrow (\wW^+)_{U_j}$
differ by an element $t_{ij}\in \cP(U_i\cap U_j)$.
But $t_{ij}$ acts on the sheaves 
$R^1h^+_*\underline{\bC}_{V^+}$, $R^1h^+_*\cO_{V^+}$ trivially. 
Thus, the local identifications of 
$R^1h^+_*\underline{\bC}_{V^+}$, 
$R^1\widetilde{g}^+_*\underline{\bC}_{\wW^+}$
and of $R^1h^+_*\cO_{V^+}$, $R^1\widetilde{g}^+_*\cO_{\wW^+}$
are well-defined, independent of the choice
of translation-birational map, and globalize to $Z^+$.
See also \cite[(46.1)]{kollar_new2} for a similar argument.

Finally, we observe that the identifications
are $G$-equivariant because, under
an identification 
$(V^+)_\wU\dashrightarrow (\wW^+)_\wU$
for $\wU$ a $G$-invariant open set, the $G$-actions on
$V^+$ and $W^+$ differ by an element of $\cP(\wU)$ 
(see Prop.~\ref{mult-class}), which acts
trivially on the sheaves 
$R^1h^+_*\underline{\bC}_{V^+}$, $R^1h^+_*\cO_{V^+}$.

The proof of the second statement is similar;
the key point is that under a blow-up of a smooth center,
the sheaf $R^ih^+_*\underline{\bC}_{V^+}$ changes by summing
with a sheaf mapping to zero in $R^ih^+_*\cO_{V^+}$ (e.g., this
follows from torsion-freeness, cf.~Prop.~\ref{ss-degen}) 
and the sheaf $R^ih^+_*\cO_{V^+}$ is unchanged.
\end{proof}

\begin{corollary}\label{all-same}
We have isomorphisms
\begin{align*}\bC N^1&
\simeq {\rm im}(H^1_{G}(Z^+, R^1h^+_*\underline{\bC}_{V^+})\to 
H^1_{G,\,{\rm an}}(Z^+, R^1h^+_*\cO_{V^+})) \\ 
&\simeq {\rm im}(H^1_{G}
(Z^+, R^1\widetilde{g}^+_*\underline{\bC}_{\wW^+})
\to H^1_{G,\,{\rm an}}(Z^+, R^1\widetilde{g}^+_*\cO_{\wW^+})) \\ 
&\simeq {\rm im}(H^1(Y^+, R^1f^+_*\underline{\bC}_{X^+})
\to H^1_{\rm an}(Y^+, R^1f^+_*\cO_{X^+}))
\\ &\simeq  {\rm im}(H^1(Y^+, 
R^1f^{{\rm Alb}+}_*\underline{\bC}_{X^{{\rm Alb}+}})
 \to H^1_{\rm an}(Y^+, R^1f^{{\rm Alb}+}_*\cO_{X^{{\rm Alb}+}}))
 \end{align*}
 where $f^{{\rm Alb}+}\colon X^{{\rm Alb}+}\to Y^+$
 is any klt model of the Albanese fibration of $X$,
 for instance the quotient
 $V^+/G\to Z^+/G=Y^+$ of the Kulikov model.
\end{corollary}

\begin{proof} The first isomorphism
follows from the proof of Corollary \ref{identify-N1},
the second
follows from Proposition \ref{hi}. Now, we have
$$H^1_G(Z^+, 
R^1\widetilde{g}^+_*\underline{\bC}_{\wW^+})
\simeq 
H^1(Y^+, R^1\widetilde{f}^+_*\underline{\bC}_{\wW^+/G})$$
where $\widetilde{f}^+\colon \wW^+/G\to Y^+$
is the quotient of $\widetilde{g}^+\colon
\wW^+\to Z^+$ by $G$, and we have a 
similar statement for $\cO_{\wW^+}$ and $\cO_{\wW^+/G}$.
These isomorphisms hold generally, 
since DM stacks in characteristic zero are tame
\cite[Thm.~3.2]{AOV}.

Now observe that 
$\widetilde{f}^+\colon \wW^+/G\to Z^+/G=Y^+$
and $f^+\colon X^+\to Y^+$ are birational
over $Y^+$ and both have klt total space,
so the third isomorphism follows, as in
Proposition \ref{hi}, because $R^1f^+_*\cO_{X^+}$
and $R^1\widetilde{f}^+_*\cO_{\wW^+/G}$ are both isomorphic
to the corresponding higher direct image 
of $\cO$ on a common resolution. The argument is 
the same for the Albanese fibration.
\end{proof}

\begin{proposition}\label{ss-degen} 
Let $f^{{\rm Alb}+}\colon 
X^{{\rm Alb}+}\to Y^+$ be any
klt model of the Albanese fibration
of $f^+\colon X^+\to Y^+$.
The Leray spectral sequences
\begin{align*}
H^i_{\rm an}(Y^+, R^jf^+_*\cO_{X^+})&
\implies H^{i+j}_{\rm an}(X^+,\cO_{X^+}), \\
H^i_{\rm an}(Y^+, 
R^jf^{{\rm Alb}+}_*\cO_{X^{{\rm Alb}+}})&
\implies H^{i+j}_{\rm an}
(X^{{\rm Alb}+},\cO_{X^{{\rm Alb}+}})
\end{align*}
degenerate at the $E_2$ page.
Furthermore,
the sheaves
$R^jf^+_*\cO_{X^+}\simeq R^jf_*^{{\rm Alb}+}
\cO_{X^{{\rm Alb}+}}$ are vector bundles.
Finally,
$H^k(X^+,\cO_{X^+})\simeq 
H^k(X^{{\rm Alb}+},\cO_{X^{{\rm Alb}+}})$ for all $k$.
\end{proposition}

\begin{proof} Since 
$X^+$ has rational singularities, 
we can replace $X^+$ with a resolution
of singularities without affecting the statement. 
Now $Y^+$ is smooth and the 
discriminant divisor $\partial^+\subset Y^+$ is 
smooth (in particular, normal crossings). 
It follows from \cite[Cor.~3.10]{Kol86} that 
$Rf^+_*\cO_{X^+}\simeq \bigoplus_i R^if^+_*\cO_{X^+}[-i]$
in the derived category, and this remains true
after analytifying. The first statement follows.
Note the projectivity condition in 
{\it loc.cit.}~is 
only necessary on the morphism $f^+$, 
as we may compactify $Y^+$ and the resolution of 
$X^+$ to satisfy the necessary hypotheses.
The argument is the same for $X^{{\rm Alb}+}$.

The fact that these sheaves are vector bundles
follows from
\cite[Thm.~2.6]{Kol86}.

Note that $R^jh^+_*\cO_{V^+}$ and $R^j\widetilde{g}^+_*\cO_{\wW^+}$
 are $G$-isomorphic by Proposition \ref{hi}. So
$R^jf^+_*\cO_{X^+}$ and $R^jf_*^{{\rm Alb}+}
\cO_{X^{{\rm Alb}+}}$ are isomorphic,
and we deduce the last sentence.
\end{proof}
 
Note that we have a morphism of 
Leray spectral sequences
\[
\xymatrix{H^p(Y^+,R^qf^+_*\underline{\bC}_{X^+}) 
\ar@{=>}[r] \ar[d]
& H^{p+q}(X^+,\underline{\bC}_{X^+})\ar[d] \\
H^p_{\rm an}(Y^+,R^qf^+_*\cO_{X^+}) \ar@{=>}[r]
& H^{p+q}_{\rm an}(X^+,\cO_{X^+})}
\]
and the same for $f^{{\rm Alb}+}$.
For shorthand, we write $E_k^{p,q}(\cF)$
to mean the $(p,q)$-entry on the $k$-th page
of the Leray spectral sequence for a sheaf $\cF$,
for the morphism $f^+\colon X^+\to Y^+$,
in the analytic topology.

\begin{proposition}\label{image-infinity}
Let $f\colon X\to Y$ be an abelian
fibration of Picard type. The image of 
$E_2^{p,q}(\underline{\bC}_{X^+}) \to 
E_2^{p,q}(\cO_{X^+})$ is isomorphic 
to the image of 
$E_\infty^{p,q}(\underline{\bC}_{X^+})
\to E_\infty^{p,q}(\cO_{X^+})$ 
for $p+q=2$. 
\end{proposition}

\begin{proof}
The pullback map
$H^p(Y^+,\underline{\bC}_{Y^+}) \to 
H^p(X^+,\underline{\bC}_{X^+})$ 
is injective, 
because $Y^+$ is smooth quasiprojective, 
and $f^+$ is proper and surjective \cite[Thm.~4.1]{wells1974comparison}.
We conclude that $ 
E_2^{p,0}(\underline{\bC}_{X^+})
 = E_\infty^{p,0}(\underline{\bC}_{X^+})$
 and 
 $ 
E_2^{p,0}(\cO_{X^+})
 = E_\infty^{p,0}(\cO_{X^+})$
for any $p$, by Proposition \ref{ss-degen}. 
In particular, the spectral sequence 
for $\underline{\bC}_{X^+}$
has degenerated on the $E_2$ page,
at the $(2,0)$- and $(3,0)$-entries.

The statement follows immediately for $(p,q)=(2,0)$, 
but also for $(p,q)=(1,1)$---the only differentials
$d_2^{1,1}\colon E_2^{1,1}(\underline{\bC}_{X^+})
\to E_2^{3,0}(\underline{\bC}_{X^+})$ 
which could affect
the $(1,1)$-entry must be zero, because
of the  degeneration at the $(3,0)$-entry.
Here we also use that the
target spectral sequence, for $\cO_{X^+}$,
has already degenerated at the $E_2$ page
by Proposition \ref{ss-degen}.

It remains to consider the $(0,2)$-entry.
Since $R^2f^+_*\cO_{X^+}$
is torsion-free (Prop.~\ref{ss-degen}),
any nonzero section cannot vanish 
identically on the Zariski
open set $Y^o$.
But the image of 
$H^0(Y^o, R^2f^o_*\underline{\bC}_{X^o})\to 
H^0_{\rm an}(Y^o,R^2f^o_*\cO_{X^o})$ is zero,
by the Picard type hypothesis.
Thus, the map 
$E_2^{0,2}(\underline{\bC}_{X^+})
\to E_2^{0,2}(\cO_{X^+})$ is zero 
(which ensures it is zero on subsequent pages).
\end{proof}

\begin{remark} Since $f\colon X\to Y$ and 
$f^{\rm Alb}\colon X^{\rm Alb}\to Y$ 
induce the same variation of Hodge
structure, $f$ being of
Picard type is equivalent to $f^{\rm Alb}$
being of Picard type.
\end{remark}

\begin{corollary}\label{cor-n1} 
 If $f\colon X\to Y$ is of Picard type,
 $\bC N^1$ is isomorphic to the image
 of $$H^2(X^+,\underline{\bC}_{X^+})/
 H^2(Y^+,\underline{\bC}_{Y^+})\to 
 H^2_{\rm an}(X^+,\cO_{X^+})/H^2_{\rm an}(Y^+,\cO_{Y^+}).$$
 The same holds replacing $X^+$ with $X^{{\rm Alb}+}$.
\end{corollary}

\begin{proof}
By Corollary \ref{all-same}, $\bC N^1$
is identified with the image of either
$E_2^{1,1}(\underline{\bC}_{X^+})\to 
E_2^{1,1}(\cO_{X^+})$ or
$E_2^{1,1}(\underline{\bC}_{X^{{\rm Alb}+}})\to 
E_2^{1,1}(\cO_{X^{{\rm Alb}+}})$, which by
Proposition \ref{image-infinity}, coincides
with the corresponding image on the $E_\infty$
page. It follows that $\bC N^1$ is identified
with the image
of ${\rm gr}_L^1 H^2(X^+, \underline{\bC}_{X^+})\to 
{\rm gr}_L^1 H^2(X^+, \cO_{X^+})$ where $L^\bullet$
is the Leray filtration (or the same for $X^{{\rm Alb}+}$).

To conclude the corollary, it suffices to show that 
${\rm gr}_L^2 H^2(X^+,\underline{\bC}_{X^+})$, which
is a quotient of $H^0(Y^o, R^2f^o_*\underline{\bC}_{X^o})$,
lifts into a subspace of $H^2(X^+,\underline{\bC}_{X^+})$
mapping to zero in $H^2_{\rm an}(X^+,\cO_{X^+})$.
This is provided by the Picard type hypothesis,
which ensures by (\ref{11-lift})
that $H^{1,1}(\wX)$ surjects
onto $H^0(Y^o, R^2f^o_*\underline{\bC}_{X^o})$
for a smooth projective compactification $\wX$ of 
a resolution of $X^+$.
\end{proof}

For the next proposition, we will additionally
require $X^+$ to be $\bQ$-factorial and klt, 
but this is easily furnished by passing to a 
$\bQ$-factorialization.  Here $I\!H^\ast $ 
denotes intersection cohomology.

 \begin{proposition}\label{ih-to-h} 
 Let $\pi \colon \wX^+ \to X^+$ be a resolution of a 
 $\bQ$-factorial klt variety. The following maps 
 have the same image:
 \begin{enumerate}
 \item $H^2(X^+,\underline{\bC}_{X^+})\to 
 H^2_{\rm an}(X^+,\cO_{X^+})$;
 \item $H^2(\wX^+,\underline{\bC}_{\wX^+})\to 
 H^2_{\rm an}(\wX^+,\cO_{\wX^+}) \simeq 
 H^2_{\rm an}(X^+,\cO_{X^+})$;
 \item $I\!H^2(X^+,\underline{\bC}_{X^+})
 \hookrightarrow H^2(\wX^+,\underline{\bC}_{\wX^+})
 \to H^2_{\rm an}(\wX^+,\cO_{\wX^+}) 
 \simeq H^2_{\rm an}(X^+,\cO_{X^+})$.
 \end{enumerate}
 \end{proposition}

 \begin{proof} 
 Since $X^+$ is $\bQ$-factorial with rational
 singularities, \cite[Prop.~12.1.6]{KM92} and
 \cite[Lem.~2.1]{BL2021} imply that 
 $$H^2(\wX^+,\underline{\bC}_{\wX^+})= 
 \pi^\ast H^2(X^+,\underline{\bC}_{X^+}) \oplus 
 \textrm{span}\{[E_i^+]\},$$
  where the $E_i^+$ are the $\pi$-exceptional divisors. We have
  $[E_i^+]\mapsto 0\in H^2_{\rm an}(\wX^+,\cO_{\wX^+})$ 
  from the long exact sequence of the exponential
  exact sequence.

Since $X^+$ has rational singularities, 
$H^2_{\rm an}(\wX^+,\cO_{\wX^+})\simeq 
H^2_{\rm an}(X^+,\cO_{X^+})$.
Thus, the image of 
$H^2_{\rm an}(X^+,\underline{\bC}_{X^+})
\to H^2_{\rm an}(X^+,\cO_{X^+})$
is isomorphic to the image of 
$
H^2(\wX^+,\underline{\bC}_{\wX^+})
\to H^2_{\rm an}(\wX^+,\cO_{\wX^+})$.
Moreover, we have a factorization of the pullback map
$$
H^2(X^+,\underline{\bC}_{X^+}) \to 
I\!H^2(X^+,\underline{\bC}_{X^+}) 
\hookrightarrow H^2(\wX^+, \underline{\bC}_{\wX^+}),
$$ 
and we conclude statement (3). \end{proof}

 \begin{proposition}\label{h2-surjection} 
 Let $X$ be a klt variety and let 
 $X^+\subset X$ be a big open subset. 
 Then the restriction map 
 $I\!H^2(X,\bC)\to I\!H^2(X^+,\bC)$ 
 is an isomorphism.
 \end{proposition}

 \begin{proof}
 Let $T$ be any 
 complex algebraic variety. Let 
 $i \colon U \hookrightarrow T$ be an open embedding 
 with $j \colon Z \coloneqq T \setminus U \hookrightarrow T$
 the closed embedding of the complement. 
 The distinguished triangle in the derived category 
 of bounded constructible complexes on $T$
\[
j_*j^! \cI\cC_T \to \cI\cC_T \to 
i_* \cI\cC_U \to j_*j^! \cI\cC_T[1]
\]
gives rise to the Gysin long exact sequence
\[
\ldots \to H^{k-\dim T}(Z, j^! \cI\cC_T) \to 
 I\!H^k(T, \bC) \to I\!H^k(U, \bC) \to \ldots .
 \]
By the cosupport property of the intersection complex 
$\cI\cC_T$, we obtain $H^{k-\dim T}(Z, j^! \cI\cC_T)=0$ 
for $k \leq \codim Z$ (see \cite[Lem.\ 1]{durfee}),
which entails the isomorphisms
\begin{equation}\label{GysinIH}
I\!H^k(T, \bC) \simeq I\!H^k(U, \bC) \text{ for }
k < \codim Z.
\end{equation}
 If $T$ and $Z$ are  rational homology manifolds, then we have
 \[j^! \cI\cC_T[\dim T] = j^! \underline{\bC}_T 
 \simeq \underline{\bC}_Z [2 \codim Z],\]
 so that $H^{k-\dim T}(Z, j^! \cI\cC_T)= 
 H^{k-2\codim Z}(Z,\bC)=0$ for $k < 2\codim Z$, which 
 entails the ``improved'' isomorphisms 
 \begin{equation}\label{eq:Gysinqs}
 H^k(T, \bC) \simeq H^k(U, \bC) 
 \text{ for }k < 2\codim Z-1. \end{equation}

 Let $X^{\rm qs} \subset X$ be the largest open set in $X$ 
 with only quotient singularities, which is a rational 
 homology manifold. Since $X$ has klt singularities, the 
 codimension of $X \setminus X^{\rm qs}$ is $\geq 3$, 
 so we obtain
 $I\!H^2(X, \bC) \simeq H^2(X^{\rm qs}, \bC)$ 
 by \eqref{GysinIH}.
 Up to removing an algebraic locus of codimension 
 $\geq 3$ from $X^{\rm qs}$, we can suppose that 
 the singular locus of $X^{\rm qs}$ is smooth, 
 so we obtain
$H^2(X^{\rm qs}, \bC) \simeq H^2(X^{\rm reg}, \bC)$ by \eqref{eq:Gysinqs}.

Applying the same argument to $X^{+}$, 
we conclude that
\[
I\!H^2(X^+, \bC) \simeq H^2(X^{+,\, {\rm reg}}, \bC)
\simeq H^2(X^{\rm reg}, \bC) \simeq I\!H^2(X, \bC) .
\]
\end{proof}

 \begin{definition} 
 Let $f\colon X\to Y$ be a projective morphism. 
 We say that a divisor $\Delta$ is {\it $f$-exceptional}
 if $\codim_Y f(\Supp(\Delta)) \geq 2$.
 \end{definition}

 \begin{proposition}\label{closer} 
 Let $f\colon X\to Y$ be an abelian
 fibration of Picard type, such that $X$ 
 is proper with klt singularities
 and $K_X\sim_{\bQ,f}0$. Then, $\bC N^1$ 
 is isomorphic to the image
 of
 $$
 H^2(X,\cO_X)/H^2(Y,\cO_Y)\to 
 H^2_{\rm an}(X^+,\cO_{X^+})/H^2_{\rm an}(Y^+,\cO_{Y^+}).
 $$
 \end{proposition}

 \begin{proof} 
By Lemma \ref{contract-f-exc}, there is
 a projective model of $X$ 
 with $\mathbb Q$-factorial klt singularities
 for which $f\colon X\to Y$ admits 
 no $f$-exceptional divisors. Furthermore,
 these modifications induce isomorphisms
 on $H^2(X,\cO_X)$, $H^2_{\rm an}(X^+,\cO_{X^+})$,
 and the restriction map between them.
 Hence, we may assume without loss
 of generality that $X$ is klt,
 $\bQ$-factorial, and there are no 
 $f$-exceptional divisors.
 In particular, the preimage via $f$ 
 of a big open subset of $Y$ is a big open subset of $X$.
 
 It follows from Corollary \ref{cor-n1} and Proposition
 \ref{ih-to-h} that $\bC N^1$
 is the image of the map
 \[
 I\!H^2(X^+,\underline{\bC}_{X^+})
 \to H^2_{\rm an}(X^+,\cO_{X^+})/H^2_{\rm an}(Y^+,\cO_{Y^+}).
 \]
 Since $Y^+ \subset Y$ is a big open set,
 then $X^+\subset X$ is a big open set in a 
 $\bQ$-factorial klt variety.
 By Proposition \ref{h2-surjection}, we conclude
 that $\bC N^1$ is the image of the map 
 $$
 I\!H^2(X,\underline{\bC}_{X})\to 
 H^2_{\rm an}(X^+,\cO_{X^+})/H^2_{\rm an}(Y^+,\cO_{Y^+}).
 $$
 Again by Proposition \ref{ih-to-h}, the image 
 of $I\!H^2(X,\underline{\bC}_X) \to H^2_{\rm an}(X^+,\cO_{X^+})$ 
 coincides with the image of 
 $H^2(X,\underline{\bC}_X)\to H^2_{\rm an}(X^+,\cO_{X^+})$ 
 and factors through $H^2(X,\cO_X)$.
 The proposition now follows, 
 since $H^2(X,\underline{\bC}_X)\to H^2(X,\cO_X)$ 
 is surjective as $X$ is proper with klt singularities, see 
 \cite[Thm.~7]{schwald2016}.
 \end{proof}

 \begin{proposition} \label{h2-injective}
 Under the same hypotheses as
 Proposition \ref{closer},
 $H^2(X,\cO_X)/H^2(Y,\cO_Y)\to 
 H^2_{\rm an}(X^+,\cO_{X^+})/H^2_{\rm an}(Y^+,\cO_{Y^+})$ 
 is injective. \end{proposition}

 \begin{proof} 
Passing to a resolution, we may assume that $X$ is 
smooth, projective. 
In doing so, we may lose the minimality of $X$ 
over $Y$, but the relevant cohomology groups 
remain unchanged, and therefore the conclusion of Propostion \ref{closer} remains valid.
Let $C_Y$ be a general very ample 
curve in $Y$, which is contained in $Y^+$,
and let $C_X$ be its inverse image in $X$. Then
$C_X$ is smooth and we may extend the restriction
and pullback maps to a larger diagram:

\[\xymatrix{
H^2(X,\cO_X) \ar[r] & H^2_{\rm an}(X^+,\cO_{X^+}) 
\ar[r] & H^2(C_X,\cO_{C_X})  \\
H^2(Y,\cO_Y) \ar[u] \ar[r] & 
H^2_{\rm an}(Y^+,\cO_{Y^+}) \ar[u] \ar[r] & 
H^2(C_Y,\cO_{C_Y})=0. \ar[u]
}\]

By Hodge theory, we may represent an element
$\overline{\sigma}\in H^2(X,\cO_X)$ as an
antiholomorphic $2$-form. Suppose that the 
restriction $\overline\sigma\vert_{X^+}$
represents a (Dolbeault) cohomology class
pulled back from $Y^+$. It suffices
to show that $\overline{\sigma}$ is pulled
back from $Y$.
We have
$[\overline{\sigma}\vert_{C_X}]=0$ 
in cohomology, because $H^2(C_Y,\cO_{C_Y})=0$.

Since $C_X$ is smooth projective,
the antiholomorphic $2$-form 
$\overline{\sigma}\vert_{C_X}$
represents $0$ in cohomology if and 
only if it equals $0$ as a form.
Since $C_Y$ can be made to pass through
a general tangent vector of $Y$,
we conclude that the pairing
of $\overline{\sigma}$ with any bivector of the form
$v_1\wedge v_2\in T^{0,1}_f\wedge T^{0,1}_X\subset \wedge^2T^{0,1}_X$ vanishes---here $T_f$
is the vertical tangent bundle.
Hence 
$\overline{\sigma}$ is pulled back 
from a reflexive antiholomorphic $2$-form on $Y$,
which represents a class in $H^2(Y,\cO_Y)$ by
\cite[Thm.~7]{schwald2016}.
\end{proof}

 \begin{corollary}\label{cts}
 Let $f\colon X\to Y$ be an abelian
 fibration of Picard type, such that
 $X$ is projective with klt singularities, and 
 $K_X\sim_{\bQ,f}0$. 
 A choice of polarization 
 $L$ (cf.~Prop.~\ref{identify-N1}) 
 on $X^{\rm Alb}$
 defines an isomorphism 
 $$\bC N^1
 \simeq_L H^2(X,\cO_X)/H^2(Y,\cO_Y).$$ \end{corollary}

\begin{definition}\label{brauer} For a
variety $X$, let ${\rm BR}_X \coloneqq H^2(X,\cO_X)/
{\rm im}\,H^2(X,\bZ)$. \end{definition}

Though we do not explore the connection to Brauer
groups here, it is the motivation for the notation
in Definition \ref{brauer}.

\begin{proposition}\label{brau-prop} 
Under the assumptions of Corollary
\ref{cts}, a
relative polarization $L$ of type ${\bf d}$ on 
$f^{\rm Alb}\colon X^{\rm Alb}\to Y$
defines an isogeny, of degree bounded
solely in terms of $(Y,B,\bfM, \Phi^o)$ and ${\bf d}$,
between $\bC N^1/N^1$ and ${\rm BR}_X/{\rm BR}_Y$.
Furthermore, $\rk N^1 = b_2(X)-\rho(X) - b_2(Y) + \rho(Y)$.
\end{proposition}

\begin{proof}
It follows from Corollaries \ref{identify-N1}, 
\ref{cor-n1}, and \ref{cts} that $L$ defines 
an isogeny of bounded degree between the subgroup
$N^1\subset \bC N^1$ and the image
of the map
$$H^2(X,\bZ)\to H^2(X,\cO_X)/H^2(Y,\cO_Y).$$
This map factors through $H^2(X,\bZ)/H^2(Y,\bZ)$,
so we deduce the first part.

The rank of ${\rm im}\,H^2(X,\bZ)\subset H^2(X,\cO_X)$
is the rank of the transcendental lattice 
$T_X$ of $X$---the 
smallest $\bZ$-Hodge sub-structure of $H^2(X,\bZ)$ containing 
$H^2(X,\cO_X)$. It is complementary to 
${\rm NS}(X) = H^{1,1}(X,\bZ)$
and so we deduce that $\rk {\rm im}\,H^2(X,\bZ) = 
b_2(X)-\rho(X)$.

We claim that, up to a multiple,
any element of $T_X$ projecting
into $H^2(Y,\cO_Y)$ necessarily lies
in the image of $T_Y$. This holds from
simple linear algebra because 
$H^2(Y,\cO_Y)\to H^2(X,\cO_X)$ is 
injective---$X$ and $Y$ are klt, so 
resolve them and the discriminant locus,
then apply \cite[Cor.~3.10]{Kol86}.
We deduce that
$\rk N^1 = b_2(X)-\rho(X) -b_2(Y)+ \rho(Y)$.
\end{proof}

We summarize the results of this
section in the following theorem,
which extends work
of Raynaud \cite{ray70}, 
Gross \cite{Gross97} and Dolgachev--Gross
\cite[Thm.~2.24]{Gross1994} to the general
setting of abelian fibrations, 
of arbitrary base
and fiber dimension, with multiple fibers.

\begin{theorem}\label{identify-N1-2} 
Let $f\colon X\to Y$ be an abelian
fibration of Picard type, of a projective, 
klt variety $X$, such that $K_X\sim_{\bQ,f}0$.
Suppose that $f$ induces $(Y,B,\bfM,\Phi^o)$.
Then the translation-birational class of $f$ defines
an element $t(f)\in \Sha_G$ in a group fitting into
an exact sequence $0\to \Sha\to \Sha_G\to 
\Ash^G\to 0$. The image 
$m(f)\in \Ash^G$ encodes the structure
of the multiple fibers of $f$, and this 
multiplicity group $\Ash^G$
is finite. We have $\Sha_G\simeq 
(\Sha_{G,\,{\rm an}})_{\rm tors}$. 
The group $\Sha_G$
is isogenous to the central term of an exact sequence
\begin{align}
\label{smallerest}
0\to \bQ N^1/N^1\to \Sha_G^0\to N^2\to 0\end{align}
where $N^2$ is finite.
The group $\Sha_G^0$ is the torsion subgroup
of a subgroup of the analytic Tate--Shafarevich twists $\overline{\underline{\Sha}}_{G,\,{\rm an}}^0$
fitting into the exact sequence
\begin{align*} 0\to \bC N^1/N^1
\to 
\overline{\underline{\Sha}}_{G,\,{\rm an}}^0 \to 
N^2\to 0.\end{align*}
A polarization
on $X^{\rm Alb}$ defines an isogeny from
$\bC N^1/N^1$ to ${\rm BR}_X/{\rm BR}_Y$.
Finally, all of the finite groups, indices
of subgroups, and degrees of isogenies 
in the above statement have size bounded above,
solely in terms of $(Y,B,\bfM, \Phi^o)$.
\end{theorem}

\begin{proof} 
The birational classification of $f$ in terms
of the multiplicity class 
$m(f)\in \Ash^{G,\,{\rm pre}}$ 
and twist class $t(f)\in \Sha_G$ follows from 
Proposition \ref{define-twist} and Corollary \ref{twist-cor}.
The description of the elements $m(f)\in \Ash^G\subset 
\Ash^{G,\,{\rm pre}}$ which actually correspond to
algebraic twists is Proposition \ref{alg-mult}.
The isomorphism onto the torsion subgroup
is Theorem \ref{ray-mult}. Finally, the isogeny
of $\bC N^1/N^1$ to ${\rm BR}_X/{\rm BR}_Y$ is 
Proposition \ref{brau-prop}.
\end{proof}

\begin{corollary}\label{cor-ts-twist} 
Under the same hypotheses
as Theorem \ref{identify-N1-2},
we have
$$
\Sha_G\simeq 
(\bQ/\bZ)^{b_2(X)-b_2(Y)-\rho(X)+\rho(Y)}
\oplus N
$$
for a finite abelian group $N$
fitting into the exact sequence (\ref{leftover}) below.
\end{corollary}

\begin{proof}
Given an abelian group $H$, denote by 
$H^{\rm div}\subset H$ the maximal
divisible subgroup, and let 
$H^{\rm ndiv} \coloneqq H/H^{\rm div}$
be the {\it non-divisible part} of $H$. 
Let $H_1\to H_2\to H_3\to 0$
be an exact sequence of abelian groups.
Then, if $H_3$ is finite,
$H_1^{\rm div}\twoheadrightarrow H_2^{\rm div}$ 
is a surjection. Furthermore,
if $0\to H_1\to H_2\to H_3\to 0$ is exact, with
$H_3$ finite, then we have an exact sequence
of groups $$0\to H_1^{\rm ndiv}\to 
H_2^{\rm ndiv}\to H_3\to 0.$$ 

Define $N^3 \coloneqq {\rm im}(H^1_G(Z^+,\cP)
\to H^1_G(Z^+,\mu))$ and $N^4 
\coloneqq {\rm im}(H^0_G(Z^+,\mu)
\to H^1_G(Z^+,\cP^0))$, coming from the component long exact
sequence (\ref{component-exact}). Then, we have
exact sequences \begin{gather*}
0\to \Sha^0_G/N^4\to \Sha_G\to N^3\to 0,  \\
0\to \bQ N^1/N^1\to \Sha_G^0\to N^2\to 0,
\end{gather*}
by Remark \ref{alternate-thing};
recall that 
$N^2 \coloneqq {\rm im}(\Sha_G^0\to H^2_G(Z^+,\Gamma))=H^2_G(Z^+,\Gamma)_{\rm tors}$
by Lemma \ref{identify-N2}.
Note that the latter exact sequence shows
$(\Sha_G^0)^{\rm ndiv}\simeq N^2$.

It follows that
$\bQ N^1/N^1\twoheadrightarrow (\Sha_G)^{\rm div}$
surjects onto the divisible subgroup with a finite
kernel, equal to $N^4\cap \bQ N^1/N^1$. 
The image is isomorphic to $(\bQ/\bZ)^{\rk N^1}$ and 
the rank is given in Proposition \ref{brau-prop}.
Furthermore, the cokernel $N=(\Sha_G)^{\rm ndiv}$ 
has the following structure:
\begin{align}\label{leftover} 
0\to N^4\cap \bQ N^1/N^1\to N^4\to N^2\to N\to N^3\to 0.
\end{align}
By divisibility, the exact sequence
$0\to (\Sha_G)^{\rm div}\to \Sha_G\to N\to 0$
splits (non-canonically). 
\end{proof}

\begin{question} It is unclear to us
whether the hypothesis $K_X\sim_{\bQ,f}0$ 
in Theorem \ref{identify-N1-2} is necessary.
This hypothesis was used to produce a compactification
$X^+\subset X$ with small (i.e., $\codim \geq 2$)
boundary. Thus, we pose the following 
topological question:

Let $f\colon X\to Y$ be a surjective map
of klt varieties of Picard type, and suppose that
$X^+$ is the inverse image of a big open subset
$Y^+\subset Y$.
Then, is $I\!H^2(X,\bC)\to 
I\!H^2(X^+,\bC)$ surjective? If not surjective,
are the images in $H^2(X^+,\cO_{X^+})$ the same?
If so, the small compactification
in Proposition \ref{closer} may be sidestepped,
allowing us to remove the 
hypothesis $K_X\sim_{\bQ,f}0$.
\end{question}

\section{Birational geometry of the relative Albanese variety}
\label{sec:alb}

\subsection{\texorpdfstring{$K$}{K}-triviality 
of the relative Albanese variety}\label{bir-alb}

The main result of this section is that 
the relative Albanese fibration of an 
abelian-fibered $K$-trivial variety is an 
abelian-fibered $K$-trivial variety. We begin with:

\begin{proposition}
\label{prop.map.alb}
Let $f\colon X\to Y$ be an abelian fibration
and let $f^{\rm Alb}\colon
X^{\rm Alb} \to Y$ be a model for 
the associated Albanese fibration.
There exists a commutative diagram of morphisms 
\begin{align*}
\xymatrix{
& X' \ar[dr]^{q} \ar[dl]_{p} & 
\\
X \ar[dr]_{f} \ar@{-->}[rr]^{\psi} & 
& \Xalb \ar[dl]^{f^{\rm Alb}} 
\\
& Y & 
}
\end{align*}
such that the following properties hold:
\begin{enumerate}
    \item 
$X'$ is smooth, 
    \item 
$p$ 
is birational, and
    \item 
$\psi$, $q$ are generically finite.
\end{enumerate}
\end{proposition}

We remark that in the statement of 
Proposition~\ref{prop.map.alb}, we do not need to 
impose any requirements on the Kodaira dimensions of 
$X$ and $\Xalb$, nor on their singularities.

\begin{proof}
Let $X_\eta$ and $\Xalb_\eta$ be generic fibers of $f$ and 
$f^{\mathrm{Alb}}$ over the generic point 
$\eta \coloneqq \Spec \bC(Y)$. By construction,
$X_\eta$
is a torsor over the abelian variety
$\Xalb_\eta$. Recall that its class $[X_{\eta}] \in 
H^1(\mathbb C(Y), \Xalb_{\eta})$ is torsion, 
cf.~\cite[Prop.~XIII.2.3]{ray70}
or the discussion at the bottom of 
\cite[p. 57]{Sko}. Indeed, there exists a finite field extension 
$L \supset \bC(Y)$
over which 
$(X_\eta)_{L} \coloneqq 
X_\eta \times_{\eta} {\rm Spec }\, L$
acquires a 
$L$-rational point and thus becomes 
isomorphic to 
$(\Xalb_\eta)_L$.
Hence, there exists a sufficiently large integer 
$\bar n$ such that 
$X_{\eta} \in H^1(\mathbb C(Y), \Xalb_{\eta})[\bar{n}]$.

By 
\cite[Prop.~3.3.2(i)]{Sko},
there exists a finite morphism 
$\psi_\eta \colon X_{\eta} \to \Xalb_{\eta}$
of degree 
$\overline{n}^{2\dim X_{\eta}}$.
The morphism 
$\psi$ over $\eta$ defines a rational map over 
$Y$
\begin{align*}
    \xymatrix{
X \ar[rd] \ar@{-->}[rr]^\psi&
&
\Xalb \ar[ld]
\\
& Y &    
    }
\end{align*}
which is generically finite.
Passing to a resolution of indeterminacy of 
$\psi$ 
yields a commutative diagram
satisfying properties 
(1), (2), (3).
\end{proof}

\begin{remark}
Proposition \ref{prop.map.alb} shows
that, when $X$ has canonical singularities, 
then we always have an inequality of 
Kodaira dimensions
$\kappa(X)\geq \kappa(X^{\rm Alb})$.
\end{remark}

\begin{theorem}[cf.~Theorem \ref{thm:alb}]\label{small-comp}
Let $f\colon X\to Y$ be an 
abelian fibration of a projective,
$K$-trivial variety.
Then we may choose
a birational model of $f^{\rm Alb}\colon
X^{\rm Alb} \to Y$ such that 
$X^{\rm Alb}$ is $K$-trivial, $\bQ$-factorial, 
with canonical singularities, and there are no 
$f^{\rm Alb}$-exceptional divisors.
Furthermore, $f$ and $f^{\rm Alb}$ induce the same 
generalized pair on the base $Y$.

Finally, if $X$ is primitive symplectic,
then $X^{\rm Alb}$ is primitive symplectic,
and if $X$ is Calabi--Yau then $X^{\rm Alb}$ 
is Calabi--Yau.
\end{theorem}

\begin{proof} We begin with 
the following lemma:

\begin{lemma}\label{edge-same} If
$X$ and $X^{\rm Alb}$ have rational
singularities, 
$h^k(X,\cO_X)=h^k(X^{\rm Alb},\cO_{X^{\rm Alb}})$.
\end{lemma}

Note the lemma makes no assumption
on the Kodaira dimensions $\kappa(X)$ and
$\kappa(X^{\rm Alb})$, which may well differ
(e.g.~if $X\to Y$ is an elliptic fibration
of an Enriques surface, see Remark \ref{enriques-ex}).

\begin{proof} 
The statement is unaffected by replacing
$X$ and $X^{\rm Alb}$ with resolutions that
are isomorphisms over the generic point of $Y$.
So assume, for the lemma, that $X$ and $X^{\rm Alb}$
are smooth and that we have a 
morphism $q\colon X\to X^{\rm Alb}$
as in Proposition \ref{prop.map.alb}.
By Hodge symmetry, it suffices to prove $h^{k,0}(X)=h^{k,0}(X^{\rm Alb})$.
We have maps \begin{align*} 
q^*&\colon H^0(X^{\rm Alb},\Omega^k_{X^{\rm Alb}})
\to H^0(X, \Omega^k_X) \\ 
q_* &\colon H^0(X,\Omega^k_X)\to 
H^0(X^{\rm Alb},\Omega^k_{X^{\rm Alb}}) \end{align*}
given by the pullback and the trace map,
respectively.
By push-pull, we have 
$q_*q^*=d\cdot {\rm id}$ where $d$
is the degree of the map $X\to X^{\rm Alb}$. 
It suffices to prove
that $q^*q_*\colon 
H^0(X,\Omega^k_X)\to H^0(X,\Omega^k_X)$ 
is an isomorphism.

Let $Y^o\subset Y^{\rm reg}$ be a Zariski
open subset over which $f$ and $f^{\rm Alb}$
are smooth, and let $X^o$, $(X^o)^{\rm Alb}$
be the inverse images of 
$Y^o$ in $X$, $X^{\rm Alb}$.
Fiberwise on $X^o\to Y^o$, the map $q$ is
multiplication by $n$, on a smooth
abelian $g$-fold. So in fact, 
$X^o\to (X^o)^{\rm Alb}$ is \'etale of degree $d=n^{2g}$.
Define a filtration by vector bundles
$\cG_0\subset \cdots \subset \cG_k=\Omega^k_{X^o}$ 
of the sheaf $\Omega_{X^o}^k$ as follows: 
$$
\cG_i  \coloneqq (f^o)^*\Omega_{Y^o}^{k-i}
\wedge \Omega_{X^o}^i\subset \Omega_{X^o}^k.
$$
There is an induced filtration 
$G_0\subset \cdots \subset G_k=H^0(X,\Omega^k_X)$
by declaring that $\sigma\in G_i$ if 
$\sigma\vert_{X^o}\in H^0(X^o,\cG_i)$.
This coincides with the pullback of
the Leray filtration on $H^0(X^o, \Omega^k_{X^o})$ 
with respect to $f^o\colon X^o\to Y^o$.

We may choose an analytic open cover
$\{U_\alpha\}$ of $Y^o$ such that
we have a holomorphic coordinate system
on $X_{U_\alpha}$ of the following form: Coordinates
$w=(w_1,\dots,w_r)$ on the base $U_\alpha$ and 
coordinates $z=(z_1,\dots, z_g)$ on the fibers 
which uniformize
the fibers of $X_{U_\alpha} \coloneqq 
f^{-1}(U_\alpha)\to U_\alpha$ 
in the flat metric; in other
words, we take standard coordinates on 
$\bC^g\times D^r$ and a biholomorphism 
$\bC^g\times D^r/\Lambda_{w}\to X_{U_\alpha}$
for a holomorphically varying full rank 
sublattice $\Lambda_{w}\subset \bC^g$ 
depending on $w\in D^r$.

Suppose $\sigma\in G_i$. Then in
these local coordinates, we may write
$$
\sigma\vert_{X_{U_j}} = \sum_{|I|\,\leq\, i} 
F_{IJ}\, dz_I\wedge dw_J
$$
for appropriate holomorphic coefficient
functions $F_{IJ}(z,w)$.
Observe that for all $I$ with $|I|=i$,
the coefficient
function $F_{IJ}(z,w)=f_{IJ}(w)$
depends only on the base. This is because
$\sigma\vert_{X_{U_j}}$ is
invariant under translations by 
the lattice $\Lambda_{w}$ and thus,
for $|I|=i$, the holomorphic function
$F_{IJ}(z,w)$ is periodic, hence constant,
in the $z$-variables,
see (\ref{transform}).

Locally over the base, the action of $q^*q_*$
on a form $\sigma$ 
may be computed as follows
$$
q^*q_*\sigma\vert_{X_{U_{\alpha}}}=\sum_{t\,\in\, H^0(U_{\alpha}, \,P[n])} t^*\sigma\vert_{X_{U_{\alpha}}}.
$$
because $q^*q_*$ is the pullback under pushing forward
under multiplication by $n$. 
That is,
$q^*q_*\sigma$ is the sum of the  
pullbacks of $\sigma$ along all $n$-torsion
translations.
In
the local coordinates $(z,w)$, we may write 
such a translation $t$ as
\begin{align}\begin{aligned}
\label{transform}
z_1&\mapsto z_1 +
t_1(w_1,\dots,w_r) \\ 
&\cdots \\
z_g&\mapsto z_g +
t_g(w_1,\dots,w_r) \\
w_j&\mapsto w_j\textrm{ for all }j.
\end{aligned}\end{align}
As $f_{IJ}(w)$ depends only on the base
for $|I|=i$,
it follows that 
$\sigma - d^{-1}q^*q_* \sigma \in G_{i-1}$
because the terms with $|I|=i$ cancel.
Therefore $d^{-1}q^*q_*$ induces
the identity map on ${\rm gr}_G H^0(X,\Omega^k_X)$
and hence $q^*q_*$ is an isomorphism.
The lemma follows.
\end{proof}

It follows from the proof of Lemma \ref{edge-same}
that, if $X$ is $K$-trivial, then there is a smooth
model $\wX^{\rm Alb}$ of the Albanese for which we
have a non-zero holomorphic volume form 
$\widetilde{\Omega}^{\rm Alb}$.
Thus, we have $K_{\wX^{\rm Alb}}\sim 
\Delta^+={\rm div}(\widetilde{\Omega}^{\rm Alb})$ for some
effective Cartier divisor 
$\Delta^+$ whose support does not dominate
the base (since the image is disjoint from $Y^o$). 
Let $\hX\,\!^{\rm Alb} \to Y$ be 
a relatively good minimal model of $\wX^{\rm Alb}$ over $Y$, which exists by \cite[Thm.~1.1]{HX13},
and let $\hY \to Y$ denote the corresponding canonical model.
By construction, these satisfy 
$K_{\hX\,\!^{\rm Alb}}
\sim_{\bQ,\widehat{f}\,\!^{\rm Alb}}0$, where $\widehat{f}\,\!^{\rm Alb}$ denotes the morphism 
$\widehat{f}\,\!^{\rm Alb} \colon 
\hX\,\!^{\rm Alb} \to \hY$.
So we have
$$
K_{\hX\,\!^{\rm Alb}} \sim 
{\rm div}(\widehat{\Omega}\,\!^{\rm Alb})=
(\widehat{f}\,\!^{\rm Alb})^\ast \Gamma
$$
for an effective $\bQ$-divisor $\Gamma$
on $\hY$, and a 
nonzero section $\widehat{\Omega}\,\!^{\rm Alb}\in 
H^0(\hX\,\!^{\rm Alb}, K_{\hX\,\!^{\rm Alb}})$
which pulls back to the top form
$\widetilde{\Omega}\,\!^{\rm Alb}$.
We will show that $\widehat Y = Y$ and
$\Gamma=0$.

By Proposition~\ref{prop.map.alb}, 
we have a commutative diagram of the form
\begin{align*}
\xymatrix{
& X' \ar[dr]^{q} \ar[dl]_{p} & 
\\
X \ar[dr]_{f} \ar@{-->}[rr]^{\psi} & & \widehat X\,\!^{\rm Alb} 
\ar[dl] \ar[d]^{\widehat{f}\,^{\rm Alb}} 
\\
& Y & \hY \ar[l]
}
\end{align*}
where 
$X'$ 
is smooth and projective, 
$p$ 
is birational, and 
$q$
is generically finite. 
By the hypothesis that 
$X$ 
is 
$K$-trivial 
(in particular, canonical),
$K_{X'}\sim E$ 
for a unique effective 
$p$-exceptional 
divisor
$E$. 
Since 
$q$
is generically finite and $\hX\,\!^{\rm Alb}$
is terminal, it follows from the 
Riemann--Hurwitz formula
that 
\begin{align}
    \label{eqn:jac.ideal}
\textstyle E = 
q^* {\rm div}\,
\widehat{\Omega}\,\!^{\rm Alb} +
R =
(\widehat{f}\,\!^{\rm Alb} \circ q)^\ast
\Gamma+R
\end{align}
where 
$R$ is an effective divisor.
Let us observe that since $K_X \sim 0$, 
then $p_\ast E=0$.
Now, let $U \subset Y$ be the largest open subset 
over which $\widehat Y \to Y$ is an isomorphism.
Denote the restriction of divisors to the inverse
image of $U$ by a $U$-subscript. From (\ref{eqn:jac.ideal}), 
we deduce 
\begin{equation}\label{ineq_canonical_divisors}
E_U=R_U\textrm{ and }\Gamma_U=0, \end{equation}
since over $U$, we have the commutation
$(\widehat{f}\,\!^{\rm Alb} \circ q)^\ast
\Gamma_U = (f\circ p)^*\Gamma_U$
and $p_*E=0$.

Next, we consider the canonical bundle formula for 
the sub-pairs $(X',-E)$ 
and $(X', -R)$ over $\widehat Y$.
By construction, the moduli parts agree;
we will denote them by $\bM \hY.$.
We write $B_{\hY}$  and $B^{\rm Alb}_{\hY}$ 
for the respective boundary divisors.
By assumption,
\begin{equation}\label{cbf_equiv_to_0}
\K \widehat Y. + B_\hY + \bM \widehat Y. \sim_{\mathbb Q} 0.
\end{equation}
By \eqref{ineq_canonical_divisors}, we may write 
$B^{\rm Alb}_{\hY}  = B_\hY + G$, where 
$ \supp G \subset  \widehat Y \setminus U$; 
thus,
$G$ 
is exceptional for $\widehat Y \to Y$
(but not, a priori, effective).

By the existence of $q$ and \eqref{eqn:jac.ideal}, 
we deduce 
$0 = \kappa(X') \geq
\kappa(\widehat X ^{\rm Alb})$.
On the other hand, by the existence of 
$\widehat \Omega^{\rm Alb}$, we have 
$\kappa(\widehat X ^{\rm Alb}) \geq 0$.
Thus $\kappa(\widehat X ^{\rm Alb}) =
\kappa(\K \widehat Y. + 
B^{\rm Alb}_\hY + \bM \widehat Y.) =\kappa(G)= 0$. By the minimality of $\widehat{X}^{\rm Alb}$ 
over $Y$ and the negativity lemma, we obtain $G \leq 0$
since the pushforward of 
$\K \widehat Y. + 
B^{\rm Alb}_\hY + \bM \widehat Y.$ 
on
$Y$ is trivial.
In turn, since $\kappa(G)=0$, we obtain $G=0$.
In particular, we may take $\widehat Y = Y$,
$\Gamma=0=G$,
and the canonical bundle formula 
for $X$ and $\widehat X ^{\rm Alb}$ 
induce the same generalized pair on $Y$.

\begin{lemma}\label{contract-f-exc} 
Let $\widehat{f}\colon \hX \to Y$ be a fibration
for which $\hX$ has 
canonical (resp.~klt) singularities, and $K_\hX\sim_{\bQ,\widehat{f}}0$. 
There is a birational contraction 
$\pi\colon \hX\dashrightarrow X$
over $Y$ for which $X$ still has canonical 
(resp.~klt) singularities, 
$f \colon X\to Y$
admits no $f$-exceptional divisors,
$K_{X} \sim_{\bQ,f}0$, and $X$ is $\bQ$-factorial.
\end{lemma}

\begin{proof} 
Let $\widehat{\Delta}$ be the reduced sum of the 
$\widehat{f}$-exceptional divisors.
Pass to a small $\bQ$-factorialization of $\hX$
if necessary; it does not affect either the 
hypothesis or the conclusion of the lemma.

As $\widehat{\Delta}$ is $\widehat{f}$-exceptional, 
the 
$(K_\hX+\epsilon \widehat{\Delta})$-MMP over 
$Y$, for sufficiently small $\epsilon$,
terminates with a model 
$\pi \colon \hX \dashrightarrow X$ which 
contracts the support of $\widehat{\Delta}$, see
\cite[Lem.~2.1]{Lai2011}.
To show that the singularities of $(X, 0)$
are the same as those of 
$(\hX, 0)$, 
it suffices to notice that 
$K_{\hX} \sim_\bQ \pi^\ast K_X$,
since each step of the MMP yielding $\pi$ is 
$K_{\hX}$-trivial.
\end{proof}

Applying Lemma \ref{contract-f-exc} to 
$\widehat{f}\,\!^{\rm Alb}\colon \hX\,\!^{\rm Alb}\to Y$, 
we conclude
there is a model of the Albanese fibration
$f^{\rm Alb}\colon X^{\rm Alb}\to Y$ 
with canonical total space, 
$K_{X^{\rm Alb}}\sim 0$, 
and no $f^{\rm Alb}$-exceptional divisors.
Passing to a small $\bQ$-factorialization,
we conclude the first statement of the theorem.

Lemma \ref{edge-same}
shows that $h^k(X,\cO)=h^k(X^{\rm Alb},\cO)$.
It remains to prove that when $X$ is symplectic,
so is $X^{\rm Alb}$. By Lemma \ref{edge-same},
$h^0(X,\Omega^{[2]}_X) = h^0(X^{\rm Alb},\Omega^{[2]}_{X^{\rm Alb}})$.  Let
$\sigma^{\rm Alb}\in H^0(X^{\rm Alb},
\Omega^{[2]}_{X^{\rm Alb}})$ be a generator.
From the proof of Lemma \ref{edge-same},
$\sigma^{\rm Alb}\vert_{(X^o)^{\rm Alb}}$ is symplectic, 
because it pulls back to $\sigma\vert_{X^o}$
along an \'etale cover. 
Hence,  
$(\sigma^{\rm Alb})^{[g]}\in H^0(X^{\rm Alb}, 
\Omega^{[2g]}_{X^{\rm Alb}})= 
H^0(X^{\rm Alb}, K_{X^{\rm Alb}})$ is not identically zero, and it follows from
$K_{X^{\rm Alb}}\sim 0$ that $\sigma^{\rm Alb}$
is symplectic.
\end{proof}

\begin{remark}\label{enriques-ex}
Theorem \ref{small-comp}
does not hold under the weakened hypothesis
$K_X\sim_\bQ 0$. For instance,
if $f\colon X\to \bP^1$ is an elliptic fibration
of an Enriques surface, then $2K_X\sim 0$
but $f^{\rm Alb}\colon X^{\rm Alb}\to \bP^1$ 
is a rational elliptic surface and 
$\kappa(X^{\rm Alb})=-\infty$. 

No converse statement holds either.
It is easy to produce an elliptic fibration
$f\colon X\to \bP^1$ of Kodaira dimension 
$\kappa(X)=1$, for which $f^{\rm Alb}\colon 
X^{\rm Alb}\to \bP^1$ satisfies 
$K_{X^{\rm Alb}}\sim 0$, e.g.~by
taking logarithmic transforms on an elliptic 
K3 surface.
\end{remark}

\subsection{Albanese of a Lagrangian fibration}\label{sec:ho} 
We now explicitly
analyze the structure of the Albanese fibration 
in codimension $1$ for a Lagrangian fibration. 
The results here 
exemplify Theorem \ref{thm:alb}.  In particular, in Theorem \ref{ho-thm}, we show that all the fibers over $Y^{+}$ are of Hwang--Oguiso type, which is used in \S \ref{sec:hk-twist}. 
Let $E_\triangle$ and 
$E_\square$ denote
the triangular and square elliptic 
curves
$\bC/\bZ[\zeta_k]$ for $k=3,4$, respectively. 
We cite:

\begin{theorem}[Hwang--Oguiso]
\label{ho-thm}
Let $f\colon X\to Y$ be a Lagrangian
fibration of a terminal primitive symplectic variety
of dimension $2g$. The singular
fibers $X_y\mapsto y$ over $y\in Y^+$
fall into one of the following classes. 
Here $m$ is the gcd of the multiplicities
the irreducible components of $X_y$
and $A$ is an abelian $(g-1)$-fold.
\begin{enumerate}
\item[($m=6$)] a $C_6$-quotient of an 
$E_\triangle$-bundle over $A$,
\item[($m=5$)] a $C_5$-quotient of an 
$II$-bundle over $A$,
\item[($m=4$)] a $C_4$-quotient of an 
$E_\square$- or $IV$-bundle over $A$, 
\item[($m=3$)] a $C_3$-quotient of an 
$E_\triangle$-, $III$-, or $I_0^\ast $-bundle over $A$, 
\item[($m=2$)] a $C_2$-quotient of an
$I_0$-, $I_0^\ast $-, $IV$-, or $IV^\ast $-bundle over $A$,
\item[($m=2$)] a $C_2$-quotient of twisted
$2n$-wheel over $A$, $n\in \bN$,
\item[($m=1$)] a singular Kodaira fiber 
bundle over $A$,
\item[($m=1$)] a twisted $n$-wheel
over $A$, $n\in \bN$.
\end{enumerate}
The $C_m$-quotients are by free, symplectic actions.
\end{theorem}

\begin{definition}\label{twisted-wheel}
For $n>0$, a {\it twisted $n$-wheel} 
$\bigcup_{i=0}^{n-1} \bP_i$ over an abelian $(g-1)$-fold
$A$, is a union of smooth $\bP^1$-bundles 
$\bP_i\simeq \bP_A(\cO_A \oplus \cL)\to A$,
$i=0,\dots,n-1$, for some
fixed $\cL\in {\rm Pic}^0(A)$, which results
from gluing the $\infty$-section of $\bP_i$ to the 
$0$-section of $\bP_{i+1\textrm{ mod }n}$. All gluings 
are via the identity
map on $A$, except for the gluing of the $\infty$-section
of $\bP_{n-1}$ to the $0$-section 
of $\bP_0$, which 
is via some translation $\alpha$, 
possibly non-torsion.
Such a twisted $n$-wheel does not in general
admit a morphism to $A$ (or any abelian variety), when
this translation is non-torsion.
\end{definition}

\begin{proof}
The description of the singular fibers
over a big open set $Y'\subset Y$ follows
from \cite[Thm.~1.1]{HO2011}. Thus it only
remains to check that we may take $Y'=Y^+$
where $Y^+$ is a big open set over which
we have a $G$-base change $Z^+\to Y^+$ 
\'etale-locally birational to a Kulikov model.
Let $Z'$ be the inverse image of $Y'$ in
$Z^+$. The Kulikov models of the base changed
Hwang--Oguiso models 
over $\partial^+\cap Z'$ are either 
$E$-bundles over an abelian $(g-1)$-fold
for $E$ an elliptic curve (when the $j$-invariant
of the relevant Kodaira fiber bundle is finite), 
or a twisted $n$-wheel. 
Thus, by the local triviality of the Kulikov model, the same
description is valid over an entire component
of $\partial^+$. Furthermore, the 
affine-linear $G$-action
on the Kulikov model 
(denoted 
$(\varphi_i^{-1})^*\rho_i^{\rm aff}$ in Prop.~\ref{mult-class}) 
is analytically-locally
trivial over $\partial^+$. We deduce that 
there is, \'etale-locally,
a birational model of $X^+\to Y^+$
which extends the model as in Theorem 
\ref{ho-thm} from $Y'$ to $Y^+$.
Finally, this birational model is the
unique minimal model, because every 
rational curve in it over $Y^+$
sweeps out a divisor.
\end{proof}

\begin{notation} We denote a general singular fiber
by the following notation: $R(m)$ where 
$m$ is the gcd of the multiplities and $R$ is a symbol
for some Kodaira fiber or $E_\triangle$ or $E_\square$. 
Thus, the unique
possibility for $m=6$ is $E_\triangle(6)$ and for $m=5$
is $II(5)$.  

Our notation is slightly
incompatible with \cite{HO2011}, which classifies
singular fibers by their characteristic cycle.
In general, a reduced, twisted $n$-wheel, 
which is denoted $I_n(1)$ in our notation,
rather has a characteristic cycle of type $I_{nk}$
for the
$k\in \bN\cup \{\infty\}$ which is the torsion
order of the translation $\alpha$ in Definition 
\ref{twisted-wheel}.

We call the local models of $f\colon X^+\to Y^+$ 
listed in Theorem 
\ref{ho-thm} the {\it Hwang--Oguiso models}.
\end{notation}

\begin{remark}
We believe that there is a minor error in 
\cite[Thm.~1.1]{HO2011}, and that a fiber with
characteristic cycle
$I_{\infty}$ of multiplicity $2$
is impossible.
Indeed, in \cite[Ex.~6.3, p.~16--17]{HO2011},
the action of $g$ in {\it loc.cit.}~intertwines 
the action of $2\bfZ$ in {\it loc.cit.}~(regardless,
it does not ``commute'' as stated in {\it loc.cit.})
when $\alpha$ is $2$-torsion. The same computation
applies on a twisted $2n$-wheel. Since the general
example of an $I_\infty(2)$ fiber would necessarily be a 
free $C_2$-quotient of a twisted $2n$-wheel, it appears
that the multiple fiber type with multiplicity $m=2$
and characteristic cycle $I_\infty$ does not exist.
\end{remark}

\begin{table}
\renewcommand{\arraystretch}{1.2}
\begin{tabular}{|c||c|c|c|c|c|c|c|c|c|c|c|c|}
\hline
$X$ & $E_\triangle(6)$ & $II(5)$ & $E_\square(4)$ & $IV(4)$ & 
$E_\triangle(3)$ & $III(3)$  \\
\hline
$X^{\rm Alb}$ & $II^\ast (1)$ & $II^\ast (1)$ & $III^\ast (1)$
& $II^\ast (1)$ & $IV^\ast (1)$ & $III^\ast (1)$ \\
\hline
$B^{\rm Alb}=B$ 
& $\tfrac{5}{6}$ & $\tfrac{5}{6}$ & $\tfrac{3}{4}$
 & $\tfrac{5}{6}$ & $\tfrac{2}{3}$ & $\tfrac{3}{4}$ \\
\hline
\hline
$X$ & $I_0^\ast (3)$ & $I_{2n}(2)$ & $I_0^\ast (2)$ 
& $IV(2)$ & $IV^\ast (2)$ & $R(1)$ \\ 
\hline
$X^{\rm Alb}$ & $II^\ast (1)$ & $I_n^\ast (1)$ & $III^\ast (1)$
 & $IV^\ast (1)$ & $II^\ast (1)$ & $R(1)$ \\
 \hline
$B^{\rm Alb}=B$ 
& $\tfrac{5}{6}$ & $\tfrac{1}{2}$ & $\tfrac{3}{4}$
 & $\tfrac{2}{3}$ & $\tfrac{5}{6}$ & equal \\
 \hline
\end{tabular}\vspace{5pt}
\caption{Fibers of the Albanese of a Lagrangian fibration
in codimension $1$, and the coefficient of the boundary
divisor in each case. Observe that the coefficients
are always standard when $m\geq 2$.}
\label{magic}
\end{table}

\begin{proposition}\label{magic-table} The Albanese
fibration of any Hwang--Oguiso model admits a smooth
symplectic resolution over $Y^+$,
and the resulting change of fiber
type is listed in Table \ref{magic}.
\end{proposition}

\begin{proof} The proof is an explicit computation
of a model of the Albanese, using the Hwang--Oguiso 
models. We give two examples; the remaining
cases are similar. The first statement
also follows from Theorem \ref{thm:alb} by passing
to a $\bQ$-factorial terminalization of the Albanese.\end{proof}

\begin{example} The Type $II(5)$ fiber is the free
$C_5$-quotient of a cuspidal curve bundle over
an abelian $(g-1)$-fold. For instance, consider
the degeneration $\{y^2= x^3+s\}\to D_s$ to a cuspidal
curve over a disk $D_s$ and let 
$E\times D_t\to D_t$ be a trivial
elliptic curve bundle over $D_t$. Then 
we obtain an abelian surface 
fibration $\widetilde{f}$
given by
$$W \coloneqq 
\{y^2=x^3+s\}\times E\times D_t
\xrightarrow{\widetilde{f}} D_s\times D_t \eqqcolon Z$$ 
over a bidisk, whose discriminant divisor 
is $ \{0\}\times D_t$.
Let $z$ be a flat coordinate on $E$. We define
a free $C_5$-action by 
\begin{align*}
(x,\,y,\,s,\,z,\,t)&\mapsto 
(\zeta_5^2x,\,\zeta_5^3y , \,\zeta_5s ,\, z+\alpha_5 ,\,t)
\end{align*} where 
$\alpha_5$ is a nontrivial $5$-torsion
point on $E$. This defines a representation
$\rho\colon C_5\to \Aut W$
and we define $f\colon X\to Y$ 
as the quotient $X=Z/C_5$ 
where $Y=D_{s^5}\times D_t$
is the quotient of the base. The symplectic
form $$\widetilde{\sigma} \coloneqq 
\frac{dx}{y}\wedge ds + dz\wedge dt$$ is invariant and 
descends to a symplectic form $\sigma$ on $X$.

More generally, it follows from \cite{HO2011} that
for any Lagrangian fibration with discriminant
of fiber type $II(5)$, there exists, locally,
a degree $5$ \'etale 
cover $W$ that is a smoothing of the total 
space of family over a 
$(g-1)$-dimensional base of
cuspidal curve bundles over abelian
$(g-1)$-folds.
In appropriate analytic-local coordinates,
the $C_5$-action is the same as the model above 
(up to increasing the dimension of the base
and the abelian factor).

Now we compute a model of $X^{\rm Alb}$. This is easy;
we simply define a new representation 
$$\rho^{\rm Alb}\colon C_5\to {\rm Aut}(W)$$ 
to be the same as $\rho$ 
except that $\alpha_5=0$. 
This gives a local bimeromorphic model 
$\overline{f}\,\!^{\rm Alb}\colon \oX\,\!^{\rm Alb}\to Y$ 
of the Albanese, because the resulting abelian fibration has
the same monodromy as $X$ and also has a section,
since $\rho^{\rm Alb}$ now preserves the section at infinity of $W$.
Observe that the section goes through the 
singularities created by the fixed locus.

The form $\widetilde{\sigma}$ still descends along
the action of $\rho^{\rm Alb}$ but we find that the total
space is now singular, since the $C_5$-action stabilizes
the family of cusp points $(0,0,0,z,t)$ and the family of
points at infinity $(\infty, z,t)$. These form two families
of $A_4$ surface singularities over $(z,t)$
and taking a simultaneous resolution 
$X^{\rm Alb}\to \oX\,\!^{\rm Alb}$
(which notably, is crepant), the symplectic form
lifts to a form $\sigma^{\rm Alb}$ which is non-degenerate
and the fiber type becomes the $II^\ast (1)=\wE_8$ type.
See Figure \ref{ho-resolve} for a diagram
of the exceptional divisors.

\begin{figure}
\includegraphics[width=4in]{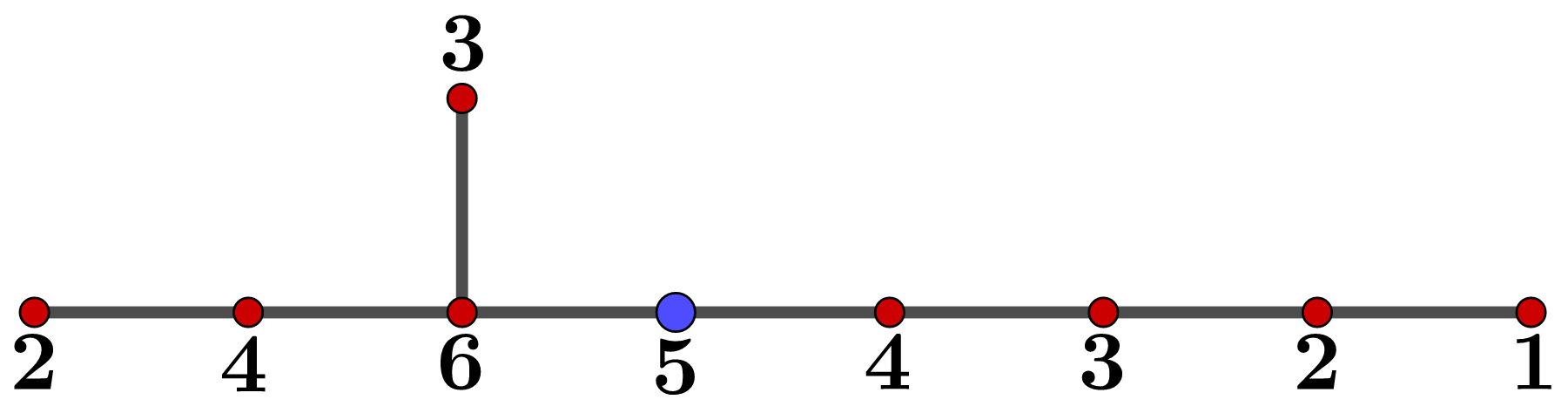}
\caption{Symplectic resolution 
$X^{\rm Alb}\to \oX\,\!^{\rm Alb}$ of the Albanese
fibration of a Type $II(5)$ Hwang--Oguiso fiber, 
resolving two $A_4$ singularities (red) on the
original multiplicity $5$ component (blue).}
\label{ho-resolve}
\end{figure}

The new fiber type can also easily be 
computed by observing
that if $M\in \Sp_2(\bZ)\subset \Sp_{2g}(\bZ)$ is the monodromy of 
the Type $II$ cuspidal fiber on $W$, then
the monodromy of the $C_5$-quotient is necessarily 
some matrix $M'$ for which $(M')^5=M$. But $M^6=1$
and so we also have $M'=M^5$.
Thus, by the classification
of Kodaira fibers via monodromy, $M'$ is necessarily
the monodromy of a Type $II^\ast (1)$ fiber.
\end{example}

\begin{example} The Type $E_\triangle(6)$ 
fiber is the free $C_6$-quotient of a triangular
elliptic curve bundle over an abelian $(g-1)$-fold. 
For example, we may take 
$$
W \coloneqq \{y^2=x^3+1\}\times D_s\times E\times D_t \to 
D_s\times D_t \eqqcolon Z
$$ 
where $E_\triangle\simeq \{y^2=x^3+1\}$
and we define the $C_6$-action as
\begin{align*}
(x,\,y,\,s,\,z,\,t)&\mapsto 
(\zeta_6^2x,\,\zeta_6^3y,\, \zeta_6s,\, z+\alpha_6,\, t)
\end{align*}
with $\alpha_6$ a nontrivial
order $6$ torsion point on $E$.
Then $f\colon X\to Y$ is $W/C_6\to Z/C_6$. 
A model of the Albanese 
$\overline{f}\,\!^{\rm Alb}\colon \oX\,\!^{\rm Alb}\to Y$ 
is defined
by modifying the $C_6$-action so that $\alpha_6=0$ and
the resulting quotient has three (families of)
singularities of types $A_1$, $A_2$, $A_5$ corresponding
to the $C_6$-orbits on $E_\Delta$ of size $3$, $2$, $1$
respectively. The crepant, symplectic resolution
is depicted in Figure \ref{ho-resolve-2}.
\end{example}

\begin{figure}
\includegraphics[width=4in]{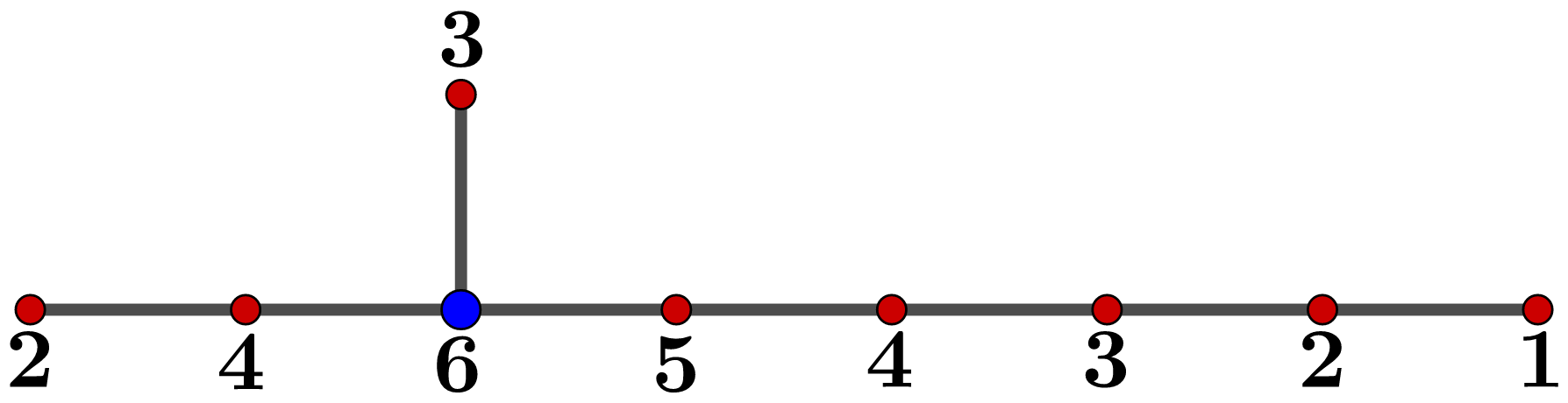}
\caption{Symplectic resolution 
$X^{\rm Alb}\to \oX\,\!^{\rm Alb}$ of the Albanese
fibration of a Type $E_\triangle(6)$ Hwang--Oguiso fiber, 
resolving $A_1$, $A_2$, $A_5$ singularities (red) on the
original multiplicity $6$ component (blue).}
\label{ho-resolve-2}
\end{figure}

Note that the Hwang--Oguiso fibers $R(1)$ 
with $m=1$ are locally isomorphic to themselves
under passing to the Albanese,
since they admit a local
section.

\section{Tate--Shafarevich twists in families}
\label{sec:ts2}

\subsection{Constructibility of the non-divisible 
part}\label{sec:constructibility}

We now discuss how to perform our
constructions in families, and to what extent
the Tate--Shafarevich groups vary in a constructible
manner. This is rather subtle; for instance
the $n$-torsion subgroup $\Sha_G[n]$ does {\it not}
vary in a constructible manner, because $\rk N^1$
is not, in general, a constructible function in families,
due to the density of Noether--Lefschetz loci.

\begin{remark}\label{constructible-remark} 
Fix a finite type family
$(\cY,\cB,\cM,\Phi^o)\to \cT$. A Kulikov model 
as in Proposition \ref{nicest-model} requires
a choice of $G$-base change $\cZ_t\to \cY_t$ for each 
$t\in \cT$ which 
\begin{enumerate}
\item is ramified to sufficient order
over the support of $\cB_t$, and
\item reduces the monodromy
into the neat subgroup $\Sp_\Lambda[3]$.
\end{enumerate}
Additionally we require a combinatorial choice of 
tilings (Construction \ref{mumford-construction})
over each component of the discriminant
to produce the relatively projective Kulikov model
$h^+_t\colon \cV_t^+\to \cZ_t^+$.

All of this data
is parameterized by a relatively finite type morphism 
$\cS\to \cT$. Thus, replacing $\cT$, 
we may as well assume (after stratifying) 
that our family of Albanese fibrations
is endowed with the data of a big open set 
$Y^+\subset Y$, fixed group $G$, 
finite $G$-base change $Z\to Y$,
and Kulikov model, which we notate
as $$\cV^+\xrightarrow{h^+} \cZ^+\subset 
\cZ\xrightarrow{/G}\cY/\cT.$$

After stratifying, the sheaves
$\mathcal{P}_t[n]$, $\mathcal{P}^0_t[n]$, 
$\mu_t$, $\Gamma_t$, etc.~are locally constant,
given restrictions of constructible
sheaves on $\cZ^+$ itself.
\end{remark}

\begin{proposition}\label{kernel-constructible}
The group $N$ of Corollary \ref{cor-ts-twist} varies
constructibly. 

Furthermore, after stratifying 
the base $\cT$ and passing to an \'etale cover,
it is possible to ``uniformly'' choose a splitting
$\Sha_{G,t} = (\Sha_{G,t})^{\rm div} + N$, 
in the following
sense: There is a fixed integer $n$, and 
a locally constant set-map
$N\to H^1_G(Z^+_t,\cP_t[n])$,
such that
the composition $$N\to H^1_G(Z^+_t,\cP_t[n])\to 
H^1_G(Z^+_t,\cP_t)=\Sha_{G,t}\to (\Sha_{G,t})^{\rm ndiv}=N$$
is the identity map ${\rm id}_N$ for all $t\in \cT$.
\end{proposition}

\begin{proof} For notational simplicity, we will
drop the subscript $t$. By (\ref{leftover}),
it suffices, for the first statement, to prove that
\begin{align*}
&N^2 \coloneqq {\rm im}(H^1_G(Z^+,\cP^0)\to 
H^1_{G,\,{\rm an}}(Z^+,\cP^0)\to H^2_G(Z^+,\Gamma)) \\
&N^3 \coloneqq {\rm im}(H^1_G(Z^+,\cP)\to 
H^1_G(Z^+,\mu)) =\coker(\Sha^0_G\to \Sha_G) \\
&N^4 \coloneqq {\rm im}(H^0_G(Z^+,\mu)\to H^1_G(Z^+,\cP^0))=
\ker(\Sha^0_G\to \Sha_G)
\end{align*}
all vary constructibly (along with the maps 
and extension data between them).
There is a sufficiently divisible integer $n\in \bN$
so that, for all $t\in \cT$, 

\begin{enumerate}
\item\label{trick2} 
$0\to \cP^0[n]\to\cP[n]\to \mu\to 0$ is exact. Furthermore,
\item\label{trick3} cohomologies of constructible sheaves
of finite groups $\cP[n]$, $\cP^0[n]$ vary constructibly,
as do maps between them, induced by
exact sequences of constructible sheaves. Finally,
\item\label{trick4} finite topological invariants, such as $H^i_G(Z^+,\Gamma)_{\rm tors}$, vary constructibly.
\end{enumerate}

Additionally, for any abelian
group $H$ with finite non-divisible part, we have

\begin{enumerate}[resume]
\item\label{trick1} 
$H^{\rm ndiv} = \coker(H\xrightarrow{\cdot n}H)$,
\end{enumerate}
and the integer $n$ may be chosen uniformly 
once $H^{\rm ndiv}$ has been shown to have
a uniformly bounded size.

{\bf Constructibility of $N^4$:}
First, we show that 
$N^4 \coloneqq  {\rm im}(H^0_G(Z^+,\mu)\to H^1_G(Z^+,\cP^0))$ 
varies constructibly. By (\ref{trick2}), we have a morphism
of short exact sequences
$$\xymatrix{ 0 \ar[r]  & \cP^0[n] \ar[r] \ar[d] & 
\cP[n] \ar[r] \ar[d] & \mu \ar[r] \ar[d] & 0 \\
0 \ar[r] & \cP^0 \ar[r] & \cP \ar[r] & \mu \ar[r] & 0}$$
once the exponent of $\mu$ divides $n$.
Taking the induced long exact sequences, 
the coboundary map
$H^0_G(Z^+,\mu)\to H^1_G(Z^+,\cP^0)$ factors as
 $$H^0_G(Z^+,\mu)\to H^1_G(Z^+,\cP^0[n])
 \to H^1_G(Z^+,\cP^0).$$
Now $N^4(n) \coloneqq {\rm im}(H^0_G(Z^+,\mu)
\to H^1_G(Z^+,\cP^0[n]))$ varies
constructibly by (\ref{trick3}). 
Next, consider the exact
sequence \begin{align}\label{trick5}
0\to \cP^0[n]\to \cP^0\xrightarrow{\cdot n}\cP^0\to 0.
\end{align}
The image $N^4(n)\twoheadrightarrow N^4$ 
is the quotient of $N^4(n)$ by the subgroup
$$
N^4(n)\cap \coker(H^0_G(Z^+,\cP^0)
\xrightarrow{\cdot n} 
H^0_G(Z^+,\cP^0))=
N^4(n)\cap H^0_G(Z^+,\cP^0)^{\rm ndiv}_{\rm tors}
$$
by (\ref{trick1}), once the exponent
of the group $H^0_G(Z^+,\cP^0)^{\rm ndiv}_{\rm tors}$
divides $n$. 
So we bound the size 
of this group uniformly over $\cT$.
By the exponential exact sequence (\ref{ses}),
meromorphic extendibility of sections
over codimension $2$, and 
GAGA, we have an exact sequence 
$$0\to H^0_G(Z^+,\Gamma)\to 
H^0_G(Z^+,\mathfrak{p})
\to H^0_G(Z^+,\cP^0)\to 
H^1_G(Z^+,\Gamma)\to 
H^1_{G,\,{\rm an}}(Z^+,\mathfrak{p}).$$
Passing to torsion, we deduce the exact sequence
$$0\to 
H^0_G(Z^+,\Gamma_\bQ)/H^0_G(Z^+,\Gamma)\to 
H^0_G(Z^+,\cP^0)_{\rm tors}\to 
H^1_G(Z^+,\Gamma)_{\rm tors}\to 0,$$
where $\Gamma_\bQ=\Gamma\otimes \bQ$.
As the first group is divisible,
we conclude that 
$H^0_G(Z^+,\cP^0)_{\rm tors}^{\rm ndiv}=
H^1_G(Z^+,\Gamma)_{\rm tors}$ which is finite
and varies
constructibly by (\ref{trick4}). 

It follows that, as subgroups of $H^1_G(Z^+,\cP^0[n])$,
the groups $H^0_G(Z^+,\cP^0)^{\rm ndiv}_{\rm tors}$
and $N^4(n)$ vary constructibly. Hence
so does their intersection, and we deduce
that the quotient $$0\to N_4(n)\cap 
H^0_G(Z^+,\cP^0)^{\rm ndiv}_{\rm tors}
\to N_4(n)\to N_4\to 0$$ also varies constructibly.
\smallskip

{\bf Reduction from $N^3$ to $N^2$:}
Consider the group
$$N^3 \coloneqq {\rm im}(H^1_G(Z^+,\cP)\to 
H^1_G(Z^+,\mu))
=\ker(H^1_G(Z^+,\mu)\to H^2_G(Z^+,\cP^0)).$$
Again by (\ref{trick2}), we have a
factorization of the map 
$$H^1_G(Z^+,\mu)\to H^2_G(Z^+,\cP^0[n])\to 
H^2_G(Z^+,\cP^0)$$ for sufficiently divisible $n$.
It follows that $N^3$ is the preimage 
in $H^1_G(Z^+,\mu)$ of the kernel $$
\ker (H^2_G(Z^+,\cP^0[n])\to 
H^2_G(Z^+,\cP^0)).$$
So by (\ref{trick3}), it suffices to prove
that this kernel varies constructibly,
as a subgroup of $H^2_G(Z^+,\cP^0[n])$. By (\ref{trick5}) and
(\ref{trick1}), we have
$$
\ker (H^2_G(Z^+,\cP^0[n])\to 
H^2_G(Z^+,\cP^0))\simeq \coker 
(H^1_G(Z^+,\cP^0)\xrightarrow{\cdot n} 
H^1_G(Z^+,\cP^0))\simeq (\Sha_G^0)^{\rm ndiv}
$$
once the exponent of 
$(\Sha_G^0)^{\rm ndiv}$ divides $n$.
Note that $(\Sha_G^0)^{\rm ndiv}\simeq N^2$,
see (\ref{smallerest}).
Thus,
we reduce the constructibility of $N^3$
to that of $N^2$. 
\smallskip

{\bf Constructibility of $N^2$:}
Since
$\Sha_G^0\to (\Sha_{G,\,{\rm an}}^0)_{\rm tors}$
is an isomorphism (Thm.~\ref{ray-mult}), we have
$N^2 = 
(\Sha_{G,\,{\rm an}}^0)_{\rm tors}^{\rm ndiv} $
which by (\ref{biggg}) coincides with 
$H_G^2(Z^+,\Gamma)_{\rm tors}$
by Lemma \ref{identify-N2}.
This group
varies constructibly by (\ref{trick4}).

\smallskip

For the second part of the proposition,
we first construct a splitting of
$$
H^1_G(Z^+,\cP)\to N^3\subset H^1_G(Z^+,\mu)$$
in families.
We begin by observing that we have a
commutative square of group schemes
\[\xymatrix{ \cP[n] \ar[r] \ar[d] & \cP\ar[d] \\ 
\cP[n]/\cP^0[n]\ar[r]^{\sim} & \cP/\cP^0=\mu}\]
so long as the exponent
of the component group scheme $\mu$ divides $n$.
Since $\cP[n]$ and $\mu$ are constructible 
sheaves on the total space of the family,
we may, after stratifying 
and passing to a finite \'etale cover of $\cT$,
produce a set-theoretic splitting of the map 
$H^1_G(Z^+,\cP[n])\to N^3(n)$,
where by definition,
$$N^3(n) \coloneqq {\rm im}(H^1_G(Z^+,\cP[n])
\to H^1_G(Z^+,\mu)).$$ Thus, we have a sequence of maps
$$N^3(n)\to H^1_G(Z^+,\cP[n])\to 
H^1_G(Z^+,\cP)\to H^1_G(Z^+,\mu)$$ which is an isomorphism
onto $N^3(n)$. We will have constructed 
a splitting of $N^3$, so long as $N^3\subset N^3(n)$.
The map 
$\Sha_G \xrightarrow{\cdot n}
\Sha_G$ factors as
$$\Sha_G \xrightarrow{\cdot n}
\Sha_G^0 \to \Sha_G$$ since
${\rm im}(\cP\xrightarrow{\cdot n} \cP)=\cP^0$.
The exponent of the group
$H^0_G(Z^+,\mu)$ divides $n$,
and thus, the kernel
of
$\Sha_G^0\to \Sha_G$ is $n$-torsion.
We conclude that we have a containment 
$$\ker(\Sha_G
\xrightarrow{\cdot n} \Sha_G^0\xrightarrow{\cdot n}\Sha_G^0)=
\ker(\Sha_G \xrightarrow{\cdot n^2} \Sha_G^0)\supset 
\ker(\Sha_G
\xrightarrow{\cdot n} \Sha_G).$$
That is, the image of $H^1_G(Z^+,\cP[n^2])$ 
in $H^1_G(Z^+,\cP)=\Sha_G$ contains $\Sha_G[n]$.
Thus $N^3(n^2)\supset N^3$, so long as
$\Sha_G[n]\to N^3$ is surjective. By the bound
on the size of $N$ (Thm.~\ref{identify-N1-2}), 
there is some uniform $n$ for which 
$N\subset \Sha_G[n]$, though note that
this group splitting is non-canonical.
For such
an $n$, we have the desired surjection
$\Sha_G[n]\twoheadrightarrow N^3$.
Let $$\alpha\colon 
N^3\to H^1_G(Z^+,\cP[n^2])\to \Sha_G$$
denote the splitting we have constructed
of $N^3$.

Since $N$ varies constructibly, we may again stratify
and pass to an \'etale cover for which we have
a coset trivialization 
$$
N = 
N^3 + (N^2/{\rm im}\,N^4),
$$ see (\ref{leftover}).
It suffices to produce a set-theoretic
splitting of the map
$$
H^1_G(Z^+,\cP^0) = \Sha_G^0\to 
(\Sha_G^0)^{\rm ndiv}=N^2,
$$
see Lemma \ref{identify-N2}.
Again,
by constructibility of $N^2$, we have 
$N^2\subset \Sha_G^0[n]$ (non-canonically). So by
(\ref{trick5}), we have a surjection 
(the same would hold with $n$ instead of $n^2$)
$$H^1_G(Z^+,\cP^0[n^2])\twoheadrightarrow 
H^1_G(Z^+,\cP^0)[n^2]\twoheadrightarrow N^2.$$ 
Passing to a stratification
and \'etale cover of $\cT$, we get a set-theoretic
splitting
of the homomorphism
$H^1_G(Z^+,\cP^0[n^2])\to N^2$, and compose
with the map to $\Sha_G$ to get a map 
$$\beta \colon N^2\to 
H^1_G(Z^+,\cP^0[n^2])\to 
H^1_G(Z^+,\cP[n^2])\to \Sha_G.$$
Choosing coset representatives of 
${\rm im}\,N^4\subset N^2$, we thus have a 
set-map 
$$\alpha+\beta\colon N\to 
H^1_G(Z^+,\cP[n^2])\to \Sha_G$$ which splits
$N=(\Sha_G)^{\rm ndiv}$ set-theoretically (and
in a locally constant manner) and factors
through $H^1_G(Z^+,\cP[n^2])$, as 
desired.
\end{proof}

\begin{remark} A priori,
there is no reason why a 
group-theoretic splitting $N\to H^1_G(Z^+,\cP[n])$
can be constructed. For instance, perhaps for all
$n$, the group $H^1_G(Z^+,\cP[n])$ contains a copy
of $\bZ/4\bZ$ which surjects onto a copy of
$\bZ/2\bZ\subset N\subset \Sha_G$. Note that the number
of group-theoretic splittings 
$\Sha_G = \Sha_G^{\rm div}\oplus N$ varies wildly in families,
as it depends on $\rk N^1$, which is not 
generally a constructible function.
\end{remark}

Similar techniques can be used
to verify that both $\Ash^G$
and $\Ash^{G,\,{\rm an}}$ vary constructibly. Thus,
the property of whether a given pre-multiplicity 
class $m(f)\in \Ash^{G,\,{\rm pre}}$ is realizable,
either algebraically or analytically, is an
essentially topological property. 

\begin{theorem}[cf.~Theorem 
\ref{thm:icy-ab}(AV)]\label{thm:bir-bdd-abelian-CY}
Abelian-fibrations of irreducible
Calabi--Yau varieties, in a fixed dimension $d$,
form a birationally bounded class.
\end{theorem}

\begin{proof} By Lemma \ref{cor:bases}
and Example \ref{rc-base-ex},
there is a finite type family 
$(\cY,\cB,\cM)/\cT$
such that the base 
$(Y,B,\Gamma)$ with $\Gamma\sim_\bQ \bfM_Y$ of any abelian
fibration $f\colon X\to Y$ of an irreducible
CY variety of dimension $d$
is isomorphic in codimension $1$
to some $(\cY_t,\cB_t+\Gamma_t)$. It follows
from Theorem \ref{bd-families}(AV) that 
there is a finite type family 
$\cX^{\rm Alb}\to \cY/\cS$ of abelian
fibrations with rational section, 
of Picard type, for which
$f^{\rm Alb}\colon X^{\rm Alb}\to Y$
is birational to
$\cX^{\rm Alb}_s\to \cY_s$ 
for some $s\in \cS$, via a birational map of bases
$\phi\colon Y\dashrightarrow \cY_t$
which is an isomorphism in 
codimension $1$. So replace
$\cT$ with this new base $\cS$.
We note that $f^{\rm Alb}$ indicates 
the same tuple $(Y,B,\bfM, \Phi^o)$
as $f$, see Theorem \ref{thm:alb}.

There is a birational model
$f'\colon X'\to \cY_t$ of 
$f\colon X\to Y$ satisfying
the hypotheses of Theorem 
\ref{identify-N1-2} (and also inducing
the same tuple);
in particular, that $K_{X'}\sim 0$,
and $X'$ is $\bQ$-factorial
with canonical singularities. 
For instance, see \cite[Prop.~2.9]{Fil24}.

We may replace $\cT$ with the open subset
of $t\in \cT$ for which $h^2(\cX^{\rm Alb}_t,
\cO_{\cX^{\rm Alb}_t})=0$. 
This necessarily holds for the Albanese
fibration of any abelian fibration of an irreducible
Calabi--Yau variety, by Theorem
\ref{thm:alb}.
Now pass to a family of Kulikov models
of the Albanese fibration as in Remark 
\ref{constructible-remark}.
We may replace $\cT$ with an 
\'etale cover of 
a stratification, which records the additional data
of a class in the finite group $\Sha_{G,t}=N$ 
(Cor.~\ref{cor-ts-twist}). This group varies constructibly
over $\cT$ by Proposition \ref{kernel-constructible}.

Now we construct relatively the universal
twist. Here, to perform the 
constructions in a relative
setting, one may use
Proposition \ref{kernel-constructible},
to lift all elements of $\Sha_{G,t}$
in a locally constant manner
to twists in $H^1_G(Z^+_t, \cP_t[n])$, i.e.,
valued in the constructible
subgroup scheme $\cP_t[n]\subset \cP_t$.
The resulting universal twist 
is (upon compactification and resolution)
a finite type family of 
projective varieties $\cX\to \cY/\cQ$.
Finally, we apply Corollary \ref{twist-cor}
to conclude that $f'\colon X'\to \cY_t$
and in turn $f\colon X\to Y$,
are birational to a fiber
of this universal twist, the one
whose twist class is $t(f)\in \Sha_{G,t}$.
\end{proof}

\subsection{Special twistor lines}\label{sec:hk-twist}

To prove Theorem \ref{thm:ps}, it remains to understand
the meaning of the group
$\bC N^1\simeq H^2(X,\cO_X)\simeq \bC$,
when $f\colon X\to Y$ is a 
Lagrangian fibration
of a primitive symplectic variety.
References for this section
are \cite{markman2014}, 
\cite[Sec.~3.4, 3.5]{AR2021}, and
\cite{abasheva}.

\begin{proposition}\label{abasheva}
Let $f\colon X\to Y$ be a 
Lagrangian fibration of a $\bQ$-factorial
terminal primitive symplectic variety.
Then, $H^1(Y, f_*T_{X/Y})\simeq 
H^1(Y,\Omega^{[1]}_Y)\simeq \bC$.
\end{proposition}

Here $T_{X/Y} \coloneqq \ker(T_X\to (f^*T_Y)_{\rm tf})$,
as in \cite[Def.~1.1.3]{abasheva}.

\begin{proof}
We copy the argument in \cite[Thm.~2.1.11]{abasheva},
where this result is proven under the hypothesis
that $X$ is smooth.
Since $T_X$ is reflexive and $(f^*T_Y)_{\rm tf}$
is torsion-free, $T_{X/Y}$ is reflexive. Since
the fibers of $f$ are equidimensional,
$f_*T_{X/Y}$ is reflexive as well, by 
\cite[Cor.~1.7]{hartshorne80}. 

Over $Y^+$, the total space $X^+$ is smooth, 
since the models in codimension $1$ 
are smooth.
Note also 
that there are no floppable curves on the 
Hwang--Oguiso models
(see~Thm.~\ref{ho-thm}), 
since every vertical rational curve deforms 
to sweep out a divisor. 
Hence the models are unique over $Y^+$.
Alternatively,
we may use that the singular locus has
codimension $\geq 4$,
and is itself Lagrangian-fibered over its image
under $f$ \cite[Thm.~3.1]{Matsushita2015}. 
Thus \cite[Thm.~2.1.11]{abasheva} directly applies
over $Y^+$,
to show that interior multiplication with the $2$-form
defines an isomorphism $$\iota\colon 
f_*T_{X/Y}\vert_{Y^+}\xrightarrow{\sim} \Omega^{1}_{Y^+}.$$ 
Since both $f_*T_{X/Y}$
and $\Omega^{[1]}_Y$ are reflexive, and $Y^+\subset Y$
is a big open set, $\iota$ 
extends to an isomorphism $f_*T_{X/Y}\to \Omega^{[1]}_{Y}$.
Proposition \ref{prop:base}(3) implies 
$H^1(Y,\Omega^{[1]}_Y)=\bC$. 
\end{proof}

The sheaf $f_*T_{X/Y}$ is the sheaf
of vertical infinitesimal automorphisms
and is the Lie algebra of the sheaf
${\rm Aut}^0_{X/Y}$, the identity
component of the sheaf of vertical
biholomorphisms.

\begin{remark} We have hitherto only used
the existence of $P$ as a sheaf of abelian groups
on $Y^+$, but here, $P$ is the sheaf of sections
of a group scheme---the N\'eron model of 
$$f^{{\rm Alb}+}\colon X^{{\rm Alb}+}\to Y^+.$$
We are unsure of a general existence result
for the N\'eron model over $Y^+$, but by 
\cite[Thm.~1.2]{yoonjoo}, the N\'eron model
does exist if we assume $f\colon X\to Y$ is a 
Lagrangian fibration, 
since then $f^{{\rm Alb}+}\colon X^{{\rm Alb}+}\to Y^+$ 
can be assumed a smooth projective 
Lagrangian fibration with every fiber
having a reduced component, by either Theorem \ref{thm:alb}, 
\cite[Thm.~3.1]{Matsushita2015},
or \S~\ref{sec:ho}.

Recall that $P(U) \coloneqq \cP(\wU)^G$ where $\wU\to U$
is the inverse image of $U$ under the quotient
map $Z^+\to Y^+$ by $G$. Since $\cP$
is the N\'eron model of the Kulikov model
$h^+\colon V^+\to Z^+$,
see \cite[Prop.~5.18]{HM18} and  
discussion in \cite[(23)]{kollar_new1},
any section of the N\'eron model over $U$
defines a unique
$G$-invariant section of $\cP(\wU)$,
and the converse is also true by the 
universal property of the N\'eron model.
Hence $P$ is the sheaf of sections
of the N\'eron model.
\end{remark}

Let $P^{00}$ denote
the identity component of the group scheme
$P$ (which
may not agree with $P^0(U) \coloneqq \cP^0(\wU)^G$ 
because the $G$-invariants of a connected group
scheme need not be connected). Let
$\mathfrak{p}^G$ denote the Lie algebra of $P^{00}$.

We have
a natural map of sheaves 
${\rm Aut}^0_{X/Y}\vert_{Y^+}\to
P^{00}$:
Any vertical
biholomorphism over $U\subset Y^+$ pulls back
to a $G$-invariant vertical biholomorphism of 
the normalized base change $g^+\colon W^+\to Z^+$
over the inverse image $\wU\to U$. In particular,
it pulls back to a translation-bimeromorphic map
in the identity component of $\cP^0(\wU)^G$.
Conversely, any element of $P^{00}(U)$ 
defines not just a bimeromorphic, 
but a biholomorphic automorphism of
$X_U\to U$, see \cite[Thm.~1.3(2)]{yoonjoo}. 
So ${\rm Aut}^0_{X/Y}\vert_{Y^+}
\to P^{00}$ 
is an
isomorphism of group schemes over $Y^+$.

\begin{proposition}\label{compactify}
Let $f\colon X\to Y$ satisfy the same
hypotheses as Proposition \ref{abasheva}. 
There is a natural
map 
$H^1(Y, f_*T_{X/Y})\to 
H^1_{\rm an}(Y^+,\mathfrak{p}^G)$
mapping 
$\bC\to \bC N^1$ 
isomorphically.
\end{proposition}

\begin{proof}

The isomorphism 
${\rm Aut}^0_{X/Y}\vert_{Y^+}
\to P^{00}$ of group schemes over $Y^+$ induces an isomorphism of Lie 
algebras $f_*T_{X/Y}\vert_{Y^+}\to 
\mathfrak{p}^G\coloneqq {\rm Lie}(P^{00})$ and in turn,
a map on cohomology $$\phi\colon H^1(Y, 
f_*T_{X/Y})
\to H^1_{\rm an}(Y^+,\mathfrak{p}^G).$$
First, we claim that $\phi$ is 
injective. The restriction map $H^1(Y, 
f_*T_{X/Y})\to H^1_{\rm an}(Y^+,f_*^+T_{X^+/Y^+})$
is injective by Proposition \ref{abasheva}, 
because the map $H^1(Y,\Omega^{[1]}_{Y})
\to H^1_{\rm an}(Y^+,\Omega^1_{Y^+})$ is---for instance 
$H^2(Y,\bC)=H^1(Y,\Omega^{[1]}_{Y})=\bC$ is generated by
an ample divisor class, which pairs nontrivially
with any smooth projective curve $C\subset Y^+$.

Let $\Xi \coloneqq \ker({\rm exp}\colon 
f_*T_{X/Y}\to {\rm Aut}^0_{X/Y})$.
We have an isomorphism of kernels 
$\Xi\vert_{Y^+}\simeq\ker({\rm exp}\colon 
\mathfrak{p}^G \to P^{00})$. Furthermore,
this isomorphism identifies 
$$\textrm{im}(H^1(Y^+,\Xi_{Y^+,\,\bC})\to 
H^1(Y^+, f^+_*T_{X^+/Y^+}))\simeq \bC N^1$$
by the definition of $\bC N^1$.
By Corollary \ref{cts}, a polarization $L$
provides us with an isomorphism
$ H^2(X,\cO_X)\simeq_L \bC N^1$, 
and these spaces can respectively be viewed as sitting
inside $$H^1_{\rm an}(Y^+,R^1f^{{\rm Alb}+}_*
\cO_{X^{{\rm Alb}+}})\simeq_L 
H^1_{G,\,{\rm an}}(Z^+, \mathfrak{p}),$$
by Corollary \ref{all-same}.
We have $H^1_{G,\,{\rm an}}(Z^+, \mathfrak{p})
\simeq H^1_{\rm an}(Y^+,\mathfrak{p}^G)$ 
by the infinitesimal version of (\ref{first-answer}).
Thus, to prove that $\phi(
H^1(Y,f_*T_{X/Y}))= \bC N^1$,
it suffices to show that the polarization $L$
{\it also} sends $H^1(Y,f_*T_{X/Y})$
into $H^2(X,\cO_X).$
This is so---the isomorphism defined by 
$L\in H^1(X,\Omega^{[1]}_X)$ is the contraction
map $c(L\otimes -)$,
$c\colon  H^1(X,\Omega^{[1]}_X)\otimes 
H^1(Y,f_*T_{X/Y})\to H^2(X,\cO)$.
\end{proof}

\begin{corollary}\label{analytic-ps-twist}
Under the hypotheses of Proposition \ref{abasheva},
there is a fibration
of proper analytic spaces $F\colon \cX\to Y\times \bC N^1$,
such that $F_s \colon \cX_s\to Y$ is an (analytic)
abelian fibration for all $s\in \bC N^1$, 
which is a twist of $F_0 = f$ over $Y^+\subset Y$, 
and whose Tate--Shafarevich
class is $t(F_s)=[s]\in \bC N^1/N^1\subset 
\Sha_{\rm an}^0$.
Furthermore:
\begin{enumerate}
\item there is an analytic
open neighborhood $N_\epsilon(0)\subset \bC N^1$
of the origin,
for which $\cX_s$ is an analytic primitive 
symplectic variety\footnote{{\it Analytic primitive 
symplectic varieties} are proper analytic spaces
satisfying Def.~\ref{def:CY}(PS), of Fujiki class $\cC$.
}, for all $s\in N_\epsilon(0)$, and
\item $F$ is relatively minimal and locally trivial.
\end{enumerate}
\end{corollary}

\begin{proof} By Proposition
\ref{compactify}, a class in $\bC N^1$,
which a priori only lives only on $Y^+$,
can be extended to an element of 
$H^1(Y,f_*T_{X/Y})$.
Analytifying and
exponentiating
to an element of $H^1_{\rm an}(Y, {\rm Aut}^0_{X/Y})$,
we see that it is possible to 
glue analytic cylinders over 
all of $Y$, in a manner which reproduces
the Tate--Shafarevich twist over the open
subset $Y^+$.

Since $F$ is constructed
by regluing cylinders, it
is a locally trivial family which is still
relatively minimal.
The second statement follows from
the fact that every small locally
trivial deformation of a primitive symplectic
variety is analytic primitive symplectic,
see \cite[Cor.~4.11]{BL2022}. \end{proof}

\begin{remark}\label{is-hol-symp}
In fact, by a similar
proof to Lemma \ref{edge-same},
see also \cite[Cor.~3.7]{abasheva},
one has that all fibers $\cX_s$ admit
a non-degenerate reflexive $2$-form 
$\sigma_s\in H^0(\cX_s,\Omega^{[2]}_{\cX_s})$. In particular,
for the $s\in \bC N^1$ such that $[s]\in \bQ N^1/N^1$,
the fiber $\cX_s$ is of Fujiki class $\cC$ (because
it is an algebraic space, see Prop.~\ref{mult-surj}), 
and is thus a primitive symplectic algebraic space.

It follows immediately, as in 
\cite[Thm.~B, Thm.~5.5]{AR2021},
that $\cX_s$ is 
analytic primitive symplectic 
for all $s\in \bC N^1$
if $H^{2,0}(X)\oplus H^{0,2}(X)$
contains no integral vectors
(a condition which 
implies that $\bC N^1/N^1$ has
trivial Hausdorff reduction).
\end{remark}

\begin{theorem}[cf.~Theorem \ref{thm:ps}] \label{theorem-ps}
The set of
$\bQ$-factorial, terminal primitive symplectic 
varieties of fixed dimension $2d$, which deform to one 
admitting a fibration, lie within a finite number 
of locally trivial deformation classes of 
(algebraic) primitive symplectic varieties.
\end{theorem}

\begin{proof}
Note that any fibration of a PS variety 
is a Lagrangian abelian fibration.

The proof is similar to that of 
Theorem \ref{thm:bir-bdd-abelian-CY}.
By Lemma \ref{cor:basesLagrbound},
there is a finite type family 
$(\cY,\cB,\cM)/\cR$
such that the base 
$(Y,B,\Gamma)$ with $\Gamma\sim_\bQ \bfM_Y$ of any 
Lagrangian fibration $f\colon X\to Y$ of a PS
variety of dimension $2d$
is isomorphic
to some $(\cY_r,\cB_r,\cM_r)$, $r\in \cR$. It follows
from Theorem \ref{bd-families}(AV) that 
there is a finite type family 
$\cX^{\rm Alb}\to \cY/\cS$ of abelian
fibrations with rational section, 
of Picard type, for which
$f^{\rm Alb}\colon X^{\rm Alb}\to Y$
is birational to
$\cX^{\rm Alb}_s\to \cY_r$ 
for some $s\in \cS$, $s\mapsto r$, 
via an isomorphism
$\phi\colon Y\to \cY_r$
on the base. 

Replacing $\cS$ with a stratification and running
a relative MMP, we may assume without loss
of generality that 
$\cX^{\rm Alb}\to \cY/\cS$ is itself
a family of PS varieties,
by Theorem \ref{thm:alb}.
Now pass to a family of Kulikov models
$$\cV^+\xrightarrow{h^+} 
\cZ^+\subset \cZ \xrightarrow{/G}\cY/\cT$$
as in Remark 
\ref{constructible-remark}.

We have that 
$\bC N^1_t={\rm im}(H^1_G(\cZ^+_t, R^1h^+_*\underline{\bC}_{\cV^+_t})\to
H^1_G(\cZ^+_t, R^1h^+_*\cO_{\cV^+_t}))$ is 1-dimensional,
and the group $\Sha_{G,\,{\rm an},\,t}$ contains
a subgroup of the form
$(\bC N^1_t/N^1_t) \oplus N_t$, with
$N_t\simeq (\Sha_{G,\,t})^{\rm ndiv}$.
We may assume
$N_t=N$ is constant on connected
components of $\cT$ by Proposition 
\ref{kernel-constructible}. By the second part 
of Proposition \ref{kernel-constructible}, we
may produce
\begin{enumerate}
\item a set-theoretic splitting $(\bC N^1_t/N^1_t) 
+ N$ defined globally on $\cT$, which
\item 
factors through a locally constant map
$\gamma\colon N\to H^1_G(\cZ_t^+,\cP_t[n])$ 
for some fixed $n$.
\end{enumerate}

Using $\gamma$,
we may construct a universal 
{\it non-divisible twist} 
$$(\cX^+)^{\rm ndiv}\to \cY^+/\cT\times N,$$
relatively over $\cT\times N$.
For a given $t\in N$, the twist by $\gamma(t)$ 
is a $G$-equivariant $n$-torsion \'etale twist,
which may be constructed relatively by 
the constructiblity of $\cP_t[n]$.
Since $(\cX^+)^{\rm ndiv}\to \cY^+/\cT\times N$ 
is (upon taking appropriate birational models)
a family of fibrations of 
quasi-projective varieties,
we may produce, after further stratification
of the base, a relative compactification\footnote{We
conjecture, but do not know, 
in general, whether
the analytic twists over $Y^+$ with class $t(f)\in 
\underline{\overline{\Sha}}_{G,\,{\rm an}}\coloneqq
\bC N^1/N^1\oplus N$,
cf.~(\ref{smaller}), 
are compactifiable analytic spaces, 
or of Fujiki
class $\cC$. We avoid this issue
by first producing the universal non-divisible
twist in the algebraic category, then applying
Corollary \ref{analytic-ps-twist}.}
$$\cX^{\rm ndiv}\to \cY/\cT\times N.$$ 

Again replacing $\cT\times N$ with a stratification
$\cQ$ of a Zariski open subset,
we may run a relative MMP and
pass to a relative $\bQ$-factorial
terminalization, so that
the fibers of
$\cX^{\rm ndiv}\to \cY/\cQ$ are Lagrangian
fibrations of $\bQ$-factorial,
terminal PS varieties.\footnote{
For the $\bQ$-factoriality, we note
that if a twist $X'$ of $X$ were not $\bQ$-factorial,
then we can untwist a small $\bQ$-factorialization
$(X')^{\rm fact}\to X'$.
The result will be a PS variety 
admitting a nontrivial crepant birational
morphism to $X$, which is impossible, by
the $\bQ$-factorial terminality of $X$.}
By Corollary \ref{analytic-ps-twist}
and the fact that the non-divisible twist
universally
represents every element of $N$, every Lagrangian
fibration $f\colon X\to Y$ of a PS variety,
is a twist of $\cX_q^{\rm ndiv}\to \cY_q$ for 
some $q\in \cQ\subset \cT\times N$,
whose twist class $t(f)\in \bQ N^1/N^1=
(\Sha_G)^{\rm div}$ lies in the divisible subgroup.
The groups $\bC N^1$ for $q\in \cQ$ form a $\bC$-bundle
$\bC \cN^1\to \cQ$, and by Corollary \ref{analytic-ps-twist},
we may construct the universal analytic twist
$$
\cX\to \cY/\bC \cN^1.
$$
The fibers
are proper, abelian-fibered,
$\bQ$-factorial (cf.~\cite[Thm.~1.1(4)]{BL2022}),
terminal analytic spaces. They are not, a priori, 
of Fujiki class $\cC$ when $\rk N^1=2$,
see Remark \ref{is-hol-symp}, but they are
for all fibers we ``care'' about, over 
the set $\bQ \cN^1\subset \bC \cN^1$.

Note that the base
$\bC \cN^1$ is a quasi-projective variety.
It follows from Corollary \ref{twist-cor} that
$f\colon X\to Y$ is birational to some fiber
of $\cX\to \cY$ over a point in $\bQ \cN^1$.
Furthermore,
birational
primitive symplectic varieties 
are deformation-equivalent. It follows
that $X$ deforms analytically, and locally 
trivially,
to some fiber of 
$\cX\to \cY/\bC \cN^1$.\footnote{
Our base
$\bC \cN^1$ has only finitely
many connected components, but the
theorem does not follow
immediately. In the theorem statement, 
we require our deformation to 
be {\it through algebraic primitive 
symplectic varieties}, but even
removing the word ``algebraic'', 
it is unclear whether all fibers 
are of Fujiki class $\cC$. 
This issue can only arise
when $H^{2,0}(X)\oplus H^{0,2}(X)$
contains an integral vector. By Appendix
\ref{sec:deformation}, when $b_2\geq 5$ 
we may throw away
the (closed, and likely empty, set of) 
non-Fujiki class fibers, and the base
$$\bC \cN^1\setminus 
(\bC \cN^1)_{\textrm{non-Fujiki}}$$ 
will still be connected. Thus, the only
special case is $b_2=4$, where
$\bC \cN^1\setminus 
(\bC \cN^1)_{\textrm{non-Fujiki}}$
is a dense open set of a genus
$1$ curve $\bC N^1/N^1$
containing all torsion
points $\bQ N^1/N^1$. It
seems preposterous that
such a dense open set would not
equal all of $\bC N^1/N^1$. But it
was not clear to us how to rule 
this out, see also 
\cite[Rem.~5.12]{AR2021}.
So we argue in a manner 
sidestepping the issue.}
We conclude
that
the function 
\begin{align*} 
\nu \colon H^2(X,\bZ)&\to \bZ \\
\gamma &\mapsto \gamma^{2d}
\end{align*}
is bounded, i.e., only finitely many such
functions occur for some $\bQ$-factorial,
terminal primitive symplectic variety $X$, 
deforming to one admitting a fibration. 
This follows
from the fact that $\nu$ is 
defined topologically and so is locally
constant in a topologically trivial deformation
(even if one deforms through analytic spaces
which are not of Fujiki class $\cC$).

Let $L$ be an ample line bundle
on $X$. The Torelli
theorem for $L$-polarized PS varieties holds,
even in the $b_2=4$ case not covered by
\cite{BL2022}---(polarized) local Torelli holds
when $b_2=4$ by
\cite[Prop.~5.5, Cor.~5.9]{BL2022}, and 
in the polarized setting,
we have surjectivity of the period mapping
by \cite[Sec.~4, Proof of Thm.~0.7]{KLSV}, 
see also Proposition \ref{extend4}.
Thus, $X$ admits a locally
trivial deformation, through primitive
symplectic varieties, to one whose period
point is an arbitrary element of the polarized
period domain
$$\bD_L\subset \bP\{x\in L^\perp\otimes \bC \,:\,
x\cdot x=0, \,x\cdot \overline{x}>0\}.$$

We require
the following lemma:

\begin{lemma}\label{lemma_smaller_norm}
Let $(M_{\rm np},\cdot)$ be an
integral lattice of signature $(3,b_2-3)$,
$b_2\geq 4$, and 
fix elements $L, L'\in M_{\rm np}$ with 
$L\cdot L>0$, $L'\cdot L'>0$. 
Let $\Gamma\subset O(M_{\rm np})$ be any finite index subgroup.
Then, there 
exists an isometry
$\gamma\in \Gamma$ for which 
$\langle L,\gamma(L')\rangle$ has hyperbolic
signature $(1,1)$.
\end{lemma}

\begin{proof}
Passing to a sublattice, we may
assume without loss of generality
that our lattice has signature $(3,1)$.
The Gram matrix of $\langle L,\gamma(L')\rangle$ 
is $$\twobytwo{L\cdot L}{L\cdot \gamma(L')}{\gamma(L')\cdot L}{L'\cdot L'}$$ 
which defines a hyperbolic lattice if and
only if its determinant
is negative.
Thus, it suffices to prove that
$\sup_{\gamma\in \Gamma}|L
\cdot \gamma(L')|=\infty$. 
Replacing
$L$, $L'$ by nonzero rescalings leaves
this statement invariant.
So we may as well assume $L$, $L'$ define
points on the hypersurface 
\begin{align}\label{hypers}
x_1^2+x_2^2+x_3^2-x_4^2=1
\end{align}
in $\bR^{3,1}$,
and say $L=(1,0,0,0)$. 
In $\bP(\bR^{3,1})$, 
the accumulation set of any $\Gamma$-orbit
contains the projectivized light cone
$\bP\{x_1^2+x_2^2+x_3^2-x_4^2=0\}$.
Thus, we may take a sequence of 
$\gamma_i \in \Gamma$ such that $\gamma_i(L')$
lies on the hypersurface (\ref{hypers})
and $\bP\gamma_i(L')\rightarrow [1:0:0:1]$
as $i\to \infty$. For such a sequence, the
$x_1$-coordinate $L\cdot \gamma(L')$ necessarily
grows arbitrarily large in absolute value.
\end{proof}

Since $\nu$ is bounded, there
exists some vector $L'\in H^2(X,\bZ)$
for which $\nu(L')$ is bounded and positive.
Note that we may restrict to the case
$b_2\geq 4$ because $X$
is Lagrangian-fibered. So by Lemma 
\ref{lemma_smaller_norm}, there exists
an element $\gamma\in O(M_{\rm np})$ for which
$\langle L, \gamma(L')\rangle$ has hyperbolic
signature. Note that $\nu(\gamma(L'))=\nu(L')$
since the BBF form and $\nu$ differ
by a fixed scalar (the Fujiki constant).

Thus, we may choose a period $x\in \bD_L$ for which
$\langle L,\gamma(L')\rangle\subset x^\perp$.
So $X$ deforms in a locally
trivial manner to a $\bQ$-factorial terminal
PS variety $X'$
admitting a line bundle $\gamma(L')$
with bounded $\nu(\gamma(L'))>0$. Applying local
 Torelli, we may take a further small deformation
 to a $\bQ$-factorial terminal
PS variety $X''$
of Picard rank $1$, generated
by $\gamma(L')$, which is necessarily projective.
Such $X''$ are bounded by
\cite[Cor.~10]{MST20}.
The theorem follows.
\end{proof}

\begin{remark} Our proof also shows that when 
$b_2(X)\geq 4$, the projective $\bQ$-factorial
terminal PS varieties, which lie in one locally
trivial analytic deformation class,
lie in finitely many locally trivial 
algebraic deformation classes (even though the 
locus of projective PS varieties in an analytic 
family may not be of finite type). 
This is visibly
false when $b_2(X)=3$---the unpolarized 
period domain is a two-sphere
$\bS^2$, which is connected, 
but the algebraic PS varieties 
inside it form a countable set.
\end{remark}

\begin{proposition}\label{triv-to-loc-triv} 
A primitive symplectic variety $X$
of fixed dimension $2d$, with $b_2(X)\geq 4$,
whose $\bQ$-factorial terminalization lies in a fixed
analytic, locally trivial deformation class,
itself lies within a finite number of analytic, 
locally trivial deformation classes. \end{proposition}

\begin{proof} 
Let 
$X^{\rm term}\to X$
be a $\bQ$-factorial terminalization.
Now \cite[Lem.~5.21]{BL2022}
implies that $$H_2(X^{\rm term},\bQ) = 
H^2(X,\bQ)\oplus N_\bQ$$ where $N$ is 
negative-definite sub-lattice for the BBF form.
For any contracted curve $C$,
the class $[C]\in 
N_1(X^{\rm term}/X)\subset H_2(X^{\rm term},\bQ)$ 
defines a nonzero 
element $[C]^*\in N_\bQ\subset H^2(X^{\rm term},\bQ)$
via the BBF form.
Thus, if $b_2(X)=b_2(X^{\rm term})=4$, then 
$X^{\rm term}=X$ and there is nothing to prove.
So assume without loss of generality that 
$b_2(X^{\rm term})\geq 5$.

In fact, 
$N\subset H^2(X^{\rm term},\bZ)_{\rm tf}$ 
is the saturated sub-lattice generated
by such classes $[C]^*$.
By \cite[7.2, 7.4, 7.7]{LMP24},
$N$ is generated by ``primitive wall
divisors'' $v_i\in N$ for which $0>v_i\cdot v_i\geq -B$ for
some fixed integer $B$ depending only on the locally
trivial deformation class of $X^{\rm term}$, which
is bounded by hypothesis.
Since $\textrm{span}\{v_i\}$ is negative-definite,
we must have $|v_i\cdot v_j|<B$,
to ensure
that the rank $2$ sublattice $\langle v_i,v_j\rangle\subset N$ 
is negative-definite. So the Gram
matrix of $\{v_i\}$ has bounded entries, and there
are only finitely many possible abstract isometry
types for $N$.

Since ${\rm Mon}\subset O(M_{\rm np},\cdot)$
has finite index \cite[Thm.~1.1(1)]{BL2022}, we conclude
that there are only finitely many ${\rm Mon}$-orbits
of lattice embeddings $N\hookrightarrow M_{\rm np}$.
Thus, up to a change of marking,
we may choose
a marking $H^2(X^{\rm term},\bZ)_{\rm tf}\to M_{\rm np}$
for which the lattice $N$ is identified with
one of finitely many fixed reference lattices. 
Now, by \cite[Prop.~2.12]{LMP24},
\cite[Prop.~5.22]{BL2022} the locally trivial analytic
deformation space of $X$ is identified locally
with the subdomain
$$\bD_N \subset 
\bP\{x\in N^\perp\otimes 
\bC\,:\,x\cdot x=0,\,x\cdot \bar x>0\},$$
which is a period domain for 
locally trivial analytic deformations
of $X^{\rm term}$ keeping $N$ of $(1,1)$-type.
By \cite[Thm.~1.1(3)]{BL2022},
the period mapping for the connected
component of the moduli space of analytic primitive
symplectic varieties, which are locally-trivially deformation
equivalent to $X$, surjects onto the complement
of a countable union of maximal
Picard rank points in $\bD_N$.

Thus, we may deform $X$ in a locally trivial manner
to a primitive symplectic variety $X_N$ with a
fixed reference period $[x_N]\in \bD_N$
of non-maximal Picard rank.
Furthermore, this deformation lifts to a locally
trivial deformation of $X^{\rm term}$ to a 
$\bQ$-factorial terminalization 
$(X_N)^{\rm term}$.
By \cite[Thm.~1.1(2)]{BL2022}, any two such
$\bQ$-factorial terminalizations are birational
because they have the same period and lie in the same
locally trivial deformation class.
So by \cite[Cor.~1.4]{LMP24}, there are only
finitely many possible isomorphism
classes of $(X_N)^{\rm term}$. In turn by
\cite[Cor.~1.3]{LMP24}, there are only finitely
many possible isomorphism classes for the
contraction $X_N$.
Since $X$ deforms to $X_N$ locally trivially,
the proposition follows.
\end{proof}

Since $b_2(X)\geq 4$, Theorem \ref{thm:ps} follows
from combining Theorem \ref{theorem-ps} and
Proposition \ref{triv-to-loc-triv}.

\begin{remark}[Special twistor lines] 
Let $f\colon X\to Y$ be a Lagrangian
fibration of a PS variety,
and let $\delta\in H^2(X,\bZ)$ be the pullback of the
class of an ample line bundle on $Y$. When 
$b_2(X)\geq 5$ and $H^{2,0}(X)\oplus H^{0,2}(X)$
contains no integral vector, the fibration 
$F\colon \cX\to Y\times \bC N^1$
of Corollary \ref{analytic-ps-twist}
is  a locally trivial 
analytic fibration of PS varieties,
over a contractible base $\bC N^1$, 
see Remark \ref{is-hol-symp}.
So there is a holomorphic period map
$$\bC N^1\to \bD_\delta \subset
\bP\{x\in 
\delta^\perp\otimes \bC\,:\,
x\cdot x=0,\,x\cdot \overline{x}>0\}.$$
The domain $\bD_\delta$ admits a fibration 
$\bD_\delta\to \bD_0$ over the Type IV 
Hermitian symmetric domain 
$$\bD_0 \subset 
\bP\{y\in \delta^\perp/\delta\otimes \bC\,:\,
y\cdot y=0,\,y\cdot \overline{y}>0\}$$ sending 
$[x]\mapsto
[x]\textrm{ mod }\bC\delta$. The fibers of 
$\bD_\delta\to \bD_0$ are copies of 
$\bC=\{[x+a\delta]\,:\,a\in \bC\}$,
called {\it special twistor lines}.
It is easy to see that $\bC N^1$
must map to a special twistor line---the
composed period mapping $\bC N^1\to \bD_\delta\to \bD_0$
is a holomorphic map from $\bC N^1\simeq \bC$ 
to a bounded domain,
and hence constant.

It is natural to expect that the map from $\bC N^1$
to a fiber of $\bD_\delta\to \bD_0$ is an isomorphism;
indeed, this is proven in \cite[Prop.~3.10]{AR2021}
when the total space is smooth, but the proof
appears to rely on the existence of a 
hyperk\"ahler metric \cite[Thm.~2.7]{SV}.
Our results suggest that
$\bC N^1$ maps isomorphically to a special twistor
line, since $\bC N^1/N^1$ and the period point
both record the bimeromorphism
class of the twist (note that 
any special twistor line is stabilized by the 
unipotent stabilizer of $\delta$ in ${\rm Mon}$, 
which under the hypothesis $b_2\geq 5$ acts 
by translation by a finite index subgroup of $N^1$
on $\bC N^1$).
\end{remark}

\section{From birational boundedness to boundedness}
\label{sec:bir-bd-to-bd}

In this section we discuss how the boundedness of fibered 
$K$-trivial varieties can be deduced from their 
birational boundedness. One important and fruitful 
approach for solving this problem is provided by the 
so-called {\it Kawamata--Morrison cone conjecture}.
In our context, the version of the conjecture that 
we use is the following:
\begin{conjecture}[Geometric Movable Cone Conjecture]
\label{conj:KM.mov}
Let $X$ be a klt variety and $f \colon X \to Y$ 
an lc-trivial fibration.
The number of orbits of the action of 
${\rm PsAut}(X/Y)$
on the set 
\begin{align}
\label{eqn:conj.KM.set.models}
\left \{
(X', \alpha)
\, 
\middle \vert
\, 
\begin{array}{l}
\text{$X'$ 
is 
klt $\mathbb Q$-factorial 
and 
$\alpha \colon X' \dashrightarrow X$
is an}
\\
\text{isomorphism in codimension 1 over Y}\,
\end{array}
\right \}
\end{align}
is finite.
\end{conjecture}

We recall that 
\begin{align*}
{\rm PsAut}(X/Y) \coloneqq 
\left \{
\phi \colon X \dashrightarrow X
\,
\middle \vert 
\,
\begin{array}{l}
\text{$\phi$ is an isomorphism in codimension}
\\ 
\text{one and it is a relative map over $Y$}
\end{array}
\right \}.
\end{align*}
Conjecture~\ref{conj:KM.mov} implies
the finiteness of the set 
\begin{align}
\label{eqn:conj.KM.set.models.no.marking}
\left \{
X'
\, 
\middle \vert
\, 
\begin{array}{l}
\text{$X'$ 
is 
klt $\mathbb Q$-factorial 
and there exists
$\alpha \colon X' \dashrightarrow X$}
\\
\text{an isomorphism in codimension 1 over }Y\,
\end{array}
\right \}
\end{align}
obtained from the set 
\eqref{eqn:conj.KM.set.models}
by dropping the marking 
$\alpha'$.
It is important to observe that, when keeping 
track of the marking 
$\alpha \colon X' \dashrightarrow X$, 
the set defined in 
\eqref{eqn:conj.KM.set.models}
can certainly be infinite.
Hence, only considering the 
orbits for the natural action of 
${\rm PsAut}(X/Y)$ 
on the effective part of the relative movable 
cone of 
$X$
over 
$Y$
can we hope to obtain finiteness.

Conjecture~\ref{conj:KM.mov} is one of the many
versions of the Kawamata--Morrison 
Cone Conjecture. For more details and versions of the conjecture,
we refer the interested reader to the survey  \cite{FS20b}.

\medskip

In order to realize our stated goal of deducing boundedness for 
classes of fibered Calabi--Yau varieties once we know that they are 
birationally bounded, we will need to use a version of the above 
conjecture that works for families of algebraic varieties.

The main result of this section is the following resolution of 
Conjecture~\ref{conj:KM.mov} under some natural cohomological 
restrictions on the total space.
In particular, such restrictions are naturally satisfied in the 
case of Calabi--Yau varieties.

\begin{theorem}[{cf.~\cite[Thm.~6.18]{FHS2024}}]
\label{thm:kawamata-morrison}
Let $\mathfrak{F}$ be a collection of fibrations 
between projective varieties $\{X_i \to Y_i\}_{i \in I}$.
Let $\mathfrak{M}$ denote the set of general fibers 
of the fibrations in $\mathfrak{F}$.
Assume that
every fibration $f\colon X \to Y$ in $\mathfrak{F}$
satisfies the following properties:
\begin{itemize}
    \item 
$X$ is a  projective terminal 
$\QQ$-factorial variety; 
    \item 
$h^1(X,\O X.)=h^2(X, \O X.)=0$; and
    \item 
$f \colon X \rar Y$ is an lc-trivial fibration.
\end{itemize}
If $\mathfrak F$ is generically bounded with 
bounded base and the 
geometric
movable
cone conjecture
holds
for the fibrations whose general fiber 
is in $\mathfrak{M}$, then $\mathfrak{F}$ is 
bounded.
\end{theorem}

Throughout the rest of this section, we shall follow 
the exposition in \cite[\S~6.5]{FHS2024}, where the 
above version of the conjecture was proven in the case 
of relative dimension 1, i.e., for elliptic fibrations.
In {\it loc.cit.}, only the case of elliptic
fibrations is treated;
on the other hand, the methods developed 
are completely general and the proofs can 
be extended to cover a much 
broader generality.
For this reason,  we limit ourselves to only 
mentioning the main differences between the two proofs.

\begin{proof}
The proof of \cite[Thm.~6.18]{FHS2024} is phrased for 
elliptic fibrations but conceptually applies verbatim 
to our setup. We limit ourselves to mentioning the main steps.

By assumption,
there exist quasi-projective varieties 
$\mathcal X$, $\mathcal Y$, $\cT$ 
and a commutative diagram
\begin{align}
\label{diag.meta.thm}
\xymatrix{
\mathcal X \ar[rr]^{\psi} \ar[dr]_\pi& & 
\mathcal Y \ar[dl]^\phi \\
& \cT &
}
\end{align}
of projective morphisms such that for any fibration 
$f \colon X \to Y$ 
in 
$\mathfrak{F}$ 
there exists a closed point 
$t \in \cT$ 
such that 
\begin{enumerate}
    \item
\label{cond:isom.km.fam.1}
there exists an isomorphism 
$h_t \colon Y \to \mathcal{Y}_t$;
    \item
\label{cond:isom.km.fam.2}
there exists an isomorphism 
$\gamma_t$ 
making the following diagram commutative
\begin{align*}
\xymatrix{
X \times_Y {\rm Spec} \ \mathbb C(Y)
\ar[rr]^{\gamma_{t}}
\ar[d]
& &
\mathcal X_t \times_{\mathcal Y_t} {\rm Spec} 
\ \mathbb C(\mathcal Y_t)
\ar[d]
\\
{\rm Spec} \ \mathbb C(Y)
\ar[rr]^{h_t}
& &
{\rm Spec} \ \mathbb C(\mathcal Y_t),
}
\end{align*}
i.e., $X$ and $\mathcal{X}_t$ are isomorphic 
over the generic point of $Y=\mathcal{Y}_t$.
\end{enumerate}

Up to passing to a locally closed stratification of $\cT$, 
which is then finite by Noetherian induction, we can  
resolve singularities of the fibers and 
run appropriate instances of the relative MMP 
(see \cite{HMX18,HX13}) and we can assume that
$\mathcal{X}$ is terminal, 
$\mathbb Q$-factorial, and 
$\mathcal{X} \to \mathcal{Y}$ is an 
lc-trivial fibration.
In particular, for every $X \to Y$, 
we may find a closed point $t \in \cT$ such that the 
condition \eqref{cond:isom.km.fam.1} above still 
holds and moreover 
\begin{enumerate}
    \item[(2')]
\label{cond:isom.km.fam.2'}
$X$ and $\mathcal{X}_t$ are isomorphic in codimension 1.
\end{enumerate} 

Up to passing to a stratification and an \'{e}tale 
base change of the original parameter space $\cT$, 
we may suppose 
that the restriction map of class groups 
$\mathrm{Cl}(\cX) \to \mathrm{Cl}(\cX_t)$ is 
surjective by \cite[Thm.~4.1]{FHS2024}.
In particular, the 
$\mathbb Q$-factoriality 
of 
$\mathcal{X}$ 
implies that of $\mathcal{X}_t$, so $X$ and 
$\mathcal{X}_t$ are connected by a sequence of 
flops over $Y=\mathcal{Y}_t$, by \cite{Kaw08}.

Now, we may conclude as in the proof of \cite[Thm.~6.18]{FHS2024}:
indeed, we can apply \cite[Thm.~4.2]{FHS2024} to lift 
the isomorphism in codimension 1 
$\mathcal{X}_t \dashrightarrow X$ 
over 
$Y \simeq \mathcal{Y}_t$ (by the identification 
\eqref{cond:isom.km.fam.1}) to an isomorphism in 
codimension 1 $\mathcal{X} \dashrightarrow \mathcal{X}'$ 
of families over 
$\mathcal{Y}$ 
such that 
$\mathcal{X}'_t$ 
is isomorphic to 
$X$.
In turn, the assumption regarding the geometric 
Kawamata--Morrison cone conjecture guarantees that 
there are only finitely many such models $\mathcal{X}'$, 
up to isomorphism over $\mathcal{Y}$, and thus we conclude.
\end{proof}

\begin{corollary}[Corollary \ref{thm:cy3}]
\label{cor.bounded.3folds.CY}
The set of 
Calabi--Yau 
3-folds admitting a fibration is bounded.
\end{corollary}

\begin{proof}
First, we settle the statement under the additional assumption that the Calabi--Yau 3-folds are terminal and $\mathbb Q$-factorial.

The case of elliptic fibrations is settled in \cite{FHS2024}.
Thus, it suffices to show that fibrations of 
relative dimension 2 satisfy the assumption of 
Theorem~\ref{thm:kawamata-morrison}.
In this case, the general fiber of such a fibration 
$X \to Y$ is either an abelian surface or a K3 surface.
Furthermore, 
$Y$ 
is 
$\mathbb P ^1$ 
which definitely forms a bounded collection.
Furthermore, Conjecture~\ref{conj:KM.mov} is known 
to hold in relative dimension 2, see \cite{Li23,MS24}.
Thus, we may conclude by Theorem \ref{thm:kawamata-morrison}, 
in view of the birational boundedness proven in Theorem 
\ref{thm:bir-bdd-abelian-CY} for the case of abelian 
fibrations, and in Theorem \ref{bd-families} for the 
case of K3 fibrations.

Now, we settle the general case.
Let $X \to Y$ be a fibered Calabi--Yau 3-fold, and let $X'$ denote a $\mathbb Q$-factorial terminalization of $X$.
Furthermore, let $D$ denote an ample Cartier divisor on $X$ and $D'$ its pullback to $X'$.
Then, by the above special case, the 
fibration $X' \to Y$ appears as a closed fiber of a 
family of fibered varieties $\cX' \to \cY \to T$, 
and say that $X \to Y$ corresponds to the fiber 
over $t_0 \in T$. 
Up to a stratification and \'etale base change 
of $T$, we may suppose that the 
divisor $D'$ is the restriction of a divisor 
$\cD'$ on $\cX'$ by \cite[Thm.~4.1]{FHS2024}.

By \cite[Thm.~4.2]{FHS2024}, $\cD'$ is relatively big and movable over $T$ and $\cY$.
Let $\cX$ denote the relatively ample model of $(\cX',\epsilon \cB')$ over $\cY$, where $0 \leq \cB' \sim_{\mathbb Q , \cY} \cD'$ and $\epsilon$ is sufficiently small.
By Kawamata--Viehweg vanishing applied to the nef and big divisor $D'$ on $X'$ and cohomology and base change, it follows that $X \to Y$ is isomorphic to the fiber of $\cX \to \cY \to T$ over $t_0$.

By \cite[Thm.~4.2]{FHS2024}, $\cB'$ is movable.
Thus, there exists a sequence of $\K \cX'.$-flops over $\cY$ that terminates with a relatively good minimal model for $(\cX',\epsilon \cB')$ over $\cY$.
In particular, $\cX' \drar \cX$ is a birational contraction that is an isomorphisms over a neighborhood of $t_0$ as $D'$ is big and nef on $X'$, 
and that corresponds to a face of the relative effective movable cone of $\cX'$ over $\cY$.
Then, by \cite[Thm.~3.4]{FHS2024} and \cite[Thm.~1.2]{MS24}, there are finitely many such models $\cX \to \cY$ up to isomorphism.
Thus, the claim follows.
\end{proof}

\begin{remark}
The same strategy as in the proof of Corollary \ref{cor.bounded.3folds.CY} shows that, for fixed a positive integer $d$, boundedness of a collection $\mathfrak C$ of $d$-dimensional  fibered terminal $\mathbb Q$-factorial Calabi--Yau varieties holds once we have the following three ingredients:
\begin{enumerate}
\item 
$\mathfrak C$ is birationally bounded;

\item 
the collection 
$\mathfrak B$ 
of the 
bases of the fibrations in 
$\mathfrak C$ 
is bounded; and
\item \label{cond.3.rem.birbound.bound} Conjecture~\ref{conj:KM.mov} holds in relative dimension at most $d-1$.
\end{enumerate}
Notice that, if we are interested in all crepant models and not only the terminal $\mathbb Q$-factorial ones, we need to require a stronger version of Conjecture~\ref{conj:KM.mov}, as used in the proof of Corollary~\ref{cor.bounded.3folds.CY}.
We observe also that it is possible to weaken condition 
\eqref{cond.3.rem.birbound.bound}
to just require that 
all generic fibers of the fibrations in 
$\mathfrak C$
satisfy the analogous version of the Kawamata--Morrison Conjecture over their field of definition, cf.
\cite{Li23, li.KM}.
\end{remark}

\begin{corollary}\label{cor:fano-base}
The set of terminal, $\QQ$-factorial irreducible
Calabi--Yau varieties of dimension $d$,
admitting an abelian fibration 
over a variety of Fano type, is bounded.
\end{corollary}

\begin{proof}
By Theorem 
\ref{thm:bir-bdd-abelian-CY}, the set of such
fibrations is generically bounded, where we can assume that the identification among the bases of the abelian fibrations and the fibers of a bounding family is an isomorphism
in codimension 
$1$.
On the other hand, since we are assuming that the base of the abelian fibrations we are considering are Fano type, then the set  of bases is bounded by 
Corollary~\ref{cor:bases} 
and 
\cite[Thm. 1.3]{HX15}.
Thus, the proof of Theorem 
\ref{thm:bir-bdd-abelian-CY} applies to obtain that the 
set of such Calabi--Yau varieties is generically bounded 
with bounded base.

Let $\{\mathcal{X}_i \to \mathcal{Y}_i \to T_i\}_{i \in I}$ denote the finitely many families thus obtained.
Then, we want to conclude by applying Theorem 
\ref{thm:kawamata-morrison}.

To this end, we observe that the needed version of the 
geometric movable cone conjecture follows by combining 
\cite[Thm.~1.4]{Li23}, \cite{Prendergast12}, and \cite{gachet}.
Indeed, by \cite{Prendergast12}, we may find a rational
polyhedral fundamental domain for the action of 
$\mathrm{PsAut}(\mathcal{X}_{i,\overline{\eta}})$
on 
$\mathrm{Mov}^e(\mathcal{X}_{i,\overline{\eta}})$, i.e.,
the effective movable cone of the geometric generic fiber.
Now, since the fundamental domain is rational polyhedral,
the divisors defining its rays are defined over a finite Galois extension of 
$\bC(\cY_i)$.
Similarly, the component group of the automorphism group of an abelian variety is finitely presented by 
\cite{borel_finite}, 
and thus its generators are defined over a finite extension of $\bC(\cY_i)$.
Thus, we may descend the conclusions of \cite{Prendergast12} to a finite Galois extension of 
$\bC(\cY_i)$.
Then, by 
\cite[Thm.~1.7 and Ex.~3.6(7)]{gachet}, 
we may descend the polyhedral fundamental domain to one for the action of 
$\mathrm{PsAut}(\mathcal{X}_{i,\eta})$ 
on 
$\mathrm{Mov}^e(\mathcal{X}_{i,\eta})$.

Since 
$\mathcal{X}_{i,\overline{\eta}}$ 
is an abelian variety, 
$\mathrm{PsAut}=\mathrm{Aut}$ 
and 
$\mathrm{Mov}^e=\mathrm{Eff}$.
We conclude by observing that the full assumption of the MMP in 
\cite[Thm.~1.4]{Li23} is not necessary:
in the proof of the cited result, the author only uses the existence of the relative MMP for
the fibration under consideration. 
Since the general fiber of $\mathcal X_i \to \mathcal Y_i$ is 
an abelian variety, 
then the conclusion follows by \cite{HX13}.
\end{proof}

\section{Fundamental group of the regular 
locus of an irreducible symplectic variety}\label{sec:pi1}
Recall that irreducible symplectic varieties admit a universal 
quasi-\'{e}tale cover by the following results due to 
Greb--Guenancia--Kebekus 
\cite[Cor.~13.3, Cor.~13.2]{GrebGuenanciaKebekus} 
and Campana \cite[Cor.~5.3]{Campana1995}.

\begin{theorem}\label{thm:algebraicfundamentalgroup} 
Any irreducible symplectic variety $X$ is simply-connected. 
Moreover, the algebraic fundamental group of its regular 
locus is finite, i.e.,
\begin{align*}|\hat{\pi}_1(X^{\mathrm{reg}})|< \infty.\end{align*}

\end{theorem} This means that, up to a quasi-\'{e}tale cover, 
we can suppose that the irreducible symplectic factor in the 
Beauville--Bogomolov decomposition is algebraically simply-connected. 
It is natural to ask whether the same result holds for 
the topological fundamental group.

\begin{conjecture}[Finiteness conjecture]\label{finitenessconj} 
The topological fundamental group of the regular locus of an 
irreducible symplectic variety $X$ is finite, i.e., we have
\begin{align*}|\pi_1(X^{\mathrm{reg}})|< \infty.\end{align*}
\end{conjecture}

The finiteness conjecture is implicit in 
\cite[\S~1.4,\S~13]{GrebGuenanciaKebekus}, and explicitly 
stated in \cite[Rem.~1.5 and Rem.~1.7]{PR2018} and in 
\cite[Conj.~3]{J2022}. When $X$ has only quotient 
singularities, e.g., if $\dim X =2$, then the finiteness 
conjecture holds. Indeed, in this case, 
$\pi_1(X^{\mathrm{reg}})$ is virtually abelian by 
\cite[Cor.~6.3]{Campana2004}, thus actually 
finite by \Cref{thm:algebraicfundamentalgroup}. 

In this section, we prove \Cref{finitenessconj}, 
conditionally to the existence of a Lagrangian 
fibration, or equivalently to the SYZ 
conjecture (\Cref{SYZ conjecture}).

\begin{theorem}[\Cref{thm:pi1}]
\label{thm:finiteness} \label{thm:Lagrangianpi1}
Let $X$ be an irreducible symplectic variety admitting a Lagrangian fibration 
$f \colon X \to Y$. 
Then the topological fundamental group of the 
regular locus of $X$ is finite.
\end{theorem}

\begin{proof}
By \Cref{lem:virtuallyabelianfinite} or 
\Cref{thm:algebraicfundamentalgroup}, it suffices to 
show that $\pi_1(X^{\mathrm{reg}})$ is virtually abelian. 
Since Lagrangian fibrations are equidimensional and the 
general fiber of $f$ is an abelian variety 
(\Cref{prop:base}), \Cref{prop:exact sequence} gives 
the exact sequence
\begin{equation}\label{eq:exactsequencefund}
\ZZ^{\dim X} \to \pi_1(X^{\mathrm{reg}}) \to 
\pi_1^{\mathrm{orb}}(Y, \Delta_f) \to 1,
\end{equation}
where $\pi_1^{\mathrm{orb}}(Y, \Delta_f)$ is the 
orbifold fundamental group of the 
multiplicity pair $(Y, \Delta_f)$, see \Cref{defn:orbifoldpi1} 
and \eqref{eq:multiplicitycouple}. 
Finally, note that $\pi_1^{\mathrm{orb}}(Y, \Delta_f)$ 
is finite by \Cref{cor:finiteorbfund}.  
\end{proof}

\begin{remark}[Calabi--Yau 3-folds]
The same proof of \Cref{thm:Lagrangianpi1} shows 
that the topological fundamental group of the 
regular locus of an irreducible Calabi--Yau 
3-fold $X$, fibred over $\PP^1$, is finite. 
When $X$ has terminal singularities, an 
alternative proof is provided in 
\cite[Thm.~6.1]{BF2024} 
(without assuming the existence of a fibration). 
\end{remark}

The fundamental group $\pi_1(X^{\mathrm{reg}})$ is 
invariant under locally trivial deformations, so the 
SYZ conjecture implies its finiteness.

 \begin{corollary}
 Let $X$ be an irreducible symplectic variety with 
 $b_2(X) \geq 5$. If Conjecture \ref{gen-abund} or 
\ref{SYZ conjecture} is true,
then the topological fundamental group of the 
regular locus of $X$ is finite.
\end{corollary}

In addition, \Cref{thm:ps} implies that the cardinality of 
$\pi_1(X^{\mathrm{reg}})$ is bounded in each dimension.

\begin{corollary}[Boundedness of fundamental groups of 
the regular locus]\label{conj:boundpi1} The set of 
fundamental groups of the regular locus of irreducible 
symplectic varieties of dimension $2n$ endowed with 
a Lagrangian fibration is finite.
\end{corollary}

In the following sections, we provide full details 
of the proof of \Cref{thm:finiteness}. 

\subsection{From finite to virtually abelian fundamental group} 

\begin{definition}
A group is \emph{virtually abelian} if it contains 
an abelian subgroup of finite index.
\end{definition}

\begin{proposition}\label{reg=totcohom}
Let $X$ be a complex
variety with rational singularities. Then, we have
\begin{align*}
H^1(X^{\mathrm{reg}}, \QQ)\simeq H^1(X, \QQ).
\end{align*}
\end{proposition}
\begin{remark}
\Cref{reg=totcohom} is proved in \cite[Thm.~6.3]{J2022} 
when $X$ has klt singularities (not necessarily 
in the quasi-projective case).
This suffices for our applications.
Here, however, we provide an alternative and more 
elementary proof and observe that the statement relies 
only on the rationality of the singularities. 
\end{remark}

\begin{proof}
The Leray spectral sequence for the inclusion 
$j \colon X^{\mathrm{reg}}
\,\hookrightarrow X$ gives the exact sequence
\begin{align*}
0 \to H^1(X, \QQ) \to H^1(X^{\mathrm{reg}}, \QQ) 
\to H^0(X,R^1j_\ast \QQ).
\end{align*}
Let $(\widetilde{X}, E = \sum_{i \in I} E_i) \to X$ be a 
projective log resolution which is an isomorphism on the 
regular locus $X^{\mathrm{reg}}$, where $E$ denotes the 
reduced exceptional divisor and each $E_i$ is a prime divisor.
The following cohomological isomorphisms hold: 
\begin{itemize}
\item $H^\ast (X^{\mathrm{reg}}, \QQ)\simeq 
H^\ast (\widetilde{X} \setminus E, \QQ)$ by construction;
\item $H^1(X, \QQ)\simeq H^1(\widetilde{X}, \QQ)$ since 
$X$ has rational singularities, see e.g.~\cite[Lem.~2.1]{BL2021}; 
\item $\bigoplus_{i \in I} H^{0}(E_i, \QQ) 
\simeq H^{2n-2}_{c}(E, \QQ)^{\vee} 
\simeq H^{2}(\widetilde{X}, \widetilde{X} \setminus E, \QQ)$. 
\end{itemize}
Then the long exact sequence of the pair 
$(\widetilde{Y}, \widetilde{X} \setminus E)$ gives
\begin{align*}
0 \to H^1(X, \QQ)\to H^1(X^{\mathrm{reg}}, \QQ)\to 
\bigoplus_{i \in I} H^{0}(E_i, \QQ) \to H^2(\widetilde{X}, \QQ).
\end{align*}
The result then follows from \Cref{lem:injectiveexc}.
\end{proof}

\begin{lemma}
\label{lem:injectiveexc}
Let $f\colon \widetilde{X} \to X$ be a projective birational morphism of complex quasi-projective varieties with reduced exceptional divisors $\{E_i\}_{i \in I}$.
Then, the Gysin morphism
\begin{align*}
\bigoplus_{i \in I} H^0(E_i, \QQ) \, \hookrightarrow H^2(\widetilde{X}, \QQ)
\end{align*}
is injective.
\end{lemma}

\begin{proof}
If $X$ is a surface, the result follows from the fact that the 
intersection form on the vector space 
$\bigoplus_{i \in I} H^0(E_i, \QQ)$ is negative definite.
We reduce to the surface case as follows. 
 
For any $i \in I$, choose a curve $C_i \subseteq E_i$ contracted 
by the morphism $f$ and not contained in any other exceptional 
divisor $E_{i'}$ with $i' \neq i$.
Let $g \colon \widetilde{X}' \to \widetilde{X}$ be the 
blow-up of $X$ along $ \bigcup_{i \in I} C_i$.
Take a surface $S$ that is a general complete 
intersection of ample divisors in $ \widetilde{X}'$, and 
let $h \colon \tilde{S} \to g(S)$ be a resolution of 
singularities.
By construction, $S$ intersects all $g$-exceptional 
divisors, $g(S)$ contains all the curves $C_i$, 
and any $h^{-1}(C_i)$
is a union of $(f \circ h)$-exceptional curves.
Denote the $(f \circ h)$-exceptional curves by $F_j$ with 
$j \in J$. Then, the pullback map 
$\bigoplus_{i \in I} H^0(E_i, \QQ) \hookrightarrow 
\bigoplus_{j \in J} H^0(F_j, \QQ)$ is injective, 
and the commutative square
\begin{align*}
\xymatrix{
\bigoplus_{i \in I} H^0(E_i, \QQ) \ar@{^{(}->}[d] 
\ar[r] & H^2(\widetilde{X}, \QQ) \ar[d]\\
\bigoplus_{j \in J} H^0(F_j, \QQ) \ar@{^{(}->}[r] & 
H^2(\widetilde{S}, \QQ)
}
\end{align*} 
gives the required injectivity. 
\end{proof}
 
\begin{lemma}\label{lem:virtuallyabelianfinite}
If $X$ is an irreducible symplectic variety such that 
$\pi_1(\Xreg)$ is virtually abelian, then 
$\pi_1(\Xreg)$ is finite.
\end{lemma}

\begin{proof}
By hypothesis, there exists a quasi-\'{e}tale cover 
$\widetilde{X} \to X$ such that 
$\pi_1(\widetilde{X}^{\mathrm{reg}})=A$ is abelian.
Since $\widetilde{X}$ has rational singularities, we 
have $A \otimes \QQ \simeq 
H^1(\widetilde{X}^{\mathrm{reg}}, \QQ)
\simeq H^1(\widetilde{X}, \QQ)$ by \Cref{reg=totcohom}.
But $H^1(\widetilde{X}, \QQ)=0$ since 
$\widetilde{X}$ is irreducible symplectic, which 
implies that $A$ is finite, since it 
is also finitely generated.
\end{proof}

\subsection{Orbifold fundamental group of the 
base of an equidimensional fibration}

\begin{definition}\label{defn:orbifoldpi1}
Let $Y$ be a complex space that is regular in codimension 
$1$, and let 
$(Y, \Delta_Y = \sum_{i}(1 - \frac{1}{m_i}) \Delta_i)$ 
be a couple\footnote{Contrary to the definition of pair, 
we do not require for a couple that $K_{Y}+\Delta_Y$ 
is $\RR$-Cartier.} with standard coefficients. 
The \emph{orbifold fundamental group} of the couple 
$(Y, \Delta)$, denoted by $\pi_1^{\mathrm{orb}}(Y, \Delta)$, 
is the quotient $\pi_1(Y^{\mathrm{reg}}\setminus 
\mathrm{Supp}(\Delta))/\langle \delta^{m_i}_i\rangle$,
where $\langle \delta^{m_i}_i\rangle$ is the normal subgroup 
generated by the loops $\delta^{m_i}_i$.
Here, $\delta_i$ is a 
simple loop winding one time counterclockwise
around $D_j$, and it is uniquely 
determined up to conjugation.
\end{definition}

\Cref{prop:exact sequence} is a special case 
of \cite[Proposition 11.7]{Campana2011}.
We recall the proof for completeness.

\begin{proposition}[Campana]\label{prop:exact sequence}
Let $f \colon X \to Y$ be an equidimensional surjective
morphism of normal complex algebraic varieties
with smooth connected general fiber $F$.
Then, there exists an exact sequence
\begin{equation}\label{eq:Campana}
  \pi_1(F) \to \pi_1(X^{\mathrm{reg}}) 
  \to \pi_1^{\mathrm{orb}}(Y, \Delta_f) \to 1,  
\end{equation}
where the multiplicity couple $(Y, \Delta_f)$
is defined in \eqref{eq:multiplicitycouple}.
\end{proposition}

\begin{proof}
Let $Y^{o}$ be the largest non-empty open set 
$U \subset Y$ such that the restriction 
$f|_{f^{-1}(U)}$ is a fiber bundle, see, e.g.,
\cite[Lem.~1.5.A]{Nori1983}.
Since the sequence \eqref{eq:Campana} is stable under 
the deletion of codimension 2 subsets, we can suppose 
$Y^{o} \subseteq Y^{\mathrm{reg}}$ and that the 
complement of $Y^{o}$ has pure codimension 1.
Write $f^{o}\colon X^{o}\coloneqq f^{-1}(Y^{o}) \to Y^{o}$, 
and set $D= \sum_{j \in J} D_j \coloneqq X^{\mathrm{reg}} 
\setminus X^{o}$.
The kernel of the restriction $\pi_1(X^{o}) 
\twoheadrightarrow \pi_1(X^{\mathrm{reg}})$ is the normal 
subgroup generated by simple loops $\gamma_j$ around 
$D_j$, see, e.g., \cite[Fact 1.2]{Nori1983}.
For any $\Delta_i  \subseteq \mathrm{Supp}(\Delta_f)$,
write $f^\ast (\Delta_i)=\sum_{j \in J} m_j D_j$, 
and $m(\Delta_i) \coloneqq \gcd_{j \in J}\{m_j\}$ 
for the multiplicity of the general fiber over $\Delta_i$.
Let $\delta_i$ be a simple loop around $\Delta_i$ in 
$\mathrm{Supp}(\Delta_f)$, and $\xi_k$ be a simple 
loop around the divisor $\Delta_k$ in 
$f(D) \setminus \mathrm{Supp}(\Delta_f)$.
Note that $\pi_1(Y^{o})/\langle \xi_k\rangle \simeq 
\pi_1(Y^{\mathrm{reg}}\setminus \mathrm{Supp}(\Delta_f))$.
By the definition of orbifold fundamental group, 
we have the following commutative diagram
\begin{align*}
\xymatrix{
& \langle \gamma_j \rangle \ar@{->>}[r] \ar@{^{(}->}[d] & 
\langle f(\gamma_j)=
\delta_j^{m_j} \rangle = 
\langle \delta_i^{m(\Delta_i)}, \xi_k \rangle \ar@{^{(}->}[d] \\
\pi_1(F) \ar[r] \ar[d]_-{\simeq} &  
\pi_1(X^{o})\ar@{->>}[r] \ar@{->>}[d] & \pi(Y^{o})
 \ar@{->>}[d]\\
\pi_1(F) \ar[r] &  \pi_1(X^{\mathrm{reg}})\ar@{->>}[r] & 
\pi_1^{\mathrm{orb}}(Y, \Delta_f),\\
}
\end{align*}
where the second row is the standard long exact 
sequence in homotopy for fiber bundles. 
In particular, the third row is exact.
\end{proof}

\Cref{prop:basepair} is a refinement of 
\Cref{prop:base}.\eqref{item:base}, i.e., the base of a 
Lagrangian fibration is a Fano variety with 
$\QQ$-factorial klt singularities. 

\begin{proposition}\label{prop:basepair}
Let $X$ be a projective irreducible symplectic variety 
endowed with a Lagrangian fibration $f \colon X \to Y$.
 The pair $(Y, \Delta_f)$ is of Fano type, i.e., it is
 klt and $-(K_{X}+\Delta_{f})$ is ample.
\end{proposition}

\begin{proof}
By the canonical bundle formula, there exist an 
effective $\QQ$-divisor divisor $B_Y \geq \Delta_{f}$ 
and a $\QQ$-effective linear series $\mathbf{M}_Y$ 
such that
\begin{align*}
0 \sim K_{X} \sim f^\ast (K_Y + B_Y + \mathbf{M}_Y)
\end{align*}
and the $\QQ$-factorial pair $(Y, B_Y)$ has klt 
singularities, see \eqref{eq:multiplicitycouple}, 
\Cref{base-same-sing} and \Cref{prop:base}. 
Since $\rho(Y)=1$ by \Cref{prop:base}, 
$-(K_Y + \Delta_f)$ is either ample or 
$\QQ$-linearly trivial.

We show that the latter case cannot occur.
Otherwise, let $m$ be the smallest positive integer 
such that $m\Delta_f$ is a $\ZZ$-divisor and 
$\omega^{[m]}_Y(m\Delta_f)$ is a trivial line bundle.
An $m$-pluricanonical form 
$s \in H^0(Y, \omega^{[m]}_Y(m\Delta_f))$ with 
pole divisor $m\Delta_f$
determines a degree $m$ cover $p \colon \widetilde{Y} \to Y$, 
called index 1 cover, see \cite[Def.~2.49]{Kollar2013}.
At the generic point of $\Delta_i$, $p$ factors as 
$p=p_{\mathrm{ram}} \circ p_{\text{\'{e}t}}$, where 
$p_{\mathrm{ram}}$ is a finite morphism totally ramified 
along $\Delta_i$ of degree $m(\Delta_i)$, and 
$p_{\text{\'{e}t}}$ is \'{e}tale.
Thus, it follows 
$K_{\widetilde{Y}} \sim p^\ast (K_Y + \Delta_f) \sim 0$ by 
the Riemann--Hurwitz formula, see 
\cite[\S~2.48 and (2.41.6)]{Kollar2013}.
For this step, we crucially use that $\Delta_f$ 
has standard coefficients.

Let $\widetilde{X}$ be the normalization of the 
(irreducible) fiber product $X \times_{Y} \widetilde{Y}$. 
Write $f^{*}(\Delta_i)=\sum_j m_j D_j$.
At a general point of $D_j$, the equidimensional 
morphism $f$ is given by
$$
(x_1, x_2, \ldots, x_{2n}) 
\mapsto (x^{m_j}_1, x_2, \ldots, x_{n})$$
in suitable local coordinates.
Hence, $\widetilde{X}$ is \'{e}tale locally 
isomorphic over $D_i$ to the normalization of 
$y^{m(\Delta_i)}=x_1^{m_j}$ in 
$\mathbb{A}^{2n+1}_{x_1, \ldots, x_{2n}, y}$, 
where $m(\Delta_i) = \gcd \{m_j : f(D_j)=\Delta_i\}$.
We conclude that the map $\widetilde{X} \to X$ is \'{e}tale 
at the generic point of $D_j$, and so $\widetilde{X} \to X$ 
is quasi-\'{e}tale. Now, since $\widetilde{X} \to X$ is 
quasi-\'{e}tale and $X$ is an irreducible symplectic 
variety, then $\widetilde{X}$ is an  irreducible symplectic 
variety as well.
However, by \Cref{prop:base}, the base of the fibration 
$\widetilde{X} \to \widetilde{Y}$ is of Fano type, 
which is a contradiction, since 
$K_{\widetilde{Y}} \sim 0$ 
by construction. 
\end{proof}

\begin{corollary}\label{cor:finiteorbfund}
Let $X$ be a projective irreducible symplectic 
variety endowed with a Lagrangian fibration $f \colon X \to Y$.
Then, the orbifold fundamental group 
$\pi^{\mathrm{orb}}_1(Y, \Delta_f)$ is finite.
\end{corollary}

\begin{proof}
It follows from \Cref{prop:basepair} and 
\cite[Thm.~2]{Braun2021}.
\end{proof}

\begin{example}
The irreducibility in the hypothesis of \Cref{prop:basepair} 
cannot be dropped.
The singular Kummer variety $X = A/C_2$ is a quartic 
surface in $\PP^3$ with 16 nodes obtained as the 
quotient of an abelian surface $A$ by the inverse 
of the group structure. 
Suppose that $A \simeq E \times E'$ is the 
product of two elliptic curves.
Then, the fibration 
$f \colon X = E \times E'/C_2 \to E/C_2 \simeq \mathbb{P}^1$ 
is Lagrangian, but the multiplicity pair 
$(\PP^1, \Delta_f = \sum^4_{i=1} \frac{1}{2}p_i)$, 
where $p_i$ is the branch locus of the 
quotient $E \to \PP^1$, is not of Fano type.
The index 1 cover of $(\PP^1, \Delta_f)$ is the 
double cover $p \colon E \to \PP^1$, and the 
normalized fiber product in the proof of 
\ref{prop:basepair} is 
\begin{align*}
\xymatrix{
\widetilde{X} = E \times E' \ar[d] \ar[r] & X = 
E \times E'/C_2 \ar[d]\\
E \ar[r] & \mathbb{P}^1= E/C_2.
}
\end{align*}
Note that $\pi_1^{\mathrm{orb}}(\PP^1, \Delta_f)$ 
and $\pi_1(X^{\mathrm{reg}})$ are infinite groups, 
as they are extensions of $C_2$ by 
$\ZZ^2 = \pi_1(E)$ and 
$\ZZ^4=\pi_1(E \times E')$ respectively.
\end{example}

\appendix
\section{Deformation of Lagrangian fibrations on hyperk\"{a}hler 
varieties}\label{sec:deformation}

Let $X$ be a primitive symplectic variety endowed with a 
Lagrangian fibration $f \colon X \to Y$. Let $L$ be the 
pullback of an ample line bundle on $Y$.
Let $\pi \colon (\mathcal{X}^{\mathrm{uni}}, \mathcal{L}) 
\to \Deflt(X, L)$ be the universal family of locally 
trivial deformations of the pair $(X, L)$; 
see also \cite[Cor.~4.11, Lem.~4.13]{BL2022}.
The germ $\Deflt(X, L)$ is smooth of dimension 
$h^{1,1}(X)-1$, and the period map
\begin{equation}\label{eq:perioddeflt}
\Deflt(X, L) \to \PP(L^{\perp}_{\CC}) \cap \bD
\end{equation}
is a local isomorphism, see 
\cite[Thm.~4.7, Cor.~5.9]{BL2022}.

We show that the Lagrangian fibration $f$ 
is unobstructed and deforms to the whole $\Deflt(X, L)$.
In the smooth case, this is the content of 
\cite{Matsushita2016}.
In this appendix, we adapt the proof to the singular setting. 

\begin{theorem}\label{thm:basechangelocallyfree}
$R^k \pi_* \mathcal{L}$ is locally free. 
\end{theorem}

\begin{theorem}\label{thm:semiampleLagr}
The line bundle $\mathcal{L}$ is $\pi$-semiample,
i.e., there is a morphism of complex spaces 
$\pi_U \colon 
(\mathcal{X}^{\mathrm{uni}}_U, \mathcal{L}_U) \to U$ 
representing $\pi$, and 
a morphism $F \colon \mathcal{X}^{\mathrm{uni}}_{U} \to 
\mathbb{P}(\pi_* \mathcal{L}_U^{\otimes k})$ over $U$ such that: 
\begin{enumerate}
\item $F|_{\mathcal{X}^{\mathrm{uni}}_{t}}$ is a 
Lagrangian fibration for all $t \in U$, and
\item $F|_{\mathcal{X}^{\mathrm{uni}}_{0}}$ 
coincides with the original fibration $f$.
\end{enumerate}
\end{theorem}

\begin{corollary}\label{cor:projectivedeformation}
Any Lagrangian fibration $f \colon X \to Y$ is deformation 
equivalent to a projective Lagrangian fibration via a 
locally trivial deformation.
\end{corollary}

\begin{proof}
Combine \ref{thm:semiampleLagr} and 
\cite[Cor.~6.11]{BL2022} because $\dim \Deflt(X, L) = 
h^{1,1}(X)-1 > 0$. 
Note that $h^{1,1}(X)\geq 2$, as $H^{1,1}(X)$ 
contains at least a K\"{a}hler class  
and the class of $L$, 
which is nef but not K\"ahler.
\end{proof}

The family $\pi \colon \mathcal{X}^{\mathrm{uni}} \to \Deflt(X, L)$ 
is a proper complex analytic family that is not algebraizable; 
cf.~\cite[Ex.~2.5.12]{Sernesi2006}. 
However, if $X$ is projective and polarized by an ample class 
$H$, the universal locally trivial family of 
$\langle H, L \rangle$-polarized deformations of $X$ can be 
algebraized to a projective morphism 
$\xi \colon \mathcal{X} \to S$ of quasi-projective 
varieties by Artin's approximation, see \cite[Thm.~1.6]{Artin1969} 
and \cite[Thms.~2.5.13 and 2.5.14]{Sernesi2006}. 
Recall that the deformation space of 
$\langle H, L \rangle$-polarized locally trivial 
deformation of $X$, denoted $\Deflt(X, \langle H, L \rangle)$, 
is locally isomorphic to 
$\PP(\langle H, L \rangle^{\perp}_{\CC}) \cap \bD$, 
thus smooth of dimension $h^{1,1}(X)-2$. 

\begin{corollary}\label{cor:projectivefamilywithfibr}
Suppose that $X$ is projective, polarized by an ample 
class $H$, and endowed with the Lagrangian fibration 
$f \colon X \to Y$.
Then there exist
\begin{enumerate} 
\item a quasi-projective variety $S$ of dimension 
$h^{1,1}(X)-2$ with a reference point $0 \in S$, 
\item a projective locally trivial family 
$\xi \colon \mathcal{X} \to S$ of $H$-polarized 
primitive symplectic varieties,  
\item a projective fibration 
$\cF \colon \mathcal{X} \to \mathcal{Z}$ over $S$ 
\end{enumerate}
with the following properties:
\begin{enumerate} 
\item $\cF|_{\mathcal{X}_{t}}$ is a Lagrangian fibration 
for all $t \in S$;
\item $\cF|_{\mathcal{X}_{0}}$ coincides with the 
original fibration $\pi$; and
\item $\xi$ is universal at each point of $S$ with 
respect to all $H$-polarized 
deformations of $X$ admitting a Lagrangian 
fibration deformation equivalent to $f$.
In particular, the classifying map from the 
germ of $S$ at $0$ to $\Deflt(X, \langle H, L \rangle)$ 
is an isomorphism.
\end{enumerate}
\end{corollary}

\begin{remark}
Although the statement of \ref{cor:projectivefamilywithfibr} 
is purely algebraic, its proof is not algebraic.
In fact, it is unclear how to deform the Lagrangian 
fibration without first showing the semiampleness of the 
corresponding nef line bundle on a very general 
non-projective deformation of $X$, see 
\S~\ref{sec:semiamplenoproj}.
For this reason, we work with the whole deformation 
space $\Deflt(X, L)$ and not just with the 
subspace $\Deflt(X, \langle H, L \rangle)$.
\end{remark}

\subsection{Deformation of Lagrangian tori} In this section, 
we show that a small deformation of $X$ in $\Deflt(X, L)$ 
contains a Lagrangian torus deformation equivalent to a 
smooth fiber of $f$.
We follow closely \cite[\S~2]{Lehn2016}.

Let $i \colon A \, \hookrightarrow X$ be the inclusion 
of a smooth fiber of $f$. Recall that $A$ is a Lagrangian 
torus lying in the regular locus of $X$. 
Let $\mathcal{D}(i)$ be a complex analytic germ of the 
relative Douady space 
of $\mathcal{X}^{\mathrm{uni}}$ over 
$\Deflt(X, L)$  near $[i \colon A \, \hookrightarrow X ]$, 
which is the complex analytic analogue of the Hilbert scheme, 
see \cite{Douady1966} or \cite[Ch.~VIII]{GPR1994}. 
There exists a commutative square
\[
\xymatrix{
\mathcal{A}^{\mathrm{uni}} \, \ar[r]^{\epsilon} \ar[d]_{u} &
\mathcal{X}^{\mathrm{uni}} \ar[d]^{\pi}\\
\Deflti \ar[r]^-{p} & \Deflt(X, L),
}
\]
where $u \colon \mathcal{A}^{\mathrm{uni}} \to \mathcal{D}(i)$ 
is the universal family, 
$\epsilon \colon \mathcal{A}^{\mathrm{uni}} 
\to \mathcal{X}^{\mathrm{uni}}$ the evaluation morphism, 
and $p \colon \Deflti \to \Deflt(X, L)$ the forgetful 
morphism $[\mathcal{A}^{\text{uni}}_{t}] \mapsto f(\mathcal{A}^{\text{uni}}_{t})$. 
The inclusion $i$ corresponds to the 
reference point $0 \in \Deflti$.  

\begin{proposition}\label{prop:deformtori} 
(1) The forgetful morphism $p$ is a submersion of 
smooth germs of complex spaces, and (2) 
$\epsilon_{t} \colon \mathcal{A}^{\text{uni}}_t \, 
\hookrightarrow \mathcal{X}^{\text{uni}}_{t}$ is the 
inclusion of a Lagrangian torus for any $t$ in a 
sufficiently small complex space 
representing the germ $\Deflti$. 
\end{proposition}
\begin{proof}
It follows from \Cref{prop:defismooth}, 
\Cref{lem:subforget}, and \Cref{lem:Lagrangiandef}.
\end{proof}

\begin{lemma}\label{prop:defismooth}
$\Deflti$ is smooth.
\end{lemma}
\begin{proof} The proof of \cite[Theorem 3.8]{Lehn2016} 
works in the singular context too with minor adjustments, 
as follows.   Let $\iota \colon \mathcal{A} \, 
\hookrightarrow \mathcal{X}$ be a locally trivial 
deformation of $i$ over the spectrum $S=\mathrm{Spec} R$ 
of an Artinian local $\mathbb{C}$-algebra $R$.
Define $T_{\mathcal{X}/S}\langle \mathcal{A}\rangle$ 
via the exact sequence of sheaves over $\mathcal{X}$
\begin{equation}\label{eq:relativetangentsequence}
0 \to T_{\mathcal{X}/S}\langle \mathcal{A}\rangle \to  
T_{\mathcal{X}/S} \to \iota_* N_{\mathcal{A}/\mathcal{X}} \to 0.
\end{equation}
To prove the smoothness of $\Deflti$, it suffices 
to show that 
$H^1(\mathcal{X}, T_{\mathcal{X}/S}\langle \mathcal{A}\rangle)$ 
is a free $R$-module for all infinitesimal deformations 
of $i$ over a local Artinian scheme,  
via Ran’s $T^1$-lifting principle; cf. \cite[\S~3.7]{Lehn2016} 
and references therein. 

Taking cohomology in \eqref{eq:relativetangentsequence}, 
we obtain the exact sequence
\begin{equation}\label{eq:defomcohom}
H^0(\mathcal{X}, T_{\mathcal{X}/S}) \to H^0(\mathcal{A}, 
N_{\mathcal{A}/\mathcal{X}}) \to H^1(\mathcal{X}, 
T_{\mathcal{X}/S}\langle \mathcal{A}\rangle) \to  H^1(\mathcal{X}, 
T_{\mathcal{X}/S}) \to H^1(\mathcal{A}, N_{\mathcal{A}/\mathcal{X}}).
\end{equation}
Let $j \colon \mathcal{X}^{\mathrm{reg}} \, 
\hookrightarrow \mathcal{X}$ be the inclusion of the 
regular locus. The non-degenerate pairing defined by the 
holomorphic symplectic form induces the following isomorphisms:
\begin{enumerate}
\item $T_{\mathcal{X}/S} 
\simeq j_* \Omega^1_{\mathcal{X}^{\mathrm{reg}}/S}$ 
by \cite[Lem.~4.5.(2)]{BL2022};
\item $N_{\mathcal{A}/\mathcal{X}} \simeq 
\Omega^{1}_{\mathcal{A}/\mathcal{X}} \simeq 
\mathcal{O}^{\otimes \dim A}_{\mathcal{A}}$, 
since $\mathcal{A}$ is Lagrangian in $\mathcal{X}$ 
by \cite[Lem.~3.4]{Lehn2016}.
\end{enumerate}
In particular, $H^0(\mathcal{X}, T_{\mathcal{X}/S})=
H^0(\mathcal{X}, j_* \Omega^1_{\mathcal{X}^{\mathrm{reg}}/S})=0$, 
and \eqref{eq:defomcohom} can be rewritten as follows:
\begin{equation}\label{eq:defomcohII}
0 \to H^0(\mathcal{A}, \Omega^{1}_{\mathcal{A}/\mathcal{X}}) \to
H^1(\mathcal{X}, T_{\mathcal{X}/S}\langle \mathcal{A}\rangle) \to 
H^1(\mathcal{X}, j_* \Omega^1_{\mathcal{X}^{\mathrm{reg}}/S}) \to 
H^1(\mathcal{A}, \Omega^{1}_{\mathcal{A}/\mathcal{X}}).
\end{equation}
Note that $H^1(\mathcal{X}, 
j_* \Omega^1_{\mathcal{X}^{\mathrm{reg}}/S})$, 
$H^0(\mathcal{A}, \Omega^{1}_{\mathcal{A}/\mathcal{X}})$ 
and $H^1(\mathcal{A}, \Omega^{1}_{\mathcal{A}/\mathcal{X}})$ 
are free $R$-modules by 
\cite[Lemma 2.4.(3)]{BL2021} and \cite[Thm.~5.5]{Deligne68}. 
The cokernel of the map 
$H^1(\mathcal{X}, j_* \Omega^1_{\mathcal{X}^{\mathrm{reg}}/S}) 
\to H^1(\mathcal{A}, \Omega^{1}_{\mathcal{A}/\mathcal{X}})$ 
is free; cf \cite[Proof of Thm.~4.17]{Lehn2015}. 
From \eqref{eq:defomcohII}, we deduce that 
$H^1(\mathcal{X}, T_{\mathcal{X}/S}\langle \mathcal{A}\rangle)$ 
is free, so $\Deflti$ is smooth.
\end{proof}

\begin{lemma}\label{lem:subforget} The forgetful morphism
 $p \colon \Deflti \to \Deflt(X, L)$ is a submersion.  
\end{lemma}
\begin{proof}
The differential $dp_{[i]} \colon T_{[i]}\Deflti \to T_0 \Deflt(X)$
at the point $[i] \in \Deflti$ can be identified with the
cohomological map $H^1(X, T_{X}\langle A\rangle) \to  
H^1(X, T_{X})$ by \eqref{eq:defomcohom}. 
It surjects onto $T_0 \Deflt(X, L)$, since
\begin{align*}
\mathrm{Im}(dp_{[i]}) & = \ker \{ H^{1}(X, \Omega^{[1]}_X) 
\to H^{1}(A, \Omega^1_{A})\} = 
\ker \{ H^{2}(X, \mathbb{C}) \to 
H^{2}(A, \mathbb{C})\} \cap H^{1}(X, \Omega^{[1]}_X)\\
& = L^{\perp}_{\mathbb{C}} \cap H^{1}(X, \Omega^{[1]}_X) = 
T_0 \Deflt(X, L).
\end{align*}
Here the identities follow from \eqref{eq:defomcohII}, the 
Lagrangianity of $A$, 
\Cref{eq:Lperprestr} and \eqref{eq:perioddeflt}. 
\end{proof}

\Cref{lem:Lagrangiandef} asserts that small 
deformations of $i$ are Lagrangian inclusions. 

\begin{lemma}\label{lem:Lagrangiandef} The universal family 
$u \colon \mathcal{A}^{\mathrm{uni}} \to \mathcal{D}(i)$ is a 
family of Lagrangian submanifolds.
\color{black}
\end{lemma}
\begin{proof}
Consider the inclusion $\mathcal{A}^{\mathrm{uni}} \, 
\hookrightarrow \mathcal{X}^{\mathrm{uni}} \times \Deflti$, 
and let $\sigma$ be a nonvanishing section of 
$\Omega^{[2]}_{\mathcal{X}^{\mathrm{uni}}\times \Deflti/\Deflti}$.
The locus $\DefltL(i)$ of Lagrangian inclusions is an analytic 
subset of $\Deflti$ cut out by the vanishing of the section
$\sigma|_{\mathcal{A}^{\mathrm{uni}}}$ of the locally free 
sheaf $u_*\Omega^{[2]}_{\mathcal{A}^{\mathrm{uni}}/\Deflti}$.
Infinitesimal truncations of the inclusion 
$\DefltL(i)\subseteq \Deflti$ are isomorphisms by 
\cite[Lem.~3.4]{Lehn2016} 
(cf.~also \cite[Lem.~3.9, Prop.~4.13]{Lehn2015}), 
so the formal completions of $\DefltL(i)$ and $\Deflti$ 
are isomorphic. Hence, $\DefltL(i) = \Deflti$.
\end{proof}

\begin{lemma}\label{eq:Lperprestr} 
$L^{\perp} = \ker \{ H^2(X, \ZZ) \to H^2(A, \ZZ)\}$. 
\end{lemma}
\begin{proof}
The proof by Matsushita \cite[Lem.~2.2]{Matsushita2016} 
in the smooth case uses the Hodge index theorem and Fujiki's 
relations, which hold in the singular cases too by 
\cite[Prop.~5.15]{BL2021}.
\end{proof}

\subsection{A semiampleness result for 
non-projective symplectic klt varieties}\label{sec:semiamplenoproj}
\begin{definition} Given a complex normal variety $X$ and $L$ a 
$\mathbb{Q}$-Cartier nef divisor, the numerical Iitaka dimension 
$\nu(X, L)$ is the maximal integer $m$ such that 
$L^{m} \cdot Y \neq 0$ for some $m$-dimensional subvariety $Y$ 
in $X$. Note that $\kappa(X, L) \leq \nu(X, L)$, where 
$\kappa(X, L)\coloneqq \dim 
\mathrm{Proj} \bigoplus_{k \geq 0} \newline H^0(X, L^{\otimes k})$ 
is the Iitaka dimension of $L$. The 
divisor $L$ is abundant if $\kappa(X, L)= \nu(X, L)$.
\end{definition} 
\begin{proposition}\label{prop:non-projective}
Let $(\mathcal{X}_{t}, \mathcal{L}_{t})$ be a very 
general fiber of 
$\pi_{U} \colon 
(\mathcal{X}^{\mathrm{uni}}_U, \mathcal{L}_U) \to U$. 
Then $\mathcal{L}_{t}$ is semiample.
\end{proposition}
\begin{proof} 
A nef and abundant line bundle on a klt K\"{a}hler 
variety with trivial canonical bundle is semiample 
by \cite[Thm.~4.8]{Fujino2011b}.
Since $\mathcal{L}_{t}$ is a deformation of the semiample 
line bundle $\mathcal{L}_{0}$,  $\mathcal{L}_{t}$ is nef 
by \cite[Prop.~1.4.14]{Laz} and 
$\kappa(\mathcal{X}_{t}, \mathcal{L}_{t}) 
\leq \nu(\mathcal{X}_{t}, \mathcal{L}_{t}) 
\leq \nu(\mathcal{X}_{0}, \mathcal{L}_{0})=n$, 
where $\dim \mathcal{X}_{t}=2n$. Therefore, it suffices 
to show that $\mathcal{L}_{t}$ is abundant, i.e., 
$\kappa(\mathcal{X}_{t}, \mathcal{L}_{t}) \geq n$. 

Recall that $\mathcal{X}_{t}$ is a non-projective klt 
K\"{a}hler variety. An algebraic reduction 
$h \colon \mathcal{X}_t \dashrightarrow B$ is a 
meromorphic map from $\mathcal{X}_t$ to a normal 
projective variety $B$ with the same field of global 
meromorphic functions, that is, 
$h^\ast \CC(B)=\CC(\mathcal{X}_t)$. Let $H$ 
be a general fiber of $h$.
By \cite[Cor.~2.5]{COP2010} and \cite[Prop.~4.5]{GLR2013} 
applied to a resolution of $\mathcal{X}_t$, 
we have the following dichotomy:
\begin{enumerate}
\item $\dim B = n$, and so 
$\kappa(\mathcal{X}_{t}, \mathcal{L}_{t})\geq n$, or 
\item $\dim B < n$ and $H$ cannot be covered by 
positive-dimensional projective varieties.
\end{enumerate}
The latter case cannot occur, since 
$\mathcal{X}_{t}$ contains a Lagrangian torus by 
\Cref{prop:deformtori}, which is projective by 
\cite[Prop.~2.1]{Campana2006}, and its deformations 
cut out on $H$ a covering family of positive-dimensional 
subvarieties, see for instance the proof of 
\cite[Thm.~4.1]{GLR2013}.
We conclude that the nef line bundle
$\mathcal{L}_{t}$ is abundant, and so semiample. 
\end{proof}

\begin{proof}[Proof of \Cref{thm:basechangelocallyfree} and \Cref{thm:semiampleLagr}]
Let $(D,0) \subset U$ be a very general small disk 
passing through the reference point, and $\mathcal{X}_{D}$ 
be the restriction of $\mathcal{X}^{\mathrm{uni}}_{U}$ over 
$D$. Observe that $\mathcal{L}_t$ is semiample on 
$\mathcal{X}_t$ for a very general $t \in D$ by 
\Cref{prop:non-projective}, hence for any 
$t \in D \setminus \{0\}$ by 
\cite[Lem.~2.5]{Matsushita}. Then 
$R^k (\pi_{U})_* \mathcal{L}_U|_{D}$ is locally free 
by applying \cite[Cor.~3.14]{Nakayama1987} to a simultaneous 
resolution of singularities of $\mathcal{X}_{D}$, 
see \cite[Lem.~4.9]{BL2022} or \cite[\S~2.24]{BGL2022}.

The function $t \to h^{k}(\mathcal{X}_t, \mathcal{L}_t)$ is 
then constant along $D$ and, as usual, uppersemicontinous 
on $U$, so constant up to shrinking $U$.
By cohomology and base change, $R^k \pi_* \mathcal{L}$ 
is now locally free, and the sections 
$(\pi_* \mathcal{L}^{\otimes k})_{0} \simeq H^0(X, L^{\otimes k})$, 
defining $\pi$, extend in a neighborhood of $X$ in 
$\mathcal{X}^{\mathrm{uni}}$ with no base point. 
See also \cite[Proof of Thm.~1.2 and Cor.~1.3]{Matsushita} 
for further details.
\end{proof}

\bibliographystyle{mrl}
\bibliography{HyperHM}

\begin{thebibliography}{100}

\bibitem{abasheva}
A.~Abasheva, \emph{Shafarevich--Tate groups of holomorphic Lagrangian fibrations II}, arXiv preprint, arXiv:2407.09178  (2024)\hskip -.1cm.

\bibitem{AR2021}
A.~Abasheva and V.~Rogov, \emph{Shafarevich--{T}ate groups of holomorphic {L}agrangian fibrations}, Math. Z. \textbf{311} (2025), no.~1,  Paper No. 4.

\bibitem{AOV}
D.~Abramovich, M.~Olsson, and A.~Vistoli, \emph{Tame stacks in positive characteristic}, Ann. Inst. Fourier (Grenoble) \textbf{58} (2008) 1057--1091.

\bibitem{AO2002}
D.~Abramovich and A.~Vistoli, \emph{Compactifying the space of stable maps}, J. Amer. Math. Soc. \textbf{15} (2002), no.~1,  27--75.

\bibitem{alexeev2002}
V.~Alexeev, \emph{Complete moduli in the presence of semiabelian group action}, Ann. of Math. (2)  (2002) 611--708.

\bibitem{alexeev2022root}
---{}---{}---, \emph{Root systems and hyperk\"ahler varieties}, arXiv preprint arXiv:2206.14070  (2022)\hskip -.1cm.

\bibitem{AE2023}
V.~Alexeev and P.~Engel, \emph{Compact moduli of K3 surfaces}, Ann. of Math. (2) \textbf{198} (2023), no.~2,  727--789.

\bibitem{AE2025}
---{}---{}---, \emph{On lattice-polarized {K}3 surfaces}, arXiv preprint, arXiv:2505.22557  (2025)\hskip -.1cm.

\bibitem{alexeev-nakamura}
V.~Alexeev and I.~Nakamura, \emph{{On Mumford's construction of degenerating abelian varieties}}, Tohoku Math. J. (2) \textbf{51} (1999), no.~3,  399 -- 420.

\bibitem{Amb04}
F.~Ambro, \emph{Shokurov's boundary property}, J. Differential Geom. \textbf{67} (2004), no.~2,  229--255.

\bibitem{Amb05}
---{}---{}---, \emph{The moduli {$b$}-divisor of an lc-trivial fibration}, Compos. Math. \textbf{141} (2005), no.~2,  385--403.

\bibitem{ACSS}
F.~{Ambro}, P.~{Cascini}, V.~V. {Shokurov}, and C.~{Spicer}, \emph{{Positivity of the Moduli Part}}, arXiv preprint, arXiv:2111.00423  (2021)\hskip -.1cm.

\bibitem{artin1969algebraic}
M.~Artin, \emph{Algebraic approximation of structures over complete local rings}, Publ. Math. Inst. Hautes \'Etudes Sci. \textbf{36} (1969) 23--58.

\bibitem{Artin1969}
---{}---{}---, \emph{Algebraization of formal moduli. {I}}, in Global {A}nalysis ({P}apers in {H}onor of {K}. {K}odaira), 21--71, Univ. Tokyo Press, Tokyo (1969).

\bibitem{artin1970algebraization}
---{}---{}---, \emph{Algebraization of formal moduli: II. Existence of modifications}, Ann. of Math. (2) \textbf{91} (1970), no.~1,  88--135.

\bibitem{babwild}
K.~Ascher, D.~Bejleri, H.~Blum, K.~DeVleming, G.~Inchiostro, Y.~Liu, and X.~Wang, \emph{Moduli of boundary polarized Calabi--Yau pairs}, arXiv preprint, arXiv:2307.06522  (2023)\hskip -.1cm.

\bibitem{mumford-amrt}
A.~Ash, D.~Mumford, M.~Rapoport, and Y.-S. Tai, Smooth compactifications of locally symmetric varieties, Cambridge Mathematical Library, Cambridge University Press, Cambridge, second edition (2010), ISBN 978-0-521-73955-9. With the collaboration of Peter Scholze.

\bibitem{BB66}
W.~L. Baily, Jr. and A.~Borel, \emph{Compactification of arithmetic quotients of bounded symmetric domains}, Ann. of Math. (2) \textbf{84} (1966) 442--528.

\bibitem{BGL2022}
B.~Bakker, H.~Guenancia, and C.~Lehn, \emph{Algebraic approximation and the decomposition theorem for {K}\"{a}hler {C}alabi--{Y}au varieties}, Invent. Math. \textbf{228} (2022), no.~3,  1255--1308.

\bibitem{BL2021}
B.~Bakker and C.~Lehn, \emph{A global {T}orelli theorem for singular symplectic varieties}, J. Eur. Math. Soc. (JEMS) \textbf{23} (2021), no.~3,  949--994.

\bibitem{BL2022}
---{}---{}---, \emph{The global moduli theory of symplectic varieties}, J. Reine Angew. Math. \textbf{790} (2022) 223--265.

\bibitem{barlet1989}
D.~Barlet and J.~Varouchas, \emph{Fonctions holomorphes sur l'espace des cycles}, Bull. Soc. Math. France \textbf{117} (1989), no.~3,  327--341.

\bibitem{barth2015compact}
W.~Barth, K.~Hulek, C.~Peters, and A.~van~de Ven, Compact Complex Surfaces, Ergebnisse der Mathematik und ihrer Grenzgebiete. 3. Folge / A Series of Modern Surveys in Mathematics, Springer Berlin Heidelberg (2015), ISBN 9783642577390.

\bibitem{B1983}
A.~Beauville, \emph{Vari\'{e}t\'{e}s {K}\"{a}hleriennes dont la premi\`ere classe de {C}hern est nulle}, J. Differential Geom. \textbf{18} (1983), no.~4,  755--782 (1984).

\bibitem{Bir12}
C.~Birkar, \emph{Existence of log canonical flips and a special {LMMP}}, Publ. Math. Inst. Hautes \'Etudes Sci. \textbf{115} (2012) 325--368.

\bibitem{BDCS}
C.~Birkar, G.~Di~Cerbo, and R.~Svaldi, \emph{Boundedness of elliptic Calabi--Yau varieties with a rational section}, J. Differential Geom. \textbf{128} (2024), no.~2,  463--519.

\bibitem{BZ16}
C.~Birkar and D.-Q. Zhang, \emph{Effectivity of {I}itaka fibrations and pluricanonical systems of polarized pairs}, Publ. Math. Inst. Hautes \'Etudes Sci. \textbf{123} (2016) 283--331.

\bibitem{birkenhake-lange}
C.~Birkenhake and H.~Lange, Complex abelian varieties, Vol. 302 of \emph{Grundlehren der mathematischen Wissenschaften [Fundamental Principles of Mathematical Sciences]}, Springer-Verlag, Berlin, second edition (2004), ISBN 3-540-20488-1.

\bibitem{BKV}
F.~{Bogomolov}, L.~{Kamenova}, and M.~{Verbitsky}, \emph{{Sections of Lagrangian fibrations on holomorphic symplectic manifolds}}, arXiv preprint, arXiv:2407.07877  (2024)\hskip -.1cm.

\bibitem{BSNWS}
S.~Boissi\`ere, M.~Nieper-Wi{\ss}kirchen, and A.~Sarti, \emph{Higher dimensional {E}nriques varieties and automorphisms of generalized {K}ummer varieties}, J. Math. Pures Appl. (9) \textbf{95} (2011), no.~5,  553--563.

\bibitem{borel_finite}
A.~Borel, \emph{Arithmetic properties of linear algebraic groups}, in Proc. I.C.M. Stockholm, 10--22 (1962).

\bibitem{Borel1969}
---{}---{}---, Introduction aux groupes arithm\'{e}tiques, Publications de l'Institut de Math\'{e}matique de l'Universit\'{e} de Strasbourg, XV. Actualit\'{e}s Scientifiques et Industrielles, No. 1341, Hermann, Paris (1969).

\bibitem{borel-ext}
---{}---{}---, \emph{Some metric properties of arithmetic quotients of symmetric spaces and an extension theorem}, J. Differential Geom. \textbf{6} (1972), no.~4,  543--560.

\bibitem{Braun2021}
L.~Braun, \emph{The local fundamental group of a {K}awamata log terminal singularity is finite}, Invent. Math. \textbf{226} (2021), no.~3,  845--896.

\bibitem{BF2024}
L.~{Braun} and F.~{Figueroa}, \emph{{Fundamental groups, coregularity, and low dimensional klt Calabi-Yau pairs}}, arXiv preprint, arXiv:2401.01315  (2024)\hskip -.1cm.

\bibitem{Campana1995}
F.~Campana, \emph{Fundamental group and positivity of cotangent bundles of compact {K}\"{a}hler manifolds}, J. Algebraic Geom. \textbf{4} (1995), no.~3,  487--502.

\bibitem{Campana2004}
---{}---{}---, \emph{Orbifoldes \`a premi\`ere classe de {C}hern nulle}, in The {F}ano {C}onference, 339--351, Univ. Torino, Turin (2004).

\bibitem{Campana2006}
---{}---{}---, \emph{Isotrivialit\'{e} de certaines familles k\"{a}hl\'{e}riennes de vari\'{e}t\'{e}s non projectives}, Math. Z. \textbf{252} (2006), no.~1,  147--156.

\bibitem{Campana2011}
---{}---{}---, \emph{Orbifoldes g\'{e}om\'{e}triques sp\'{e}ciales et classification bim\'{e}romorphe des vari\'{e}t\'{e}s k\"{a}hl\'{e}riennes compactes}, J. Inst. Math. Jussieu \textbf{10} (2011), no.~4,  809--934.

\bibitem{Campana2021}
---{}---{}---, \emph{The {B}ogomolov--{B}eauville--{Y}au decomposition for {KLT} projective varieties with trivial first {C}hern class---without tears}, Bull. Soc. Math. France \textbf{149} (2021), no.~1,  1--13.

\bibitem{COP2010}
F.~Campana, K.~Oguiso, and T.~Peternell, \emph{Non-algebraic hyperk\"{a}hler manifolds}, J. Differential Geom. \textbf{85} (2010), no.~3,  397--424.

\bibitem{carlson}
J.~Carlson, S.~M{\"u}ller-Stach, and C.~Peters, Period Mappings and Period Domains, Cambridge Studies in Advanced Mathematics, Cambridge University Press (2003), ISBN 9780521814669.

\bibitem{chandra}
A.~Chandra, A.~Constantin, C.~S. Fraser-Taliente, T.~R. Harvey, and A.~Lukas, \emph{Enumerating Calabi--Yau Manifolds: Placing Bounds on the Number of Diffeomorphism Classes in the Kreuzer-Skarke List}, Fortschritte der Physik \textbf{72} (2024), no.~5,  2300264.

\bibitem{charles}
F.~Charles, \emph{Birational boundedness for holomorphic symplectic varieties, Zarhin's trick for K3 surfaces, and the Tate conjecture}, Ann. of Math. (2)  (2016) 487--526.

\bibitem{nagata}
B.~Conrad, M.~Lieblich, and M.~Olsson, \emph{Nagata compactification for algebraic spaces}, J. Inst. Math. Jussieu \textbf{11} (2012), no.~4,  747--814.

\bibitem{debarre_book}
O.~Debarre, Tores et vari{\'e}t{\'e}s ab{\'e}liennes complexes, Soci{\'e}t{\'e} math{\'e}matique de France (1999).

\bibitem{Debarre2001}
---{}---{}---, Higher-dimensional algebraic geometry, Universitext, Springer-Verlag, New York (2001), ISBN 0-387-95227-6.

\bibitem{Deligne68}
P.~Deligne, \emph{Th\'{e}or\`eme de {L}efschetz et crit\`eres de d\'{e}g\'{e}n\'{e}rescence de suites spectrales}, Publ. Math. Inst. Hautes \'Etudes Sci.  (1968), no.~35,  259--278.

\bibitem{deligne1970}
---{}---{}---, Equations Differentielles a Points Singuliers Reguliers, Vol. 163 of \emph{Lecture Notes in Mathematics}, Springer-Verlag (1970).

\bibitem{deligne1974}
---{}---{}---, \emph{Th{\'e}orie de {H}odge: {III}}, Publ. Math. Inst. Hautes \'Etudes Sci. \textbf{44} (1974) 5--77.

\bibitem{deligne87}
---{}---{}---, \emph{Un théorème de finitude pour la monodromie}, Discrete groups in geometry and analysis, {Pap}. {Hon}. {G}. {D}. {Mostow} 60th {Birthday}, {Prog}. {Math}. 67, 1-19 (1987).

\bibitem{Gross92}
I.~Dolgachev and M.~Gross, \emph{Elliptic three-folds I: Ogg-Shafarevich theory}, arXiv preprint, alg-geom/9210009  (1992)\hskip -.1cm.

\bibitem{Douady1966}
A.~Douady, \emph{Le probl\`eme des modules pour les sous-espaces analytiques compacts d'un espace analytique donn\'{e}}, in Contemporary {P}roblems in {T}heory {A}nal. {F}unctions ({I}nternat. {C}onf., {E}revan, 1965) ({R}ussian), 141--143, Izdat. ``Nauka'', Moscow (1966).

\bibitem{Druel2018}
S.~Druel, \emph{A decomposition theorem for singular spaces with trivial canonical class of dimension at most five}, Invent. Math. \textbf{211} (2018), no.~1,  245--296.

\bibitem{DruelGuenancia}
S.~Druel and H.~Guenancia, \emph{A decomposition theorem for smoothable varieties with trivial canonical class}, J. \'{E}c. polytech. Math. \textbf{5} (2018) 117--147.

\bibitem{durfee}
A.~H. Durfee, \emph{Intersection homology Betti numbers}, Proceedings of the American Mathematical Society \textbf{123} (1995), no.~4,  989--993.

\bibitem{dutta}
Y.~Dutta, D.~Mattei, and E.~Shinder, \emph{Twists of intermediate Jacobian fibrations}, arXiv preprint, arXiv:2411.01953  (2024)\hskip -.1cm.

\bibitem{Elkik1981}
R.~Elkik, \emph{Rationalit\'{e} des singularit\'{e}s canoniques}, Invent. Math. \textbf{64} (1981), no.~1,  1--6.

\bibitem{faltings}
G.~Faltings, \emph{Endlichkeitssätze für abelsche Varietäten über Zahlkörpern.}, Invent. Math. \textbf{73} (1983) 349--366.

\bibitem{Fil20}
S.~Filipazzi, \emph{On a generalized canonical bundle formula and generalized adjunction}, Ann. Sc. Norm. Super. Pisa Cl. Sci. (5) \textbf{21} (2020) 1187--1221.

\bibitem{Fil24}
---{}---{}---, \emph{On the boundedness of $n$-folds with $\kappa(X)=n-1$}, Algebr. Geom. \textbf{11} (2024), no.~3,  318--345.

\bibitem{FHS2024}
S.~Filipazzi, C.~D. Hacon, and R.~Svaldi, \emph{Boundedness of elliptic Calabi–Yau threefolds}, J. Eur. Math. Soc.  (2024)\hskip -.1cm. In print.

\bibitem{FM20}
S.~Filipazzi and J.~Moraga, \emph{Strong {$(\delta,n)$}-complements for semi-stable morphisms}, Doc. Math. \textbf{25} (2020) 1953--1996.

\bibitem{FS20b}
S.~Filipazzi and R.~Svaldi, \emph{Invariance of plurigenera and boundedness for generalized pairs}, Mat. Contemp. \textbf{47} (2020) 114--150.

\bibitem{FS20}
---{}---{}---, \emph{On the connectedness principle and dual complexes for generalized pairs}, Forum Math. Sigma \textbf{11} (2023) Paper No. e33, 39.

\bibitem{friedman91}
R.~Friedman, \emph{On threefolds with trivial canonical bundle}, Complex geometry and Lie theory (Sundance, UT, 1989) \textbf{53} (1991) 103--134.

\bibitem{Fu2025}
L.~{Fu}, Z.~{Li}, T.~{Takamatsu}, and H.~{Zou}, \emph{{Finiteness of pointed families of symplectic varieties: a geometric Shafarevich conjecture}}, arXiv preprint, arXiv:2505.15295  (2025)\hskip -.1cm.

\bibitem{Fu-Menet}
L.~Fu and G.~Menet, \emph{On the {B}etti numbers of compact holomorphic symplectic orbifolds of dimension four}, Math. Z. \textbf{299} (2021), no. 1-2,  203--231.

\bibitem{Fuj03}
O.~Fujino, \emph{A canonical bundle formula for certain algebraic fiber spaces and its applications}, Nagoya Math. J. \textbf{172} (2003) 129--171.

\bibitem{Fuj09}
---{}---{}---, \emph{Effective base point free theorem for log canonical pairs---{K}oll\'ar type theorem}, Tohoku Math. J. (2) \textbf{61} (2009), no.~4,  475--481.

\bibitem{Fujino2011b}
---{}---{}---, \emph{On {K}awamata's theorem}, in Classification of algebraic varieties, EMS Ser. Congr. Rep., 305--315, Eur. Math. Soc., Z\"{u}rich (2011).

\bibitem{Fuj15}
---{}---{}---, \emph{Kodaira vanishing theorem for log-canonical and semi-log-canonical pairs}, Proc. Japan Acad. Ser. A Math. Sci. \textbf{91} (2015), no.~8,  112--117.

\bibitem{FG14}
O.~Fujino and Y.~Gongyo, \emph{On the moduli b-divisors of lc-trivial fibrations}, Ann. Inst. Fourier (Grenoble) \textbf{64} (2014), no.~4,  1721--1735.

\bibitem{FM00}
O.~Fujino and S.~Mori, \emph{A canonical bundle formula}, J. Differential Geom. \textbf{56} (2000), no.~1,  167--188.

\bibitem{Fuj86}
T.~Fujita, \emph{Zariski decomposition and canonical rings of elliptic threefolds}, J. Math. Soc. Japan \textbf{38} (1986), no.~1,  19--37.

\bibitem{FKL}
M.~Fulger, J.~Koll{\'a}r, and B.~Lehmann, \emph{Volume and {Hilbert} functions of {{\(\mathbb{R}\)}}-divisors}, Mich. Math. J. \textbf{65} (2016), no.~2,  371--387.

\bibitem{gachet}
C.~Gachet, \emph{Well-clipped cones behave themselves under all finite quotients, the cone conjecture under most}, arXiv preprint, arXiv:2504.01753  (2025)\hskip -.1cm.

\bibitem{GHS2003}
T.~Graber, J.~Harris, and J.~Starr, \emph{Families of rationally connected varieties}, J. Amer. Math. Soc. \textbf{16} (2003), no.~1,  57--67.

\bibitem{grassi91}
A.~Grassi, \emph{On minimal models of elliptic threefolds}, Math. Ann. \textbf{290} (1991), no.~2,  287--301.

\bibitem{grassi-wen}
A.~Grassi and D.~Wen, \emph{Higher dimensional elliptic fibrations and Zariski decompositions}, Commun. Contemp. Math. \textbf{24} (2022), no.~04,  2150024.

\bibitem{Grauert1962}
H.~Grauert, \emph{\"{U}ber {M}odifikationen und exzeptionelle analytische {M}engen}, Math. Ann. \textbf{146} (1962) 331--368.

\bibitem{GPR1994}
H.~Grauert, T.~Peternell, and R.~Remmert, editors, Several complex variables. {VII}, Vol.~74 of \emph{Encyclopaedia of Mathematical Sciences}, Springer-Verlag, Berlin (1994), ISBN 3-540-56259-1.

\bibitem{GrebGuenanciaKebekus}
D.~Greb, H.~Guenancia, and S.~Kebekus, \emph{Klt varieties with trivial canonical class: holonomy, differential forms, and fundamental groups}, Geom. Topol. \textbf{23} (2019), no.~4,  2051--2124.

\bibitem{GrebKebekusKovacsPeternell}
D.~Greb, S.~Kebekus, S.~J. Kov\'{a}cs, and T.~Peternell, \emph{Differential forms on log canonical spaces}, Publ. Math. Inst. Hautes \'{E}tudes Sci.  (2011), no. 114,  87--169.

\bibitem{GKP2016}
D.~Greb, S.~Kebekus, and T.~Peternell, \emph{Singular spaces with trivial canonical class}, in Minimal models and extremal rays ({K}yoto, 2011), Vol.~70 of \emph{Adv. Stud. Pure Math.}, 67--113, Math. Soc. Japan, [Tokyo] (2016).

\bibitem{GLR2013}
D.~Greb, C.~Lehn, and S.~Rollenske, \emph{Lagrangian fibrations on hyperk\"{a}hler manifolds---on a question of {B}eauville}, Ann. Sci. \'{E}c. Norm. Sup\'{e}r. (4) \textbf{46} (2013), no.~3,  375--403 (2013).

\bibitem{griffithsIII}
P.~A. Griffiths, \emph{Periods of integrals on algebraic manifolds, {III} (Some global differential-geometric properties of the period mapping)}, Publications Math{\'e}matiques de l'IH{\'E}S \textbf{38} (1970) 125--180.

\bibitem{Gross1994}
M.~Gross, \emph{A finiteness theorem for elliptic {C}alabi--{Y}au threefolds}, Duke Math. J. \textbf{74} (1994), no.~2,  271--299.

\bibitem{Gross97}
---{}---{}---, \emph{Elliptic three-folds II: Multiple fibres}, Transactions of the American Mathematical Society \textbf{349} (1997), no.~9,  3409--3468.

\bibitem{artin-comparison}
A.~Grothendieck and J.-L. Verdier, \emph{Th{\'e}orie des topos et cohomologie {\'e}tale des sch{\'e}mas}, Lecture Notes in Mathematics  (1972)\hskip -.1cm.

\bibitem{guan.betti}
D.~Guan, \emph{On the Betti numbers of irreducible compact hyperk{\"a}hler manifolds of complex dimension four}, Math. Res. Lett. \textbf{8} (2001), no.~5,  663--669.

\bibitem{Guenancia2016}
H.~Guenancia, \emph{Semistability of the tangent sheaf of singular varieties}, Algebr. Geom. \textbf{3} (2016), no.~5,  508--542.

\bibitem{HMX18}
C.~D. Hacon, J.~McKernan, and C.~Xu, \emph{Boundedness of moduli of varieties of general type}, J. Eur. Math. Soc. (JEMS) \textbf{20} (2018), no.~4,  865--901.

\bibitem{HX13}
C.~D. Hacon and C.~Xu, \emph{Existence of log canonical closures}, Invent. Math. \textbf{192} (2013), no.~1,  161--195.

\bibitem{HX15}
---{}---{}---, \emph{Boundedness of log {C}alabi--{Y}au pairs of {F}ano type}, Math. Res. Lett. \textbf{22} (2015), no.~6,  1699--1716.

\bibitem{HM18}
L.~H. Halle and J.~Nicaise, \emph{Motivic zeta functions of degenerating Calabi--Yau varieties}, Math. Ann. \textbf{370} (2018) 1277--1320.

\bibitem{HJ22}
J.~{Han} and C.~{Jiang}, \emph{{Birational boundedness of rationally connected log Calabi--Yau pairs with fixed index}}, Algebr. Geom. Phys. \textbf{1} (2024), no.~1,  59--79.

\bibitem{hartshorne80}
R.~Hartshorne, \emph{Stable reflexive sheaves}, Math. Ann. \textbf{254} (1980), no.~2,  121--176.

\bibitem{HH2020}
K.~Hashizume and Z.-Y. Hu, \emph{On minimal model theory for log abundant lc pairs}, J. Reine Angew. Math. \textbf{767} (2020) 109--159.

\bibitem{Hironaka1975}
H.~Hironaka, \emph{Flattening theorem in complex-analytic geometry}, Amer. J. Math. \textbf{97} (1975) 503--547.

\bibitem{HoringPeternell}
A.~H\"{o}ring and T.~Peternell, \emph{Algebraic integrability of foliations with numerically trivial canonical bundle}, Invent. Math. \textbf{216} (2019), no.~2,  395--419.

\bibitem{hunt}
B.~Hunt, \emph{A bound on the {E}uler number for certain {C}alabi-{Y}au {$3$}-folds}, J. Reine Angew. Math. \textbf{411} (1990) 137--170.

\bibitem{Huy03}
D.~Huybrechts, \emph{Finiteness results for compact hyperk\"{a}hler manifolds}, J. Reine Angew. Math. \textbf{558} (2003) 15--22.

\bibitem{huybrechts2016}
---{}---{}---, Lectures on {K}3 surfaces, Vol. 158 of \emph{Cambridge Studies in Advanced Mathematics}, Cambridge University Press, Cambridge (2016), ISBN 978-1-107-15304-2.

\bibitem{HO2009}
J.-M. Hwang and K.~Oguiso, \emph{Characteristic foliation on the discriminant hypersurface of a holomorphic Lagrangian fibration}, Amer. J. Math. \textbf{131} (2009), no.~4,  981--1007.

\bibitem{HO2011}
---{}---{}---, \emph{Multiple fibers of holomorphic Lagrangian fibrations}, Commun. Contemp. Math. \textbf{13} (2011), no.~02,  309--329.

\bibitem{ivash}
S.~Ivashkovich, \emph{The Hartogs-type extension theorem for meromorphic maps intor compact Kähler manifolds.}, Invent. Math. \textbf{109} (1992), no.~1,  47--54.

\bibitem{javlitt}
A.~Javanpeykar and D.~Litt, \emph{Integral points on algebraic subvarieties of period domains: from number fields to finitely generated fields}, Manuscripta Math. \textbf{173} (2024), no.~1,  23--44.

\bibitem{wati}
V.~Jejjala, W.~Taylor, and A.~Turner, \emph{{Identifying equivalent Calabi--Yau topologies: A discrete challenge from math and physics for machine learning}}, arXiv preprint, arXiv:2202.07590v1  (2022)\hskip -.1cm.

\bibitem{mj2025}
M.~J. Jeon, \emph{Resolution of indeterminacy of rational maps to proper tame stacks}, arXiv preprint, arXiv:2506.14969v1  (2025)\hskip -.1cm.

\bibitem{Kam2018}
L.~Kamenova, \emph{Survey of finiteness results for hyperk\"{a}hler manifolds}, in Phenomenological approach to algebraic geometry, Vol. 116 of \emph{Banach Center Publ.}, 77--86, Polish Acad. Sci. Inst. Math., Warsaw (2018).

\bibitem{Kamenova2024}
---{}---{}---, \emph{Finiteness of stable {L}agrangian fibrations}, S\~{a}o Paulo J. Math. Sci. \textbf{18} (2024), no.~2,  801--806.

\bibitem{KL2022}
L.~{Kamenova} and C.~{Lehn}, \emph{{Non-hyperbolicity of holomorphic symplectic varieties}}, arXiv preprint, arXiv:2212.11411  (2022)\hskip -.1cm.

\bibitem{Kaw98}
Y.~Kawamata, \emph{Subadjunction of log canonical divisors. {II}}, Amer. J. Math. \textbf{120} (1998), no.~5,  893--899.

\bibitem{Kaw08}
---{}---{}---, \emph{Flops connect minimal models}, Publ. Res. Inst. Math. Sci. \textbf{44} (2008), no.~2,  419--423.

\bibitem{KS2021}
S.~Kebekus and C.~Schnell, \emph{Extending holomorphic forms from the regular locus of a complex space to a resolution of singularities}, J. Amer. Math. Soc. \textbf{34} (2021), no.~2,  315--368.

\bibitem{kerr-pearlstein}
M.~Kerr and G.~Pearlstein, \emph{{Boundary components of Mumford–Tate domains}}, Duke Math. J. \textbf{165} (2016), no.~4,  661 -- 721.

\bibitem{Kim2024}
H.~{Kim}, \emph{{A Remark on Fujino's work on the canonical bundle formula via period maps}}, arXiv preprint, arXiv:2405.05489  (2024)\hskip -.1cm.

\bibitem{yoonjoo}
Y.-J. Kim, \emph{The N\'eron model of a higher-dimensional Lagrangian fibration}, arXiv preprint, arXiv:2410.21193  (2024)\hskip -.1cm.

\bibitem{KimLaza2020}
Y.-J. Kim and R.~Laza, \emph{A conjectural bound on the second {B}etti number for hyper-{K}\"{a}hler manifolds}, Bull. Soc. Math. France \textbf{148} (2020), no.~3,  467--480.

\bibitem{Kirschner}
T.~Kirschner, Period Mappings with Applications to Symplectic Complex Spaces, Lecture Notes in Mathematics, 2140, Springer International Publishing, Cham, 1st ed. 2015. edition (2015), ISBN 3-319-17521-1.

\bibitem{Kod1}
K.~Kodaira, \emph{On the structure of compact complex analytic surfaces. {II}}, Amer. J. Math. \textbf{88} (1966) 682--721.

\bibitem{Kod2}
---{}---{}---, \emph{On the structure of compact complex analytic surfaces. {III}}, Amer. J. Math. \textbf{90} (1968) 55--83.

\bibitem{Ko86:KollarHigherI}
J.~Koll\'{a}r, \emph{Higher direct images of dualizing sheaves. {I}}, Ann. of Math. (2) \textbf{123} (1986), no.~1,  11--42.

\bibitem{Kol86}
---{}---{}---, \emph{Higher direct images of dualizing sheaves. {II}}, Ann. of Math. (2) \textbf{124} (1986), no.~1,  171--202.

\bibitem{Kollar07}
---{}---{}---, \emph{Kodaira's canonical bundle formula and adjunction}, in Flips for 3-folds and 4-folds, Vol.~35 of \emph{Oxford Lecture Ser. Math. Appl.}, 134--162, Oxford Univ. Press, Oxford (2007).

\bibitem{Kollar2013}
---{}---{}---, Singularities of the minimal model program, Vol. 200 of \emph{Cambridge Tracts in Mathematics}, Cambridge University Press, Cambridge (2013).

\bibitem{kollar_new2}
---{}---{}---, \emph{Abelian fiber spaces and their Tate-Shafarevich twists with sections}, arXiv preprint, arXiv:2504.21705  (2025)\hskip -.1cm.

\bibitem{kollar_new1}
---{}---{}---, \emph{N\'eron models, minimal models, and birational group actions}, arXiv preprint, arXiv:2502.13800  (2025)\hskip -.1cm.

\bibitem{KL2009}
J.~Koll\'ar and M.~Larsen, \emph{Quotients of {C}alabi--{Y}au varieties}, in Algebra, arithmetic, and geometry: in honor of {Y}u. {I}. {M}anin. {V}ol. {II}, Vol. 270 of \emph{Progr. Math.}, 179--211, Birkh\"auser Boston, Boston, MA (2009), ISBN 978-0-8176-4746-9.

\bibitem{KLSV}
J.~Koll\'{a}r, R.~Laza, G.~Sacc\`a, and C.~Voisin, \emph{Remarks on degenerations of hyper-{K}\"{a}hler manifolds}, Ann. Inst. Fourier (Grenoble) \textbf{68} (2018), no.~7,  2837--2882.

\bibitem{KM92}
J.~Koll\'{a}r and S.~Mori, \emph{Classification of three-dimensional flips}, J. Amer. Math. Soc. \textbf{5} (1992), no.~3,  533--703.

\bibitem{KM98}
---{}---{}---, Birational geometry of algebraic varieties, Vol. 134 of \emph{Cambridge Tracts in Mathematics}, Cambridge University Press, Cambridge (1998), ISBN 0-521-63277-3. With the collaboration of C. H. Clemens and A. Corti, Translated from the 1998 Japanese original.

\bibitem{KSCYdata}
M.~Kreuzer and H.~Skarke, \emph{Calabi--Yau data}, online website, \url{http://hep.itp.tuwien.ac.at/~kreuzer/CY/}.

\bibitem{Kurnosov2016}
N.~Kurnosov, \emph{On an inequality for {B}etti numbers of hyper-{K}\"{a}hler manifolds of dimension six}, Mat. Zametki \textbf{99} (2016), no.~2,  309--313.

\bibitem{Lai2011}
C.-J. Lai, \emph{Varieties fibered by good minimal models}, Math. Ann. \textbf{350} (2011), no.~3,  533--547.

\bibitem{Laz}
R.~Lazarsfeld, Positivity in algebraic geometry. {I}, Vol.~48 of \emph{Ergebnisse der Mathematik und ihrer Grenzgebiete. 3. Folge. A Series of Modern Surveys in Mathematics [Results in Mathematics and Related Areas. 3rd Series. A Series of Modern Surveys in Mathematics]}, Springer-Verlag, Berlin (2004), ISBN 3-540-22533-1. Classical setting: line bundles and linear series.

\bibitem{LP20}
V.~Lazi\'c and T.~Peternell, \emph{On generalised abundance, {I}}, Publ. Res. Inst. Math. Sci. \textbf{56} (2020), no.~2,  353--389.

\bibitem{Lehn2015}
C.~Lehn, \emph{Normal crossing singularities and {H}odge theory over {A}rtin rings}, Asian J. Math. \textbf{19} (2015), no.~2,  235--250.

\bibitem{Lehn2016}
---{}---{}---, \emph{Deformations of {L}agrangian subvarieties of holomorphic symplectic manifolds}, Math. Res. Lett. \textbf{23} (2016), no.~2,  473--497.

\bibitem{LMP24}
C.~Lehn, G.~Mongardi, and G.~Pacienza, \emph{The {M}orrison-{K}awamata cone conjecture for singular symplectic varieties}, Selecta Math. (N.S.) \textbf{30} (2024), no.~4,  Paper No. 79, 36.

\bibitem{Li23}
Z.~{Li}, \emph{{On the relative Morrison--Kawamata cone conjecture (II)}}, arXiv preprint, arXiv:2309.04673  (2023)\hskip -.1cm.

\bibitem{li.KM}
Z.~Li and H.~Zhao, \emph{{On the relative Morrison-Kawamata cone conjecture}}, arXiv preprint, arXiv:2206.13701v5  (2022)\hskip -.1cm.

\bibitem{LLX2024}
Y.~{Liu}, Z.~{Liu}, and C.~{Xu}, \emph{{Irreducible symplectic varieties with a large second Betti number}}, arXiv preprint, arXiv:2410.01566  (2024)\hskip -.1cm.

\bibitem{Markman2010}
E.~Markman, \emph{Modular {G}alois covers associated to symplectic resolutions of singularities}, J. Reine Angew. Math. \textbf{644} (2010) 189--220.

\bibitem{markman2014}
---{}---{}---, \emph{Lagrangian fibrations of holomorphic-symplectic varieties of K 3 [n]-type}, in Algebraic and Complex Geometry: In Honour of Klaus Hulek's 60th Birthday, 241--283, Springer (2014).

\bibitem{MST20}
D.~Martinelli, S.~Schreieder, and L.~Tasin, \emph{On the number and boundedness of log minimal models of general type}, Ann. Sci. \'Ec. Norm. Sup\'er. (4) \textbf{53} (2020), no.~5,  1183--1207.

\bibitem{matsusaka}
T.~Matsusaka, \emph{Polarized varieties with a given Hilbert polynomial}, Amer. J. Math. \textbf{94} (1972), no.~4,  1027--1077.

\bibitem{matsusaka2}
---{}---{}---, \emph{On polarized normal varieties, {I}}, Nagoya Math. J. \textbf{104} (1986) 175--211.

\bibitem{Matsushita}
D.~Matsushita, \emph{On fibre space structures of a projective irreducible symplectic manifold}, Topology \textbf{38} (1999), no.~1,  79--83.

\bibitem{Matsushita2000}
---{}---{}---, \emph{Equidimensionality of {L}agrangian fibrations on holomorphic symplectic manifolds}, Math. Res. Lett. \textbf{7} (2000), no.~4,  389--391.

\bibitem{Matsushita2015}
---{}---{}---, \emph{On base manifolds of {L}agrangian fibrations}, Sci. China Math. \textbf{58} (2015), no.~3,  531--542.

\bibitem{Matsushita2016}
---{}---{}---, \emph{On deformations of {L}agrangian fibrations}, in K3 surfaces and their moduli, Vol. 315 of \emph{Progr. Math.}, 237--243, Birkh\"{a}user/Springer, [Cham] (2016).

\bibitem{meyer}
A.~Meyer, \emph{Mathematische Mittheilungen}, Vierteljahrschrift der Naturforschenden Gesellschaft in Zürich \textbf{29} (1884) 209–222.

\bibitem{MS24}
J.~{Moraga} and T.~{Stark}, \emph{{The geometric cone conjecture in relative dimension two}}, arXiv preprint, arXiv:2409.13068  (2024)\hskip -.1cm.

\bibitem{Morrison1984}
D.~Morrison, \emph{The {C}lemens--{S}chmid exact sequence and applications}, in Topics in transcendental algebraic geometry ({P}rinceton, {N}.{J}., 1981/1982), Vol. 106 of \emph{Ann. of Math. Stud.}, 101--119, Princeton Univ. Press, Princeton, NJ (1984).

\bibitem{Mum70}
D.~Mumford, Abelian varieties, Vol.~5 of \emph{Tata Institute of Fundamental Research Studies in Mathematics}, Tata Institute of Fundamental Research, Bombay; by Oxford University Press, London (1970).

\bibitem{mumford-construction}
---{}---{}---, \emph{An analytic construction of degenerating abelian varieties over complete rings}, Compos. Math. \textbf{24} (1972), no.~3,  239--272.

\bibitem{Mum77}
---{}---{}---, \emph{Hirzebruch's proportionality theorem in the noncompact case}, Invent. Math. \textbf{42} (1977) 239--272.

\bibitem{Nakayama1987}
N.~Nakayama, \emph{The lower semicontinuity of the plurigenera of complex varieties}, in Algebraic geometry, {S}endai, 1985, Vol.~10 of \emph{Adv. Stud. Pure Math.}, 551--590, North-Holland, Amsterdam (1987).

\bibitem{Nak04}
---{}---{}---, Zariski-decomposition and abundance, Vol.~14 of \emph{MSJ Memoirs}, Mathematical Society of Japan, Tokyo (2004), ISBN 4-931469-31-0.

\bibitem{namikawabook}
Y.~Namikawa, Toroidal compactification of Siegel spaces, Vol. 812, Springer (1980).

\bibitem{Namikawa:deformation}
---{}---{}---, \emph{Deformation theory of singular symplectic {$n$}-folds}, Math. Ann. \textbf{319} (2001), no.~3,  597--623.

\bibitem{Namikawa2006}
---{}---{}---, \emph{On deformations of $\mathbb{Q}$-factorial symplectic varieties}, J. Reine Angew. Math. 599  (2006) 97--110.

\bibitem{narasimhan}
R.~Narasimhan, Introduction to the theory of analytic spaces, Vol.~25, Springer (1966).

\bibitem{NX2016}
J.~Nicaise and C.~Xu, \emph{The essential skeleton of a degeneration of algebraic varieties}, Amer. J. Math. \textbf{138} (2016), no.~6,  1645--1667.

\bibitem{nikulin}
V.~V. Nikulin, \emph{Integral symmetric bilinear forms and some of their applications}, Mathematics of the USSR-Izvestiya \textbf{14} (1980), no.~1,  103.

\bibitem{Nori1983}
M.~V. Nori, \emph{Zariski's conjecture and related problems}, Ann. Sci. \'{E}cole Norm. Sup. (4) \textbf{16} (1983), no.~2,  305--344.

\bibitem{ohno}
K.~Ohno, \emph{The Euler Characteristic Formula for Logarithmic Minimal Degenerations of Surfaces with Kodaira Dimension Zero and its application to Calabi-Yau Threefolds with a pencil}, arXiv preprint, arXiv:0710.3641  (2007)\hskip -.1cm.

\bibitem{Olsson2007}
M.~Olsson, \emph{A boundedness theorem for {H}om-stacks}, Math. Res. Lett. \textbf{14} (2007), no.~6,  1009--1021.

\bibitem{OG10}
K.~G. O’Grady, \emph{Desingularized moduli spaces of sheaves on a K3}, J. Reine Angew. Math. \textbf{1999} (1999), no. 512,  49--117.

\bibitem{OG6}
---{}---{}---, \emph{A new six-dimensional irreducible symplectic variety}, J. Algebraic Geom. \textbf{12} (2003) 435--505.

\bibitem{PR2018}
A.~{Perego} and A.~{Rapagnetta}, \emph{{Irreducible symplectic varieties from moduli spaces of sheaves on K3 and Abelian surfaces}}, arXiv:1802.01182.v1  (2018)\hskip -.1cm.

\bibitem{peters}
C.~Peters, \emph{Rigidity for variations of Hodge structure and Arakelov-type finiteness theorems}, Compos. Math. \textbf{75} (1990), no.~1,  113--126.

\bibitem{pozzi}
A.~Pozzi, The Kuga--Satake Construction: A Modular Interpretation, Master's thesis, Concordia University (2013).

\bibitem{Prendergast12}
A.~Prendergast-Smith, \emph{The cone conjecture for abelian varieties}, J. Math. Sci. Univ. Tokyo \textbf{19} (2012), no.~2,  243--261.

\bibitem{procesi}
C.~Procesi, \emph{The invariant theory of n × n matrices}, Adv. Math. \textbf{19} (1976) 306--381.

\bibitem{Procesi2007}
---{}---{}---, Lie groups, Universitext, Springer, New York (2007), ISBN 978-0-387-26040-2; 0-387-26040-4. An approach through invariants and representations.

\bibitem{PS09}
{\relax Yu}.~G. Prokhorov and V.~V. Shokurov, \emph{Towards the second main theorem on complements}, J. Algebraic Geom. \textbf{18} (2009), no.~1,  151--199.

\bibitem{ray70}
M.~Raynaud, Faisceaux amples sur les sch\'emas en groupes et les espaces homog\`enes, Vol. 119 of \emph{Lecture Notes in Mathematics}, Springer-Verlag, Berlin-New York (1970).

\bibitem{reid}
M.~Reid, \emph{The moduli space of 3-folds with K= 0 may nevertheless be irreducible}, Math. Ann. \textbf{278} (1987) 329--334.

\bibitem{rizov}
J.~Rizov, \emph{Kuga--Satake abelian varieties of K3 surfaces in mixed characteristic}, J. Reine Angew. Math. \textbf{2010} (2010), no. 648,  13--67.

\bibitem{sacca}
G.~Sacc{\`a}, \emph{Compactifying Lagrangian fibrations}, arXiv:2411.06505  (2024)\hskip -.1cm.

\bibitem{Sawon2008}
J.~Sawon, \emph{On the discriminant locus of a {L}agrangian fibration}, Math. Ann. \textbf{341} (2008), no.~1,  201--221.

\bibitem{Sawon2016}
---{}---{}---, \emph{A finiteness theorem for {L}agrangian fibrations}, J. Algebraic Geom. \textbf{25} (2016), no.~3,  431--459.

\bibitem{Sawon2022}
---{}---{}---, \emph{A bound on the second {B}etti number of hyperk\"{a}hler manifolds of complex dimension six}, Eur. J. Math. \textbf{8} (2022), no.~3,  1196--1212.

\bibitem{schmid}
W.~Schmid, \emph{Variation of {H}odge Structure: The Singularities of the Period Mapping.}, Invent. Math. \textbf{22} (1973) 211--320.

\bibitem{schreieder2020kuga}
S.~Schreieder and A.~Soldatenkov, \emph{The Kuga--Satake construction under degeneration}, J. Inst. Math. Jussieu \textbf{19} (2020), no.~6,  2165--2182.

\bibitem{schwald2016}
M.~Schwald, \emph{Low degree Hodge theory for klt varieties}, arXiv preprint, arXiv:1612.01919  (2016)\hskip -.1cm.

\bibitem{Schwald2020}
---{}---{}---, \emph{Fujiki relations and fibrations of irreducible symplectic varieties}, \'{E}pijournal G\'{e}om. Alg\'{e}brique \textbf{4} (2020) Art. 7, 19.

\bibitem{Sernesi2006}
E.~Sernesi, Deformations of algebraic schemes, Vol. 334 of \emph{Grundlehren der mathematischen Wissenschaften [Fundamental Principles of Mathematical Sciences]}, Springer-Verlag, Berlin (2006), ISBN 978-3-540-30608-5; 3-540-30608-0.

\bibitem{Shin2019}
M.~Shin, Computations of the cohomological Brauer group of some algebraic stacks, Ph.D. thesis, University of California, Berkeley (2019).

\bibitem{Sko}
A.~Skorobogatov, Torsors and rational points, Vol. 144 of \emph{Cambridge Tracts in Mathematics}, Cambridge University Press, Cambridge (2001), ISBN 0-521-80237-7.

\bibitem{SV}
A.~Soldatenkov and M.~Verbitsky, \emph{The Moser isotopy for holomorphic symplectic and C-symplectic structures}, Annales de l'Institut Fourier  (2025) 1--18.

\bibitem{Steenbrink1980}
J.~H.~M. Steenbrink, \emph{Cohomologically insignificant degenerations}, Compos. Math. \textbf{42} (1980/81), no.~3,  315--320.

\bibitem{voisin}
C.~Voisin, Th{\'e}orie de Hodge et g{\'e}om{\'e}trie alg{\'e}brique complexe, Collection SMF, Soci{\'e}t{\'e} Math{\'e}matique de France (2002), ISBN 9782856291290.

\bibitem{J2022}
J.~Wang, \emph{Structure of projective varieties with nef anticanonical divisor: the case of log terminal singularities}, Math. Ann. \textbf{384} (2022), no. 1-2,  47--100.

\bibitem{wells1974comparison}
R.~Wells, \emph{Comparison of de Rham and Dolbeault cohomology for proper surjective mappings}, Pac. J. Math. \textbf{53} (1974), no.~1,  281--300.

\bibitem{Wilson2021}
P.~M.~H. {Wilson}, \emph{Boundedness questions for {C}alabi--{Y}au threefolds}, J. Algebraic Geom. \textbf{30} (2021), no.~4,  631--684.

\bibitem{Wilson2022}
---{}---{}---, \emph{{The topology of Calabi--Yau threefolds with Picard number three}}, arXiv preprint, arXiv:2202.05202  (2022)\hskip -.1cm.

\bibitem{zarhin83}
{\relax Yu}.~G. Zarhin, \emph{Hodge groups of K3 surfaces.}, J. Reine Angew. Math. \textbf{341} (1983) 193--220.

\bibitem{zarhin85}
---{}---{}---, \emph{A finiteness theorem for unpolarized Abelian varieties over number fields with prescribed places of bad reduction}, Invent. Math. \textbf{79} (1985), no.~2,  309--321.

\bibitem{zarhin}
---{}---{}---, \emph{Abelian varieties, quaternion trick and endomorphisms}, in International Conference on Birational Geometry, Kaehler--Einstein Metrics and Degenerations, 857--864, Springer (2019).

\end{thebibliography}

\end{document}